\documentclass[11pt, leqno]{amsart}

\usepackage[margin=1in]{geometry}
\usepackage{graphicx,xcolor}
\usepackage{amssymb,amsmath,amsthm,amscd}
\usepackage{mathrsfs}
\usepackage{enumerate}
\usepackage{upgreek}
\usepackage{lscape}
\usepackage{booktabs}
\usepackage[normalem]{ulem}
\usepackage[all]{xy}
\usepackage{verbatim}
\usepackage{tikz}
\usepackage{xcolor}

\tikzset{dynkdot/.style={circle,draw,scale=.38}}
\usepackage[centertableaux]{ytableau}

\usepackage[colorlinks=true,pdfstartview=FitV,linkcolor=blue,citecolor=blue,urlcolor=blue]{hyperref}

\newtheorem{theorem}{Theorem}[section]
\newtheorem{proposition}[theorem]{Proposition}
\newtheorem{corollary}[theorem]{Corollary}
\newtheorem*{corollary*}{Corollary}
\newtheorem*{corollaryA}{Corollary A}
\newtheorem*{convention}{Convention}
\newtheorem{conjecture}[theorem]{Conjecture}
\newtheorem{problem}[theorem]{Problem}

\newtheorem{lemma}[theorem]{Lemma}

\theoremstyle{definition}
\newtheorem{definition}[theorem]{Definition}
\newtheorem{example}[theorem]{Example}
\newtheorem{remark}[theorem]{Remark}

\numberwithin{equation}{section}

\newcommand{\stens}{\mathop{\mbox{\normalsize$\bigotimes$}}\limits}
\newcommand{\Sk}{\bbS_\bsk}
\newcommand{\Sc}{\bbS_\bsc}

\newcommand{\tbsk}{\widetilde{\bsk}}
\newcommand{\Stk}{\bbS_{\tbsk}}

\newlength{\mylength}
\newcommand{\de}{\mathfrak{d}}
\newcommand{\g}{\mathfrak{g}}
\newcommand{\h}{\mathfrak{h}}
\newcommand{\ko}{\mathbf{k}}
\newcommand{\wvee}{\widetilde{\vee}}
\newcommand{\iso}{\simeq}
\newcommand{\tens}{\mathop\otimes}
\newcommand{\scb}{\scalebox}
\newcommand{\akete}[1][0ex]{\rule[{#1}]{0ex}{1ex}}
\newcommand{\dtens}{\mathop{\mbox{\scb{1.1}{\akete[-.6ex]\normalsize$\bigotimes$}}}\limits}
\newcommand{\seteq}{\mathbin{:=}}
\newcommand{\wl}{P_0}
\newcommand{\rl}{Q_0}
\newcommand{\cmA}{A}

\newcommand{\lan}{\langle}
\newcommand{\ran}{\rangle}

\newcommand{\la}{\lambda}
\newcommand{\redez}{{\widetilde{w}_0}}

\newcommand{\abs}[1]{\left\lvert #1 \right\rvert}

\newcommand{\cl}{\operatorname{cl}}

\newcommand{\Hom}{\operatorname{Hom}}

\newcommand{\gmod}{\mbox{-$\operatorname{gmod}$}}
\newcommand{\Image}{\operatorname{Im}}

\newcommand{\hd}{{\operatorname{hd}}}
\newcommand{\soc}{{\operatorname{soc}}}
\newcommand{\dist}{\operatorname{dist}}
\newcommand{\head}{\operatorname{hd}}

\newcommand{\univ}{\mathrm{univ}}
\newcommand{\norm}{\mathrm{norm}}
\newcommand{\ren}{\mathrm{ren}}
\newcommand{\wt}{\mathrm{wt}}
\newcommand{\rch}{\mathrm{rch}}

\newcommand{\clr}{\mathrm{clr}}
\newcommand{\exrch}{\mathrm{exrch}}
\newcommand{\fin}{\mathrm{fin}}
\newcommand{\gfin}{{\g_{\fin}}}
\newcommand{\Rnorm}{R^\norm}
\newcommand{\Runiv}{R^\univ}
\newcommand{\Rren}{R^\ren}
\newcommand{\rmat}[1]{{\mathbf r}_{\mspace{-2mu}\raisebox{-.5ex}{${\scriptstyle{#1}}$}}}

\newcommand{\Vkm}[1]{ \ttV \big( \hspace{0.2ex} \scalebox{.8}{\boxed{#1}} \hspace{0.2ex} \big)  }
\newcommand{\Wkm}[1]{ \ttW \big( \hspace{0.2ex} \scalebox{.8}{\boxed{#1}} \hspace{0.2ex} \big)  }

\newcommand{\PS}[1]{ {}_s\hspace{-.5ex}\left[ #1 \right] }
\newcommand{\PSN}[1]{ {}_s\hspace{-.5ex}\left[ #1 \right]_{(n+k)} }

\newcommand{\PNZ}[2]{ {}_s\hspace{-.5ex}\left[ #1 \right]_{(#2)} }
\newcommand{\PA}[1]{ \left\langle #1 \right] }

\newcommand{\PSF}[4]{ [#1][#2][#3][#4] }
\newcommand{\PPA}[1]{ {}^{(2)}\hspace{-.5ex}[\hspace{-.5ex}\langle #1 ]\hspace{-.5ex}] }
\newcommand{\PPAP}[1]{ {}^{(2)}\hspace{-.5ex}[\hspace{-.5ex}[ #1 ]\hspace{-.5ex}]'}
\newcommand{\CE}[1]{ {}^{(2)}\hspace{-.5ex}\left[#1 \right]}
\newcommand{\CO}[1]{ {}^{(2)}\hspace{-.5ex}\left\langle #1 \right]}

\newcommand{\mqs}{(-q^2)}
\newcommand{\al}{\alpha}
\newcommand{\be}{\beta}
\newcommand{\ga}{\gamma}

\newcommand{\Z}{\mathbb{Z}}

\newcommand{\ZZ}{\mathbb{Z}}
\newcommand{\Q}{\mathbb{Q}}

\newcommand{\CC}{\mathbb{C}}
\newcommand{\C}{\mathbb{C}}

\newcommand{\Ca}{\mathscr{C}}

\newcommand{\mC}{\mathscr{C}}

\newcommand{\twz}{{\widetilde{w}_0}}
\newcommand{\um}{\underline{m}}

\newcommand{\us}{\underline{s}}
\newcommand{\up}{\underline{p}}
\newcommand{\tb}{\mathtt{b}}

\newcommand{\conv}{\mathop{\mathbin{\mbox{\large $\circ$}}}}
\newcommand{\dct}[1]{ \underset{#1}{\conv} }
\newcommand{\uqpg}{\calU_q(\g)}
\newcommand{\gaff}{\g}
\newcommand{\uqaff}{U_q(\gaff)}

\newcommand{\qs}{q_{\mathfrak{s}}}
\newcommand{\qt}{q_{\mathfrak{t}}}

\newcommand{\bon}{\overline{1}}
\newcommand{\btw}{\overline{2}}
\newcommand{\bth}{\overline{3}}
\newcommand{\bnm}{\overline{n-1}}
\newcommand{\bn}{\overline{n}}

\newcommand{\zero}{\mathbf{0}}
\newcommand{\Zero}{\mathrm{zero}}
\newcommand{\tabcol}[1]{[#1]}
\newcommand{\crystal}{\mathcal{B}}
\newcommand{\ke}{\widetilde{e}}
\newcommand{\kf}{\widetilde{f}}

\newcommand{\hconv}{\mathbin{\scalebox{.9}{$\nabla$}}}
\newcommand{\sconv}{\mathbin{\scalebox{.9}{$\Delta$}}}

\newcommand{\diagramone}[6]{
\xymatrix@R=6.5ex@C=9ex{
\Vkm{#1} \otimes \Vkm{#2}_{(-q)^{#3}z} \otimes \Vkm{#4}_{(-q)^{#5}z}
\ar[rr]^{\qquad \qquad \quad \Vkm{#1}\otimes -} \ar[d]_{\Runiv_{#1,#2}((-q)^{#3}z) \otimes \Vkm{#4}_{(-q)^{#5}z}} &&
\Vkm{#1} \otimes \Vkm{#6}_{z} \ar[dd]_{\Runiv_{#1,#6}(z)} \\
\Vkm{#2}_{(-q)^{#3}z}  \otimes \Vkm{#1} \otimes \Vkm{#4}_{(-q)^{#5}z}
\ar[d]_{\Vkm{#2}_{(-q)^{#3}z} \otimes \Runiv_{#1,#4}((-q)^{#5}z)} &\\
\Vkm{#2}_{(-q)^{#3}z} \otimes \Vkm{#4}_{(-q)^{#5}z}  \otimes \Vkm{#1}\ar[rr] && \Vkm{#6}_{z} \otimes \Vkm{#1}
}}

\newenvironment{myequation}
{\relax\setlength{\arraycolsep}{1pt}\begin{eqnarray}}
{\end{eqnarray}}

\newcommand{\eq}{\begin{myequation}}
\newcommand{\eneq}{\end{myequation}}

\definecolor{UQpurple}{RGB}{73,7,94}
\definecolor{UQgold}{RGB}{189,161,78}
\definecolor{UQgreen}{RGB}{140,184,0}
\definecolor{OCUenji}{RGB}{153,0,51} 
\definecolor{OCUsapphire}{RGB}{0,51,102} 
\definecolor{HUgreen}{RGB}{33,98,26} 

\definecolor{darkred}{rgb}{0.7,0,0} 
\newcommand{\defn}[1]{{\color{darkred}\emph{#1}}} 

\usepackage[colorinlistoftodos]{todonotes}

\newcommand{\chq}{\check{q}}

\newcommand{\td}{\widetilde{d}}

\newcommand{\tk}{\widetilde{k}}

\newcommand{\tp}{\widetilde{p}}

\newcommand{\tw}{\widetilde{w}}

\newcommand{\hI}{\widehat{I}}

\newcommand{\uL}{\underline{L}}
\newcommand{\uM}{\underline{M}}

\newcommand{\uw}{\underline{w}}

\newcommand{\frakR}{\mathfrak{R}}

\newcommand{\frakc}{\mathfrak{c}}

\newcommand{\sff}{\mathsf{f}}

\newcommand{\sfs}{\mathsf{s}}

\newcommand{\bbB}{\mathbb{B}}

\newcommand{\bbS}{\mathbb{S}}

\newcommand{\bsc}{{\boldsymbol{c}}}

\newcommand{\bsk}{{\boldsymbol{k}}}

\newcommand{\bsu}{{\boldsymbol{u}}}

\newcommand{\bla}{{\boldsymbol{\lambda}}}
\newcommand{\bmu}{{\boldsymbol{\mu}}}
\newcommand{\bnu}{{\boldsymbol{\nu}}}

\newcommand{\bfb}{\mathbf{b}}
\newcommand{\bfc}{\mathbf{c}}

\newcommand{\bfg}{\mathbf{g}}

\newcommand{\bfk}{\mathbf{k}}

\newcommand{\bfr}{\mathbf{r}}

\newcommand{\calC}{\mathcal{C}}

\newcommand{\calF}{\mathcal{F}}
\newcommand{\calG}{\mathcal{G}}

\newcommand{\calM}{\mathcal{M}}
\newcommand{\calN}{\mathcal{N}}

\newcommand{\calP}{\mathcal{P}}
\newcommand{\calQ}{\mathcal{Q}}
\newcommand{\calR}{\mathcal{R}}
\newcommand{\calS}{\mathcal{S}}

\newcommand{\calU}{\mathcal{U}}
\newcommand{\calV}{\mathcal{V}}

\newcommand{\calY}{\mathcal{Y}}

\newcommand{\scrC}{\mathscr{C}}
\newcommand{\scrD}{\mathscr{D}}

\newcommand{\ttP}{\mathtt{P}}

\newcommand{\ttR}{\mathtt{R}}

\newcommand{\ttV}{\mathtt{V}}
\newcommand{\ttW}{\mathtt{W}}

\newcommand{\ttb}{\mathtt{b}}

\newcommand{\rmQ}{\mathrm{Q}}

\newcommand{\Ang}[1]{  \lan #1 \ran  }
\newcommand{\dpar}[1]{  ( \hspace{-0.4ex} ( #1 ) \hspace{-0.4ex} )  }
\newcommand{\dbracket}[1]{  [ \hspace{-0.2ex} [ #1 ] \hspace{-0.2ex} ]  }

\newcommand{\UpLa}{\Uplambda}
\newcommand{\upal}{\upalpha}
\newcommand{\upla}{\uplambda}

\newcommand{\ee}{\end{enumerate}}
\newcommand{\bitem}{\begin{itemize}}
\newcommand{\eitem}{\end{itemize}}
\newcommand{\ben}{\begin{enumerate}[{\rm (1)}]}
\newcommand{\bnum}{\begin{enumerate}[{\rm (i)}]}
\newcommand{\bnump}{\begin{enumerate}[{\rm (i)$'$}]}
\newcommand{\bna}{\begin{enumerate}[{\rm (a)}]}
\newcommand{\bnap}{\begin{enumerate}[{\rm (a$'$)}]}
\newcommand{\bnA}{\begin{enumerate}[{\rm (A)}]}
\newcommand{\bc}{\begin{cases}}
\newcommand{\ec}{\end{cases}}
\newcommand{\ba}{\begin{array}}
\newcommand{\ea}{\end{array}}
\newcommand{\noi}{\noindent}

\newcommand{\mnoi}{\medskip \noindent}

\newcommand{\Cg}{\scrC_\g}
\newcommand{\qtq}[1][{and}]{\quad\text{{#1}}\quad}

\newcommand{\KR}{\mathtt{KR}}
\newcommand{\La}{\Lambda}
\newcommand{\Li}{\Lambda^\infty}
\newcommand{\Deg}{\operatorname{Deg}}
\newcommand{\vph}{\varphi}
\newcommand{\im}{\imath}
\newcommand{\jm}{\jmath}
\newcommand{\Dynkin}{\triangle}  
\newcommand{\bDynkin}{\blacktriangle}

\newcommand{\oim}{{\overline{\im}}}
\newcommand{\ojm}{{\overline{\jm}}}

\newcommand{\Qdatum}{(\bDynkin,\sigma,\xi)} 
\newcommand{\uii}{{\underline{\boldsymbol{\im}}}}

\newcommand{\hbDynkin}{{\widehat{\bDynkin}}}

\newcommand{\tauarrow}{\ar@{.>}[rr]^{  \tau_\calQ}}
\newcommand{\tauarrows}{\ar@{.>}[rrrr]^{ \tau^2_\calQ}}
\newcommand{\tauarrowq}{\ar@{.>}[rrrrrr]^{ \tau^3_\calQ}}
\newcommand{\seq}[1]{ \boldsymbol{(} {#1} \boldsymbol{)}   }
\newcommand{\exponent}[1]{ \boldsymbol{\{} {#1} \boldsymbol{\}}   }
\newcommand{\range}[1]{ \boldsymbol{[  } {#1} \boldsymbol{ ]}   }
\newcommand{\drange}[1]{  [  \hspace{-.3ex} [ {#1} ] \hspace{-.3ex} ]  }

\newcommand{\ceil}[1]{ \lceil #1 \rceil   }
\newcommand{\floor}[1]{ \lfloor #1 \rfloor   }

\newcommand{\sR}[1]{s_{\mspace{-2mu}\raisebox{-.6ex}{${\scriptstyle{#1}}$}}}
\newcommand{\pR}[1]{\uppsi_{\mspace{-2mu}\raisebox{-.6ex}{${\scriptstyle{#1}}$}}}

\newcommand{\isoto}[1][]{\mathop{\xrightarrow%
[{\raisebox{.3ex}[0ex][.3ex]{$\scriptstyle{#1}$}}]%
{{\raisebox{-.6ex}[0ex][-.6ex]{$\mspace{2mu}\sim\mspace{2mu}$}}}}}

\title[Denominators, higher Dorey's rules, and generalized T-systems] {Denominators of $R$-matrices, higher Dorey's rules and a generalization of T-systems for quantum affine algebras}

\author[S.-j. Oh]{Se-jin Oh}
\thanks{ The research of S.-j.\ Oh was supported by the National Research Foundation of
	Korea (NRF) Grant funded by the Korea government(MSIT) (NRF-2022R1A2C1004045).}
\address[S.-j. Oh]{ Department of Mathematics, Sungkyunkwan University, Suwon, South Korea}
\email[S.-j. Oh]{sejin092@gmail.com}
\urladdr{https://sites.google.com/site/mathsejinoh/}

\author[T.~Scrimshaw]{Travis Scrimshaw}
\thanks{T.S.~was partially supported by Grant-in-Aid for JSPS Fellows 21F51028 and for Scientific Research for Early-Career Scientists 23K12983, as well as by the Australian Research Council DP170102648.}
\address[T. Scrimshaw]{Department of Mathematics, Hokkaido University, 5 Ch\=ome Kita 8 J\=onishi, Kita Ward, Sapporo, Hokkaid\=o 060-0808}
\email{tcscrims@gmail.com}
\urladdr{https://tscrim.github.io/}


\date{\today}

\begin{document}

\begin{abstract}
We construct a higher-level analogue of Dorey’s rule, which describe certain surjective morphisms between Kirillov--Reshetikhin (KR) modules over quantum affine algebras.
Building on this, we establish a generalized T-system of short exact sequences and prove the denominator formula between KR modules in all nonexceptional types, except with only mild ambiguities persisting in type $C_n^{(1)}$.
As a consequence, we can completely classify when a tensor product of KR modules is simple.
These results have further applications to Schur positivity statements, quiver Hecke algebras, and the recently introduced $\de$-invariants in monoidal categories over quantum affine algebras and quiver Hecke algebras.
\end{abstract}

\maketitle
\tableofcontents

\section*{Introduction}

One of the fundamental methods for studying (quantum) integrable systems is the use of the Yang--Baxter (or star-triangle) equation, which originally arose as a consistency condition in the construction of solutions to the Bethe ansatz~\cite{Bethe31} (see also~\cite{Baxter89}).
In 1985, Drinfel'd and Jimbo independently introduced the notion of a quantum group $U_q(\g)$ associated with a Lie algebra $\g$~\cite{Drinfeld85,Jimbo85}, motivated by aspects of the Yang--Baxter equation.
These quantum groups can be viewed as $q$-deformations of the universal enveloping algebra $U(\g)$ of $\g$ and have been extensively studied due to their connections with numerous areas of mathematics and physics, including statistical mechanics, algebraic geometry, dynamical systems, and number theory.

Here, we focus on the quantum affine algebras $\uqpg$ (without derivation), which are associated with the Lie algebra $\g' = [\g, \g]$, where $\g$ is an affine (Kac--Moody) Lie algebra.
The category $\Ca_{\g}$ of finite dimensional $\uqpg$-modules plays a central role in the study of quantum integrable systems such as the XXZ Heisenberg spin chain.
The simple modules in $\Ca_{\g}$ have been completely parameterized by Chari and Pressley~\cite{CP91,CP95A}, building on Drinfel'd's work~\cite{Drinfeld87} on the corresponding Yangian (a degeneration of $\uqpg$; see, e.g.,~\cite{Faddeev95}). This parameterization uses a tuple of (monic) polynomials $(\mathcal{P}_i \mid i \in I_0)$, called Drinfel'd polynomials, where $I_0$ denotes the set of nodes of the Dynkin diagram of the associated finite dimensional Lie algebra $\g_0$.
By factoring the Drinfel'd polynomials over $\CC$, the classification of simple modules reduces to the roots $a$ (counted with multiplicity) of $\mathcal{P}_i$, which are encoded as monomials in the variables $Y_{i,a}$. More precisely, Chari and Pressley showed that every simple module is a subquotient of a tensor product of fundamental modules $\Vkm{i}_a$ corresponding to each $Y_{i,a}$.
There exist numerous explicit constructions of $\Vkm{i}_a$, such as via evaluation modules of the projection of a level-zero extremal weight module~\cite{Kashiwara02}, or as minimal affinizations of $V(\Lambda_i)$~\cite{Chari95,CP95II,CP96III,Her07}, the simple highest weight $U_q(\g_0)$-module with fundamental weight $\Lambda_i$. 
Consequently, understanding the finite dimensional simple modules reduces to understanding tensor products of fundamental modules.

Let us step back slightly and look at the category of finite dimensional $U_q(\g_0)$-modules. 
This is a semisimple category, which affords us a simple and straightforward method to construct the (highest weight) simple modules. 
Consider $\lambda = \sum_{i \in I_0} m_i \Lambda_i$, and let
\[
T_{\lambda} := \bigotimes_{i \in I_0} V(\Lambda_i)^{\otimes m_i}.
\] 
Being a semisimple category means we have
\[
T_{\lambda} \iso \bigoplus_{\mu \in P^+} V(\mu)^{\oplus n_{\mu}}
\] 
with $P^+$ denotes the set of dominant integral weights.
An explicit construction of each $V(\Lambda_i)$ is known (see, e.g.,~\cite{FH91}), which shows that its $\Lambda_i$-weight space is one dimensional.
Consequently, the highest weight simple module $V(\lambda) \subseteq T_{\lambda}$ appears with multiplicity one (and its $\lambda$-weight space is one dimensional).
Hence, we can realize $V(\lambda) = U_q(\g_0) v_{\lambda}$, where $v_{\lambda}$ is any nonzero vector of weight $\lambda$.
Moreover, for any modules $M$ and $N$, there exists an isomorphism $M \otimes N \to N \otimes M$ called the $R$-matrix, although this map is usually not simply $m \otimes n \mapsto n \otimes m$ since $U_q(\g_0)$ is not a cocommutative Hopf algebra, that satisfies the Yang--Baxter equation.
Therefore, the structure of these representations can be completely understood by examining their characters, and the character of a simple module $\chi(V(\lambda))$ can be computed using specific formulas such as the Weyl character formula (see, e.g.,~\cite{CP95,FH91,Humphreys08}).
Taken together with Schur's lemma, this gives a complete understanding of the category of finite dimensional $U_q(\g_0)$-modules.

Despite the parallel method of constructing simple modules, performing computations in $\Ca_{\g}$ is hard because it is not a semisimple category.
In particular, there exist tensor products of $\uqpg$-modules that contain nontrivial submodules that do not decompose into a direct sum of simple submodules.
Nevertheless, $\Ca_{\g}$ has a rich structure as a rigid abelian monoidal category: the tensor functor is exact and every $\uqpg$-module has a left and right dual.
The quantum affine algebra is also not a cocommutative Hopf algebra, and we do have examples where $M \otimes N \not\iso N \otimes M$ (contrast this with $U_q(\g_0)$-modules).
There is still a homomorphism called the $R$-matrix 
\[
R \colon M_a \otimes N_b \to N_b \otimes M_a
\]
that depends only on a single parameter $z = b/a$ and satisfies the (parameterized) Yang--Baxter equation.
Unlike the $R$-matrix for $U_q(\g_0)$-modules, the $R$-matrix for $\uqpg$-modules is not always an isomorphism.
However, if $M$ and $N$ are simple modules and one of them is real, the tensor product $M \otimes N$ is simple if and only if the $R$-matrix is an isomorphism~\cite{KKKO15,KKKO18}.
This extends to multiple tensor products for simple modules by examining each pair~\cite{Her10S}.
Hence, understanding when the $R$-matrix is not an isomorphism is very significant in understanding the structure of~$\Ca_{\g}$.

There is an analog of characters for $\uqpg$-modules, known as $q$-characters, that was developed by Frenkel and Reshetikhin~\cite{FR99} (see also~\cite{FM01,Her10,Knight95}).
The $q$-characters are based on a generalization of the weight space decomposition and are expressed in terms of the variables $\{Y_{i,a} \mid i \in I_0, a \in \CC\}$.
The monomial that characterizes a simple module dominates all others and occurs with multiplicity one as with characters of $U_q(\g_0)$-modules.
While the $q$-character also cannot compute the submodule structure, it can compute which simple modules appear in the composition series of any representation in $\Ca_{\g}$.
However, there is no known method to compute, much less closed formula for, the $q$-character of a simple module in general as simple modules can have multiple dominant monomials (contrast this with $U_q(\g_0)$-characters).
Many cases can be handled using the Frenkel--Mukhin algorithm~\cite{FM01}, which is recursive and known to work when the $q$-character has a unique dominant monomial, though there are examples where it fails~\cite{NN11} and a complete characterization of when the FM algorithm succeeds is not known.
There is also a cluster algebra algorithm by Hernandez and Leclerc~\cite{HL16} and a recent algorithm of Kanakubo, Koshevoy, and Nakashima~\cite{KKN25} for computing $q$-characters of certain simple modules.

Additionally, unlike for $U_q(\g_0)$-modules, the tensor product of simple $\uqpg$-modules is generically simple as determining when the $R$-matrix is an isomorphism is governed by whether or not the parameter $z = b/a$ is a root of a \emph{polynomial} called the denominator formula $d_{M,N}(z)$~\cite{FM01}.
Thus, roughly speaking, we can say the representation theory of $\uqpg$ is controlled by the denominator formula, even though the denominator formulas do not provide information about the image of the $R$-matrix in the tensor product nor the composition series when the $R$-matrix is not an isomorphism.
The fact that the denominator formula depends only on the ratio of parameters implies we can reduce our study to a rigid monoidal subcategory $\Ca_{\g}^0 \subseteq \Ca_{\g}$ (see~\cite[Section~10]{HL10} and~\cite[Section~3.1]{KKKOIV}) that contains all prime simple modules in $\Ca_{\g}$ up to parameter shifts whose Grothendieck ring $K(\Ca_{\g}^0)$ admits a Poincar\'e--Birkhoff--Witt-type (PBW-type) basis with monomials of the form $\bigotimes_{i,a} \Vkm{i}_a$ given in an order determined by the denominator formula~\cite{AK97,Chari02,Kashiwara02,VV02}.
So far, the denominator formula has been completely computed only when $M$ and $N$ are fundamental modules~\cite{Fuj22a,KKK15,KKKOIV,KO18,OhS19,OhS19Add} (see also~\cite{KKMMNN92,KKO19}).
Hence, a major open problem is therefore to compute the denominator formula for all simple modules in $\Ca_{\g}^0$.

One important property of this reducibility of tensor products is it allows surjections
\begin{equation}
\label{eq:original_Doreys_rule}
\Vkm{i}_a \otimes \Vkm{j}_b \twoheadrightarrow \Vkm{k}_c
\tag{DR}
\end{equation}
to exist.
The epimorphisms~\eqref{eq:original_Doreys_rule} are called Dorey's rules as they are the quantum group interpretation (due to Chari and Pressley~\cite{CP96}) of the relations between three point couplings in the simple-laced affine Toda field theories and Lie theories as described by Dorey~\cite{Dorey91,Dorey93}.
Indeed, Chari and Pressley's interpretation relies on the connection with (twisted) Coxeter elements for types $A_n^{(1)}$ and $D_n^{(1)}$ (and, respectively, $B_n^{(1)}$ and $C_n^{(1)}$) observed by Dorey, with analogous results for other types given in~\cite{FH15,KKKOIV,Oh19,OhS19,OhS19Add,YZ11}.
The first named author reinterpreted Dorey's rule and the PBW ordering in terms of properties of the Auslander--Reiten (AR) quiver constructed from a Dynkin quiver in simply-laced types~\cite{Oh14A,Oh14D}.
Later, he and Suh introduced the notion of a folded AR quiver to extend these statements to non-simply-laced types by using the natural diagram foldings~\cite{OS15,OS19}.
Dorey's rules also include antisymmetric fusion, which is the special case of $i = 1$ and $k = j+1$, owes its name to the fact that $V(\Lambda_k) \iso \bigwedge^k V(\Lambda_1)$, and has appeared previously in the (mathematical) physics literature (see, e.g.,~\cite{BS06,Kuniba22}).

One other well-known homomorphism is generally known as the (symmetric) fusion rule:
\begin{equation}
\label{eq:fusion_rule}
\Vkm{i}_{q_i^{m-1}} \otimes \Vkm{i}_{q_i^{m-3}} \otimes \cdots \otimes \Vkm{i}_{q_i^{3-m}} \otimes \Vkm{i}_{q_i^{1-m}} \twoheadrightarrow \Vkm{i^m},
\tag{FR}
\end{equation}
where $\Vkm{i^m}$ is the simple module given by the head of the tensor product, which coincides with the image of the $R$-matrix that completely reverses the tensor product.
The modules
\[
\{ \Vkm{i^m}_a \mid i \in I_0, \; m \in \ZZ_{>0}, \; a \in \CC \}
\]
are known as Kirillov--Reshetikhin (KR) modules and form an important and well-studied class of $\uqpg$-modules with many amazing (conjectural) properties.
For example, KR modules are:
\bnA
\item prime modules in $\Ca_{\g}$~\cite{CP91,CP95II,CP96III}; i.e., they are not isomorphic to a nontrivial tensor product of $\uqpg$-modules;
\item real modules~\cite{HL16}, which means that $M \otimes M$ is a simple $\uqpg$-module;
\item conjectured to have crystal (psuedo)bases~\cite{HKOTT02,HKOTY99}, which is known for all but a few nodes in exceptional types~\cite{BS20,KKMMNN92,Naoi18,NS21,OS08}, that are simple (resp.\ (generally) perfect; see~\cite{KKMMNN91}) crystals, which is known in nonexceptional types~\cite{Okado13} (resp.~\cite{KKMMNN92,FOS10}).
\ee
KR modules appear in the study of integrable systems (see, e.g.,~\cite{IKT12,KNS11} and references therein), such as being solutions to T-systems~\cite{Her06,Her10,Nak04,Nak10}, which is a short exact sequence that also implies particular homomorphisms.
By the T-system relations, the $q$-characters of KR modules are computable by the Frenkel--Mukhin algorithm, by a sequence of mutations in a cluster algebra due to Hernandez and Leclerc~\cite{HL16}, or by tableaux in certain cases~\cite{BR90,KOS95,NN06,NN07,NN07II}.
This also implies characters of their tensor products are given by fermionic formulas; see ~\cite{HKOTT02,Kirillov83,Kirillov84,KR87,OSS18} and references therein.

The main purpose of this paper is to compute the denominator formulas for all KR modules.
We achieve this goal for all nonexceptional types, with the exception of $\Vkm{k^m} \otimes \Vkm{k^p}$ in type $C_n^{(1)}$ for odd integers $m$ and $p$ (Theorem~\ref{thm:denominators_untwisted}, Theorem~\ref{thm:denominators_twisted}, Conjecture~\ref{conj: denom BC}).
We give conjectural formulas for types $E_{6,7,8}^{(1)}$ based on the results for types $A_n^{(1)}$ and $D_n^{(1)}$, which are also given by an appropriate shift of the denominator formulas between the corresponding fundamental modules (Conjecture~\ref{conj: denom E}).
This extends to type $E_6^{(2)}$ since the twisted type is based off the corresponding untwisted type (\textit{cf}.\ Theorem~\ref{thm:denominators_twisted}). 

Let us now discuss our proof.
We only need to focus on untwisted types because, as mentioned above, this allows us to easily compute the denominator formula for twisted types using generalized Schur--Weyl duality functors constructed in~\cite{KKKOIII}, which preserve KR modules.
Our primary method is to compute upper and lower bounds for the orders of zeros of the denominator formulas from known denominator formulas by using~\cite[Lemma C.15]{AK97} (Proposition~\ref{prop: aMN} in the text).
This requires us determine the correct homomorphisms to apply this relation to, and this is the main technical difficulty of these proofs.
In other words, our problem is reduced to identifying sufficient Dorey's rules to make the upper and lower bounds coincide, coupled with an induction argument.
For $\Vkm{k^m} \otimes \Vkm{l^p}$ when $\abs{m-p}$ is ``large,'' this can be achieved using the classical Dorey's rules, the fusion rule~\eqref{eq:fusion_rule} and duality.
However, when $\abs{m - p}$ is ``small,'' we require additional generalized Dorey's rules to resolve these remaining ambiguities.
Hence, we must not only produce new homomorphisms, but we must also ensure that they remove these ambiguities.
Case in point, we are only able to prove the denominator formulas for $\Vkm{k^m} \otimes \Vkm{k^p}$ with  $k>1$ and $m,p$ odd in type $C_n^{(1)}$ (Conjecture~\ref{conj: denom BC}) when $\abs{m-p}$ is ``large''   (Remark~\ref{rem:small_case_C}; Proposition~\ref{prop: dlpkm C})   as we cannot resolve the remaining ambiguities.

Therefore we come to our first main result, a generalization of Dorey's rules involving tensor products of KR modules that we call higher Dorey's rule (Theorem~\ref{thm: Higher Dorey I}, Theorem~\ref{thm: Higher Dorey II}; see also Theorem~\ref{thm: Higher Dorey twisted} for twisted types).
Here we emphasize that the $\de$-invariants, which can be computed by the denominator formula of the $R$-matrix~\cite[Proposition 3.16]{KKOP20} (Proposition~\ref{prop: de ge 0} in the text) and yield $\La$-invariants for pairs of simple modules~\cite{KKOP20,KKOP23P} (Proposition~\ref{prop: Lambda property} in the text), play a crucial role in the monoidal categorification theory via quantum affine algebras and quiver Hecke algebras.
For our proof of the higher Dorey's rule, we want to apply~\cite[Corollary 3.17]{KKOP20} (Proposition~\ref{prop: hconv simple} in the text) and~\cite[Lemma 7.3]{KO18} (Proposition~\ref{prop: length 2} in the text) to show such homomorphisms could exist as a composition of other homomorphisms.
Since KR modules are real modules, we can use this theory.
In particular, our proof that there could be such a surjection uses the fact that the composition lengths of the tensor products is at most $2$, the normality of the sequence of modules, and a combination of induction arguments and showing certain morphisms are not injective or surjective.
This imposes the restriction to $\Vkm{k^m} \otimes \Vkm{l^m}$ with either $k$ or $l$ equal to $1$ in Theorem~\ref{thm: Higher Dorey I} as we need to utilize $i$-boxes and root modules to show the corresponding $\de$-invariant is $1$ to bound the composition length by $2$.
However, this is not sufficient to prove for Theorem~\ref{thm: Higher Dorey I} as we still need to verify the composition does not vanish.

In order to show the composition is nonzero, our strategy is to find a unique subspace in each module, which will yield the claim since we are composing injections and surjections.
For most cases, we use the $U_q(\g_0)$-decomposition of KR modules~\cite{Chari01,Her06,Her10,KKMMNN92,Nakajima03II} (see also~\cite{HKOTT02,HKOTY99,Scr20}) and show a particular simple $V(\lambda)$ exists in all of them by utilizing the theory of Kashiwara crystals~\cite{Kashiwara90, Kashiwara91}.
For this argument, only the $U_q(\g_0)$-crystal structure is needed and not the $\uqpg$-crystal structure~\cite{Shimozono02,OS08,FOS09}.
However, for one computation in type $B_n^{(1)}$ (Proposition~\ref{prop: q-character uniquely appear folded B}), we have multiplicities that force us to utilize the $q$-character combinatorics from~\cite{KOS95,NN06} to show there is a unique such dominant monomial to show the composition is nonzero.

With this we are able to finish our proof of the denominator formulas between any pair of KR modules.
Indeed, using the higher Dorey's rule in Theorem~\ref{thm: Higher Dorey I}, we then perform a series of tedious-but-straightforward computations to obtain tight bounds on all of the root multiplicities of the denominator formulas between KR modules (Section~\ref{sec:denom_proofs}).
As part of this computation, we also determine all of the universal coefficients of the $R$-matrices (Section~\ref{subsec: UCF}).
We also remark that in type $C_n^{(1)}$ our computations have completely determined the zeros of the denominator formula, just not their exact multiplicities.
There exist auxiliary techniques that, in certain cases, can be used to resolve the ambiguities in multiplicities, although it is not clear whether they apply in full generality (see Appendix~\ref{appensec: resolving}). 
However, even with these ambiguities, we obtain the following as a consequence:

\begin{corollaryA}
\label{cor: tensor product of KR modules}
There is a complete and precise characterization of which tensor products of KR modules are simple, which is determined by the roots of denominator formulas.
\end{corollaryA}

One other immediate application of our denominator formulas is to complete the proof of the higher Dorey's rule (Theorem~\ref{thm: Higher Dorey II}, Theorem~\ref{thm: higher mesh II}) by computing additional $\de$-invariants between KR modules.
We note that almost all $\de$-invariants appearing in Theorem~\ref{thm: Higher Dorey II} and Theorem~\ref{thm: higher mesh II} are larger than $1$, which implies that the composition series of tensor products involved need not to be of length $2$, but they still have KR modules as their simple heads.
Moreover, we can compute $\Lambda$-invariants between KR modules in all cases except for the exceptional cases in type $C_n^{(1)}$ previously mentioned, which provides valuable information on $R$-matrices and, in turn, on the representation theory of quantum affine algebras.
We also determine exactly when the socle in the higher Dorey’s rule is prime (Corollaries~\ref{cor: socle prime} and~\ref{cor: prime socle}) by applying a classical result of Chari--Pressley (Corollary~\ref{prop: CP partition}).

Let us discuss some applications of the higher Dorey's rule.
Using Theorem~\ref{thm: Higher Dorey I}, we obtain the higher Dorey's rule of mesh type in Theorem~\ref{thm: higher mesh} (also Corollary~\ref{cor: length 2 C} and~\ref{cor: length 2 D}) and Theorem~\ref{thm: higher mesh II}, which also yields the short exact sequences with two components are always tensor product of KR modules (Theorem~\ref{thm: generalization of T-system b>1}).
Since the usual T-system is a special family of the sequences in Theorem~\ref{thm: generalization of T-system b>1}, Theorem~\ref{thm: generalization of T-system b>1} can be understood as a generalization of the T-system.  
This is different than the extended T-system of Mukhin and Young~\cite{MY12} (see also~\cite{Naoi24}) since these extended T-systems
\begin{enumerate}
\item \label{it: extT} arise from certain sequences of root modules $\{ \ttR[k] \mid  c \le k  \le d  \}$ and pairs of intervals $([a,b],[a+1,b+1])$ on the sequence and
\item work only for affine types $A_n^{(1)}$ and $B_n^{(1)}$ due to the minuscule property (i.e., the nonzero monomials in the $q$-characters of all fundamentals have coefficient $1$).
\end{enumerate}
Later work by Li and Mukhin~\cite{LM13} (resp.\ Li~\cite{Li15}) constructed an extended T-system for type $G_2^{(1)}$ (resp.\ $C_3^{(1)}$).
In contrast, our generalized T-system in Theorem~\ref{thm: generalization of T-system b>1} works for all types and, in general, cannot be described in either the form~\eqref{it: extT} nor in the typical form of T-system (see~\cite{KKOP24A}).
In Appendix~\ref{sec:extended_T_system}, we give a different expression of the extended T-system in terms of the $\de$-invariants utilizing the framework of Naoi~\cite{Naoi24}.
In there, the arguments and construction presented rely on a particular order of the sequence of fundamental modules~\eqref{eq: tk condi} determined by the denominator formulas between fundamental modules (which was previously known).
Therefore, we obtain the extended T-system satisfying the framework for all affine types (removing the necessity of the minuscule property).

Next, by applying our denominator formulas and higher Dorey's rules, we can deduce further interesting applications.
We recover~\cite[Theorem 2.1]{FH14} as a simple corollary (Theorem~\ref{thm:dimhom_FH}), as well as producing a generalization of this result (Theorem~\ref{thm:gen_dimhom_FH}).

Another application concerns quiver Hecke algebras~\cite{KL1,R08}.
Let $\calR^\gfin$ denote the quiver Hecke algebra of simply-laced $\g_{\fin}$.
Since 
\begin{enumerate}[(i)]
\item the heart subcategory $\scrC_\calQ$ of $\Cg^0$ is equivalent to the category $\calR^\gfin\gmod$ of finite dimensional graded $R$-modules via the generalized Schur--Weyl duality functor~\cite{KKK18,KKK15,KKKOIV,KO18,OhS19} and
\item the categories $\scrC_\calQ$ and $\calR^\gfin\gmod$ categorify the negative half $U^-_q(\g_\fin)$ of the quantum group $U_q(\g_\fin)$~\cite{KL1,R08},
\end{enumerate}
we can interpret
\begin{enumerate}[(i$'$)]
\item the $\de$-invariants from the denominator formulas in this paper as those for determinantial modules over $\calR^\gfin$ (Theorem~\ref{thm: Dorey in Rgmod}) and
\item the higher Dorey's rule and the generalized T-system in this paper as the multiplication behavior of dual canonical/upper global basis of $U^-_q(\g_\fin)$ (see Remark~\ref{rmk: multiplication structure}).
\end{enumerate}

As a final application, we consider the generalized Schur--Weyl duality functors in~\cite{KKO19} between the categories $\scrC^0_{A_{2n-1}^{(1)}}$ and $\scrC^0_{B_{n}^{(1)}}$, and the functors in~\cite{KKK15,KKKOIV,KO18} among heart subcategories of $\scrC_\g^0$.
Remarkably, these functors sends simple modules to simple modules bijectively, but they do \emph{not} preserve the fundamental ones, much less KR modules; see~\cite[Section~3]{KKO19} and \cite[Section~12]{HO19}.
Consequently, by applying these functors to our higher Dorey's rule and generalized T-system, we obtain interesting relations among a certain family of modules (Corollary~\ref{cor: non KR rel} and Corollary~\ref{cor: non KR rel2}).

Let us remark on a number of distinctive features that we can see from our results.
In our generalized T-systems, the socle is prime unless it is the classical T-system and \emph{not} involving fundamental modules. 
From the denominator formulas, we observe that KR modules $M = \Vkm{1^m}, \Vkm{k}$ ($m \geq 1$, $k<n$) in types $A_{n-1}^{(1)}$ and $B_n^{(1)}$, typically described as one-row or one-column, possess the distinguished property that, for \emph{any} KR module $N$, we have $\de(M,N) \le 1$.
We call such modules $M$ plain.
Hence the tensor product $M \tens N$ has composition length at most $2$ (Corollary~\ref{cor: plain A} and Corollary~\ref{cor: root module B}).

On the other hand, we note that the existence of the (higher) Dorey's rules does not fundamentally rely on $\de(M, N) \leq 1$ or the composition length at most $2$.
For example, in type $A_4^{(1)}$ with $M =\Vkm{2^3}_{(-q)^{-2}}$ and $N = \Vkm{2^3}_{(-q)^2}$, we can see $M \tens N$ has a composition series of length $3$ by using $q$-characters, and so $\de(M, N) > 1$ (indeed, we can check $\de(M, N) = 2$).
However, this is the higher Dorey's rule~\eqref{eq: k+l<n homo II}, which says the head of $M \tens N$ is $\Vkm{4^3}$ (and also remains simple when restricted to $U_q(\g_0)$-module). 

Furthermore, our results yield new and noteworthy observations concerning root modules.
First, the denominator formulas provide the first explicit construction of root modules that are not KR modules (Remark~\ref{rmk: non KR root}).
Second, the higher Dorey’s rules give the first example of a root module appearing as the head of a tensor product of two non-root KR modules $M$ and $N$ with $\de(M,N) > 1$ (Remark~\ref{rmk: non roots to root}).

We end this introduction by placing our results in the broader context of monoidal categorification in the sense of~\cite{HL10}. 
In~\cite{HL16}, Hernandez--Leclerc proved that the Grothendieck group $K(\Cg^-)$ of the half subcategory $\Cg^-$ of $\Cg^0$ is isomorphic to a cluster algebra and every KR module in $\Cg^-$ corresponds to a cluster variable in $K(\Cg^-)$. 
This result was later extended to the entire category $\Cg^0$ in~\cite{KKOP22A,KKOP24A}.
Subsequently, it was shown in~\cite{KKOP20,KKOP24A} that $\Cg^0$ provides a monoidal categorification of the cluster algebra $K(\Cg^0)$; that is, every cluster monomial corresponds to a real simple module in $\Cg^0$. 
In the proofs of~\cite{HL16,KKOP24A}, the T-system plays the role of the exchange relation in cluster algebra theory. 
Hence, it is natural to ask whether the generalized T-system studied in this paper also corresponds to an exchange relation in $K(\Cg^0)$.

\subsection*{Organization} 

This paper is organized as follows.
In Section~\ref{sec:background_afffine}, we review the necessary background on quantum affine algebras and the relevant categories of finite dimensional representations.
In Section~\ref{sec:background_repr_theory}, we recall $R$-matrices and root modules with their properties.
In Section~\ref{sec:background_doreys_rule}, we give the necessary background on KR modules, $i$-boxes, and Dorey's rule.
In Section~\ref{sec:crystals}, we briefly review the Kashiwara crystals.
In Section~\ref{sec: new morphisms}, we show the restricted version of the higher Dorey's rules.
In Section~\ref{sec:denominator_formulas}, we state our main result, the denominator formulas in all nonexceptional affine types.
In Section~\ref{sec:denom_proofs}, we prove the denominator formulas, which is the technical core of the paper.
In Section~\ref{sec:applications}, we discuss several applications of our main results.

\subsection*{Conventions}

Throughout this paper, we use the following conventions.
\ben
\item For a statement $\ttP$, we set $\delta(\ttP)$ to be $1$ or $0$ depending on whether $\ttP$ is true or not. In particular, we set $\delta_{i,j}=\delta(i=j)$. 
\item For $a\in \Z \cup \{ -\infty \} $ and $b\in \Z \cup \{ \infty \} $ with $a\le b$, we set 
\begin{align*}
	& [a,b] =\{  k \in \Z \ | \ a \le k \le b\}, &&  [a,b) =\{  k \in \Z \ | \ a \le k < b\}, \allowdisplaybreaks\\
	& (a,b] =\{  k \in \Z \ | \ a < k \le b\}, &&  (a,b) =\{  k \in \Z \ | \ a < k < b\},
\end{align*}
and call them \defn{intervals}. 
When $a> b$, we understand them as empty sets. For simplicity, when $a=b$, we write $[a]$ for $[a,b]$. 
\item For a totally ordered set $J = \{ \cdots > j_{2} > j_1 > j_0 > j_{-1} > j_{-2} > \cdots \}$, write
$$
\stens^\to_{j \in J} A_j \seteq \cdots A_{j_2} \otimes A_{j_1} \otimes  A_{j_0}\otimes  A_{j_{-1}}\otimes  A_{j_{-2}} \cdots.
$$
\item For a field $\bfk$, $a\in\bfk$ and $f(z)\in\bfk(z)$,
we denote by $\Zero_{z=a}f(z)$ the order of zero of $f(z)$ at $z=a$.
\item For a set $A$, we denote by $|A|$ the cardinality of $A$.
\item For set $A$ and $B$, a map $A \twoheadrightarrow B$ denotes a surjective map and a map $A \rightarrowtail B$  denotes an injective map.
\item For integers $a,b \in \Z$ and $m \in \Z_{\ge 1}$, we write $a \equiv_m b$ if $a-b$ is divisible by $m$,
and $a \not\equiv_m b$ otherwise. 
\ee

\subsection*{Acknowledgments}

The authors thank Matheus Brito, Vyjayanthi Chari, Patrick Dorey, Ryo Fujita, David Hernandez, Masaki Kashiwara, Rinat Kedem, Myungho Kim, and Euiyong Park and Heizo
Sakamoto for useful discussions and comments.
S-jO thanks Hokkaido University for its hospitality during his visit in January, 2025.
TS thanks Sungkyunkwan University for its hospitality during his visit in March, 2025 and Ewha Womans University for its hospitality during his visit in March, 2019.
This material is based upon work supported by the National Science Foundation under Grant No.~DMS-1929284 while the authors were in residence at the Institute for Computational and Experimental Research in Mathematics in Providence, RI, during the ``Categorification and Computation in Algebraic Combinatorics'' Fall 2025 semester program. 
This work benefited from computations done using \textsc{SageMath}~\cite{sage}.

\section{Background}
\label{sec:background_afffine}

In this section, we briefly review the necessary background of this paper including quantum affine algebras $\uqpg$, finite dimensional modules over 
$\uqpg$  and combinatorics related them.

\subsection{Quantum affine algebras}\label{subsec:Quantum affine}

Let $q$ be an indeterminate. 
Let $(A,\widehat{P},\Uppi,\widehat{P}^\vee,\Uppi^\vee)$ be an \defn{affine Cartan datum} consisting of an \defn{affine Cartan matrix} $A=(a_{i,j})_{i,j \in I}$
with an index set $I$, a \defn{weight lattice} $\widehat{P}$, a set of \defn{simple roots} $\Uppi =\{\upal_i \}_{i \in I} \subset \widehat{P}$, 
a \defn{coweight lattice} $\widehat{P}^\vee \seteq \Hom_{\Z}(\widehat{P},\Z)$ and a set of
\defn{simple coroots} $\{ h_i \}_{i \in I} \subset \widehat{P}^\vee$. The datum satisfies
$\Ang{h_i,\upal_j}=a_{i,j}$ for all $i,j\in I$, where $\Ang{ \ ,\ }\colon \widehat{P}^\vee \times \widehat{P} \to \Z$
is the canonical pairing. We \emph{choose} $\{\UpLa_i\}_{i \in I}$ such that $\Ang{h_j,\UpLa_i}=\delta_{i,j}$ for $i,j \in I$ and call them the \defn{fundamental weights}. 
Note that there exist a diagonal matrix 
\begin{align}\label{eq: D and di}
\text{$D = \operatorname{diag}(d_i \in \Z_{\ge 1} \mid i \in I)$ such that $D \cmA$ is symmetric.}
\end{align}
We take the diagonal matrix $D$ such that $\min(d_i)_{i \in I}=1$.

We take the primitive \defn{imaginary root} (also known as the \defn{null root}) $\updelta=\sum_{i\in I}u_i\upal_i$ and the \defn{canonical central element} $c=\sum_{i \in I}c_i h_i$ such that
$ \{ \upmu \in \bigoplus_{i\in I} \Z \upal_i \mid \Ang{h_i,\upmu}=0 \text{ for every } i\in I \} =\Z\updelta$ and $ \{ h \in \bigoplus_{i\in I} \Z h_i \ | \ \Ang{h,\upal_i}=0 \text{ for every } i\in I \} =\Z c$.
We choose $\uprho \in \widehat{P}$ (resp.\ $\uprho^\vee \in \widehat{P}^\vee$) such that $\lan h_i,\uprho \ran=1$ (resp.\ $\lan \uprho^\vee,\upal_i\ran =1$) for all $i \in I$.

Set $\h \seteq \Q \tens_{\Z} \widehat{P}^\vee$.
Then there exists the symmetric bilinear form $( \ , \ )$ on $\h^*$ satisfying
\[
\lan h_i,\upmu \ran = \dfrac{2(\upal_i,\upmu)}{(\upal_i,\upal_i)}
\qquad \text{ and } \qquad
\lan c,\upmu\ran = (\updelta,\upmu) \quad \text{ for any $i \in I$ and $\upmu \in \h^*$}.
\]
Let $\gamma$ be the smallest positive integer such that $\gamma(\al_i,\al_i)/2 \in \Z$ for every $i \in I$. 
Note that $(\al_i,\al_i)$ may not be integer, $(\al_i,\al_i)/2$ takes the values $1,2,3,1/2,1/3$.

Let $q$ be an indeterminate. 
For $m,n \in \Z_{\ge 0}$ and $i\in I$, we define 
\begin{equation*}
[n]_i =\frac{ q^n_{i} - q^{-n}_{i} }{ q_{i} - q^{-1}_{i} },
\qquad\quad [n]_i! = \prod^{n}_{k=1} [k]_i ,
\qquad\quad \left[\begin{matrix}m \\ n\\ \end{matrix} \right]_i=  \frac{ [m]_i! }{[m-n]_i! [n]_i! },
\qquad\quad (z;q_i)_\infty \seteq \prod_{s=0}^\infty (1-q_i^s z) 
\end{equation*}
and set $q_i \seteq q^{(\upal_i,\upal_i)/2}$ in $\{ q^{1/m}, q^{n} \ | \ m,n \in \Z_{\ge 1} \}$.

We denote by $\uqaff$ the \defn{quantum affine algebra} (with derivation) over $\Q(q^{1/\ga})$ and $\gaff$ the \defn{affine Kac--Moody algebra} associated with $(A,\widehat{P},\Uppi,\widehat{P}^\vee,\Uppi^\vee)$, respectively. 
We also denote by $\Dynkin=(\Dynkin_0,\Dynkin_1)$ the Dynkin diagram of $\g$ consisting of (i) the set of vertices $\Dynkin_0$ and
(ii) the set of edges $\Dynkin_1$, respectively. 
We will use the standard convention in~\cite{Kac90} for labeling of affine Dynkin diagrams $\triangle$ for non-exceptional affine types except type $A_{2n}^{(2)}$, in which case we take the longest simple root as $\upal_0$.
For the exceptional types, we use the labeling given in Figure~\ref{fig:exceptional_types} mostly following~\cite{Bou}.
We denote by $d_{\triangle}(i,j)$ the number of edge between vertices $i$ and $j$ in $\triangle$.

\begin{figure}[t]
\begin{align*}
&
 E^{(1)}_6:  \raisebox{2.3em}{\xymatrix@C=3.4ex@R=3ex{ && *{\bullet}<3pt>\ar@{-}[d]^<{0}  \\ && *{\circ}<3pt>\ar@{-}[d]^<{2} \\
*{ \circ }<3pt> \ar@{-}[r]_<{1}  &*{\circ}<3pt>
\ar@{-}[r]_<{3} &*{ \circ }<3pt> \ar@{-}[r]_<{4} &*{\circ}<3pt>
\ar@{-}[r]_<{5} &*{\circ}<3pt>
\ar@{-}[l]^<{\ \ 6}}} \
\hspace{-5ex} E^{(1)}_7: \raisebox{1.3em}{\xymatrix@C=3.4ex@R=3ex{ &&& *{\circ}<3pt>\ar@{-}[d]^<{2} \\
*{ \bullet }<3pt> \ar@{-}[r]_<{0}  & *{ \circ }<3pt> \ar@{-}[r]_<{1}  &*{\circ}<3pt>
\ar@{-}[r]_<{3} &*{ \circ }<3pt> \ar@{-}[r]_<{4} &*{\circ}<3pt>
\ar@{-}[r]_<{5} &*{\circ}<3pt>
\ar@{-}[r]_<{6} &*{\circ}<3pt>
\ar@{-}[l]^<{7} } } \
E_8 ^{(1)}: \raisebox{1.3em}{\xymatrix@C=3.4ex@R=3ex{ && *{\circ}<3pt>\ar@{-}[d]^<{2} \\
*{ \circ }<3pt> \ar@{-}[r]_<{1}  &*{\circ}<3pt>
\ar@{-}[r]_<{3} &*{ \circ }<3pt> \ar@{-}[r]_<{4} &*{\circ}<3pt>
\ar@{-}[r]_<{5} &*{\circ}<3pt>
\ar@{-}[r]_<{6} &*{\circ}<3pt>
\ar@{-}[r]_<{7} &*{\circ}<3pt>
\ar@{-}[l]^<{8} &*{\bullet}<3pt>
\ar@{-}[l]^<{0} } }
 \allowdisplaybreaks \\
& E^{(2)}_6 :\  \xymatrix@C=4ex@R=3ex{
*{ \bullet }<3pt> \ar@{-}[r]_<{0}  &*{\circ}<3pt>
\ar@{-}[r]_<{1} &*{ \circ }<3pt> \ar@{<=}[r]_<{2} &*{\circ}<3pt>
\ar@{-}[r]_<{3} &*{\circ}<3pt> \ar@{-}_<{4} }  \
F^{(1)}_4:   \xymatrix@C=4ex@R=3ex{
*{ \bullet }<3pt> \ar@{-}[r]_<{0}  &*{\circ}<3pt>
\ar@{-}[r]_<{1} &*{ \circ }<3pt> \ar@{=>}[r]_<{2} &*{\circ}<3pt>
\ar@{-}[r]_<{3} &*{\circ}<3pt> \ar@{-}_<{4} }  \
G^{(1)}_2 :\  \xymatrix@C=4ex@R=3ex{
*{ \bullet }<3pt> \ar@{-}[r]_<{0}  &*{\circ}<3pt>
\ar@{=>}[r]_<{1} &*{ \circ }<3pt>  \ar@{-}[l]^<{ \ \ 2}  }  \
D^{(3)}_4:   \xymatrix@C=4ex@R=3ex{
*{ \bullet }<3pt> \ar@{-}[r]_<{0}  &*{\circ}<3pt>
\ar@{<=}[r]_<{1} &*{ \circ }<3pt>  \ar@{-}[l]^<{ \ \ 2}  }
\end{align*}
\caption{Dynkin diagrams $\triangle$ for the exceptional affine types. The affine node is marked black.}
\label{fig:exceptional_types}
\end{figure}

Let $\uqpg \seteq U_q([\gaff, \gaff])$ be the affine quantum group without derivation.
Let $P \seteq \widehat{P} / \ZZ\updelta$ denote its corresponding weight lattice, and let $\cl \colon \widehat{P}\to P$ denote the canonical projection.
When there is no danger of confusion, we will use the same notation for the fundamental weights, simple (co)roots, etc.\ as for $\uqaff$.

We define $\g_0$ to be the subalgebra of $\g$ generated by the Chevalley generators $e_i$, $f_i$  and $h_i$ for $i \in I_0 \seteq I \setminus \{ 0 \}$ and $W_0 = \lan s_i \mid  i \in I_0 \ran$ to be the Weyl group of $\g_0$ generated by the simple reflections $\{ s_i\}_{i \in I_0}$.
Let $U_q(\g_0)$ denote the corresponding quantum group.
Note that $\g_0$ is a finite dimensional simple Lie algebra and $W_0$ is a finite group that contains a unique longest element $w_0$.
Let $\wl \subseteq P$ and $\rl \subseteq Q$ denote the corresponding weight and root lattices, respectively, with $\wl^+ \seteq \bigoplus_{i \in I_0} \ZZ_{\geq 0} \La_i$ and $\rl^+ \seteq \bigoplus_{i \in I_0} \ZZ_{\geq 0} \alpha_i$ being the set of \defn{dominant weights} and \defn{positive root cone}, respectively.

\subsection{Finite dimensional integrable modules} 

When considering $\uqpg$-modules, we take the algebraic closure of $\C(q)$ in $\bigcup_{m>0} \C\dpar{q^{1/m}}$ as our base field $\bfk$. 
For a $\uqpg$-module $M$, let $\soc(M)$ (resp.\ $\hd(M)$) denote the \defn{socle} (resp.\ \defn{head}) of $M$, the largest semisimple submodule (resp.\ quotient) of $M$.

We say that a $\uqpg$-module $M$ is \defn{integrable} if
(i) $M$ decomposes into $P$-weight spaces; that is, $ M = \bigoplus_{ \mu \in P} M_\mu$, where $M_\mu \seteq \{ v \in M \mid K_i v = q^{\langle h_i, \mu \rangle } v \}$, and (ii) $e_i$ and $f_i$ $(i \in I)$ act on $M$ nilpotently.
Here $K_i \seteq q_i^{ h_i}$.

We denote by $\Ca_\g$ the category of finite dimensional integrable $\uqpg$-modules. Note that $\Ca_\g$ is a tensor category from the coproduct $\Delta$ of $\uqpg$:
\[
\Delta(q^h)=q^h \otimes q^h,
\qquad
\Delta(e_i)=e_i \otimes 1 + K_i^{-1} \otimes  e_i,
\qquad
\Delta(f_i)= f_i \otimes K_i + 1 \otimes  f_i.
\]

A simple module $M$ in $\Ca_\g$ contains a non-zero vector $\bsu$ of weight $\upla\in P$ such that (i) $\langle h_i,\upla \rangle \ge 0$ for all $i \in I_0$,  (ii) all the weight of $M$ are contained in $\upla- \sum_{i \in I_0} \Z_{\ge 0} {\cl}(\upal_i)$.
Such a $\upla$ is unique and $\bsu$ is unique up to a constant multiple.
We call $\upla$ the \defn{dominant extremal weight} of $M$ and $\bsu$ the \defn{dominant extremal weight vector} of $M$.

We note that (ii) implies that if we restrict $M$ to a $U_q(\g_0)$-module, then there exists a unique weight $\lambda \in \wl$ such that every nontrivial weight space has weight in $\lambda - \rl^+$.
We say the weights $\mu \in \lambda - \rl^+$ are \defn{dominated} by $\la$; that is, $\mu - \lambda \in \rl^+$.
This is a well-known fact (see, e.g.,~\cite[Ch.~1.6]{Humphreys08}) that finite dimensional simple $U_q(\g_0)$-modules are parameterized by $\lambda \in \wl^+$, which we denote by $V(\lambda)$.
Hence $M \iso \bigoplus_{\mu} V(\mu)^{\oplus m_{\mu}}$, for some multiplicities $m_{\mu}$, where $\mu$ is dominated by $\la$.

For an integrable $\uqpg$-module $M$, the \defn{affinization} $M_z \seteq \ko[z,z^{-1}] \otimes  M$ of $M$ is considered as a vector space over $\ko$ and is equipped with a $\uqpg$-module structure
\[
e_i(u_z) = z^{\delta_{i,0}}(e_iu)_z,
\qquad\qquad
f_i(u_z) = z^{-\delta_{i,0}}(f_iu)_z,
\qquad\qquad
K_i(u_z) = (K_iu)_z,
\]
for all $i \in I$.
Here $u_z$ denotes $\mathbf{1} \otimes  u \in M_z$ for $u \in M$. We sometimes write $z_M$ as the action of $z$ on $M_z$
to emphasize the module $M$.
For $\zeta \in \ko^\times$, we define
\[
M_{\zeta} \seteq M_z / (z_M - \zeta)M_z.
\]
We call $\zeta$ the \defn{spectral parameter}.
Note that, for a module $M \in \Ca_\g$ and $\zeta \in \ko^\times$, we have $M_{\zeta} \in \Ca_\g$.

For each $i \in I_0$, we define the \defn{fundamental level zero weight} as
\[
\varpi_i \seteq \gcd(\mathsf{c}_0,\mathsf{c}_i)^{-1}\cl(\mathsf{c}_0\Uplambda_i-\mathsf{c}_i \Uplambda_0) \in P_{\cl}.
\]
Then there exists a unique simple module $V(\varpi_i)$ in $\Ca_\g$, called the \defn{fundamental module of $($level $0)$ weight $\varpi_i$}, satisfying the certain conditions (see, e.g.,~\cite[\S 5.2]{Kashiwara02}).
We call $V(\varpi_i)_x$ $(x \in \bfk^\times)$ also a fundamental module.

\begin{remark} \label{rem:m_i}
Let $m_i$ be a positive integer such that
\[
W(\Uplambda_i-d_i\Uplambda_0)=(\Uplambda_i-d_i\Uplambda_0)+\Z m_i\updelta.
\]
We have $m_i=(\upal_i,\upal_i)/2$ in the case $\g$ is the dual of an untwisted affine algebra, and $m_i = 1$ otherwise.
Then, for $x,y\in \ko^\times$, we have~\cite[\S 1.3]{AK97}
\[
V(\varpi_i)_x \iso V(\varpi_i)_y  \quad \text{ if and only if }   \quad   x^{m_i}=y^{m_i}.
\]
Thus, for twisted types, we will simply write $M_z$ instead of $M_{z^{m_i}}$.
\end{remark}

For a module $M$ in $\Ca_\g$, let us denote the right and the left dual of $M$ by $\scrD M$ and $\scrD^{-1}M$, respectively.
That is, we have isomorphisms
\begin{align*}
&\Hom_{\uqpg}(M\hspace{-.4ex} \tens \hspace{-.4ex} X,Y) \hspace{-.2ex} \iso \hspace{-.2ex} \Hom_{\uqpg}(X, \hspace{-.4ex} \scrD M \hspace{-.4ex} \tens \hspace{-.4ex} Y), \ \quad \
\Hom_{\uqpg}(X \hspace{-.4ex} \tens \hspace{-.4ex} \scrD M,Y)\hspace{-.2ex} \iso \hspace{-.2ex} \Hom_{\uqpg}(X, Y \hspace{-.4ex} \tens \hspace{-.4ex} M),\\
&\Hom_{\uqpg}(\scrD^{-1} M \hspace{-.4ex} \tens \hspace{-.4ex}  X,Y)\hspace{-.2ex} \iso \hspace{-.2ex} \Hom_{\uqpg}(X, M \hspace{-.4ex}  \tens \hspace{-.4ex}  Y),\
\Hom_{\uqpg}(X \hspace{-.4ex} \tens \hspace{-.4ex}  M,Y)\hspace{-.2ex} \iso \hspace{-.2ex} \Hom_{\uqpg}(X, Y \hspace{-.4ex}  \tens \hspace{-.4ex} \scrD^{-1} M),
\end{align*}
which are functorial in $\uqpg$-modules $X$ and $Y$.
In particular, $V(\varpi_i)_x$ $(x\in \ko^\times)$ has the left dual and right dual
as follows:
\begin{equation*}
\scrD^{-1} \bigl( V(\varpi_i)_x\bigr)  \iso   V(\varpi_{i^*})_{x(p^*)^{-1}},
\qquad \scrD \bigl(V(\varpi_i)_x \bigr) \iso   V(\varpi_{i^*})_{xp^*}
\end{equation*}
where $p^* \seteq (-1)^{\langle \rho^\vee ,\delta \rangle}q^{(\rho,\delta)}$ and $i^*$ is the image of $i$ under the involution of $I_0$ determined by the action of $w_0$  (see \cite[Appedix A]{AK97}); i.e., $w_0(\al_i) =-\al_{i^*}$.
\renewcommand{\arraystretch}{1.5}
\begin{align} \label{Table: p*}
\small
p^* = 
\begin{array}{ccccccccccccc}
\toprule
A_{n-1}^{(1)} & B_n^{(1)} & C_n^{(1)} & D_n^{(1)} & G_2^{(1)} & A_{n}^{(2)} & D_{n+1}^{(2)} & D_4^{(3)}
& E_6^{(1)}& E_7^{(1)}& E_8^{(1)}& F_4^{(1)}& E_6^{(2)}
\\ \midrule
(-q)^n &  q^{2n-1} & q^{n+1} & q^{2n-2} & q^{4} & -q^{n+1} & -(-q^2)^n & q^6 & q^{12}& q^{18}& q^{30}& q^{9}& -q^{12}
\\ \bottomrule
\end{array}
\end{align}

We say that a $\uqpg$-module $M$ is \defn{good} if it has a \emph{bar involution}, a \emph{crystal basis} with \emph{simple
crystal graph}, and a \emph{global basis} (see \cite{Kashiwara02} for the precise definition).  We say that a $\uqpg$ module $M$ is \defn{quasi-good} if
$M \iso V_c$ for some good module $V$ and $c \in \bfk^\times$. It is known that the
fundamental modules are (quasi-)good modules.

For simple modules $M$ and $N$ in $\Ca_\g$, we say that $M$ and $N$ \defn{strongly commute} (or $M$ \defn{strongly commutes with} $N$) if $M \tens N$ is simple.
We say that a simple module $L$ in $\Ca_\g$ is \defn{real} if $L$ strongly commutes with itself, i.e., if $L \tens L$ is simple.
We say that a simple module $L$ in $\Ca_\g$ is \defn{prime} if there exist no non-trivial modules $M_1$ and $M_2$ such that $L\iso M_1 \otimes M_2$.

A monoidal subcategory $\calC$ of $\Ca_\g$ is called a \defn{skeleton subcategory} if every prime simple module in $\mC_{\g}$ is a parameter shifts of some prime simple module in $\calC$.

\subsection{Convex orders} \label{subsec: Convex}

Next, we will briefly review convex orders  on the set of positive roots of finite Lie algebra $\bfg$.
We denote by $\Phi^+_{\bfg}$ the set of positive roots, $\Pi_{\bfg} = \{ \al_i \mid i \in I_0\}$ the set of simple roots, and $\{ \La_i \mid i \in I_0\}$   
the set of fundamental weights for $\bfg$, respectively.
Throughout this manuscript, we sometimes omit the subscript $_\bfg$ in our notation for simplicity when there is no danger of confusion.

It is well-known that for any reduced expression $\tw_0 = s_{i_1} s_{i_2} \dotsm s_{i_\ell}$ of $w_0$, we have
\begin{align}\label{eq: labeling redex}
\Phi^+_{\bfg} = \{ \be^{\tw_0}_{k} = s_{i_1}s_{i_2} \cdots s_{i_{k-1}}(\al_{i_k}) \mid 1 \le k \le \ell  \} \quad\text{ and }  \quad \abs{\Phi^+_{\bfg}} = \ell.
\end{align}
Thus, we can define a total order $<_{\twz}$  on $\Phi^+_{\bfg}$ as $\be^\twz_k <_\twz \be^\twz_l$ if and only if $k < l$.

We say that two reduced expression $\twz$ and $\twz'$ of the longest element $w_0 \in W_\bfg$ are \defn{commutation equivalent} if one can obtain $\twz'$
from $\twz$ by applying the commutation relations $s_i s_j = s_j s_i$ ($d_{\Dynkin^{\bfg}}(i,j) > 1$). We denote by $[\twz]$ the commutation class of $\twz$ under this equivalence.

We say that a partial order $\prec$ on $\Phi^+$ is \defn{convex} if $\al,\be,\al+\be \in \Phi^+$, we have either
$\al \prec \al+\be \prec \be$ or $\be \prec \al+\be \prec \al$.
In~\cite{Papi,Zhe87}, the following order $\prec_{[\twz]}$ is shown to be convex for any $[\twz]$:
\[
\al \prec_{[\twz]}  \be \  \text{ if and only if } \   \al <_{\twz'}  \be  \text{ for all } \twz' \in [\twz].
\]

An element $\um = \exponent{\um_\beta}_{\be \in \Phi^+} \in \Z_{\ge0}^{\Phi^+}$ is called an \defn{exponent}, parameterized by $\Phi^+$.
For an exponent $\um$, we set $\wt(\um) \seteq \sum_{\be \in \Phi^+} \um_\be \be \in Q_0^+$ of $\bfg$.

\begin{definition}[{\cite{McN15,Oh19}}]
We define the partial orders $<^\tb_{\redez}$ and $\prec^\tb_{[\redez]}$ on $\Z_{\ge 0}^{\Phi^+}$ as follows:
\begin{enumerate}[{\rm (i)}]
\item $<^\tb_{\redez}$ is the bi-lexicographical partial order induced by $<_{\redez}$. Namely, $\um<^\tb_{\redez}\um'$ if
\bna
\item $\wt(\um)=\wt(\um')$, 
\item there exists $\al \in \Phi^+$ such that $\um_\al < \um'_\al$ and $\um_\be = \um'_\be$ for any $\be$ such that $\be <_\twz \al$,
\item there exists $\eta \in \Phi^+$ such that $\um_\eta < \um'_\eta$ and $\um_\zeta = \um'_\zeta$ for any $\be$ such that $\eta <_\twz \zeta$.
\ee
\item For exponents $\um$ and $\um'$, we have $\um \prec^\tb_{[\redez]} \um'$ if the following conditions are satisfied:
$$
\um <_{\twz'}^\tb \um' \quad \text{ for all } \twz' \in [\twz]. 
$$
\end{enumerate}
\end{definition}

We say an exponent $\um=\exponent{\um_\beta}_{\be \in \Phi^+}$ is \defn{$[\redez]$-simple} if it is minimal with respect to the partial order $\prec^\tb_{[\redez]}$.
For a given $[\redez]$-simple exponent $\us=\exponent{\us_\beta}_{\be \in \Phi^+}$, we say $\um$ is \defn{$[\redez]$-minimal exponent of $\us$} if it is a \emph{cover} of $\us$ under $\prec^{\tb}_{[\redez]}$ (that is, there is no $\um'$ such that $\us \prec^{\tb}_{[\redez]} \um' \prec^{\tb}_{[\redez]} \um$).
The \defn{$[\redez]$-distance} of an exponent $\um$ is the largest integer $k \geq 0$ such that
\[
\um^{(0)} \prec^\tb_{[\redez]} \cdots \prec^\tb_{[\redez]} \um^{(k)} = \um
\]
and $\um^{(0)}$ is $[\redez]$-simple. We denote by the number $k$ as $\dist_{[\redez]}(\um)$.

We sometimes identify $\ga \in \Phi^+$ with the exponent $\um$ such that $\um_\be = \delta(\be = \gamma)$. 
We call an exponent $\um$ a \defn{pair} if $|\um|\seteq \sum_{\be \in \Phi^+} m_\be=2$ and $m_\be \le 1$ for $\be \in \Phi^+$. We mainly use the notation $\up$ for a pair
and identify $\up$ with $ \{ \al, \be \}  \in \Phi^+$ with $\up_\al=\up_\be=1$.
Consider a pair $\up$ such that there exists a unique $[\redez]$-simple exponent $\us$ satisfying
$\us \preceq^\tb_{[\redez]} \up$, we call $\us$ the \defn{$[\redez]$-socle} of $\up$ and denoted it by $\soc_{[\redez]}(\up)$.

We use the notation $\exponent{ a_1\be_1,\ldots,a_r\be_r}_{\be \in \Phi^+}$ to represent the exponent $\um$ such that $\um_{\be_k}=a_k$ for $1 \le k \le r$,
and $\um_\al=0$ for $\al \ne \be_k$.

\subsection{$\rmQ$-datum and quivers}

In this subsection, we briefly review the $\rmQ$-data and their related quivers by following \cite{FO21} mainly.
We first consider the quantum affine algebra $\uqpg$ of \emph{untwisted} type.
For each \emph{untwisted} quantum affine algebra $\uqpg$,  we fix a function $\epsilon \colon I_0 \to \{0,1\}$ such that
\begin{align} \label{eq: parity function}
\epsilon_i \equiv \epsilon_j + \min(d_i,d_j) \pmod{2} \quad \text{ whenever } a_{i,j}<0 \ \text{ for } i,j \in I_0, 
\end{align}
where $A=(a_{i,j})_{i \in I}$ denotes the affine Cartan matrix of $\g$. 
We call $\epsilon$ a \emph{parity function}.   Set 
$$
\hI_0 \seteq \{ (i,p) \in I_0 \times \Z \ | \  p \equiv \epsilon_i \pmod{2} \}. 
$$

For each \emph{untwisted} quantum affine algebra $\uqpg$, we assign the finite simple Lie algebra $\gfin$ of symmetric type as follows:
\renewcommand{\arraystretch}{1.5}
\begin{align} \label{Table: root system}
\small
\begin{array}{|c||c|c|c|c|c|c|c|}
\hline 
 \g  & A_n^{(1)} \ (n \ge 1)  & B_n^{(1)}  \ (n \ge 2)  & C_n^{(1)}   \ (n \ge 3)  & D_n^{(1)}  \ (n \ge 4)   & E_{6,7,8}^{(1)} & F_{4}^{(1)}  & G_{2}^{(1)}   \\ \hline 
 \g_\fin  & A_n & A_{2n-1}    & D_{n+1}   &  D_n & E_{6,\,7,\,8} & E_{6} & D_{4}  \\
\hline  
\end{array}
\end{align}
The finite dimensional simple Lie algbera $\g_\fin$ corresponding to the untwisted affine Kac--Moody algebra $\g$
can be understood as an \defn{unfolding} of $\g_0$ since there exists a Dynkin diagram automorphism $\sigma= {\rm id}$, $\vee$, $\wvee$ on $\Dynkin^\gfin=(\Dynkin^{\gfin}_0,\Dynkin^{\gfin}_1)$ whose orbits yield
a Dynkin diagram $\Dynkin^{\g_0}=(\Dynkin^{\g_0}_0,\Dynkin^{\g_0}_1)$ of $\g_0$.

\begin{figure}[ht]
\begin{center}
\begin{tikzpicture}[xscale=1.25,yscale=.7]
\node (A2n1) at (-0.2,4.5) {$(\mathrm{A}_{2n-1}, \vee)$};
\node[dynkdot,label={below:\footnotesize$n+1$}] (A6) at (4,4) {};
\node[dynkdot,label={below:\footnotesize$n+2$}] (A7) at (3,4) {};
\node[dynkdot,label={below:\footnotesize$2n-2$}] (A8) at (2,4) {};
\node[dynkdot,label={below:\footnotesize$2n-1$}] (A9) at (1,4) {};
\node[dynkdot,label={above:\footnotesize$n-1$}] (A4) at (4,5) {};
\node[dynkdot,label={above:\footnotesize$n-2$}] (A3) at (3,5) {};
\node (Au) at (2.5, 5) {$\cdots$};
\node (Al) at (2.5, 4) {$\cdots$};
\node[dynkdot,label={above:\footnotesize$2$}] (A2) at (2,5) {};
\node[dynkdot,label={above:\footnotesize$1$}] (A1) at (1,5) {};
\node[dynkdot,label={above:\footnotesize$n$}] (A5) at (5,4.5) {};
\path[-]
 (A1) edge (A2)
 (A3) edge (A4)
 (A4) edge (A5)
 (A5) edge (A6)
 (A6) edge (A7)
 (A8) edge (A9);
\path[-] (A2) edge (Au) (Au) edge (A3) (A7) edge (Al) (Al) edge (A8);
\path[<->,thick,blue] (A1) edge (A9) (A2) edge (A8) (A3) edge (A7) (A4) edge (A6);
\path[->, thick, blue] (A5) edge [loop below] (A5);
\def\Foffset{6.5}
\node (Bn) at (-0.2,\Foffset) {$\mathrm{B}_n$};
\foreach \x in {1,2}
{\node[dynkdot,label={above:\footnotesize$\x$}] (B\x) at (\x,\Foffset) {};}
\node[dynkdot,label={above:\footnotesize$n-2$}] (B3) at (3,\Foffset) {};
\node[dynkdot,label={above:\footnotesize$n-1$}] (B4) at (4,\Foffset) {};
\node[dynkdot,label={above:\footnotesize$n$}] (B5) at (5,\Foffset) {};
\node (Bm) at (2.5,\Foffset) {$\cdots$};
\path[-] (B1) edge (B2) (B2) edge (Bm) (Bm) edge (B3) (B3) edge (B4);
\draw[-] (B4.30) -- (B5.150);
\draw[-] (B4.330) -- (B5.210);
\draw[-] (4.55,\Foffset) -- (4.45,\Foffset+.2);
\draw[-] (4.55,\Foffset) -- (4.45,\Foffset-.2);
\draw[-,dotted] (A1) -- (B1);
\draw[-,dotted] (A2) -- (B2);
\draw[-,dotted] (A3) -- (B3);
\draw[-,dotted] (A4) -- (B4);
\draw[-,dotted] (A5) -- (B5);
\draw[|->] (Bn) -- (A2n1);
\node (Dn1) at (-0.2,0) {$(\mathrm{D}_{n+1}, \vee)$};
\node[dynkdot,label={above:\footnotesize$1$}] (D1) at (1,0){};
\node[dynkdot,label={above:\footnotesize$2$}] (D2) at (2,0) {};
\node (Dm) at (2.5,0) {$\cdots$};
\node[dynkdot,label={above:\footnotesize$n-2$}] (D3) at (3,0) {};
\node[dynkdot,label={above:\footnotesize$n-1$}] (D4) at (4,0) {};
\node[dynkdot,label={above:\footnotesize$n$}] (D6) at (5,.5) {};
\node[dynkdot,label={below:\footnotesize$n+1$}] (D5) at (5,-.5) {};
\path[-] (D1) edge (D2)
  (D2) edge (Dm)
  (Dm) edge (D3)
  (D3) edge (D4)
  (D4) edge (D5)
  (D4) edge (D6);
\path[<->,thick,blue] (D6) edge (D5);
\path[->,thick,blue] (D1) edge [loop below] (D1)
(D2) edge [loop below] (D2)
(D3) edge [loop below] (D3)
(D4) edge [loop below] (D4);
\def\Coffset{1.8}
\node (Cn) at (-0.2,\Coffset) {$\mathrm{C}_n$};
\foreach \x in {1,2}
{\node[dynkdot,label={above:\footnotesize$\x$}] (C\x) at (\x,\Coffset) {};}
\node (Cm) at (2.5, \Coffset) {$\cdots$};
\node[dynkdot,label={above:\footnotesize$n-2$}] (C3) at (3,\Coffset) {};
\node[dynkdot,label={above:\footnotesize$n-1$}] (C4) at (4,\Coffset) {};
\node[dynkdot,label={above:\footnotesize$n$}] (C5) at (5,\Coffset) {};
\draw[-] (C1) -- (C2);
\draw[-] (C2) -- (Cm);
\draw[-] (Cm) -- (C3);
\draw[-] (C3) -- (C4);
\draw[-] (C4.30) -- (C5.150);
\draw[-] (C4.330) -- (C5.210);
\draw[-] (4.55,\Coffset+.2) -- (4.45,\Coffset) -- (4.55,\Coffset-.2);
\draw[-,dotted] (C1) -- (D1);
\draw[-,dotted] (C2) -- (D2);
\draw[-,dotted] (C3) -- (D3);
\draw[-,dotted] (C4) -- (D4);
\draw[-,dotted] (C5) -- (D6);
\draw[|->] (Cn) -- (Dn1);
\node (E6desc) at (6.8,4.5) {$(\mathrm{E}_6, \vee)$};
\node[dynkdot,label={above:\footnotesize$2$}] (E2) at (10.8,4.5) {};
\node[dynkdot,label={above:\footnotesize$4$}] (E4) at (9.8,4.5) {};
\node[dynkdot,label={above:\footnotesize$5$}] (E5) at (8.8,5) {};
\node[dynkdot,label={above:\footnotesize$6$}] (E6) at (7.8,5) {};
\node[dynkdot,label={below:\footnotesize$3$}] (E3) at (8.8,4) {};
\node[dynkdot,label={below:\footnotesize$1$}] (E1) at (7.8,4) {};
\path[-]
 (E2) edge (E4)
 (E4) edge (E5)
 (E4) edge (E3)
 (E5) edge (E6)
 (E3) edge (E1);
\path[<->,thick,blue] (E3) edge (E5) (E1) edge (E6);
\path[->, thick, blue] (E4) edge[loop below] (E4); 
\path[->, thick, blue] (E2) edge[loop below] (E2); 
\def\Foffset{6.5}
\node (F4desc) at (6.8,\Foffset) {$\mathrm{F}_4$};
\foreach \x in {1,2,3,4}
{\node[dynkdot,label={above:\footnotesize$\x$}] (F\x) at (\x+6.8,\Foffset) {};}
\draw[-] (F1.east) -- (F2.west);
\draw[-] (F3) -- (F4);
\draw[-] (F2.30) -- (F3.150);
\draw[-] (F2.330) -- (F3.210);
\draw[-] (9.35,\Foffset) -- (9.25,\Foffset+.2);
\draw[-] (9.35,\Foffset) -- (9.25,\Foffset-.2);
\draw[|->] (F4desc) -- (E6desc);
\path[-, dotted] (F1) edge (E6)
(F2) edge (E5) (F3) edge (E4) (F4) edge (E2);

\node (D4desc) at (6.8,0) {$(\mathrm{D}_{4}, \widetilde{\vee})$};
\node[dynkdot,label={above:\footnotesize$1$}] (D1) at (7.8,.6){};
\node[dynkdot,label={above:\footnotesize$2$}] (D2) at (8.8,0) {};
\node[dynkdot,label={left:\footnotesize$3$}] (D3) at (7.8,0) {};
\node[dynkdot,label={below:\footnotesize$4$}] (D4) at (7.8,-.6) {};
\draw[-] (D1) -- (D2);
\draw[-] (D3) -- (D2);
\draw[-] (D4) -- (D2);
\path[->,blue,thick]
(D1) edge [bend left=0] (D3)
(D3) edge [bend left=0](D4)
(D4) edge[bend left=90] (D1);
\path[->, thick, blue] (D2) edge[loop below] (D2); 
\def\Goffset{1.8}
\node (G2desc) at (6.8,\Goffset) {$\mathrm{G}_2$};
\node[dynkdot,label={above:\footnotesize$1$}] (G1) at (7.8,\Goffset){};
\node[dynkdot,label={above:\footnotesize$2$}] (G2) at (8.8,\Goffset) {};
\draw[-] (G1) -- (G2);
\draw[-] (G1.40) -- (G2.140);
\draw[-] (G1.320) -- (G2.220);
\draw[-] (8.25,\Goffset+.2) -- (8.35,\Goffset) -- (8.25,\Goffset-.2);
\draw[|->] (G2desc) -- (D4desc);
\path[-, dotted] (D1) edge (G1) (D2) edge (G2);
\end{tikzpicture}
\end{center}
\caption{The Dynkin diagram and automorphism $(\bDynkin,\sigma)$ for non-simply-laced $\g_0$.} \label{Fig:unf}
\end{figure}

For the notational simplicity, we set $\bDynkin \seteq \Dynkin^\gfin$ for untwisted affine Kac--Moody algebra $\g$. Thus
(i) $\bDynkin =\Dynkin^{\g_0}$ and $\sigma = {\rm id}$
if $\g$ is symmetric, and (ii)  $\bDynkin \ne \Dynkin^{\g_0}$ and $\sigma \ne {\rm id}$ otherwise. Hence (a) we can associate
the pair $(\bDynkin,\sigma)$ for each untwisted affine Kac--Moody algebra $\g$, and (b)  
the index set $I_0$ of $\g_0$ can be considered as an orbits  $I_0=\{i,j,\ldots\}$ of $\bDynkin_0= \{\im,\jm,\ldots\}$ under the action $\sigma$.
Hence we understand $I_0 \ni i = \oim$ the orbit of $\im \in \bDynkin_0$.

\begin{definition} \label{def: Q-datum}
A \defn{height function} on $(\bDynkin,\sigma)$ of $\g$ is a function $\xi\colon \bDynkin_0 \to \Z$
satisfying the following conditions (here we write $\xi_\im \seteq \xi(\im)$):
\bnum
\item Let $\im,\jm \in \bDynkin_0$ with $d_{\bDynkin}(\im,\jm)=1$ and $d_\oim = d_\ojm$. Then we have
$|\xi_\im-\xi_\jm|=d_\oim=d_\ojm$. 
\item Let $i,j \in I_0$ with $d_{\Dynkin}(i,j)=1$ and $d_i =1 < d_j =r$. Then there exists a unique 
$\jm \in j$ such that $|\xi_\im-\xi_\jm|$ and $\xi_{\sigma^k(\jm)} = \xi_\jm +2k$  
for any $1 \le k < r$, where $i=\{ \im \}$. 
\ee
Such triple $\calQ = (\bDynkin,\sigma,\xi)$ is called a \defn{$\rmQ$-datum} for $\g$. 
\end{definition}

For $\im , \jm \in \bDynkin_0$, we define
$$
\td_{\bDynkin,\sigma}(\im,\jm) \seteq \sum_{k=1}^{s} \min(d_{\overline{\im_k}},d_{\overline{\im_{k+1}}})
$$
where $(\im=\im_1,\im_2,\ldots,\im_{s+1}=\jm)$ is a path from $\im$ to $\jm$ in $\bDynkin$ with $s=d_{\bDynkin}(\im,\jm)$. Note that it is well-defined
since there exists only one such path in $\bDynkin$.  Note that $\td_{\bDynkin,\sigma}(\im,\jm) \ne d_{\bDynkin}(\im,\jm) $ only if $\g=B_n^{(1)},F_4^{(1)}$. For instance $\g=B_3^{(1)}$ and $(\im,\jm)=(1,5)$, we have $\gfin=A_5$ and $\td_{\bDynkin,\sigma}(\im,\jm)=6$.

Let $\calQ = (\bDynkin,\sigma,\xi)$ be a $\rmQ$-datum for $\g$. 
A vertex $\im \in \Dynkin_0$
is called a \defn{sink} of $\calQ$ if we have $\xi_\im < \xi_\jm$ for any $\jm \in \bDynkin_0$
with $d_{\bDynkin}(\im,\jm)=1$. When $\im$ is a sink of $\calQ$, we define a new $\rmQ$-datum
$\sfs_\im \calQ = (\bDynkin,\sigma,\sfs_\im\xi)$ of $\g$ with   
$$
(\sfs_\im \xi)_\jm \seteq \xi_\jm + 2d_\oim \delta_{\im,\jm} \quad \text{ for any } \jm \in \bDynkin_0.
$$

For a sequence $\uii=(\im_1,\im_2,\ldots,\im_r)$   in $\bDynkin_0$,
 $\uii$ is said to be \defn{reduced} if $w^{\uii} \seteq s_{\im_1} \cdots s_{\im_r}$ be a reduced expression of an element $w \in W_\gfin$. 
For a $\rmQ$-datum $\calQ$ of $\g$, $\uii$ is said to be \defn{$\calQ$-adapted} or \defn{adapted to $\calQ$} if $\im_k$ is a sink of
the $\calQ$-datum $\sfs_{\im_{k-1}}\cdots \sfs_{\im_2}\sfs_{\im_1}\calQ$ for all $1 \le k \le r$.

Then the following are known (see~\cite{FO21} for more details):
\bna
\item For each $\rmQ$-datum $\calQ=\Qdatum$, there exist reduced expressions $\uw_0$ of $w_0$ adapted to $\calQ$. 
Also all reduced expressions adapted to $\calQ$ form a  commutation class, denoted by $[\calQ]$. 
\item For each $\rmQ$-datum $\calQ=\Qdatum$, there exists a unique element $\tau_\calQ \in W_\gfin \rtimes \Ang{\sigma}$ satisfying certain properties,
which we call the \defn{$\calQ$-Coxeter element}. 
\ee
Note that for each $\calQ$-Coxeter element $\tau_Q$, there exists a unique height function $\xi$ up to $\Z$.

Let $\xi$ be a height function on $(\bDynkin,\sigma)$. We define a quiver $\hbDynkin^\sigma =(\hbDynkin^\sigma_0,\hbDynkin^\sigma_1)$ as follows:
\begin{align*}
\hbDynkin^\sigma_0 & = \{ (\im,p) \in \bDynkin_0 \times \Z \ | \  p - \xi_\im \in 2d_\oim\Z \},  \allowdisplaybreaks\\
\hbDynkin^\sigma_1 & = \{ (\im,p) \to (\jm,s) \ | \ (\im,p),(\jm,s) \in \hbDynkin^\sigma_0, \ d(\im,\jm)=1, \ s-p =\min(d_\oim,d_\ojm) \}. 
\end{align*}
We call $\hbDynkin^\sigma$ the \defn{repetition quiver} of $\g$.

Recall we have fixed a parity function $\epsilon: I_0 \to \{0,1\}$ for $\g_0$ and defined the set $\hI_0$. Hereafter, we always assume that a height function $\xi$ on $(\bDynkin,\sigma)$
always satisfies the condition 
$$
\xi_\im \equiv \epsilon_\oim \pmod{2} \quad \text{ for any } \im \in \bDynkin_0.  
$$

\begin{example} \label{ex: Repetition quiver}
Here are some examples of the repetition quiver $\hbDynkin^\sigma$.
\bnum
\item For $\g$ of type $A^{(1)}_{5}$, we have $\gfin=A_5$, $\sigma={\rm id}$ and 
the repetition quiver $\hbDynkin^{{\rm id}}$ is depicted as:
$$
\raisebox{3mm}{
\scalebox{0.65}{\xymatrix@!C=0.5mm@R=2mm{
(\im\setminus p) & -8 & -7 & -6 &-5&-4 &-3& -2 &-1& 0 & 1& 2 & 3& 4&  5
& 6 & 7 & 8 & 9 & 10 & 11 & 12 & 13 & 14 & 15 & 16 & 17& 18 \\
1&\bullet \ar@{->}[dr]&& \bullet \ar@{->}[dr] &&\bullet\ar@{->}[dr]
&&\bullet \ar@{->}[dr] && \bullet \ar@{->}[dr] &&\bullet \ar@{->}[dr] &&  \bullet \ar@{->}[dr]
&&\bullet \ar@{->}[dr] && \bullet \ar@{->}[dr] &&\bullet \ar@{->}[dr]  && \bullet\ar@{->}[dr] &&
\bullet\ar@{->}[dr] && \bullet\ar@{->}[dr]  && \bullet\\
2&&\bullet \ar@{->}[dr]\ar@{->}[ur]&& \bullet \ar@{->}[dr]\ar@{->}[ur] &&\bullet \ar@{->}[dr]\ar@{->}[ur]
&& \bullet \ar@{->}[dr]\ar@{->}[ur]&& \bullet\ar@{->}[dr] \ar@{->}[ur]&& \bullet \ar@{->}[dr]\ar@{->}[ur]&&\bullet \ar@{->}[dr]\ar@{->}[ur]&
&\bullet \ar@{->}[dr]\ar@{->}[ur]&&\bullet\ar@{->}[dr] \ar@{->}[ur]&& \bullet \ar@{->}[dr]\ar@{->}[ur]
&&\bullet \ar@{->}[dr]\ar@{->}[ur]&& \bullet \ar@{->}[dr] \ar@{->}[ur]&&\bullet \ar@{->}[dr]\ar@{->}[ur] & \\
3&\bullet \ar@{->}[dr] \ar@{->}[ur]&& \bullet \ar@{->}[dr] \ar@{->}[ur] &&\bullet\ar@{->}[dr] \ar@{->}[ur]
&&\bullet \ar@{->}[dr] \ar@{->}[ur] && \bullet \ar@{->}[dr]\ar@{->}[ur] &&\bullet \ar@{->}[dr] \ar@{->}[ur]&&  \bullet \ar@{->}[dr] \ar@{->}[ur]
&&\bullet \ar@{->}[dr] \ar@{->}[ur] && \bullet \ar@{->}[dr] \ar@{->}[ur]&&\bullet \ar@{->}[dr] \ar@{->}[ur] && \bullet\ar@{->}[dr] \ar@{->}[ur]&&
\bullet\ar@{->}[dr] \ar@{->}[ur] && \bullet\ar@{->}[dr] \ar@{->}[ur]  && \bullet\\
4&& \bullet \ar@{->}[ur]\ar@{->}[dr]&&\bullet \ar@{->}[ur]\ar@{->}[dr]&&\bullet \ar@{->}[ur]\ar@{->}[dr] &&\bullet \ar@{->}[ur]\ar@{->}[dr]&& \bullet \ar@{->}[ur]\ar@{->}[dr]
&&\bullet \ar@{->}[ur]\ar@{->}[dr]&& \bullet \ar@{->}[ur]\ar@{->}[dr] &&\bullet \ar@{->}[ur]\ar@{->}[dr]&&\bullet \ar@{->}[ur]\ar@{->}[dr]&&
\bullet \ar@{->}[ur]\ar@{->}[dr]&&\bullet \ar@{->}[ur]\ar@{->}[dr]&&\bullet \ar@{->}[ur]\ar@{->}[dr]&&\bullet\ar@{->}[ur]\ar@{->}[dr]\\
5&\bullet  \ar@{->}[ur]&& \bullet  \ar@{->}[ur] &&\bullet \ar@{->}[ur]
&&\bullet  \ar@{->}[ur] && \bullet \ar@{->}[ur] &&\bullet  \ar@{->}[ur]&&  \bullet  \ar@{->}[ur]
&&\bullet  \ar@{->}[ur] && \bullet  \ar@{->}[ur]&&\bullet  \ar@{->}[ur] && \bullet \ar@{->}[ur]&&
\bullet \ar@{->}[ur] && \bullet \ar@{->}[ur]  && \bullet }}}
$$
\item \label{it: Rq B3}
For $\g$ of type $B^{(1)}_{3}$, we have $ \gfin =A_{5}$, $\sigma=\vee$ and the repetition quiver $\hbDynkin^\vee$ is depicted as:
$$\raisebox{3.1em}{\scalebox{0.65}{\xymatrix@!C=0.1ex@R=0.5ex{
(\im\setminus p) & -8 & -7 & -6 &-5&-4 &-3& -2 &-1& 0 & 1& 2 & 3& 4&  5
& 6 & 7 & 8 & 9 & 10 & 11 & 12 & 13 & 14 & 15 & 16 & 17& 18 \\
1&&&\bullet \ar@{->}[ddrr]&&&& \bullet \ar@{->}[ddrr]&&&&  \bullet \ar@{->}[ddrr] &&&& \bullet\ar@{->}[ddrr]
&&&& \bullet \ar@{->}[ddrr]&&&& \bullet \ar@{->}[ddrr] &&&& \bullet \\ \\
2&\bullet \ar@{->}[dr] \ar@{->}[uurr]&&&&\bullet\ar@{->}[dr] \ar@{->}[uurr]
&&&& \bullet \ar@{->}[dr]\ar@{->}[uurr] &&&&  \bullet \ar@{->}[dr] \ar@{->}[uurr]
&&&& \bullet \ar@{->}[dr]\ar@{->}[uurr]&&&& \bullet \ar@{->}[dr]\ar@{->}[uurr]&&&& \bullet \ar@{->}[dr]\ar@{->}[uurr] \\
3&& \bullet \ar@{->}[dr]&& \bullet \ar@{->}[ur] &&\bullet \ar@{->}[dr] && \bullet \ar@{->}[ur] && \bullet \ar@{->}[dr]
&& \bullet \ar@{->}[ur] && \bullet \ar@{->}[dr] &&
\bullet \ar@{->}[ur]&&\bullet \ar@{->}[dr] &&\bullet \ar@{->}[ur]&&\bullet \ar@{->}[dr] &&\bullet \ar@{->}[ur]
&&\bullet \ar@{->}[dr]\\
4&&&\bullet\ar@{->}[ddrr]\ar@{->}[ur]&&&&\bullet\ar@{->}[ddrr]\ar@{->}[ur] &&&&  \bullet \ar@{->}[ddrr]\ar@{->}[ur] &&&&
\bullet \ar@{->}[ddrr] \ar@{->}[ur]
&&&& \bullet \ar@{->}[ddrr]\ar@{->}[ur]&&&& \bullet\ar@{->}[ddrr]\ar@{->}[ur] &&&& \bullet \\ \\
5& \bullet  \ar@{->}[uurr]&&&&\bullet \ar@{->}[uurr] &&&& \bullet \ar@{->}[uurr]
&&&& \bullet   \ar@{->}[uurr] &&&& \bullet \ar@{->}[uurr]
&&&&\bullet \ar@{->}[uurr] &&&& \bullet \ar@{->}[uurr]}}}
$$
\item For $\g$ of type $C^{(1)}_{4}$, we have $ \gfin =D_{5}$, $\sigma=\vee$ and the repetition quiver $\hbDynkin^\vee$ is depicted as:
$$
\raisebox{3mm}{
\scalebox{0.65}{\xymatrix@!C=0.1ex@R=0.5ex{
(\im\setminus p) & -8 & -7 & -6 &-5&-4 &-3& -2 &-1& 0 & 1& 2 & 3& 4&  5
& 6 & 7 & 8 & 9 & 10 & 11 & 12 & 13 & 14 & 15 & 16 & 17& 18 \\
1&\bullet \ar@{->}[dr]&& \bullet \ar@{->}[dr] &&\bullet\ar@{->}[dr]
&&\bullet \ar@{->}[dr] && \bullet \ar@{->}[dr] &&\bullet \ar@{->}[dr] &&  \bullet \ar@{->}[dr]
&&\bullet \ar@{->}[dr] && \bullet \ar@{->}[dr] &&\bullet \ar@{->}[dr]  && \bullet\ar@{->}[dr] &&
\bullet\ar@{->}[dr] && \bullet\ar@{->}[dr]  && \bullet\\
2&&\bullet \ar@{->}[dr]\ar@{->}[ur]&& \bullet \ar@{->}[dr]\ar@{->}[ur] &&\bullet \ar@{->}[dr]\ar@{->}[ur]
&& \bullet \ar@{->}[dr]\ar@{->}[ur]&& \bullet\ar@{->}[dr] \ar@{->}[ur]&& \bullet \ar@{->}[dr]\ar@{->}[ur]&&\bullet \ar@{->}[dr]\ar@{->}[ur]&
&\bullet \ar@{->}[dr]\ar@{->}[ur]&&\bullet\ar@{->}[dr] \ar@{->}[ur]&& \bullet \ar@{->}[dr]\ar@{->}[ur]
&&\bullet \ar@{->}[dr]\ar@{->}[ur]&& \bullet \ar@{->}[dr] \ar@{->}[ur]&&\bullet \ar@{->}[dr]\ar@{->}[ur] & \\
3&\bullet \ar@{->}[dr] \ar@{->}[ur]&& \bullet \ar@{->}[ddr] \ar@{->}[ur] &&\bullet\ar@{->}[dr] \ar@{->}[ur]
&&\bullet \ar@{->}[ddr] \ar@{->}[ur] && \bullet \ar@{->}[dr]\ar@{->}[ur] &&\bullet \ar@{->}[ddr] \ar@{->}[ur]&&  \bullet \ar@{->}[dr] \ar@{->}[ur]
&&\bullet \ar@{->}[ddr] \ar@{->}[ur] && \bullet \ar@{->}[dr] \ar@{->}[ur]&&\bullet \ar@{->}[ddr] \ar@{->}[ur] && \bullet\ar@{->}[dr] \ar@{->}[ur]&&
\bullet\ar@{->}[ddr] \ar@{->}[ur] && \bullet\ar@{->}[dr] \ar@{->}[ur]  && \bullet\\
4&& \bullet \ar@{->}[ur]&&&&\bullet \ar@{->}[ur] &&&& \bullet \ar@{->}[ur]
&&&& \bullet \ar@{->}[ur] &&&&\bullet \ar@{->}[ur]&&&&\bullet \ar@{->}[ur]&&
&&\bullet \ar@{->}[ur]\\
5&&&&\bullet \ar@{->}[uur]&&&&\bullet \ar@{->}[uur]&&&&  \bullet \ar@{->}[uur]&&&&
\bullet \ar@{->}[uur]&&&& \bullet \ar@{->}[uur]&&&& \bullet \ar@{->}[uur]&& }}}
$$
\item
 For $\g$ of type $D^{(1)}_{4}$, we have $\gfin=D_4$, $\sigma={\rm id}$ and 
the repetition quiver $\hbDynkin^\sigma$ is depicted as:
$$
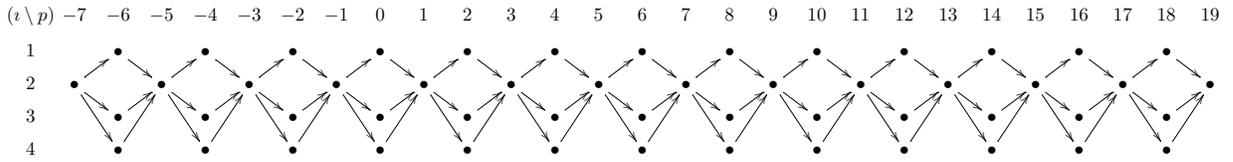
\raisebox{3mm}{
\scalebox{0.65}{\xymatrix@!C=0.5mm@R=2mm{
(\im\setminus p) & -7 & -6 &-5&-4 &-3& -2 &-1& 0 & 1& 2 & 3& 4&  5
& 6 & 7 & 8 & 9 & 10 & 11 & 12 & 13 & 14 & 15 & 16 & 17& 18 & 19 \\
1&&\bullet \ar@{->}[dr]&& \bullet \ar@{->}[dr]&&\bullet \ar@{->}[dr]
&& \bullet \ar@{->}[dr]&& \bullet\ar@{->}[dr] && \bullet \ar@{->}[dr]&&\bullet \ar@{->}[dr]&
&\bullet \ar@{->}[dr]&&\bullet\ar@{->}[dr]&& \bullet \ar@{->}[dr]
&&\bullet \ar@{->}[dr]&& \bullet \ar@{->}[dr]&&\bullet \ar@{->}[dr] & \\
2&\bullet \ar@{->}[dr] \ar@{->}[ur]&& \bullet \ar@{->}[dr] \ar@{->}[ur] &&\bullet\ar@{->}[dr] \ar@{->}[ur]
&&\bullet \ar@{->}[dr] \ar@{->}[ur] && \bullet \ar@{->}[dr]\ar@{->}[ur] &&\bullet \ar@{->}[dr] \ar@{->}[ur]&&  \bullet \ar@{->}[dr] \ar@{->}[ur]
&&\bullet \ar@{->}[dr] \ar@{->}[ur] && \bullet \ar@{->}[dr] \ar@{->}[ur]&&\bullet \ar@{->}[dr] \ar@{->}[ur] && \bullet\ar@{->}[dr] \ar@{->}[ur]&&
\bullet\ar@{->}[dr] \ar@{->}[ur] && \bullet\ar@{->}[dr] \ar@{->}[ur]  && \bullet\\
3&& \bullet \ar@{->}[ur]&&\bullet \ar@{->}[ur]&&\bullet \ar@{->}[ur] &&\bullet \ar@{->}[ur]&& \bullet \ar@{->}[ur]
&&\bullet \ar@{->}[ur]&& \bullet \ar@{->}[ur] &&\bullet \ar@{->}[ur]&&\bullet \ar@{->}[ur]&&
\bullet \ar@{->}[ur]&&\bullet \ar@{->}[ur]&&\bullet \ar@{->}[ur]&&\bullet\ar@{->}[ur]&\\
4&& \bullet \ar@{<-}[uul]\ar@{->}[uur]&&\bullet \ar@{<-}[uul]\ar@{->}[uur]&&\bullet \ar@{<-}[uul]\ar@{->}[uur] &&\bullet \ar@{<-}[uul]\ar@{->}[uur]&& \bullet \ar@{<-}[uul]\ar@{->}[uur]
&&\bullet \ar@{<-}[uul]\ar@{->}[uur]&& \bullet \ar@{<-}[uul]\ar@{->}[uur] &&\bullet \ar@{<-}[uul]\ar@{->}[uur]&&\bullet \ar@{<-}[uul]\ar@{->}[uur]&&
\bullet \ar@{<-}[uul]\ar@{->}[uur]&&\bullet \ar@{<-}[uul]\ar@{->}[uur]&&\bullet \ar@{<-}[uul]\ar@{->}[uur]&&\bullet\ar@{<-}[uul]\ar@{->}[uur]&
}}}
$$
\item
 For $\g$ of type $G^{(1)}_{2}$, we have $ \gfin =D_{4}$, $\sigma=\widetilde{\vee}$ and the repetition quiver $\hbDynkin^{\widetilde{\vee}}$ is depicted as:
$$
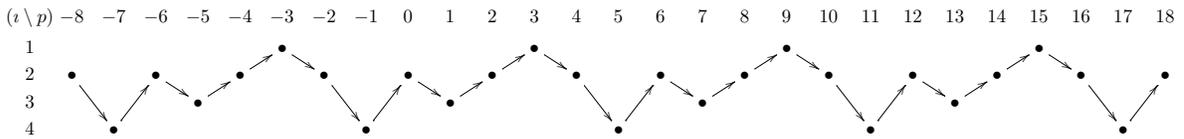
\raisebox{3mm}{
\scalebox{0.65}{\xymatrix@!C=0.1ex@R=0.5ex{
(\im\setminus p) & -8 & -7 & -6 &-5&-4 &-3& -2 &-1& 0 & 1& 2 & 3& 4&  5
& 6 & 7 & 8 & 9 & 10 & 11 & 12 & 13 & 14 & 15 & 16 & 17& 18 \\
1&&&&&&\bullet \ar@{->}[dr]&&&&&&\bullet \ar@{->}[dr]&&&
&&&\bullet \ar@{->}[dr]&&&&&&\bullet \ar@{->}[dr]&&& \\
2&\bullet \ar@{->}[ddr]&&\bullet \ar@{->}[dr]&&\bullet\ar@{->}[ur] &&
\bullet \ar@{->}[ddr]&&\bullet \ar@{->}[dr]&&\bullet\ar@{->}[ur] &&
\bullet \ar@{->}[ddr]&&\bullet \ar@{->}[dr]&&\bullet\ar@{->}[ur] &&
\bullet \ar@{->}[ddr]&&\bullet \ar@{->}[dr]&&\bullet\ar@{->}[ur] &&
\bullet \ar@{->}[ddr]&&\bullet \\
3&&&&\bullet \ar@{->}[ur] &&
&&&&\bullet \ar@{->}[ur] &&
&&&&\bullet \ar@{->}[ur] &&
&&&&\bullet \ar@{->}[ur] &&
&&&\\
4&& \bullet \ar@{->}[uur]&&&&&&
\bullet \ar@{->}[uur]&&&&&&
\bullet \ar@{->}[uur]&&&&&&
\bullet \ar@{->}[uur]&&&&&&
\bullet \ar@{->}[uur]& \\
}}}
$$
\ee
\end{example}

A sequence $\frakR = \seq{ (\im_k,p_k) }_{k \in \Z}$ of elements in $\hbDynkin^\sigma_0$ is called a \defn{compatible reading} of $\hbDynkin^\sigma$
if it satisfies the following conditions (\cite[\S 3.3]{KO23}):
\bnum 
\item $(\im_k,p_k) = (\im_l,p_l)$ if and only if $k=l$.
\item $\{ (\im_k,p_k) \}_{k \in \Z} =\hbDynkin^\sigma_0$.
\item We have $k<l$ whenever there is an arrow $(\im_k,p_k) \to (\im_l,p_l)$ in $\hbDynkin^\sigma$. 
\ee
Understanding $\hbDynkin^\sigma$ as a Hasse quiver on $\hbDynkin^\sigma_0$ via the partial order $\preceq$ determined by arrows, 
each compatible reading $\frakR$ can be interpreted as a completion of the partial order $\preceq$.

\begin{convention} \label{conv: identify R}
Throughout this paper, we frequently identify a compatible reading $\frakR$ with the bijection 
\[
\pR{\frakR} \colon \Z \to \frakR \qquad \text{defined by } \Z \ni k \mapsto (\im_k,p_k) \in\hbDynkin^{\sigma}_0.
\]
\end{convention}

\begin{example} \label{ex: c reading}
In this example, we present several compatible readings of $\hbDynkin^\sigma$ by using the $\hbDynkin^\sigma$ of untwisted affine 
type $B^{(1)}_3$ in Example~\ref{ex: Repetition quiver}~\eqref{it: Rq B3}.
\bna 
\item \label{it: se reading} By reading $\hbDynkin^\sigma$ from north west to south east, we obtain a compatible reading 
$$
\seq{\ldots, (1,-6),(2,-4),(3,-3),(4,-2),(5,0), (3,-1), (1,-2),(2,0),(3,1),(4,2),(5,4),\ldots}.
$$
\item \label{it: ne reading}
By reading $\hbDynkin^\sigma$ from south west to north east,  we obtain a compatible reading 
$$
\seq{\ldots, (5,-4),(4,-2),(3,-1),(2,0),(1,2), (3,1), (5,0),(4,2),(3,3),(2,4),(1,6),\ldots}.
$$
\item By reading $\hbDynkin^\sigma$ from north to south,  we obtain a compatible reading 
$$
\seq{\ldots, (2,-4),(5,-4),(3,-3),(1,-2),(4,-2), (3,-1), (2,0),(5,0),(3,1),(1,2),(4,2),\ldots}.
$$
\ee
\end{example}

Note that we have the bijection $\sff \colon \hbDynkin^\sigma_0 \to \hI_0$ given by $(\im,p)\longmapsto(\oim,p)$. 
Hence we can give a quiver structure on $\hI_0$ by using the ones of $\hbDynkin^\sigma$ and we sometimes identify 
$\hbDynkin^\sigma$ (resp.\ $\hbDynkin^\sigma_0$) and $\hI_0$ as quivers (resp.\ as vertices).

For each $\rmQ$-datum $\calQ=(\bDynkin,\sigma,\xi)$ of $\g$, there exists a unique bijection
\[
\phi_\calQ \colon \hbDynkin^\sigma_0 \to \Phi^+_\gfin \times \Z,
\]
which is defined by using $\tau_\calQ$ (see~\cite{HL15,FO21}).

Let $\Qdatum$ be a $\rmQ$-datum for $\g$. We set $\Gamma^\calQ=(\Gamma^\calQ_0,\Gamma^\calQ_1)$ the full-subquiver of $\hbDynkin^\sigma$ whose set $\Gamma^\calQ_0$ of vertices is given as follows:
\begin{align} \label{eq: ga vertex}
\Gamma^\calQ_0 \seteq \phi_\calQ^{-1}(\Phi^+_\gfin \times \{ 0 \}) \subset \hbDynkin^\sigma_0.     
\end{align}
We call $\Gamma^\calQ$ the \defn{Auslander--Reiten $($AR$)$ quiver} associated with $\calQ$.
Hence each vertex in $\Gamma^\calQ$ corresponds to a positive root $\be \in \Phi^+_\gfin$. 

Throughout this paper, to give an example of $\Gamma^\calQ$, we take $\calQ=\Qdatum$ and hence a $\calQ$-Coxeter element $\tau_\calQ \in W_\gfin  \rtimes \Ang{\sigma}$ for an untwisted affine type $\g=X_n^{(1)}$ as follows:
\begin{align} \label{eq: choice tau}
\tau_\calQ = s_{1} s_{2} \cdots s_{n-1} s_{n}\sigma \qquad \text{(resp.\ } \tau = s_{1} s_{3} s_4 s_2 \vee \text{ if } \g = F_4^{(1)}).
\end{align}

\begin{example}
In the following quivers, (i) the vertex $\al \in \Phi^+_{\gfin}$ corresponds to  $(\al,0)$ as in ~\eqref{eq: ga vertex}, 
(ii) $\xymatrix@C=2ex@R=1ex{ \al \ar@{.>}[rr]^{ \tau^k_\calQ} && \be  }$ denotes
$\tau^k_\calQ(\be)=\al$, whose arrow is \emph{not} contained in $\Gamma^\calQ_1$.  
\ben
\item For $\g$ of type $A_5^{(1)}$ with  $\tau_\calQ=s_1s_2s_3s_4s_5 \in W_{A_5}$ and 
$\xi=(\xi_1=-4,\xi_2=-3,\xi_3=-2,\xi_4=-1,\xi_5=0)$, $\Gamma^\calQ$ in $\hbDynkin^\sigma_0$
can be depicted as follows:
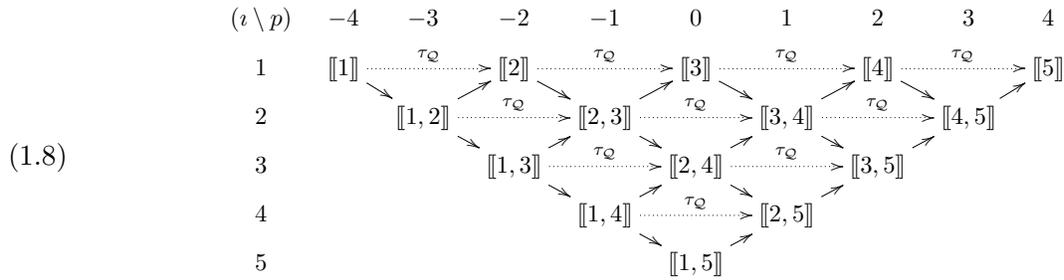
\begin{align}\label{eq: AR-quiver A}
\raisebox{5em}{\scalebox{0.84}{\xymatrix@C=2ex@R=1ex{
(\im \setminus p)& -4 & -3 & -2 & -1 & 0 & 1 & 2 & 3 & 4 \\
1&\drange{1}\ar@{->}[dr] \tauarrow && \drange{2}\ar@{->}[dr]\tauarrow & & \drange{3}\ar@{->}[dr]\tauarrow & &\drange{4} \ar@{->}[dr]\tauarrow & &\drange{5} \\
2&&\drange{1,2}\ar@{->}[ur]\ar@{->}[dr] \tauarrow  & &\drange{2,3} \ar@{->}[ur]\ar@{->}[dr]\tauarrow&& \drange{3,4} \ar@{->}[ur]\ar@{->}[dr]\tauarrow&& \drange{4,5} \ar@{->}[ur] \\
3&&& \drange{1,3}\ar@{->}[ur]\ar@{->}[dr] \tauarrow&& \drange{2,4}\tauarrow\ar@{->}[ur]\ar@{->}[dr] && \drange{3,5} \ar@{->}[ur]\\
4&&&&  \drange{1,4} \ar@{->}[ur]\ar@{->}[dr]  \tauarrow && \drange{2,5} \ar@{->}[ur] \\
5&&&&&\drange{1,5} \ar@{->}[ur]}}}
\end{align}
Here $\drange{a,b}$ ($1 \le a,b \le 5$) denotes the positive root $\sum_{k=a}^b \al_k \in \Phi^+_{A_5}$ and $\drange{a} \seteq \al_a$.
\item For $\g$ of type $B_3^{(1)}$ with $\tau_\calQ=s_1s_2s_3\vee \in W_{A_5} \rtimes \Ang{\vee}$ and $\xi=(\xi_1=-6,\xi_2=-4,\xi_3=-3,\xi_4=-2,\xi_5=-4)$, 
$\Gamma^\calQ$  in $\hbDynkin^\sigma_0$ can be depicted as follows:
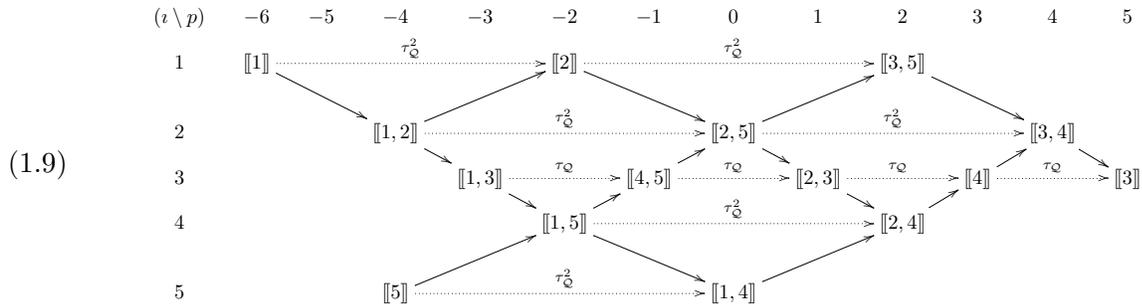
\begin{equation}\label{eq: B3 AR}
 \raisebox{5.3em}{\scalebox{0.70}{\xymatrix@C=3ex@R=1.5ex{
(\im \setminus p) & -6 & -5 & -4 & -3 & -2 & -1 & 0 & 1 & 2 & 3 & 4 & 5\\
1& \drange{1}\ar@{->}[ddrr] \tauarrows
&&&& \drange{2}\ar@{->}[ddrr] \tauarrows
&&&& \drange{3,5} \ar@{->}[ddrr] \\ \\
2&&& \drange{1,2} \ar@{->}[dr]\ar@{->}[uurr] \tauarrows
&&&& \drange{2,5} \ar@{->}[dr]\ar@{->}[uurr] \tauarrows
&&&& \drange{3,4}  \ar@{->}[dr] \\
3&&&& \drange{1,3} \ar@{->}[dr] \tauarrow
&& \drange{4,5} \ar@{->}[ur] \tauarrow
&& \drange{2,3} \ar@{->}[dr] \tauarrow
&& \drange{4}\ar@{->}[ur] \tauarrow
&& \drange{3}   \\
4 &&&&& \drange{1,5} \ar@{->}[ur]\ar@{->}[ddrr] \tauarrows  
&&&&  \drange{2,4} \ar@{->}[ur] 
&&  \\ \\ 
5&&& \drange{5}  \ar@{->}[uurr]  \tauarrows
&&&&  \drange{1,4}  \ar@{->}[uurr] \\
}}}
\end{equation}
\item For $\g$ of type $C_3^{(1)}$ with $\tau_\calQ=s_1s_2s_3\vee \in W_{D_4} \rtimes \Ang{\vee}$ and  $\xi=(\xi_1=-4,\xi_2=-3,\xi_3=0,\xi_4=-2)$, 
$\Gamma^\calQ$  in $\hbDynkin^\sigma_0$ can be depicted as follows:
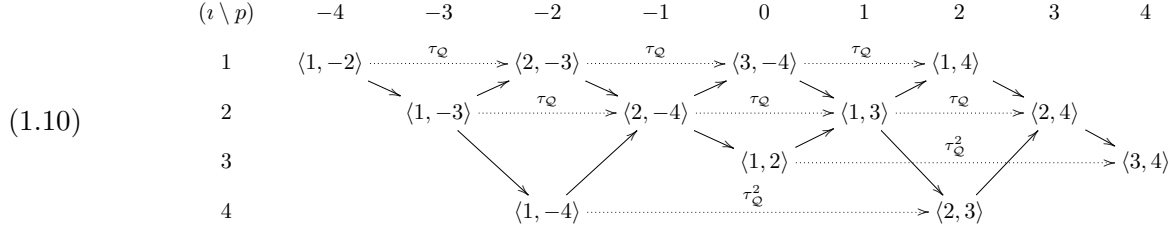
\begin{align}\label{eq: AR-quiver C}
\raisebox{4em}{\scalebox{0.77}{\xymatrix@C=3ex@R=1.5ex{
(\im\setminus p)& -4 & -3 & -2 & -1 & 0 & 1 & 2 & 3 & 4     \\
1& \lan 1,-2\ran \ar@{->}[dr] \tauarrow && \lan 2,-3\ran  \ar@{->}[dr] \tauarrow && \lan 3,-4\ran  \ar@{->}[dr] \tauarrow && \lan 1,4\ran  \ar@{->}[dr] \ \\
2&& \lan 1,-3\ran  \ar@{->}[ur]\ar@{->}[ddr] \tauarrow && \lan 2,-4\ran  \ar@{->}[ur]\ar@{->}[dr]  \tauarrow
 &&\lan 1,3\ran  \ar@{->}[ur]\ar@{->}[ddr]  \tauarrow && \lan 2,4\ran  \ar@{->}[dr]\\
3&&&&& \lan1 ,2\ran  \ar@{->}[ur]\tauarrows &&&& \lan 3,4\ran      \\
4&&& \lan 1,-4 \ran  \ar@{->}[uur] \tauarrows &&&& \lan 2,3\ran  \ar@{->}[uur]
}}}
\end{align}
Here $\lan a,\pm b \ran$ denotes $\varepsilon_a \pm \varepsilon_b \in \Phi^+_{D_4}$
by following the convention in \cite{Bou}.
\item For $\g$ of type $D_4^{(1)}$ with $\tau_\calQ=s_1s_2s_3s_4 \in W_{D_4}$ and 
 $\xi=(\xi_1=-2,\xi_2=-1,\xi_3=0,\xi_4=0)$, $\Gamma^\calQ$  in $\hbDynkin^\sigma_0$
can be depicted as follows:
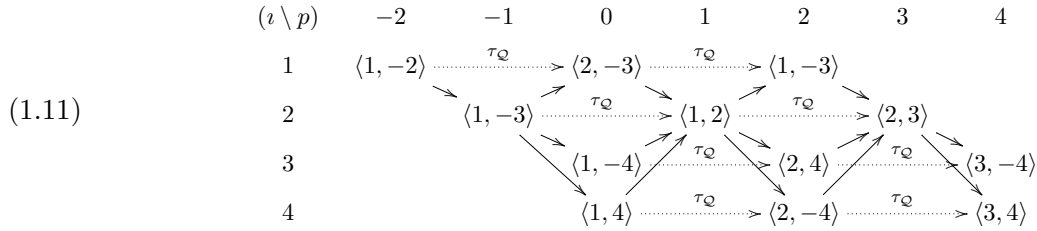
\begin{align}\label{eq: AR-quiver D}
\raisebox{3.5em}{\scalebox{0.84}{\xymatrix@C=2ex@R=1ex{
(\im\setminus p) & -2& -1 & 0 & 1 & 2 & 3 & 4 \\
1&  \lan 1,-2 \ran   \ar@{->}[dr] \tauarrow  && \lan 2,-3 \ran \ar@{->}[dr]\tauarrow && \lan 1,-3 \ran \ar@{->}[dr] \\
2&& \lan 1,-3 \ran  \ar@{->}[ur]\ar@{->}[ddr]\ar@{->}[dr] \tauarrow&& \lan 1,2 \ran \ar@{->}[ddr] \ar@{->}[ur]\ar@{->}[dr]\tauarrow && \lan 2,3 \ran  \ar@{->}[ddr]\ar@{->}[dr] \\
3&&&\lan 1,-4 \ran  \ar@{->}[ur] \tauarrow&& \lan 2,4 \ran \ar@{->}[ur]\tauarrow && \lan 3,-4 \ran   \\
4&&&\lan 1,4 \ran  \ar@{->}[uur] \tauarrow&& \lan 2,-4 \ran \ar@{->}[uur] \tauarrow&& \lan 3,4 \ran  
}}}
\end{align}
\item For $\g$ of type $G_2^{(1)}$ with $\tau_\calQ=s_1s_2\wvee \in W_{D_4} \rtimes \Ang{\wvee}$ and 
and $\xi=(\xi_1=-3,\xi_2=-2,\xi_3=-1,\xi_4=1)$, $\Gamma^\calQ$  in $\hbDynkin^\sigma_0$
can be depicted as follows:
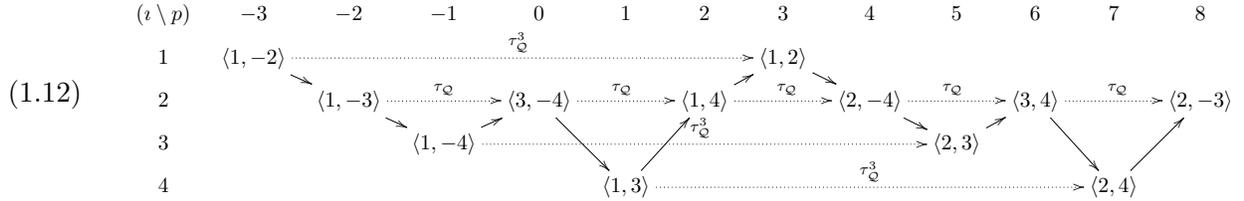
\begin{equation}\label{eq: folded G2 start}
\raisebox{3em}{\scalebox{0.74}{\xymatrix@C=2ex@R=1ex{
(\im \setminus p) & -3 & -2 & -1 & 0 & 1 & 2 & 3 & 4 & 5 & 6 & 7 & 8   \\
1 &  \lan 1, -2 \ran \ar@{->}[dr] \tauarrowq&&&& && \lan 1, 2 \ran \ar@{->}[dr] \\
2 &&  \lan 1, -3 \ran \ar@{->}[dr] \tauarrow && \lan 3, -4 \ran \ar@{->}[ddr]\tauarrow && \lan 1, 4 \ran\ar@{->}[ur] \tauarrow && \lan 2, -4 \ran \ar@{->}[dr]\tauarrow && \lan 3, 4 \ran \ar@{->}[ddr]\tauarrow&& \lan 2, -3 \ran  \\
3 &&& \lan 1, -4 \ran \ar@{->}[ur] \tauarrowq&&&&&&\lan 2, 3 \ran \ar@{->}[ur]  \\ 
4 &&&&& \lan 1, 3 \ran \ar@{->}[uur] \tauarrowq &&&&&&\lan 2, 4 \ran \ar@{->}[uur]  
}}}
\end{equation}
\ee     
\end{example}

Now let us consider a quantum affine algebra $\uqpg$ of twisted type; i.e, $\g=A_{2n-1}^{(2)},A_{2n}^{(2)},D_{n+1}^{(2)}$, $E_{6}^{(2)}$ 
and $D_{4}^{(3)}$. 
For each \emph{twisted} quantum affine algebra $\calU_q(\g^{(t)})$ $(t=2,3)$, we assign the finite simple Lie algebra $\gfin$ of symmetric type as follows:
\renewcommand{\arraystretch}{1.5}
\begin{align} \label{Table: root system twisted}
\small
\begin{array}{|c||c|c|c|c|c|}
\hline
 \g  & A_{2n-1}^{(2)} \ (n \ge 2)  & A_{2n}^{(2)}  \ (n \ge 1)   & D_{n+1}^{(2)}  \ (n \ge 3)  & E_6^{(2)}    & D_{4}^{(3)}    \\ \hline
 \g_\fin  & A_{2n-1} & A_{2n}    & D_{n+1}   &  E_6 & D_{4}    \\
\hline
\end{array}
\end{align}

Then we take $\hbDynkin^\sigma$ and $\calQ$ for $\calU_q(\g^{(t)})$ $(t=2,3)$ as the ones of $\calU_q(\g^{(1)})$.

\subsection{Subcategories}

For a quantum affine algebra $\uqpg$ of untwisted type $\g^{(1)}$, we can choose a $\rmQ$-datum $\calQ=\Qdatum$ as seen in the previous subsection.

For each $(\im,p) \in \bDynkin_0 \times \Z$, we set
\[
\calV^{(1)}(\im,p)\seteq \begin{cases}
V(\varpi_\im)_{(-q)^{p}} & \text{if $\g$ is of simply-laced affine type}, \\
V(\varpi_\oim)_{(-1)^{n+\oim}\qs^{p}} & \text{if $\g$ is of type $B^{(1)}_{n}$}, \\
V(\varpi_\oim)_{(-\qs)^{p}} & \text{if $\g$ is of type $C^{(1)}_{n}$}, \\
V(\varpi_\oim)_{(-1)^{i+p}\qs^{p}} & \text{if $\g$ is of type $F^{(1)}_{4}$}, \\
V(\varpi_\oim)_{(-\qt)^{p}} & \text{if $\g$ is of type $G^{(1)}_{2}$},
\end{cases}
\quad \text{where $q = \qs^2 = \qt^3$.}
\]

For a quantum affine algebra $\uqpg$ of twisted type $\g^{(t)}$ $(t=2,3)$ and 
$(\im,p) \in \bDynkin_0 \times \Z$, we assign the fundamental module $\calV^{(t)}(\im,p)$ $(t=2,3)$ as follows:
\begin{align*}
\calV^{(t)}(\im,p) \seteq V(\varpi_{\pi(\im)})_{A(\im,p)},
\end{align*}
where
\begin{align*}
&\pi(\im)\seteq \begin{cases} \im & \text{ if } {\rm (i)} \  \g^{(2)}=A^{(2)}_N \text{ and } \im \le \lceil N/2 \rceil  \text{ or } {\rm (ii)} \  \g^{(2)}=D^{(2)}_{n+1}  \text{ and } \im <n,  \\[1ex]
N+1- \im  & \text{ if } \g^{(2)}=A^{(2)}_N \text{ and } \im > \lceil N/2 \rceil, \\
  n  & \text{ if } \g^{(2)}=D^{(2)}_{n+1} \text{ and }  \im=n,n+1, \\
\end{cases} \allowdisplaybreaks \\[1ex]
&\bc
 \pi(1)=\pi(6)=1,  \quad \pi(3)=\pi(5)= 2,  \quad \pi(4)=3,  \quad
\pi(2)=4 &\text{if $\g^{(2)}=E_6^{(2)}$},\\
\pi(1)=\pi(3)=\pi(4)= 1, \quad \pi(2) = 2 &\text{if
$\g^{(3)}=D_4^{(3)}$,} \ec \end{align*} and
\begin{align*}
& A(\im,p)\seteq
\begin{cases}
 (-q)^{p} & \text{if } {\rm (i)} \  \g^{(2)}=A^{(2)}_N \text{ and } \im \le  \lceil N/2 \rceil  \\
& \quad  \text{ or } {\rm (ii)} \  \g^{(2)}=E^{(2)}_6 \text{ and } \im=1,3, \\
(-1)^N (-q)^{p} & \text{if } \g^{(2)}=A^{(2)}_N \text{ and } \im >  \lceil N/2 \rceil,  \\
 (\sqrt{-1})^{n+1-\im} (-q)^{p} & \text{if } \g^{(2)}=D^{(2)}_{n+1} \text{ and } \im <n, \\
(-1)^\im (-q)^{p} & \text{if } \g^{(2)}=D^{(2)}_{n+1} \text{ and } \im =n,n+1, \\
- (-q)^p & \text{if }  \g^{(2)}=E^{(2)}_6 \text{ and } \im=5,6, \\
 (\sqrt{-1}) (-q)^p & \text{if }  \g^{(2)}=E^{(2)}_6 \text{ and } \im=2, 4,\\
\bigl(\delta_{\im,1}-\delta_{\im,2}+\delta_{\im,3} \omega +
\delta_{\im,4} \omega^2\bigr)(-q)^{p}&\text{if
$\g^{(3)}=D^{(3)}_4$.}
\end{cases}
\end{align*}
Here $\omega$ is the third root of unity.  

For each quantum affine algebra $\calU_q(\g^{(t)})$ $(t=1,2,3)$, its $\rmQ$-datum $\calQ=\Qdatum$ and $(\be,k) \in \Phi^+_\gfin \times \Z$ with $\phi^{-1}_\calQ(\be,k)=(\im,p) \in \hbDynkin^\sigma_0$, we define
$$
\calV^{(t)}_\calQ(\be,k) \seteq \calV^{(t)}(\im,p).
$$
For notational simplicity, we usually write $\calV^{(t)}_\calQ(\be)$ instead of $\calV^{(t)}_\calQ(\be,0)$ if $k=0$, 
and drop $^{(t)}$ if there is no danger of confusion.

Note that, for $k \in \Z$, we have  
$$
\calV^{(t)}_\calQ(\be,k) = \scrD^{k} \calV^{(t)}_\calQ(\be).
$$

\begin{definition}[{\cite{HL15,HL16,KKKOIII,KO18,OhS19}}]
Let $U_q(\g^{(t)})$ be a quantum affine algebra $(t=1,2$ or $3)$.
\begin{enumerate}[{\rm (i)}]
\item We denote by $\mC^0_{\g}$ the smallest abelian full subcategory  of  the category $\Ca_{\g}$
\begin{itemize}
\item[{\rm (a)}] it is stable under subquotient, tensor product and extension,
\item[{\rm (b)}] it contains $\calV^{(t)}(\im,k)$ for all $(\im,k) \in \hbDynkin^\sigma_0$.
\end{itemize}
Note that the definition of $\mC^0_{\g}$ does not depend on the choice of $\mathcal{Q}$.
\item For each $k \in \Z$, we denote by $\mC^{(t)}_{\mathcal{Q}}[k]$ the smallest abelian full subcategory of the category $\Ca_{\g^{(t)}}$  such that
\begin{itemize}
\item[{\rm (a)}] it is stable under subquotient, tensor product and extension,
\item[{\rm (b)}] it contains $\scrD^k\calV^{(t)}_{\mathcal{Q}}(\beta)$ for all $\beta \in \Phi^+_\mathsf{g}$, and the trivial module $\mathbf{1}$.
\end{itemize}
When $k=0$, we write $\mC^{(t)}_{\mathcal{Q}}$ instead of $\mC^{(t)}_{\mathcal{Q}}[0]$ for simplicity. 
We call $\mC^{(t)}_{\mathcal{Q}}[k]$ a \defn{heart} subcategory of $\mC^0_{\g}$.
\item For $m \in \Z$, we denote by $\mC^{(t)}_{\calQ,\le m}$ 
the smallest abelian full subcategory  of  the category $\Ca_{\g}$
\begin{itemize}
\item[{\rm (a)}] it is stable under subquotient, tensor product and extension,
\item[{\rm (b)}] it contains $\scrD^s\calV^{(t)}_{\mathcal{Q}}(\beta)$ for all $\beta \in \Phi^+_\mathsf{g}$ and $s \le m$.
\end{itemize}
Similarly, we can define $\mC^{(t)}_{\calQ,< m}$, $\mC^{(t)}_{\calQ,\ge m}$  and $\mC^{(t)}_{\calQ,> m}$, respectively. 
\end{enumerate}
\end{definition}

\begin{remark} For each $k \in \Z$,
by considering all fundamental modules $\{ \calV(\im_s,p_s) \ | \ 1 \le s \le \ell\}$ in $\scrC_\calQ[k]$, the full subquiver of
$\hbDynkin$ with vertices $\{ (\im_s,p_s) \ | \ 1 \le s \le \ell \}$ is denoted by $\Gamma^\calQ[k]$. The quiver 
$\Gamma^\calQ[k] \iso \Gamma^\calQ$ as quivers, $\Gamma^\calQ[k] \cap \Gamma^\calQ[l] = \emptyset$ unless $k=l$, and
$\displaystyle\bigsqcup_{k \in \Z} (\Gamma^\calQ[k])_0 = \hbDynkin_0$. 
\end{remark}

These categories for simply-laced untwisted affine type $\g^{(1)}$ were introduced in~\cite{HL10,HL15}. 
Note that $\mC^0_\g$ is a skeleton category and every skeleton category can be obtained from
$\mC^0_\g$ by taking a parameter shifts.  
The Grothendieck ring $K(\mC^0_\g)$  of $\mC^0_\g$ is isomorphic to the \emph{commutative} polynomial ring generated by
the classes of modules in $\{ \scrD^k \calV^{(t)}_{\mathcal{Q}}(\beta)  \ | \ \be \in \Phi_\gfin^+, \ k \in \Z \}$~\cite{FR99}.

\begin{proposition} [{\cite[Proposition 5.4]{KKOP23P}}]\label{prop: no intersection}
For $m,n\in \Z$ with $|m-n|>1$,  
if a simple module $M$ is contained in $\scrC_\calQ[m] \cap \scrC_\calQ[n]$, then $M \iso \mathbf{1}$.
\end{proposition}

For a fundamental module $V \in \scrC^0_\g$, there exists an $(\im,p)$  such that $V \iso \calV(\im,p)$. We set 
$$
{\rm clr}(V) := \im
$$
and call it the \defn{color} of $V$.
Since we fix the parity function $\epsilon$, the color of $V$ is well-defined. 

\begin{example} Using~\eqref{eq: B3 AR}, we can see that
$$
\clr(V(\varpi_2)_{(-1)\qs^{-2}}) = 4 \qtq \clr(V(\varpi_2)_{(-1)\qs^{4}}) = 2,
$$    
since $V(\varpi_2)_{(-1)\qs^{-2}} = \calV(4,-2)$ and $V(\varpi_2)_{(-1)\qs^{4}} = \calV(2,4)$. 
\end{example}

\section{\texorpdfstring{$R$}{R}-matrices and \texorpdfstring{$\Z$}{Z}-invariants}
\label{sec:background_repr_theory}

In this section, we recall the $R$-matrix in the representation theory of quantum affine algebra and $\Z$-invariants arising from $R$-matrix, in which the latter is introduced in~\cite{KKOP20}.

\subsection{\texorpdfstring{$R$}{R}-matrices}

In this subsection, we review the notion of $R$-matrices on $\uqpg$-modules and their coefficients
by following~\cite[\S 8]{Kashiwara02}, \cite[Appendices A and B]{AK97} and \cite{KKOP20} mainly. 

For modules $M$ and $N \in \Cg$, there exists $\bfk\dpar{z} \tens \uqpg$-module isomorphism
$$
\Runiv_{M,N_z} \colon \bfk\dpar{z} \tens_{\bfk[z^{\pm 1}]} (M \tens N_z) \to \bfk\dpar{z} \tens_{\bfk[z^{\pm 1}]} (N_z \tens M)
$$
satisfying
\begin{align}  \label{eq: commutativity}
\raisebox{1.5em}{ \scalebox{0.90}{ \xymatrix@C=5.5ex{
\ko(\!(z)\!)\tens_{\ko[z^{\pm1}]}  (M\tens N\tens L_z) \ar[dr]_{M\tens \Runiv_{N,L_z}}\ar[rr]^
{\Runiv_{M\tens N,L_z}}
& & \ko(\!(z)\!)\tens_{\ko[z^{\pm1}]}  (L_z\tens M\tens N) \\
& \ko(\!(z)\!)\tens_{\ko[z^{\pm1}]}  ( M\tens L_z\tens N) \ar[ur]_{\Runiv_{M,L_z}\tens N}
}}}
\end{align}
for $L, M, N \in \Ca_\g$ and further properties (see \cite{Kashiwara02} for more details). We call $\Runiv_{M,N_z}$ the \defn{universal $R$-matrix} of $M$ and $N$. 
 
For modules  $M$ and $N \in \Cg$, we say that $\Runiv_{M,N_z}$ is \defn{rationally renormalizable} if there exists $c_{M,N}(z) \in \bfk\dpar{z}^\times$
such that $c_{M,N}(z)$ is unique up to a multiple of $\bfk[z^{\pm1}]^\times$, 
\bnum
\item $\Rren_{M,N_z} \seteq c_{M,N}(z)\Runiv_{M,N_z}\colon M \otimes N_z \to N_z \otimes M$ and 
\item $\Rren_{M,N_z}|_{z =x}$ does not vanish for any $x \in \bfk^\times$. 
\ee
We call $c_{M,N}(z)$  the \defn{renormalizing coefficient}. 
In the case, we write
\[
\text{ $\rmat{M,N}\seteq \Rren_{M,N_z}|_{z =1}$ and call it \defn{$R$-matrix}.}
\]
By the definition, $\rmat{M,N}$ never vanishes.

Now assume that $M$ and $N$ are non-zero simple $\uqpg$-modules in $\Ca_\g$.
Then the universal $R$-matrix $\Runiv_{M,N_z}$ is rationally renormalizable.
Furthermore, we have the following:
Let $u$ and $v$ be dominant extremal weight vectors of $M$ and $N$, respectively.
Then there exists $a_{M,N}(z) \in \ko[[z]]^\times$ such that
\begin{align}\label{eq: univ coeff}
\Runiv_{M,N_z}\big( u \tens v_z\big) = a_{M,N}(z)\big( v_z\tens u \big).
\end{align}
Then $\Rnorm_{M,N_z}\seteq a_{M,N}(z)^{-1}\Runiv_{M,N_z}$ is a unique $\ko(z)\tens\uqpg$-module isomorphism
\begin{equation*}
\Rnorm_{M, N_z} \colon
\ko(z)\otimes_{\ko[z^{\pm1}]} \big( M \otimes N_z\big) \overset{\iso}{\longrightarrow} \ko(z)\otimes_{\ko[z^{\pm1}]}  \big( N_z \otimes M \big)
\end{equation*}
satisfying
\begin{equation*}
\Rnorm_{M, N}\big( u  \otimes v_z\big) = v_z\otimes u .
\end{equation*}
We call $a_{M,N}(z)$ the \defn{universal coefficient} of $M$ and $N$ and $\Rnorm_{M,N_z}$ the \defn{normalized $R$-matrix}.

Let $d_{M,N}(z) \in \ko[z]$ be a monic polynomial of the smallest degree such that the image of $d_{M,N}(z) \Rnorm_{M,N_z}$ is contained in $N_z \otimes M$. We call $d_{M,N}(z)$ the \defn{denominator} of $\Rnorm_{M,N}$. Then we have
\begin{equation}
\Rren_{M,N_z} \seteq  d_{M,N}(z)\Rnorm_{M,N} \colon M \otimes N_z \to N_z \otimes M \quad \text{up to a constant multiple of $\ko[z^{\pm1}]$},
\end{equation}
and the $R$-matrix
\[
\rmat{M,N}\colon M \tens N \to N \tens M
\]
is equal to the specialization of $\Rren_{M,N_z}$ at $z=1$ up to a constant multiple. Thus
$$
\Rren_{M,N_z} = \dfrac{d_{M,N}(z)}{a_{M,N}(z)}\Runiv_{M,N_z} \qtq c_{M,N}(z) = \dfrac{d_{M,N}(z)}{a_{M,N}(z)}.
$$

The following theorem tells that the denominators control the representation theory of $\uqpg$:

\begin{theorem}[{\cite{AK97,Chari02,Kashiwara02,KKKO15}}; see also {\cite[Theorem~2.2]{KKK15}}]  \label{Thm: basic properties}
\mbox{}
\begin{enumerate}[{\rm (a)}]
\item \label{item: positivity}
For good modules $M,N$, the zeros of $d_{M,N}(z)$ belong to
$\C[[q^{1/m}]]q^{1/m}$ for some $m\in\Z_{>0}$.
\item \label{item: real commutes} For simple modules $M$ and $N$ such that one of them is real, $M_x$ and $N_y$ strongly commute to each other if and only if $d_{M,N}(z)d_{N,M}(z^{-1})$ does not vanish at $z=x/y$.
\item \label{item: good symmetry}
Let $M_k$ be a good module
with a dominant extremal vector $u_k$ of weight $\lambda_k$ and
$a_k\in\ko^\times$ for $k=1,\ldots, t$.
Assume that $a_j/a_i$ is not a zero of $d_{M_i, M_j}(z) $ for any
$1\le i<j\le t$. Then the following statements hold.
\begin{enumerate}[{\rm (i)}]
\item  $(M_1)_{a_1}\otimes\cdots\otimes (M_t)_{a_t}$ is generated by $u_1\otimes\cdots \otimes u_t$.
\item  The head of $(M_1)_{a_1}\otimes\cdots\otimes (M_t)_{a_t}$ is simple.
\item  Any non-zero submodule of $(M_t)_{a_t}\otimes\cdots\otimes (M_1)_{a_1}$ contains the vector $u_t\otimes\cdots\otimes u_1$.
\item  The socle of $(M_t)_{a_t}\otimes\cdots\otimes (M_1)_{a_1}$ is simple.
\item  $\rmat{} \colon (M_1)_{a_1}\otimes\cdots\otimes (M_t)_{a_t} \to (M_t)_{a_t}\otimes\cdots\otimes (M_1)_{a_1}$  be the specialization of $R^{{\rm norm}}_{M_1,\ldots, M_t}
\seteq\prod\limits_{1\le j<k\le t}R^{{\rm norm}}_{M_j,\,M_k}$ at $z_k=a_k$.
Then the image of $\rmat{}$ is simple and it coincides with the head of $(M_1)_{a_1}\otimes\cdots\otimes (M_t)_{a_t}$
and also with the socle of $(M_t)_{a_t}\otimes\cdots\otimes (M_1)_{a_1}$.
\end{enumerate}
\item \label{item: symmetry}
For a simple integrable $\uqpg$-module $M$, there exists
a finite sequence
\begin{align*}
\big((i_1,a_1),\ldots, (i_t,a_t)\big) \in I\times \mathbf{k}^\times
\end{align*}
such that
 $d_{V(\varpi_{i_k}),V(\varpi_{ i_{k'} })}(a_{k'}/a_k) \seteq d_{i_k,i_{k'}}(a_{k'}/a_k)\not=0$ for $1\le k<k'\le t$,
and $M$ is isomorphic to the head of $V(\varpi_{i_1})_{a_1}\otimes\cdots\otimes V(\varpi_{i_t})_{a_t}$.
In particular, $M$ has the dominant extremal weight $\sum_{k=1}^t \varpi_{i_k}$.
\end{enumerate}
\end{theorem}

For simple modules $M$ and $N$ in $\Cg$, let $M\hconv N$ and $M \sconv N$ denote the head and the socle of $M \tens N$, respectively. 

\begin{remark} \label{rmk: gen cogen}
Note that for quasi-good simple modules $M$ and $N$ with the dominant extremal weight vectors $u_M$ and $u_N$, it is well known that the vector $u_M \otimes u_N$ either generates or cogenerates $M \otimes N$.
\end{remark}

\begin{proposition}[{\cite[Theorem 3.12]{KKKO15}}] \label{prop: simple head}
Let $M$ and $N$ be simple modules in $\Cg$ and assume that one of them is real. 
\bnum
\item $\Hom(M \tens N,N\tens M)= \bfk \hspace{.2ex} \rmat{M,N}.$
\item $M \tens N$ and $N \tens M$ have simple socles and simple heads.
Moreover, $\Image(\rmat{M,N})$ is isomorphic to $M \hconv N$ and $N \sconv M$.
\item $M \tens N$ is simple whenever its head and its socle are isomorphic to each other. 
\ee
\end{proposition}

\begin{proposition}[{\cite[Corollary 3.11]{KKKO15}}] \label{prop: non-van}
\mbox{}
\bnum 
\item \label{it: non-vani1}
Let $M_k$ $(k\in \{1,2,3\})$ be modules in $\Cg$, let $\vph_1\colon L \to M_2 \tens M_3$ and $\vph_2\colon M_1 \otimes M_2 \to L'$ be non-zero homomorphism. Assume further $M_2$ is a simple module. Then the composition
$$
\xymatrix@R=3.5ex@C=9ex{
M_1 \otimes L \ar[r]^{M_1 \otimes \vph_1 \ \ \quad } &
M_1 \tens M_2 \tens M_3 \ar[r]^{ \ \quad \vph_2 \otimes M_3}&
L' \tens M_3
}
\text{ does not vanish.}
$$
\item \label{it: non-vani2}
Let $M_k$ $(k\in \{1,2,3\})$ be modules in $\Cg$, let $\vph_1\colon L \to M_1 \tens M_2$ and $\vph_2\colon M_2 \otimes M_3 \to L'$ be non-zero homomorphism. Assume further $M_2$ is a simple module. The the composition
$$
\xymatrix@R=3.5ex@C=9ex{
L \otimes M_3 \ar[r]^{M_1 \otimes \vph_1 \ \ \quad } &
M_1 \tens M_2 \tens M_3 \ar[r]^{ \ \quad \vph_2 \otimes M_3}&
M_1 \tens L'
}
\text{ does not vanish.}
$$
\ee
\end{proposition} 

We call the properties in Proposition~\ref{prop: non-van} \defn{quasi-rigidity} (see \cite[\S 6.2]{KKOP24F} for more detail).

\begin{lemma}[{\cite[Corollary 3.13]{KKKO15}}] \label{Lem: MNDM}
Let $L$ be a real simple module. Then for any simple module $X$, we have
\begin{align*}
& (L \hconv X) \hconv \scrD L \iso X, \qquad \scrD^{-1} L \hconv (X \hconv L)  \iso X,    \\
& L \hconv (X \hconv \scrD L) \iso X, \qquad (\scrD^{-1} L \hconv X) \hconv L  \iso X.
\end{align*}
\end{lemma}

\subsection{\texorpdfstring{$\Z$}{Z}-invariants arising from the \texorpdfstring{$R$}{R}-matrix} \label{subsec: de-theory}

In this subsection, we review the $\Z$-invariants arising from $R$-matrices of quantum affine algebras, which are introduced in \cite{KKOP20}.

We set  
$$
\tp \seteq p^{*2} = q^{2\Ang{c,\uprho}} \text{ and }
\vph(z) \seteq \prod_{s \in \Z_{\ge 0}} (1-\tp^s z) = \sum_{n=0}^\infty \dfrac{(-1)^n \tp^{n(n-1)/2}}{\prod_{k=1}^n (1-\tp^k)}z^n \in \bfk\dbracket{z}.
$$

We define the multiplicative subgroup $\calG$ in $\bfk\dpar{z}^\times$ containing $\bfk(z)^\times$ as follows:
\begin{align*}
\calG \seteq \left\{ 
cz^m \prod_{a \in \bfk^\times} \vph(az)^{\eta_a} \left| 
\begin{matrix}
c \in \bfk^\times, \ m \in \Z \\
\ \eta_a \in \Z \text{ vanishes except finitely many $a$'s}
\end{matrix}
\right.
\right\}. 
\end{align*}
Then it is proved in~\cite{KKOP20} that $c_{M,N}(z)$ is contained in $\calG$ for rationally renormalizable $\Runiv_{M,N_z}$.

In~\cite[Section 3]{KKOP20}, the following group homomorphisms
are introduced 
\[
\Deg\colon \calG \to \Z
\qtq
\Deg^\infty\colon \calG \to \Z
\]
defined by 
\[
\Deg(f(z)) \seteq \sum_{a \in \tp^{\Z_{\le 0}}} \eta_a - \sum_{a \in \tp^{\Z_{> 0}}} \eta_a \qtq 
\Deg^\infty(f(z)) \seteq \sum_{a \in \tp^{\Z} } \eta_a
\]
for $f(z) = cz_m \prod \vph(az)^{\eta_a} \in \calG$. Here  $\tp^S \seteq \{ \tp^k \ | \ k\in S\}$ for a subset $S$ of $\Z$.  

\begin{definition}[{\cite[Definition 3.6, 3.14]{KKOP20}}]
Let $M,N \in \Cg$. 
\ben
\item If $\Runiv_{M,N_z}$ is rationally renormalizable, we define the integers $\La(M,N)$ and $\La^\infty(M,N)$ by
\[
\La(M,N)=\Deg(c_{M,N}(z)) \qtq \Li(M,N)=\Deg^\infty(c_{M,N}(z)). 
\]
\item For simple modules $M$ and $N$ in $\Cg$, we define the integer $\de(M,N)$ by 
\[
\de(M,N) = \dfrac{1}{2}\big(\La(M,N)+\La(\scrD^{-1}M,N)\big). 
\]
\ee
\end{definition}

 \begin{proposition}[{\cite[Proposition 3.16]{KKOP20}}] \label{prop: de ge 0}
For simple modules $M$ and $N$ in $\Ca_\g$, we have
\[
\mathfrak{d}(M,N)= \operatorname{zero}_{z=1}\big( d_{M,N}(z)d_{N,M}(z^{-1})\big),
\]
where $\operatorname{zero}_{z=a}f(z)$ denotes the order of zero of $f(z)$ at $z=a$.   In particular $\mathfrak{d}(M,N) \ge 0$ 
\end{proposition}
 
\begin{proposition}[{\cite[Corollary 3.17]{KKOP20}}] \label{prop: hconv simple}
For simple modules $M, N \in \Ca_\g$, assume that one of them is real.
\bna
\item $M \otimes N$ has a simple head and a simple socle. 
\item $\mathfrak{d}(M,N) = 0$, then $M\hconv N \iso M \sconv N $ and hence $ M \tens N$ is simple.  
\ee
\end{proposition}

\begin{proposition}[{\cite[Lemma 7.3]{KO18}}]  \label{prop: length 2}
Let $M$ and $N$ be simple modules in $\Ca_\g$.
Assume that one of them is real and $\mathfrak{d}(M,N)=1$.
Then we have an exact sequence
\[
0 \to M \sconv N \to M\tens N\to M\hconv N \to 0.
\]
In particular, $M\tens N$ has composition length $2$.
\end{proposition}

\begin{proposition}[{\cite[Corollary 4.2]{KKOP20}}] \label{prop: de less than equal to}
For simple modules $L,M, N \in \Ca_\g$, we have
$$
\de(S,L) \le \de(M,L)+\de(N,L)
$$
for any simple subquotient $S$ of $M \otimes N$. 
\end{proposition}

\begin{proposition}[{\cite[Lemma 2.28]{KKOP23P}}] \label{prop: real head}
For real simple modules  $M$ and $N$, if $\de(M,N) \le 1$, then $M \hconv N$ is real. 
\end{proposition}

\begin{proposition}[{\cite{KKOP20,KKOP23P}}]\label{prop: Lambda property} 
Let $M$ and $N$ be simple modules in $\Cg$. Then we have the following:
\bnum
\item $\de(M,N) \in \Z_{\ge0}$ and $ \de(M,N)=\frac{1}{2}\big(\La(M,N)+\La(N,M)\big)=\de(N,M)$.
\item \label{it: La de} $\La(M,N) = \sum_{k\in \Z}(-1)^{k+\delta(k<0)}\de(M,\scrD^kN)$.
\item $\La(M,N)=\La(\scrD^{-1} N,M)=\La(N,\scrD M)$. 
\item \label{it: Lainf d} $\La^\infty(M,N)   = \sum_{k\in \Z}(-1)^{k}\de(M,\scrD^kN)$.
\ee 
\end{proposition}

\begin{lemma}[{\cite[Corollary 2.25]{KKOP23P}}] \label{lem: de=de}
Let $L,M$ be real simple modules and $X$ a simple module.
\bnum
\item \label{it: de=de 1}
If $\de(L,M)=\de(\scrD L,M)=0$, then we have $\de(L,X\hconv M)=\de(L,X)$.
\item  \label{it: de=de 2} If $\de(L,M)=\de(\scrD^{-1} L,M)=0$, then we have $\de(L,M\hconv X)=\de(L,X)$.
\ee
\end{lemma}

\begin{proposition}[{\cite[Proposition 2.25]{KKOP24A}}] \label{prop: three}
Let $M$, $N$ and $L$ be simple $\uqpg$-modules such that $L$ is
real. Assume that \bnum
\item $\de(\scrD M,L)=0$,
\item $\de(\scrD L,N)=0$, and
\item $M \tens L \tens N $ has a simple head.
\end{enumerate}
Then we have
\[
\de(L, \hd(M\tens L \tens N)) = \de(L,M \hconv L) + \de(L,L \hconv N).
\]
\end{proposition}

\begin{lemma}[{\cite[Corollary 3.18]{KKOP21}}] \label{lem: decrease}
Let $L$ be a real simple module and $M$ be a module in $\Cg$. 
Let $n \in \Z_{\ge0}$ and assume that any simple subquotient $S$ of $M$ satisfies $\de(L,S)\le n$. 
Then any simple subquotient $K$ of $L \tens M$ satisfies $\de(L,K)<n$. In particular, any simple subquotient of $L^{\tens n}\tens M$ 
strongly commutes with $L$. 
\end{lemma}

\subsection{Normal sequences and root modules} 

A sequence $\uL$ of simple modules is called a \defn{normal sequence} if the composition of $R$-matrices 
\begin{align*}
    \rmat{L_r,\ldots,L_1} & \seteq
\displaystyle\prod_{1\le i <k \le r} \rmat{L_k,L_i}
=(\rmat{L_{2},L_1})  \circ \cdots \circ (\rmat{L_{r-1},L_1}\circ
\cdots \circ \rmat{L_{r-1},L_{r-2}})  \circ (\rmat{L_r,L_1} \circ \cdots
\circ \rmat{L_r,L_{r-1}}) \\
& \colon L_r \tens \cdots\tens L_1 \to L_1 \tens \cdots \tens L_r \text{ does not vanish}. 
\end{align*} 

An ordered sequence of simple modules $\uL = (L_r,L_{r-1}\ldots,L_1)$ in $\Cg$ is called \defn{almost real}, if all $L_i$ $(1 \le i \le r)$ are real except for at most one. 

\begin{lemma}[{\cite{KK19,KKOP23L}}] \label{lem: normal property}
Let $\uL=(L_r,\ldots,L_1)$ be an almost real sequence. If $\uL$ is normal, then 
\begin{eqnarray} &&
\parbox{75ex}{
the image of $\rmat{\uL}$ is simple and coincides with the head of $L_r \tens \cdots \tens L_1$    
and also with the socle of $L_1 \tens \cdots \tens L_r$. 
}\label{eq: head socle}
\end{eqnarray} 
Moreover, $\uL$ is normal if and only if either
\bnum
\item $\uL'=(L_{r-1},\ldots,L_1)$ is a normal sequence and $\La\big(L_r,\Image(\rmat{\uL'})\big) = \sum_{k=1}^{r-1} \La(L_r,L_k)$, or 
\item $\uL''=(L_r,\ldots,L_{2})$ is a normal sequence and $\La\big(\Image(\rmat{\uL''}),L_1\big) = \sum_{k=2}^{r} \La(L_k,L_1)$.
\ee
\end{lemma}

\begin{proposition}[{\cite[Proposition 1.10]{KKOP25}}] \label{prop:normalsimple}
Let $\uL=(L_1,\ldots,L_r)$ be an almost real normal sequence.
\bnum
  \item Any simple subquotient $S$ of $L_{2}\tens\cdots\tens L_r$ satisfies $\La(L_1,S)\le\sum_{k=2}^r\La(L_1,L_k)$.
  \item Any simple subquotient $S$ of $L_{1}\tens\cdots\tens L_{r-1}$ satisfies $\La(S,L_r)\le\sum_{k=1}^{r-1}\La(L_k,L_r)$.
  \item $\hd(L_{1}\tens\cdots\tens L_r)$ appears only once in the composition series of $L_{1}\tens\cdots\tens L_r$.\label{itsimple}
\ee
\end{proposition}

\begin{lemma}[{\cite[Lemma 4.3, Lemma 4.17]{KKOP20}; \cite[Lemma 2.24]{KKOP23P}}] \label{lem: normal seq d}
Let $L,M,N$ be simple modules in $\Cg$ that are all real except for at most one.
\bna
\item \label{it: normal 1}
Assume that one of the following condition holds:
\bnum
\item $\de(L,M)=0$ and $L$ is real,
\item $\de(M,N)=0$ and $N$ is real,
\item $\de(L,\scrD^{-1} N)=\de(\scrD L,N)=0$ and $L$ or $N$ is real,
\ee
then $(L,M,N)$ is a normal sequence.
\item \label{it: de +} Assume that $L$ is real. 
\bnum
\item[{\rm (iv)}]  $(L,M,N)$ is normal if and only if $(M,N,\scrD L)$ is normal.
\item[{\rm (v)}]  $\de(L,M\hconv N)=\de(L,M)+\de(L,N)$ if and only if $(L,M,N)$ and $(M,N,L)$ are normal. 
\ee
\ee
\end{lemma}

\begin{definition}[{\cite{KKOP23P,KKOP24A}}]
Let $(M,N)$ be an ordered pair of simple modules in $\Cg$. 
\ben
\item We call the pair \defn{unmixed} if 
$$ \de(\scrD M,N)=0$$
and \defn{strongly unmixed} if 
$$ \de(\scrD^k M,N)=0 \quad \text{for any } k \in \Z_{>0}.$$
\item An almost real sequence $\uM=(M_s,M_{s-1},\ldots,M_1)$ is said to be \defn{$($strongly$)$ unmixed} if $(M_k,M_i)$ is (strongly) unmixed for all $s \ge k >i \ge 1$.  
\ee 
\end{definition}

\begin{proposition}[{\cite{KKOP23P}}] \label{prop: Unmix normal} \hfill 
\bna
\item \label{it: unmix normal} Any unmixed almost real sequence $\uM=(M_s,M_{s-1},\ldots,M_1)$ is normal. 
\item For a strongly unmixed almost real sequence $\uM=(M_s,M_{s-1},\ldots,M_1)$, 
\[
\bigl(\head(M_s\tens M_{s-1} \tens \cdots M_k), \head(M_{k-1}\tens M_{k-1} \tens \cdots M_1) \bigr)
\]
is strongly unmixed for any $1< k \le r$. 
\ee
\end{proposition}

We say a real simple module $L$ a \defn{root module} if
\begin{align} \label{eq: root module}
\de(L,\scrD^{k}(L)) =\delta(k = \pm 1) \quad \text{ for any } k \in \Z.       
\end{align}

\begin{proposition}[{\cite[Proposition 2.28]{KKOP24A}}]
Every fundamental module is a root module.
\end{proposition}

For the  rest of this subsection, we take root modules $L,L'$ satisfying  
\begin{align} \label{Eq: assumption}
 \de( \scrD^k L, L')  = \delta(k=0) \quad \text{ for } k\in \Z.
\end{align}
We call a pair of root modules satisfying~\eqref{Eq: assumption} a \defn{$\mathfrak{sl}_3$-pair}.

\begin{lemma}[{\cite[Lemma 3.8, Lemma 3.9]{KKOP23P}}] \label{Lem: two root modules}
We have
\[
\de(\scrD^k L,L\hconv L') = \delta(k=1)
\qtq
\de(\scrD^k L,L'\hconv L) = \delta(k=-1).
\]
Furthermore, the simple module $L\hconv L'$ is a root module.
\end{lemma}

\begin{lemma}[{\cite[Lemma 3.10]{KKOP23P}}] \label{Lem:LL'X} 
Let $X$ be a simple module.
\bnum
\item \label{it: LL'X i} If $k\ne 0,1$, then 
$$
\de(  \scrD^k L, L' \hconv X) = \de(\scrD^k L, X).
$$
 As for $k=0$ and $1$,  one and only one of the following  two statements is true.
\bna
\item $ \de(L, L' \hconv X) = \de(L,X) $ and $\de( \scrD L, L' \hconv X) = \de (\scrD L, X)-1$,
\item $ \de(L, L' \hconv X) = \de(L,X) +1 $ and $\de( \scrD L, L' \hconv X) = \de (\scrD L, X)$.
\ee 
\item \label{it: LL'X ii} If $k\ne -1,0$, then 
$$
\de(  \scrD^k L, X \hconv L') = \de(\scrD^k L, X).
$$ 
 As for $k=-1$ and $0$, one and only one of the following two
statements is true.
\bna
\item[{\rm (c)}] $ \de(L, X \hconv L') = \de(L,X) $ and $\de( \scrD^{-1} L, X \hconv L' ) = \de (\scrD^{-1} L, X)-1$,
\item[{\rm (d)}] $ \de(L,  X\hconv L') = \de(L,X) +1 $ and $\de( \scrD^{-1} L, X \hconv L') = \de (\scrD^{-1} L, X)$.
\ee
\ee
\end{lemma}

\section{KR modules, \texorpdfstring{$i$}{i}-box combinatorics and Dorey's rule}
\label{sec:background_doreys_rule}

In this section, we first recall the KR modules in $\scrC_\g^0$ and interpret them in terms of $i$-boxes,  recently introduced in~\cite{KKOP24A}.
Then, using $i$-boxes, we investigate strongly commuting conditions of two KR modules in terms of their (extended) reaches.
In the final subsection, we recall the generalized Dorey's rule among the fundamental modules.

\subsection{Kirillov--Reshetikhin modules}

Note that the simple objects in $\Ca_\g$ are parameterized by $I_0$-tuples of polynomials $\mathcal{P}=(\mathcal{P}_{i} (u))_{i \in I_0}$,
where $\mathcal{P}_{i}(u) \in \ko[u]$ and $\mathcal{P}_{i}(0)=1$ for $i \in I_0$~\cite{CP95,CP98}.
We denote by $(\mathcal{P}^V_{i}(u))_{i \in I_0}$ the polynomials corresponding to a simple module $V$, and call $(\mathcal{P}^V_{i}(u))_{i \in I_0}$  the \defn{Drinfel'd polynomials} of $V$.  The Drinfel'd polynomials $(\mathcal{P}^V_{i}(u))_{i \in I_0}$ are determined by the eigenvalues
of the simultaneously commuting actions of some \defn{Drinfel'd generators} of $\uqpg$ on a subspace of $V$ (see, e.g.,~\cite{CP95} for more details).

A \defn{Kirillov--Reshetikhin module} (KR module), denoted by $W^{(k)}_{m,a}$ for $k \in I_0$, $m \ge 1$, and $a \in \ko^{\times}$, is a simple module of dominant extremal weight $m\varpi_k$ in $\Ca_\g$ whose Drinfel'd polynomials $\mathcal{P} = ( \mathcal{P}_i, \mid i \in I_0)$ are
\[
\mathcal{P}_i(u) = \delta_{i,k} (1-au) (1-aq_k^2u) (1-aq_k^4u) \cdots (1-aq_k^{2m-2}u) + (1 - \delta_{i,k}) \quad (i\in I_0)
\]
unless $k=n$ and $\g=A^{(2)}_{2n}$, in which case
\[
\mathcal{P}_i(u) =  \delta_{i,n}(1-au) (1-aq^2u) (1-aq^4u) \cdots (1-aq^{2m-2}u)  +\delta(i\ne n).
\]

For any $a \in \bfk^\times$, it is well-known that the sequence of fundamental modules
$$
\seq{ V(\varpi_k)_{a(-\check{q}_k)^{2m}} , \ldots, V(\varpi_k)_{a(-\check{q}_k)^{2}}, V(\varpi_k)_{a} } 
$$
is contained in the same skeleton category and unmixed. Hence it is a normal sequence for $m \ge 1$ and $k \in I_0$.
Here $\check{q}_k \seteq q_k$ unless $\g = A_{2n}^{(2)}$ and $k = n$, in which case $\check{q}_n = q$.
We denote $\Vkm{k^m}\seteq W^{(k)}_{m,(-\check{q}_k)^{1-m}}$ for $k \in I_0$ and $m \in \Z_{\ge 1}$, which can be constructed as 
\[
\Vkm{k^m} \iso
\begin{cases}
\hd\left(V(\varpi_k)_{(-q_k)^{m-1}} \tens V(\varpi_k)_{(-q_k)^{m-3}} \tens \cdots  \tens V(\varpi_k)_{(-q_k)^{1-m}}\right) & \text{ if $\g$ is untwisted, }  \\
\hd\left(V(\varpi_k)_{(-q)^{m-1}} \tens V(\varpi_k)_{(-q)^{m-3}} \tens \cdots  \tens V(\varpi_k)_{(-q)^{1-m}}\right)  & \text{ otherwise},
\end{cases}
\]
(see also~\cite{Her10}). 
It is also well-known that $\Vkm{k^m}_a$ is a real simple prime module and $\scrD(\Vkm{k^m}_a) \iso \Vkm{k^{*m}}_{a(p^*)}$ (see~\cite[(2.20)]{Chari02}).

By Proposition~\ref{prop: Unmix normal} and the definition of $\Vkm{k^m}$, we have surjective homomorphisms as follows:
\begin{align}\label{eq:KR-surj}
 \Vkm{k^{m-t}}_{(-\check{q}_k)^t} \otimes \Vkm{k^t}_{(-\check{q}_k)^{t-m}} \twoheadrightarrow \Vkm{k^m}, 
\end{align}
for $k\in I_0$, $m  \in \Z_{\ge 2}$ and $m>t\ge 1$.
We call the homomorphism~\eqref{eq:KR-surj} a \defn{fusion rule}.

The branching rule of $\Vkm{k^m}_a$ to $U_q(\g_0)$-modules is known~\cite{Chari01,Her06,Her10,Nakajima03II} (see also, e.g.,~\cite{Okado13,Scr20}).
Let us describe some particular cases.
We call a node $k \in I_0$ \defn{special} if there exists an automorphism $\psi$ of the Dynkin diagram of $\g$ such that $k = \psi(0)$, in which case $\Vkm{k^m} \iso V(m\La_k)$ (as $U_q(\g_0)$-modules)~\cite[Sec.~1]{KKMMNN92}.
When $\g$ is dual to an untwisted affine type (including all simply-laced types), then we additionally call such $k$ \defn{minuscule} and can be classified by $\langle \alpha^{\vee}, \Lambda_i \rangle \leq 2$ for all positive roots $\alpha$ of $\g_0$, where $\alpha^{\vee}$ denotes the corresponding dual root.
When $k$ is the adjoint node (when $\g$ is not type $A_n^{(1)}$), then $\Vkm{k^m} \iso \bigoplus_{s=0}^m V(s\La_k)$~\cite[Sec.~1]{KKMMNN92}.
If $\langle \alpha^{\vee}, \Lambda_i \rangle \leq 2$ for all positive roots $\alpha$, then this decomposition is multiplicity free, which is true for all $k \in I_0$ in nonexceptional affine types.

Note that KR module $\Vkm{k^m}_a$ $(a \in \bfk^\times)$ with $m=1$ is a fundamental module. 
For $k,l \in I_0$, the universal coefficient $a_{l,k}(z) \seteq a_{V(\varpi_l),V(\varpi_k)}(z)$ in~\eqref{eq: univ coeff}
is described as follows (see~\cite[Appendix~A]{AK97}):
\begin{align}\label{eq: aij}
a_{l,k}(z) \equiv \dfrac{\prod_{\nu} (p^*y_\nu z;q)_\infty (p^*\overline{y_\nu}z;q)_\infty }{\prod_{\nu} (x_\nu z;q)_\infty (p^{*2}\overline{x_\nu}z;q)_\infty} \quad \pmod{\ko[z^{\pm1}]^\times},
\end{align}
where $(y; x)_\infty = \prod_{s=0}^\infty (1-y^sx)$,
\[
d_{l,k}(z) = \prod_{\nu}(z-x_\nu) \quad \text{ and } \quad d_{l^*,k}(z) = \prod_{\nu}(z-y_\nu).
\]
Here and after, for $f(z), g(z) \in \ko(\!( z )\!)$, we write
$$f(z) \equiv g(z) \quad \text{if } f(z) / g(z) \in \ko[z^{\pm 1}]^{\times}.$$

Note that the denominators of the normalized $R$-matrices $d_{l,k}(z)$, and hence the universal coefficients $a_{l,k}(z)$,
were calculated in~\cite{AK97,DO94,KKK15,Oh14R} for classical affine types and in~\cite{OhS19Add,OhS19} for exceptional affine types.

Equation~\eqref{eq: aij} can be generalized to KR modules for $\Vkm{l^p}$ and $\Vkm{k^m}$ ($m,p \in \Z_{\ge 1}$) by following the same argument in~\cite[Appendix A]{AK97}:
\begin{align}\label{eq: aimjl}
a_{l^p,k^m}(z) \seteq a_{\Vkm{l^p}),\Vkm{k^m}}(z) \equiv \dfrac{\prod_{\nu} (p^*y_\nu z;q)_\infty (p^*\overline{y_\nu}z;q)_\infty }{\prod_{\nu} (x_\nu z;q)_\infty (p^{*2}\overline{x_\nu}z;q)_\infty} \quad \pmod{\ko[z^{\pm1}]^\times},
\end{align}
where
\[
d_{l^p,k^m}(z)\seteq d_{\Vkm{l^p},\Vkm{k^m}}(z) = \prod_{\nu}(z-x_\nu), \quad d_{(l^*)^p,k^m}(z)\seteq d_{\Vkm{(l^*)^p},\Vkm{k^m}}(z)=\prod_{\nu}(z-y_\nu).
\]

The following lemmas are easily checked.  

\begin{lemma}[{cf.~\cite{Chari02}}] \label{lem:denominators_symmetric}
For $l,k \in I_0$, assume that the zeros of $d_{l^p,k^m}(z)$ belong to $\C[[q^{1/m}]]q^{1/m}$ for some $m\in\Z_{>0}$. Then we have
\[
d_{l^p,k^m}(z) \equiv  d_{k^m,l^p}(z)  \pmod{\ko[z^{\pm1}]^\times}.
\]
\end{lemma}

\begin{remark}
As a consequence of the result of this paper,
we can conclude that $d_{l^p,k^m}(z)$ belong to $\C[[q^{1/m}]]q^{1/m}$ for all classical affine types
(see Section~\ref{subsec: denom result}).
\end{remark}

\begin{lemma}
Let $k'$ $($resp.~$l')$ be in the orbit of $k$ $($resp.~$l)$ under some $($classical type$)$ Dynkin diagram automorphism.
Then
\[
d_{l^p,k^m}(z) = d_{(l')^p,(k')^m}(z).
\]
\end{lemma}

The following proposition is one of the main ingredients of this paper.

\begin{proposition}[{\cite[Lemma C.15]{AK97}}] \label{prop: aMN}
Let $M, M', M''$ and $N$ be simple $\uqpg$-modules. Assume that we have a surjective $\uqpg$-homomorphism
$M'\otimes M'' \twoheadrightarrow M$. 
Then we have
\begin{subequations}
\label{eq: AK eqns}
\begin{align} \label{eq: AK1}
\dfrac{ d_{N,M'}(z) d_{N,M''}(z) a_{N,M}(z)}{d_{N,M}(z)a_{N,M'}(z)a_{N,M''}(z)}  \in \ko[z^{\pm 1}] 
\end{align}
and
\begin{align} \label{eq: AK2}
\dfrac{ d_{M',N}(z) d_{M'',N}(z) a_{M,N}(z)}{d_{M,N}(z)a_{M',N}(z)a_{M'',N}(z)} \in \ko[z^{\pm 1}]. 
\end{align}
\end{subequations}
\end{proposition}

For the remaining discussion of this subsection, we consider $\g$ not of type $A_{2n}^{(2)}$, but the statements hold for type $A_{2n}^{(2)}$ by replacing $q_i \mapsto q$.
We can see that KR modules are, roughly speaking, the building blocks of every factorization of simple modules.

\begin{proposition}[{\cite{CP91,CP96III}}] \label{prop: CP partition}
For any simple module $M$ with Drinfel'd polynomials
\[
(\mathcal{P}^M_i(u) = (1-a_iu) (1-a_iq_i^2u) (1-a_iq_i^4u) \cdots (1-a_iq_i^{2m_i-2}u) )_{j\in I_0}
\]
for some $a_i \in \bfk^\times$ and $m_i \in \ZZ_{\geq 0}$ such that $M \iso M_1 \otimes \cdots \otimes M_{\ell}$, then there exists a set partition $\mathcal{X} = \{\mathcal{X}_1, \cdots, \mathcal{X}_{\ell}\}$ of $I_0$ of size $\ell$ such that the Drinfel'd polynomials of $M_k$ are
\[
\mathcal{P}_i(u) = \begin{cases} \mathcal{P}^M(u) & \text{if } i \in \mathcal{X}_k, \\ 1 & \text{otherwise}. \end{cases}
\]
\end{proposition}

\begin{proof}
This follows from~\cite[Thm.~4.8]{CP91} in conjunction with~\cite[Prop.~2.2]{CP96III}.
\end{proof}

A more general statement could also be made for decomposing each Drinfel'd polynomial for each $i \in I_0$ into $q$-strings in general position, but since we only need the level of generality of Proposition~\ref{prop: CP partition}, we state it in that form for brevity.
In other words, we can not separate $q$-strings between two different tensor factors in a simple module.
By restricting to the case when only two Drinfel'd polynomials are nontrivial, we have the follow corollary.

\begin{corollary}
    \label{cor:two_factor_prime}
    Assume $ k \ne l \in I_0$. 
    The module $M$ with Drinfel'd polynomials
    \begin{align} \label{eq: Drin poly}
    \mathcal{P}_i(u) = \begin{cases}
    (1-au) (1-aq_l^2u) (1-aq_l^4u) \cdots (1-aq_l^{2p-2}u) & \text{if } i = l, \\
    (1-bu) (1-bq_k^2u) (1-bq_k^4u) \cdots (1-bq_k^{2m-2}u) & \text{if } i = k, \\
    1 & \text{otherwise,}
    \end{cases}
    \end{align}
    is prime if and only if $\Vkm{l^p}_a \otimes \Vkm{k^m}_b$ is not simple.
\end{corollary}

\subsection{\texorpdfstring{$i$}{i}-boxes and T-systems} \label{subsec: i-box and T-system}

Recall that for each quantum affine algebra $\uqpg$, we fix $\gfin$ and $\hbDynkin^\sigma$.
Also we can choose a compatible reading 
\[
\frakR =(\ldots,(\im_{-1},p_{-1}),(\im_0,p_0),(\im_1,p_1),\ldots) \text{ of }\hbDynkin^\sigma.
\]
Throughout this subsection, we fix a compatible reading $\frakR$ and recall Convention~\ref{conv: identify R}. 
Thus we skip $^\frakR$ or $_\frakR$ in notations if there is no danger of confusion.

For each $k \in \Z$, we define the fundamental module 
\begin{align} \label{eq: Vkmk}
\frakR[k] \seteq \calV(\im_k,p_k), \quad   
\text{where $\pR{\frakR}(k)=(\im_k,p_k)$}. 
\end{align}

\begin{proposition} [{\cite[Proposition 5.7]{KKOP23P}}] \label{prop: lemma of KKOP23P} 
For any $a<b$, the ordered pair $(\frakR[b],\frakR[a])$ is strongly unmixed.
\end{proposition}

For $s \in \Z$ and $\jm\in \bDynkin_0$, we define
\begin{equation} \label{eq: nota +,-}
\begin{aligned}
\sR{\frakR}^+ &\seteq\min\{t \mid s<t ,\; \im_t=\im_s \},
& \sR{\frakR}^-&\seteq\max\{t \mid t<s, \; \im_t=\im_s \},
\\
\sR{\frakR}(\jm)^+ &\seteq\min\{t \mid s \le  t,\; \im_t=\jm \},
& \sR{\frakR}(\jm)^- &\seteq\max\{t \mid t \le s,\; \im_t=\jm \}.
\end{aligned}
\end{equation}

\begin{lemma}[{\cite[\S 6.1]{KKOP24}}] \label{lem: property of reading}
Let $\frakR$ be a compatible reading of $\hbDynkin^\sigma$.  
\bnum
\item $p_{s^+}=p_{s}+2d_{\oim}$.
\item $p_t = p_s + \min(d_{\overline{\im_s}},d_{\overline{\im_t}})$ if $d_\bDynkin(\im_s,\im_t)=1$ and $t^-<s < t < s^+$.
\item $p_t-(p_s+2) \in 2 d_{\overline{\im_s}}\Z$ if $s,t \in \Z$ satisfies $\im_t = \sigma(\im_s)$.
\item Let $\im_s,\im_t \in \bDynkin_0$ with $d_{\bDynkin}(\im_s,\im_t)=1$. Then
\bna
\item $p_s < p_t$ if and only if $s<t$.
\item $p_s-p_t-\min(d_{\overline{\im_s}},d_{\overline{\im_t}}) \in 2 \min(d_{\overline{\im_s}},d_{\overline{\im_t}}) \Z$. 
\ee
\ee
\end{lemma}

\begin{definition}[{\cite[Deﬁnition 4.14]{KKOP24A}}]  \hfill
\ben
\item We say that an interval $[a,b]$ is an \defn{$i$-box associated with $\frakR$} if
$-\infty<a\le b<+\infty$ and $\im_a=\im_b$. 
We sometimes write it as $[a,b]^{\frakR}$ to emphasize that it is associated with $\frakR$. 
\item
For an $i$-box $[a,b]$, we set
$$ |[a,b]|_\phi \seteq  \big| \{ s \mid \text{$s\in[a,b]$
and $\im_a=\im_s=\im_b$} \}\big|.$$
\item For an $i$-box $[a,b]$, we define
\begin{align*}
\frakR[a,b] \seteq \hd(\frakR[b] \tens \frakR[b^-] \tens \cdots \tens \frakR[a^+] \tens \frakR[a] ).
\end{align*}
\ee
\end{definition}

By Proposition~\ref{prop: lemma of KKOP23P} and Lemma~\ref{lem: property of reading}, the sequence of root modules $(\frakR[b] , \frakR[b^-], \dotsc, \frakR[a^+] ,\frakR[a] )$ is normal and $\frakR[a,b]$ is a KR module.
Note that $\clr(\frakR[a,b])$ of $\frakR[a,b]$ is canonically defined as $\im_a=\im_b$.

In order to compute the $\de$-invariants between KR modules in the sequel, we introduce the following definitions.

\begin{definition} \label{def: ranges}
Let $[a,b]$ be an $i$-box associated with $\frakR$. 
\bnum
\item
For a KR module $\frakR[a,b]$, we call an interval 
$$\rch(\frakR[a,b]) \seteq \range{p_a,p_b}$$ 
the \defn{reach} of $\frakR[a,b]$.
Here we use bold bracket $\range{ \ , \ } $ to distinguish it with an $i$-box $[ \ , \ ]$ associated with $\frakR$
and to emphasize it as a reach in $\hbDynkin^\sigma$.
\item For $p \in \Z$ and $\im \in \bDynkin_0$, we set 
\[
p_{\le}(\im) \seteq \max\{ r \ | \  r \le p, \ (\im, r) \in \hbDynkin^\sigma_0) \qtq p_{\ge}(\im) \seteq \min\{ r \ | \  r \ge p, \ (\im, r) \in \hbDynkin^\sigma_0).
\]
In particular, note that $\sR{\frakR}(\im)^+$ (resp.\ $\sR{\frakR}(\im)^-$) in~\eqref{eq: nota +,-} is different from $p_\ge(\im)$ (resp.\ $p_\le(\im)$) as the former depends on $\frakR$ while the latter do not. 
\item For a KR module $\frakR[a,b]$ with its reach $\range{p_a,p_b}$ and $\im_a=\im$, and $\jm \in \bDynkin_0$, we call an interval 
\begin{align} \label{eq: exrch}
\exrch_\jm(\frakR[a,b])  \seteq  \range{ (\tp_{a})_{\le}(\jm)-2d_{\ojm} , \; (\tp_{b})_{\ge}(\jm)+ 2d_{\ojm} }    
\end{align}
the \defn{$\jm$-extended reach} of $\frakR[a,b]$, where $\tp_{a} \seteq p_{a} -  \td_{\bDynkin,\sigma}(\im,\jm)$
and $\tp_{b} \seteq p_{b}  +  \td_{\bDynkin,\sigma}(\im,\jm)$.
\ee
\end{definition}
Note that, for a given KR module $M$ in $\Cg^0$, its reach and extended reaches do not depend on the choice of 
a compatible reading $\frakR$.

\begin{example}
Let us consider the repetition quiver $\hbDynkin^\sigma$ of type $B_3^{(1)}$ given in Example~\ref{ex: Repetition quiver}(\ref{it: Rq B3}).
\bnum
\item For $p = -4 \in \Z$ and $\im =1 \in \bDynkin_0$, we have $p_{\le}(\im)= -6$ and $p_{\ge}(\im)=-2$.
\item Let us consider the compatible reading $\frakR$ in Example~\ref{ex: c reading}~\eqref{it: ne reading} with $(\im_1,p_1)=(2,0)$.
Then we have $\rch(\frakR[0,3]) = \range{-1,1}$, $\exrch_3(\frakR[0,3])= \range{-3,3}$
$\exrch_1(\frakR[-2,4])=\range{-13,10}$ and $\exrch_5(\frakR[-1,5])=\range{-8,8}$.
\ee
\end{example}

\begin{proposition}[{\cite{KKOP24A}}] \label{prop: i-box d-value}
Let $[a,b]$ an $i$-box associated with a compatible reading $\frakR$. 
\bnum
\item For $s \in\Z$ such that $a^-<s < b^+$, we have $\de(\frakR[a,b],\frakR[s])=0$ .
\item \label{it: d=1 a- b+} We have $ \de(\frakR[b^+],\frakR[a,b]) = \de(\frakR[a^-],\frakR[a,b])=1$.
\item \label{it: d=1 1 a-b a b+} We have $\de(\frakR[a,b],\frakR[a^-,b^-])=1.$
\item \label{it: d=1 Da a} We have $\de(\scrD \frakR[b],\frakR[a,b])=1$ and $\de(\scrD^{-1}\frakR[a],\frakR[a,b])=1.$
\ee   
\end{proposition}

\begin{definition}
We say that $i$-boxes $[a_1,b_1]$ and $[a_2,b_2]$ \defn{commute} if
we have either
\[
a_1^- < a_2 \le b_2 <  b_1^+ \quad  \text{ or } \quad a_2^- < a_1 \le b_1 <  b_2^+.
\]
\end{definition}

\begin{theorem}[{\cite{KKOP24A}}] \label{thm: i-box commute}
For a compatible reading $\frakR$, if two $i$-boxes  $[a_1,b_1]$ and $[a_2,b_2]$ commute, then
$\frakR[a_1,b_1]$ and $\frakR[a_2,b_2]$ commute.   
\end{theorem} 

Recall the two extremal compatible readings in Example~\ref{ex: c reading}~\eqref{it: se reading} and~\eqref{it: ne reading}.
Then, using Lemma~\ref{lem: property of reading} and terms of reaches, Theorem~\ref{thm: i-box commute} can be re-expressed as follows:
For KR modules 
$W_1$ and $W_2$ strongly commute 
if either
\begin{align} \label{eq: range commute}
{\rm(i)} \ \ p'_{a_2}  <  p_{a_1} \le p_{b_1} <p'_{b_2}
\quad \text{ or } \quad
{\rm(ii)} \ \  p'_{a_1}  <  p_{a_2} \le p_{b_2} <p'_{b_1},
\end{align}
where 
\[
W_u = \frakR[a_u,b_u] \text{ with } \im_{a_u}=\im_{b_u}=\im_u, \ \ \rch(W_u)=\range{p_{a_u},p_{b_u}}, \ \ \text{ and } \exrch_{\im_v}(W_u)=\range{p'_{a_u},p'_{b_u}}
\]
for $\{u,v\}=\{1,2\}$ and some compatible reading $\frakR$.

\begin{theorem}[{\cite{Nak04,Her06}; see also \cite[\S 4.6]{KKOP24A}}]  
\label{th:Tsystem}
For an arbitrary quantum affine algebra $\uqpg$, a compatible reading $\frakR$ of $\hbDynkin^\sigma$ and an arbitrary
$i$-box $[a,b]^\frakR$ such that $a<b$, we have a short exact sequence
\begin{align} \label{eq: T-system}
0 \to \hspace{-2ex} \dtens_{ \substack{ \jm \in \bDynkin_0; \; d(\im_a,\jm)=1}} \hspace{-2ex}
\frakR[a(\jm)^+,b(\jm)^-]  \to   \frakR[a^+,b]  \tens \frakR[a,b^-]    \to   \frakR[a,b]
\tens \frakR[a^+,b^-] \to 0,
\end{align}
where the left and right terms are real simple modules. 
\end{theorem}
We call the exact sequence in~\eqref{eq: T-system} a \defn{T-system}.
In Appendix~\ref{Appendix: T-system}, we present the T-systems for all untwisted affine types.

\subsection{Generalized  Dorey's rule} \label{subsec: lower Dorey}

In this subsection, we review the morphisms
\begin{align}\label{eq: Dorey}
\Hom_{\uqpg}\big(V(\varpi_i)_a \tens V(\varpi_j)_b,V(\varpi_k)_c\big) \quad \text{ for } i,j,k\in I_0 \text{ and } a,b,c \in \ko^\times
\end{align}
that are studied by~\cite{CP96,FH15,KKK15,Oh14R,Oh19,YZ11}.
The morphisms in~\eqref{eq: Dorey} are called \defn{Dorey's rule} due to their origin from the work of Dorey~\cite{Dorey91,Dorey93}.

We keep the notation $\gfin$ for a given $\g^{(t)}_n$ as defined in~\eqref{Table: root system} and~\eqref{Table: root system twisted}.

\begin{proposition}[{\cite[Lemma 2.6]{BKM12}}] \label{pro: BKM minimal}
Let $\bfg$ be a finite simple Lie algebra.
For $\gamma \in \Phi^+_\bfg \setminus \Pi_\bfg$ and any $\redez$ of $w_0 \in W_\bfg$, a $[\redez]$-minimal exponent of $\gamma$ is indeed a pair $(\al,\be)$ for some $\al,\be \in \Phi^+_\bfg$ such that $\al+\be = \gamma$.
\end{proposition}

Recall the exponent $\um =\exponent{\um_\beta}_{\be \in \Phi^+} \in \Z_{\ge0}^{\Phi^+}$ in \S\ref{subsec: Convex}.
For a reduced expression $\twz$ of $w_0$ adapted to a $\rmQ$-datum $\calQ$, we set  
\[
\calV_\mathcal{Q}(\um) \seteq    \calV_\calQ(\beta^\twz_\ell)^{\tens m_{\beta_\ell}}  \tens \cdots \tens  \calV_\calQ(\beta^\twz_2)^{\tens m_{\beta_2}} \tens  \calV_\calQ(\beta^\twz_1)^{\tens m_{\beta_1}}.
\]

Dorey's rule for affine Kac--Moody algebra $\g$ can be described as follows.

\begin{theorem}[{\cite{CP96,Oh14A,Oh14D,Oh19,KO18,OhS19}}] \label{thm: Dorey}
For a quantum affine algebra $\uqpg$, let $(i,x)$, $(j,y)$, $(k,z) \in I_0 \times \mathbf{k}^\times$. Then
\[
\Hom_{\uqpg}\big( V(\varpi_{j})_y \otimes V(\varpi_{i})_x, V(\varpi_{k})_z \big) \ne 0
\]
if and only if there exists a $\rmQ$-datum $\calQ$ of $\g$  and $\al,\be,\gamma \in \Phi_{\gfin}^+$ such that
\bnum
\item  $\{ \al,\be \}$ is a pair of positive roots such that $\al \prec_{[\calQ]} \be$ and  $\al+\be=\gamma$,
\item $V(\varpi_{j})_y  = \calV_{\calQ}(\al)_t, \ V(\varpi_{i})_x  = \calV_{\calQ}(\be)_t, \ V(\varpi_{k})_z  = \calV_{\calQ}(\gamma)_t$ for some $t \in \ko^\times$.
\ee 
Furthermore the followings holds: 
\ben
\item \label{it: untwisted}
Let $\g$ be of affine type $A_n^{(t)}$, $D_n^{(t)}$ or $E_{6,7,8}^{(t)}$ $(t=1,2,3)$, and $\calQ$ be its $\rmQ$-datum.
Assume that $\up=\{ \al,\be\}$ is a non $[\calQ]$-simple pair such that $\al \prec_{[\calQ]} \be$. Then there exits a unique $[\calQ]$-simple exponent $\us = (\ga_1,\ldots,\ga_u)$
such that $\us=\soc_{[\calQ]}(\up)$ and
\[
\hd( \calV_{\calQ}(\al) \otimes \calV_{\calQ}(\be)) \iso \soc( \calV_{\calQ}(\be) \otimes \calV_{\calQ}(\al)) \iso  \calV_{\calQ}(\us) \text{ is simple.}
\]
\item \label{it: untwisted non}
Let $\g$ be of affine type $B_n^{(1)}$, $C_n^{(1)}$, $F_4^{(1)}$ or $G_{2}^{(1)}$, and $\calQ$ be its $\rmQ$-datum.
Let $\up=(\al,\be)$ be a non $[\mathcal{Q}]$-simple pair such that $\al \prec_{[\calQ]} \be$, then the following holds:
\eq &&
\parbox{93ex}{
\bna
\item[{\rm (a)}] if $\dist_{[\mathcal{Q}]}(\up)=1$, then we have
\[
\hd(\calV^{(1)}_\mathcal{Q}(\al) \tens \calV^{(1)}_\mathcal{Q}(\be)) \iso \soc(\calV^{(1)}_\mathcal{Q}(\be) \tens \calV^{(1)}_\mathcal{Q}(\al)) \iso \calV^{(1)}_{\mathcal{Q}}(\soc_{[\mathcal{Q}]}\{ \al,\beta \}),
\]
which is simple,
\item[{\rm (b)}] if $\dist_{[\mathcal{Q}]}(\up)=2$ and there exists a unique exponent $\um$ satisfying
\[
\soc_{[\calQ]}(\up) \prec^\tb_{[\mathcal{Q}]} \um \prec^\tb_{[\mathcal{Q}]} \up, \hspace{60pt}
\]
then we have $\hd(\calV^{(1)}_\mathcal{Q}(\al) \tens \calV^{(1)}_\mathcal{Q}(\be)) \iso \hd(\calV^{(1)}_\mathcal{Q}(\um))$, which is a KR module.
\ee
}\label{eq: KR-head}
\eneq
\ee
\end{theorem}

\begin{example} \label{ex: simple head}
In~\eqref{eq: AR-quiver D}, consider the pair $(\lan 1,-2 \ran,\lan 1,2 \ran)$. It is a $[\calQ]$-minimal pair for the $[\calQ]$-simple exponent $(\lan 1,-4 \ran,\lan 1,4 \ran)$.
Then Theorem~\ref{thm: Dorey} tells that we have a surjective homomorphism
\[
\calV_\calQ(\lan 1,2\ran) \otimes \calV_\calQ(\lan 1,-2\ran) \twoheadrightarrow \calV_\calQ(\lan 1,4\ran) \otimes \calV_\calQ(\lan 1,-4\ran),
\text{ which is simple}.
\]
\end{example}

\begin{remark}
By Proposition~\ref{prop: hconv simple}, $\de(\calV_\calQ(\al),\calV_\calQ(\be)) >0$ for a non $[\calQ]$-simple pair $\up =(\al,\be)$ in Theorem~\ref{thm: Dorey}. Also, there are non $[\calQ]$-simple pair
$\up=(\al,\be)$ such that $\de(\calV_\calQ(\al),\calV_\calQ(\be))>1$ and  $\up'=(\al',\be')$ whose composition length of $\calV_\calQ(\al') \otimes \calV_\calQ(\be')$ is strictly larger than $2$.  For instance,
consider the pair $(\lan 1,-3 \ran,\lan 2,3 \ran)$ in~\eqref{eq: AR-quiver D}. Then we have
$$\de(\calV_\calQ(\lan 1,-3 \ran),\calV_\calQ(\lan 2,3 \ran) =2, \qquad \calV_\calQ(\lan 1,-3 \ran)\hconv \calV_\calQ(\lan 2,3 \ran)
\iso \calV_\calQ(\lan 1,2 \ran)$$ and the composition length of $\calV_\calQ(\lan 1,-3 \ran)\tens \calV_\calQ(\lan 2,3 \ran)$
is strictly larger than $2$
\end{remark}

\begin{remark}
For $\g=B_n^{(1)}, C_n^{(1)}, F_4^{(1)}$ (resp. $G_2^{(1)}$), there exists $k \in I_0$ (resp. $1 \in I_0$), such that $\im \ne \jm \in \bDynkin_0$ (resp. $1,3,4 \in \bDynkin_0$) and $\{ \oim, \overline{\jm} \} =k$ (resp. $\{ \overline{1}, \overline{3}, \overline{4} \}=1$).
In particular, when $\g=B_n^{(1)}$, $\jm=2n-\im$.
Thus to emphasize the color of $\Vkm{k^m}_x$ within a fixed skeleton category, we sometimes write $\Wkm{\im^m}_x$ or $\Wkm{\jm^m}_x$ instead of $\Vkm{k^m}_x$, 
according to their colors. 
Throughout this paper, we specially deal with the $B_n^{(1)}$-cases, since $k = \{ \overline{k}, \overline{2n-k}\} \in I_0 \setminus \{n\}$. 
\end{remark}

As special cases of Theorem~\ref{thm: Dorey}~\eqref{it: untwisted} and~\eqref{it: untwisted non}, we have the following special families of homomorphisms for classical untwisted affine types:

\begin{corollary} \label{cor: mesh}
For classical untwisted affine types $\g=A_{n-1}^{(1)}$, $B_n^{(1)}$, $C_n^{(1)}$, $D_n^{(1)}$, and
$l < n'  \seteq n- \delta(\g = D_n^{(1)} )$, assume $1\le a,b < n'$ with $a+b < l$. Then we have 
\begin{subequations} \label{eq: mesh type Dorey's rule}
\begin{align}  \label{eq: mesh lower k+l<n}
\Vkm{l-b}_{(-\chq_{l-b})^{-b}} \hconv \Vkm{l-a}_{(-\chq_{l-a})^{a}} \iso \Vkm{l} \otimes \Vkm{l-a-b}_{(-\chq_{l-a-b})^{a-b}}.
\end{align}
\bna
\item Consider $\g=B_n^{(1)}$ and assume $l > n$. Suppose there exists $a,b \in [1,2n-1]$ such that $l-b < n < l-a$ and $a+b \le l$.
Then we have 
\begin{align} \label{eq: folded homo}
    \Wkm{l-b}_{(-q)^{-b+1}} \hconv \Wkm{l-a}_{-(-q)^{a}} \iso \Wkm{l}_{(-1)} \otimes \Wkm{l-a-b}_{(-q)^{a-b+1}}. 
\end{align}
Here we understand $\Wkm{l-a-b}_{(-q)^{a-b+1}}$ is a trivial module $\mathbf{1}$ if $l=a+b$. 
\item If $l=n'$, we have
\begin{equation} \label{eq: mesh lower k+l=n}
\begin{aligned}
&\Vkm{n'-b}_{(-\chq_{n'-b})^{-b}} \hconv \Vkm{n'-a}_{(-\chq_{n'-a})^{a}}  \\  & \hspace{5ex}\iso
\bc
\Vkm{n^{2}}_{(-1)} \otimes \Vkm{n-a-b}_{(-\chq_{n-a-b})^{a-b}} & \text{ if } \g =B_n^{(1)}, \\
\Vkm{n}  \otimes \Vkm{n-a-b}_{(-\qs)^{a-b}} & \text{ if } \g =C_n^{(1)}, \\
\Vkm{n} \otimes \Vkm{n-1}  \otimes  \Vkm{n-1-a-b}_{(-q)^{a-b}} & \text{ if } \g =D_n^{(1)}.
\ec
\end{aligned}
\end{equation}  
\ee
\end{subequations}

\end{corollary}
\noindent
We call homomorphisms in~\eqref{eq: mesh type Dorey's rule} \defn{Dorey's rule of mesh type}. 
Next, let us give an example of Dorey's rule of mesh type.

\begin{example}
In~\eqref{eq: AR-quiver A}, we have a length $2$ chain of exponents:
$$
\{ \drange{1,5}, \drange{2,3} \} \prec^\ttb_{[\calQ]} \{ \drange{1,3}, \drange{2,5} \}.
$$
Then we have
$$
 \calV_\calQ(\drange{1,3}) \hconv  \calV_\calQ(\drange{2,5}) \iso \calV_\calQ(\drange{2,3}) \otimes \calV_\calQ(\drange{1,5}),
$$
which is simple.  
\end{example}

As special cases of Theorem~\ref{thm: Dorey}~\eqref{it: untwisted non}, we have the following special families of homomorphisms in~\eqref{eq: new homo}:

\begin{corollary} \hfill
\begin{enumerate}[{\rm (a)}]
\begin{subequations} \label{eq: new homo}
\item For $1 \le k \le n-1$ and $a \in \bfk^\times$, there exists a $U_q(B_n^{(1)})$-module surjective homomorphism
\begin{align} \label{eq: new homom B}
\Vkm{k}_{(-1)^{n+1-k}aq^{-n+k}} \otimes \Vkm{n-k}_{a(-1)^{k+1}q^{k}} \twoheadrightarrow \Vkm{n^2}_a.
\end{align}
\item For $1 \le k \le n-1$ and $a \in \bfk^\times$, there exists a $U_q(C_n^{(1)})$-module surjective homomorphism
\begin{align} \label{eq: new homom C} 
\Vkm{n}_{a(-\qs)^{-1-n+k}} \otimes \Vkm{n}_{a(-\qs)^{n+1-k}} \twoheadrightarrow \Vkm{k^2}_a.
\end{align}
\item 
For $1 \le k\le n - 2$ and $n',n''\in \{n,n-1\}$ such that $n'-n'' \equiv_2 n - k$, there
exist a $U_q(D_n^{(1)})$-module surjective homomorphism
\begin{align} \label{eq: new homom D} 
\Vkm{n'}_{(-q)^{-n+l+1}} \otimes  \Vkm{n''}_{(-q)^{n-l-1}} \to \Vkm{k}. 
\end{align}
\item There exists a $G_2^{(1)}$-module surjective homomorphism 
\begin{align} \label{eq: new homom G}
\Vkm{1}_{(-\qt)^{-2}} \otimes \Vkm{2}_{(-\qt)^{5}} \twoheadrightarrow \Vkm{2^2}. 
\end{align}
\item For $a \in \bfk^\times$, there exist $U_q(F_4^{(1)})$-module surjective homomorphisms
\begin{equation} \label{eq: new homom F}
\begin{aligned} 
\Vkm{1}_{a(-\qs)^{-5}} \otimes \Vkm{1}_{a(-\qs)^{5}} \twoheadrightarrow \Vkm{4^2}_{-a} \qtq \Vkm{2}_{-aq^{-2}} \otimes \Vkm{1}_{a\qs^4} \twoheadrightarrow \Vkm{3^2}_{-a}.
\end{aligned}
\end{equation}
\end{subequations}
\end{enumerate}
\end{corollary}

Let us see morphisms described in~\eqref{eq: new homo} with examples: In~\eqref{eq: B3 AR}, one can see that
there exists a unique chain of exponents satisfying the condition:
\[
\drange{1,5} \prec^\tb_{[\calQ]} \{\drange{1,3},\drange{4,5}\}   \prec^\tb_{[\calQ]} \{ \drange{1},\drange{2,5} \}.
\]
Then 
\[
\hd(\calV_\calQ(\drange{2,5}) \otimes \calV_\calQ(\drange{1})) \iso \hd(\calV_\calQ(\drange{1,3}) \otimes\calV_\calQ(\drange{4,5})),
\]
which is a KR module. In~\eqref{eq: AR-quiver C}, one can see also that there exists a unique chain of exponents satisfying  
\[
\lan 1,3 \ran \prec^\tb_{[\calQ]} \{ \lan 1,4 \ran,\lan 3,-4 \ran \}   \prec^\tb_{[\calQ]} \{ \lan 1,-4\ran,\lan 3,4 \ran \}.
\]
Then 
\[
\hd(\calV_\calQ(\lan 3,4 \ran) \otimes \calV_\calQ(\lan 1,-4\ran)) \iso \hd(\calV_\calQ(\lan 3,-4\ran) \otimes \calV_\calQ(\lan 1,4\ran)),
\]
which is a KR module.


\section{Kashiwara Crystals}
\label{sec:crystals}

We recall some basic facts about Kashiwara crystals~\cite{Kashiwara90,Kashiwara91}.
Here, we give an abbreviated definition focusing on the category of crystals coming from finite dimensional (highest weight) $U_q(\g_0)$-modules in order to focus on the facts that we need, but we refer the reader to~\cite{BS17,Kashiwarabook} for more details.

\begin{definition}
A \defn{$U_q(\g_0)$-crystal} (which we simply call a crystal when the quantum group is clear) is a set $\crystal$ with \defn{Kashiwara operators} $\ke_i, \kf_i \colon \crystal \to \crystal \sqcup \{\zero\}$ such that $\kf_i b = b'$ if and only if $b = \ke_i b'$ for all $b,b' \in \crystal$.
We equate this with the $I_0$-edge labeled digraph with edges $b \xrightarrow{i} \kf_i b$.
We also define statistics for each $i \in I_0$
\[
\varepsilon_i(b) = \max \{k \mid \ke_i^k b \neq \zero \},
\qquad\qquad
\varphi_i(b) = \max \{k \mid \kf_i^k b \neq \zero \},
\]
and weight function $\wt \colon \crystal \to  \wl   $ by $\wt(b) = \sum_{i \in I_0} (\varphi_i(b) - \varepsilon_i(b))  \UpLa_i $.
\end{definition}

The direct sum of two crystals is simply their disjoint union with the obvious crystal structure.
Morphisms are the maps that commute with the crystal structure (normally these are called strict morphisms and have more restrictions than in the usual category of crystals, but it will be sufficient for our purposes).

\begin{definition}[Tensor product of crystals]
For crystals $\crystal_1, \dotsc, \crystal_k$, we define the tensor product crystal $\crystal_k \otimes \cdots \otimes \crystal_1$ as the Cartesian product $\crystal_k \times \cdots \times \crystal_1$ with the Kashiwara operators defined by the following \defn{bracketing rule}.
Fix an element $b = b_k \otimes \cdots \otimes b_1$ and some $i \in I_0$.
Consider the following word of the form
\[
\underbrace{) \cdots )}_{\varphi_i(b_k)} \, \underbrace{( \cdots (}_{\varepsilon_i(b_k)}
\;\cdots\;
\underbrace{) \cdots )}_{\varphi_i(b_1)} \, \underbrace{( \cdots (}_{\varepsilon_i(b_1)}
\qquad \longrightarrow \qquad
\underbrace{) \cdots )}_{\varphi_i(b)} \, \underbrace{( \cdots (}_{\varepsilon_i(b)},
\]
where we have done the reduction by removing all pairs `$()$' from the word.
Let $m_e$ and $n_f$ correspond to the leftmost `$($' (resp.\ rightmost `$)$') in the reduced word, then the Kashiwara operators are defined by
\[
\ke_i b := b_k \otimes \cdots \otimes (\ke_i b_{m_e}) \otimes \cdots \otimes b_1,
\qquad\qquad
\kf_i b := b_k \otimes \cdots \otimes (\kf_i b_{m_f}) \otimes \cdots \otimes b_1.
\]
In particular, $\wt(b) = \sum_{m=1}^k \wt(b_m)$.
\end{definition}

\begin{remark}
We follow the tensor product convention in~\cite{BS17}, which is the opposite convention of Kashiwara~\cite{Kashiwarabook}.
\end{remark}

Let $\crystal(\lambda)$ denote the crystal associated to the highest weight $U_q(\g_0)$-module $V(\lambda)$, and $u_{\lambda}$ the unique \defn{highest weight element}, that is $\ke_i b = \zero$ for all $i \in I_0$.
Kashiwara showed~\cite{Kashiwara90,Kashiwara91,Kashiwara93} that the highest weight elements in a crystal correspond to the decomposition of $U_q(\g_0)$-modules into simple modules (counted with multiplicity).
Furthermore, explicit tableau descriptions of $\crystal(\lambda)$ was given for Lie algebras of type ABCDG in~\cite{KN94,KM94} with all non-spin cases as follows.
Indeed, the set of \defn{Kashiwara--Nakashima $($KN$)$ tableaux} is generated by $u_{\lambda}$ being the tableau of shape $\lambda$ (under the usual identification by $\wl \subseteq \bigoplus_{i=1}^n \ZZ \epsilon_i$ with the coefficient of $\epsilon_i$ being the number of boxes in the $i$-th row) with the $i$-th row filled by $i$ and embedded in $\crystal(\La_1)^{\otimes \abs{\lambda}}$, where $\crystal(\La_1)$ is given in Figure~\ref{fig:natrep_crystals}, by reading from bottom-to-top along columns and columns are read left-to-right.
A precise characterization of these tableaux are given that we omit for brevity, but we note that all KN tableaux have the following version of semistandardness.

\begin{definition}[Semistandard tableau]
\label{def:semistandard}
Let $T$ be a tableaux of type $\g_0$.
Then $T$ is \defn{semistandard} if rows (resp.\ columns) are weakly (resp.\ strictly) increasing in the alphabet $\crystal(\UpLa_1)$ considered as a poset whose Hasse diagram is given by Figure~\ref{fig:natrep_crystals} with the following exceptions:
\begin{itemize}
    \item In types $B_n$ and $G_2$, an entry $0$ cannot be repeated along rows but it can be repeated down columns.
    \item In type $D_n$, columns do not need to be comparable (as opposed to rows).
    \item In type $G_2$, it has more complicated rules; see~\cite{KM94}.
\end{itemize}
\end{definition}

\begin{figure}
\[
\begin{array}{|rl|}\hline
A_{n-1}: &
\begin{tikzpicture}[xscale=1.95,baseline=-4]
\node (1) at (0,0) {$\ytableaushort{1}$};
\node (2) at (1.5,0) {$\ytableaushort{2}$};
\node (d) at (3.0,0) {$\cdots$};
\node (n-1) at (4.5,0) {$\ytableaushort{{n_-}}$};
\node (n) at (6,0) {$\ytableaushort{n}$};
\draw[->,red] (1) to node[above]{\tiny$1$} (2);
\draw[->,HUgreen] (2) to node[above]{\tiny$2$} (d);
\draw[->,brown] (d) to node[above]{\tiny$n-1$} (n-1);
\draw[->,blue] (n-1) to node[above]{\tiny$n$} (n);
\end{tikzpicture}\\
B_n: &
\begin{tikzpicture}[xscale=1.3,baseline=-4]
\node (1) at (0,0) {$\ytableaushort{1}$};
\node (d1) at (1.5,0) {$\cdots$};
\node (n) at (3,0) {$\ytableaushort{n}$};
\node (0) at (4.5,0) {$\ytableaushort{0}$};
\node (bn) at (6,0) {$\ytableaushort{\bn}$};
\node (d2) at (7.5,0) {$\cdots$};
\node (b1) at (9,0) {$\ytableaushort{\bon}$};
\draw[->,red] (1) to node[above]{\tiny$1$} (d1);
\draw[->,brown] (d1) to node[above]{\tiny$n-1$} (n);
\draw[->,blue] (n) to node[above]{\tiny$n$} (0);
\draw[->,blue] (0) to node[above]{\tiny$n$} (bn);
\draw[->,brown] (bn) to node[above]{\tiny$n-1$} (d2);
\draw[->,red] (d2) to node[above]{\tiny$1$} (b1);
\end{tikzpicture}
\\[10pt]
C_n: &
\begin{tikzpicture}[xscale=1.3,baseline=-4]
\node (1) at (0,0) {$\ytableaushort{1}$};
\node (d1) at (1.8,0) {$\cdots$};
\node (n) at (3.6,0) {$\ytableaushort{n}$};
\node (bn) at (5.4,0) {$\ytableaushort{\bn}$};
\node (d2) at (7.2,0) {$\cdots$};
\node (b1) at (9,0) {$\ytableaushort{\bon}$};
\draw[->,red] (1) to node[above]{\tiny$1$} (d1);
\draw[->,brown] (d1) to node[above]{\tiny$n-1$} (n);
\draw[->,blue] (n) to node[above]{\tiny$n$} (bn);
\draw[->,brown] (bn) to node[above]{\tiny$n-1$} (d2);
\draw[->,red] (d2) to node[above]{\tiny$1$} (b1);
\end{tikzpicture}
\\[10pt]
D_n: &
\begin{tikzpicture}[xscale=1.3,baseline=-4]
\node (1) at (0,0) {$\ytableaushort{1}$};
\node (d1) at (1.5,0) {$\cdots$};
\node (n-1) at (3,0) {$\ytableaushort{{n_-}}$};
\node (n) at (4.5,.75) {$\ytableaushort{n}$};
\node (bn) at (4.5,-.75) {$\ytableaushort{\bn}$};
\node (bn-1) at (6,0) {$\ytableaushort{{\overline{n_-}}}$};
\node (d2) at (7.5,0) {$\cdots$};
\node (b1) at (9,0) {$\ytableaushort{\bon}$};
\draw[->,red] (1) to node[above]{\tiny$1$} (d1);
\draw[->,purple] (d1) to node[above]{\tiny$n-2$} (n-1);
\draw[->,brown] (n-1) to node[above,sloped]{\tiny$n-1$} (n);
\draw[->,blue] (n-1) to node[below,sloped]{\tiny$n$} (bn);
\draw[->,blue] (n) to node[above,sloped]{\tiny$n$} (bn-1);
\draw[->,brown] (bn) to node[below,sloped]{\tiny$n-1$} (bn-1);
\draw[->,purple] (bn-1) to node[above]{\tiny$n-2$} (d2);
\draw[->,red] (d2) to node[above]{\tiny$1$} (b1);
\end{tikzpicture}
\\[30pt]
G_2: &
\begin{tikzpicture}[xscale=1.3,baseline=-4]
\node (1) at (0,0) {$\ytableaushort{1}$};
\node (2) at (1.5,0) {$\ytableaushort{2}$};
\node (3) at (3,0) {$\ytableaushort{3}$};
\node (0) at (4.5,0) {$\ytableaushort{0}$};
\node (b3) at (6,0) {$\ytableaushort{\bth}$};
\node (b2) at (7.5,0) {$\ytableaushort{\btw}$};
\node (b1) at (9,0) {$\ytableaushort{\bon}$};
\draw[->,red] (1) to node[above]{\tiny$1$} (2);
\draw[->,HUgreen] (2) to node[above]{\tiny$2$} (3);
\draw[->,red] (3) to node[above]{\tiny$1$} (0);
\draw[->,red] (0) to node[above]{\tiny$1$} (b3);
\draw[->,HUgreen] (b3) to node[above]{\tiny$2$} (b2);
\draw[->,red] (b2) to node[above]{\tiny$1$} (b1);
\end{tikzpicture}
\\\hline
\end{array}
\]
\caption{The crystal corresponding to the natural representation $V(\UpLa_1)$ for each of the Lie algebras of type ABCDG; here $n_- := n-1$.}
\label{fig:natrep_crystals}
\end{figure}

By the tensor product rule, we have the following well-known result.
(This is also known as the Yamanouchi condition from tableaux in type $A_n$.)

\begin{proposition}
\label{prop:highest_wt_recursive}
If $b_m \otimes b_{m-1} \otimes \cdots \otimes b_1$ is a highest weight element, then $b_{m-1} \otimes \cdots \otimes b_1$ is also a highest weight element.
\end{proposition}

\begin{theorem}[\cite{Kashiwara91,Kashiwara93}]
\label{thm:crystal_R_matrix}
The tensor product makes the category of highest weight crystals into a symmetric monoidal category.
\end{theorem}

An important consequence of Theorem~\ref{thm:crystal_R_matrix} is that $\crystal_1 \otimes \crystal_2 \iso \crystal_2 \otimes \crystal_1$ and that a tensor product decomposes into a direct sum of highest weight crystals.


\section{Higher Dorey's rules and generalized T-systems}
\label{sec: new morphisms}

In this section, we shall (i) show the homomorphisms among KR modules, called the higher Dorey's rule, and (ii) generalize T-system by using the results have been introduced in the previous sections. 
In this section, our higher Dorey's rule will be given with some \emph{restrictions} due to the lack of $\de$-invariants among KR modules.
We will subsequently remove such restrictions by computing the $\de$-invariants in the later sections.

\subsection{Preparation} \label{subsec: preparation}

In this subsection, we compute $\de$-invariants and composition length of modules in $\Cg$ 
and investigate homomorphism between modules over $\uqpg$, which will be used crucially in our main theorem.
Note that $\de$-invariants between fundamental modules are calculated in \cite{AK97,DO94,KKK15,Oh14R,OhS19}.
We used the results several times in this subsection.

\begin{proposition} \label{prop: q-character uniquely appear}
Let $\uqpg$ be a quantum affine algebra of untwisted type $A_{n-1}^{(1)},B_n^{(1)},C_n^{(1)}$ or $D_n^{(1)}$.
For $m \in \Z_{\ge 1}$ and $1 \le k, l < n$, let us assume $k + l < n - \delta(\g = D_n^{(1)})$. Then the highest weight $U_q(\g_0)$-module $V(m \La_{k+l})$ appears exactly once in the $U_q(\g_0)$ restriction of the following modules:
\bna
\item $\Vkm{l^m}_{(-\check{q}_l)^{-k}} \otimes \Vkm{k^m}_{(-\check{q}_k)^{l}}$,
\item $  \Vkm{k}_{(-\check{q}_k)^{l+m-1}}  \otimes  \Vkm{k^{m-1}}_{(-\check{q}_k)^{l-1}} \otimes  \Vkm{l^{m}}_{(-\check{q}_l)^{-k}}   $,
\item $  \Vkm{k^{m-1}}_{(-\check{q}_k)^{l-1}} \otimes  \Vkm{l^{m-1}}_{(-\check{q}_l)^{-k-1}} \otimes \Vkm{k+l}_{(-\chq_{k+1})^{m-1}}  $,
\item \label{it: k+l m} $\Vkm{(k+l)^m}$.
\ee
Furthermore, we have the following:
\bnum
\item If $\g = B_{n}^{(1)}$ and $k+l =n$, then $V(2m \La_n)$ appears exactly once in the $U_q(\g_0)$ restriction of the modules {\rm (a)}$\sim${\rm (c)} and $\Vkm{n^{2m}}_{(-1)}$.
\item If $\g = C_{n}^{(1)}$ and $k+l =n$, then $V(m \La_n)$ appears exactly once in the $U_q(\g_0)$ restriction of the modules {\rm (a)}$\sim${\rm (d)}.
\item If $\g = D_{n}^{(1)}$ and $k+l =n-1$, then $V(m (\La_{n-1} + \La_n))$ appears exactly once in the $U_q(\g_0)$ restriction of the modules {\rm (a)}$\sim${\rm (c)} and $\Vkm{n^{2m}} \otimes \Vkm{(n-1)^{2m}}$.
\ee
\end{proposition}

\begin{proof}
We will prove all cases (including (i)--(iii)) by using the corresponding $U_q(\g_0)$ crystals.
For~\eqref{it: k+l m}, this corresponds to the dominant extremal weight, and hence appears exactly once.
From Proposition~\ref{prop:highest_wt_recursive}, we know that the rightmost factor in the tensor product must be a highest weight element.
Then by weight considerations, we see there is precisely only one such highest weight element in each tensor product:
\begin{equation}
\label{eq:hw_tableau_generic_height}
\ytableausetup{boxsize=1.45em}
\ytableaushort{{k_+}{\cdots}{k_+}{k_+},{\vdots}{\ddots}{\vdots}{\vdots},h{\cdots}hh}
\otimes
\ytableaushort{1{\cdots}11,{\vdots}{\ddots}{\vdots}{\vdots},k{\cdots}kk},
\quad
\ytableaushort{{l_+},{\vdots},h}
\otimes
\ytableaushort{{l_+}{\cdots}{l_+},{\vdots}{\ddots}{\vdots},h{\cdots}h}
\otimes
\ytableaushort{1{\cdots}11,{\vdots}{\ddots}{\vdots}{\vdots},l{\cdots}ll},
\quad
\ytableaushort{{k_+}{\cdots}{k_+},{\vdots}{\ddots}{\vdots},h{\cdots}h}
\otimes
\ytableaushort{1{\cdots}1,{\vdots}{\ddots}{\vdots},k{\cdots}k}
\otimes
\ytableaushort{1,{\vdots},h},
\end{equation}
where $k_+ := k + 1$, $l_+ := l+1$, and $h := k+l$.
\end{proof}

Alternatively, we could prove Proposition~\ref{prop: q-character uniquely appear} inductively on $m$ using the (virtual) Kleber's algorithm~\cite{Kleber98,OSS03II} (see, e.g.,~\cite{Scr20}).

For the following cases in type $B_n^{(1)}$, the multiplicity of the highest weight $U_q(\g_0)$-module $V(m \La_{\overline{k+l}})$ is not multiplicity free.
Therefore, we are required to use a different method to show the higher Dorey's rule.
We use the theory of \defn{$q$-characters}~\cite{FR99}, which is the injective ring morphism
\begin{align} \label{eq: q-character}
\chi_q \colon K(\Ca_{\g}) \to \ZZ[Y_{i,a}^{\pm1}]_{i \in I_0, a \in \ko^{\times}},    
\end{align}
where $K(\Ca_{\g})$ is the Grothendieck ring of $\Ca_{\g}$, such that every simple $\uqpg$-module $V \in \Ca_\g$ is characterized by its maximal \defn{dominant monomial}, a monomial $m = \prod_{i,a} Y_{i,a}^{u_{i,a}}$ such that $u_{i,a} \geq 0$.
In particular, KR modules are \defn{special}~\cite{FM01,Her06,Her10,Nak04,Nak10}, which means that there is precisely one dominant monomial (and corresponds to its Drinfel'd polynomial~\cite{CP91,CP95A}). (Note this is different than what it means for the corresponding node $i$ to be special.)

\begin{proposition} \label{prop: q-character uniquely appear folded B}
For $\g$ of type $B_n^{(1)}$ and $1 \le l < n < k <2n-1$ with $k+l \le 2n-1$, the unique dominant monomial of $\Wkm{(k+l)^{m}}_{(-1)}$ appears exactly once in the $q$-characters of the following modules:
\bna
\item $\Wkm{l^m}_{(-q)^{-k+1}} \otimes \Wkm{k^m}_{-(-q)^{l}}$,
\item $\Wkm{k^{m-1}}_{-(-q)^{l-1}} \otimes \Wkm{l^m}_{(-q)^{-k+1}} \otimes  \Wkm{k}_{-(-q)^{m+l-1}}$,
\item $\Wkm{k^{m-1}}_{-(-q)^{l-1}} \otimes \Wkm{l^{m-1}}_{(-q)^{-k}}   \otimes  \Wkm{k+l}_{-(-q)^{m-1}}$.
\ee
\end{proposition}

\begin{proof}
We will use the tableau formula for $q$-characters that is given in~\cite{KOS95} (see also~\cite{NN06}).
More precisely, they are described as sums over semistandard tableaux (Definition~\ref{def:semistandard}) with $K$ rows and $M$ columns.
The box in row $r$ and column $j$ is weighted as
\[
\ytableaushort{i}_{a'} = \begin{cases}
Y_{i-1,a'\qs^{2i}}^{-1} Y_{i,a'\qs^{2(i-1)}} & \text{if } i \in J, \\
Y_{n-1,a'\qs^{2n}}^{-1} Y_{n,a'\qs^{2n-3}} Y_{n,a'\qs^{2n-1}} & \text{if } i = n, \\
Y_{n,a'\qs^{2n+1}}^{-1} Y_{n,a'\qs^{2n-3}} & \text{if } i = 0, \\
Y_{n-1,a'\qs^{2n-2}} Y_{n,a'\qs^{2n-1}}^{-1} Y_{n,a'\qs^{2n+1}}^{-1} & \text{if } i = \bn, \\
Y_{\overline{\imath}-1,a'\qs^{2(2n-\overline{\imath}-1)}} Y_{\overline{\imath},a'\qs^{2(2n-\overline{\imath})}}^{-1} & \text{if } i \in \overline{J},
\end{cases}
\]
for $a' = a\qs^{-1-2M+4j+K-2r}$, where $J = \{1, \dotsc, n-1\}$ and $\overline{J} = \{\bnm, \dotsc, \bon\}$, and by convention $Y_{0,a} = 1$ and $\overline{\overline{\imath}} = i$ for all $a \in \CC$.
Below, we will use $h := 2n - k$.

We first show existence.
Consider (a), and we first consider the special case of $m = 1$.
The dominant monomial of $\Wkm{k}_{-(-q)^{l}} \iso \Vkm{h}_{-(-q)^l}$ is $Y_{h,(-1)^{l+1}\qs^{2l}}$, and  
$\Wkm{l}_{(-q)^{1-k}}$ 
has the monomial  $Y_{h,(-1)^{-k}\qs^{2l}}^{-1} Y_{h-l,1}$ 
given by the column $\tabcol{\overline{h}, \dotsc, \overline{g}}$, where $g = h - l + 1 = 2n - k - l + 1$.
Thus the requisite dominant monomial appears in the product of the $q$-characters, and the general $m$ case follows from the $j$-th columns of each tableau is given by the respective $m = 1$ case column/tableau and they are shifted by the same amount.
Parts (b) and (c) are similar.

Now we need to show uniqueness.
We first show (a) for the case $m=1$.
Suppose the right factor's column is $\tabcol{1, \dotsc, j, i_1, \dotsc, i_{l'}}$ with $i_1 > j + 1$ and $l' > 0$.
Now let us consider the monomial $Y_{j,\qs^{\alpha}}$, where $\alpha = 2l+h-2j+1$, in this column.
As there is no $j + 1$ in the column, if we try to cancel this with a $\overline{\jmath}$, then we would require that box to be at height $2n-2j$, which is impossible since $h < n$.
Next if $\alpha = 0$, then we would again need $j > h$, which is impossible.
Hence, we must cancel this $Y_{j,\qs^{\alpha}}$ from the left factor to obtain our dominant monomial.
If we have a $j+1$ (resp.\ $\overline{j}$) in this column, it must occur at height $r = 2+(h-l)/2+2j-2n$ (resp.\ $r = 1+(h-l)/2$).
If $h < l$, this is impossible, and if $h \geq l$, then the corresponding $Y_{j-1,\qs^{\alpha'}}$ from this factor cannot be canceled by an analogous computation as there is no $Y_{j-1,\qs^{\beta}}^{-1}$ for $\beta < \alpha$.
Hence, we have shown (a) for $m = 1$.

The general case for (a) is follows by noting this argument shows the rightmost column of the right factor is $\tabcol{1, \dotsc, h}$ and forces all other columns to be the same by semistandardness, and weight considerations forces us to have the left factor be the one given in the existence proof.
An analogous argument is used for (b) and (c).
\end{proof}

\begin{proposition} \label{prop: q-character uniquely appear spin C and D} \hfill
\bna
\item For $\g=C_n^{(1)}$, the highest weight $U_q(\g_0)$-module $V(2m\La_{n-2})$ appears exactly once in the $U_q(\g_0)$ decomposition of the following modules:  
\bnum
\item $\Vkm{n^m}_{(-\qs)^{-3}} \otimes \Vkm{n^{m}}_{(-\qs)^3}$,
\item $\Vkm{n^{m-1}}_{\qs} \otimes \Vkm{n^m}_{(-\qs)^{-3}} \otimes \Vkm{n}_{(-1)^{m}\qs^{2m+1}}$,
\item $\Vkm{n^{m-1}}_{\qs} \otimes \Vkm{n^{m-1}}_{\qs^{-5}} \otimes \Vkm{(n-2)^2}_{(-1)^{m+1}\qs^{2m-2}}$,
\item $\Vkm{(n-2)^{2m}}_{(-1)^{m+1}}$.
\ee
\item For $\g=D_n^{(1)}$, the highest weight $U_q(\g_0)$-module $\Vkm{m\La_{n-3}}$ appears exactly once in the following modules:  
\bnum
\item $\Vkm{n^m}_{(-q)^{-2}} \otimes \Vkm{(n-1)^m}_{(-q)^{2}}$,
\item $\Vkm{(n-1)^{m-1}}_{(-q)} \otimes \Vkm{n^m}_{(-q)^{-2}} \otimes \Vkm{(n-1)}_{(-q)^{m+1}}$,
\item $\Vkm{(n-1)^{m-1}}_{(-q)} \otimes \Vkm{n^{m-1}}_{(-q)^{-3}} \otimes \Vkm{n-3}_{(-q)^{m-1}}$,
\item $\Vkm{(n-3)^{m}}$.
\ee 
\ee
\end{proposition}

\begin{proof}
The proof is analogous to the proof of Proposition~\ref{prop: q-character uniquely appear}.
For (iv), this corresponds to the dominant extremal weight, and hence appears exactly once.
Hence, it remains to show the claim for (i)$\sim$(iii).
The tableaux/highest weight elements for (a) are
\[
\ytableausetup{boxsize=1.45em}
\ytableaushort{1{\cdots}11,{\vdots}{\ddots}{\vdots}{\vdots},{\overline{n_-}}{\cdots}{\overline{n_-}}{\overline{n_-}},{\bn}{\cdots}{\bn}{\bn}}
\otimes u_{m\Lambda_n},
\qquad
\ytableaushort{1{\cdots}11,{\vdots}{\ddots}{\vdots}{\vdots},{\overline{n_-}}{\cdots}{\overline{n_-}}{\overline{n_-}},{\bn}{\cdots}{\bn}{\bn}}
\otimes u_{(m-1)\Lambda_n} \otimes u_{\Lambda_n},
\qquad
\ytableaushort{1{\cdots}1,{\vdots}{\ddots}{\vdots},{\overline{n_-}}{\cdots}{\overline{n_-}},{\bn}{\cdots}{\bn}}
\otimes u_{(m-1)\Lambda_n} \otimes u_{2\La_{n-2}},
\]
where $n_- = n - 1$.  
For (b) the highest weight elements are
\begin{gather*}
\kf_{n-1}^m \kf_{n-2}^m \kf_n^m u_{m\La_n} \otimes u_{m\La_{n-1}},
\\
\kf_n^{m-1} \kf_{n-2}^{m-1} \kf_{n-1}^{m-1} u_{(m-1)\La_{n-1}} \otimes \kf_{n-1} \kf_{n-2} \kf_n u_{m\La_n} \otimes u_{\La_{n-1}},
\\
\kf_n^{m-1} \kf_{n-2}^{m-1} \kf_{n-1}^{m-1} u_{(m-1)\La_{n-1}} \otimes u_{m\La_n} \otimes u_{\La_{n-3}}. \qedhere
\end{gather*}
\end{proof}

\begin{lemma} \label{lem: no surjective or injective}
Let $m \in \Z_{\ge 2}$.
\bnum
\item \label{it: no homom classical}
For $1 \le k < n$ with $k + 1 < n - \delta(\g = D_n^{(1)})$, there is  no \emph{injective} homomorphism
\begin{align*}
\Vkm{1^m}_{(-\chq_1)^{-k}} \otimes \Vkm{k^{m}}_{(-\chq_k)} 
\rightarrowtail
\Vkm{k^{m-1}} \otimes \Vkm{1^{m-1}}_{(-\chq_1)^{-k-1}}  \otimes \Vkm{k+1}_{(-\chq_{k+1})^{m-1}} 
\end{align*}
and no \emph{isomorphism} 
\begin{align*}
 \Vkm{1^m}_{(-\chq_1)^{-k}} \otimes \Vkm{k^{m}}_{(-\chq_k)} 
\isoto 
\Vkm{(k+1)^m}.  
\end{align*}

\item \label{it: no homo classical B to spin} For $\g = B_n^{(1)}$, there is no \emph{injective} homomorphism
\begin{align*}
\Vkm{1^m}_{(-q)^{-n+1}} \otimes \Vkm{(n-1)^{m}}_{(-q)} 
\rightarrowtail
\Vkm{(n-1)^{m-1}} \otimes \Vkm{1^{m-1}}_{(-q)^{-n+1}}  \otimes \Vkm{n^2}_{(-1)^m\qs^{2m-2}}   
\end{align*}
and no \emph{isomorphism}
\begin{align*}
\Vkm{1^m}_{(-q)^{-n+1}} \otimes \Vkm{(n-1)^{m}}_{(-q)}  \isoto    \Vkm{n^{2m}}_{(-1)^m}. 
\end{align*} 

\item \label{it: no homo classical C to spin} For $\g = C_n^{(1)}$, there is  no \emph{injective} homomorphism
\begin{align*}
\Vkm{1^m}_{(-\qs)^{-n+1}} \otimes \Vkm{(n-1)^{m}}_{(-\qs)} 
\rightarrowtail
\Vkm{(n-1)^{m-1}} \otimes \Vkm{1^{m-1}}_{(-\qs)^{-n}}  \otimes \Vkm{n}_{(-\qs)^{m-1}}   
\end{align*}
and no \emph{isomorphism} 
\begin{align*}
\Vkm{1^m}_{(-\qs)^{-n+1}} \otimes \Vkm{(n-1)^{m}}_{(-\qs)} 
\isoto 
\Vkm{n^{\lceil m/2 \rceil}}_{(-\qs)^{-\epsilon} }\otimes \Vkm{n^{\lfloor m/2 \rfloor}}_{(-\qs)^{\epsilon}}.   
\end{align*}
where $\epsilon = \delta(m\equiv_2 0)$. 

\item  \label{it: no homo classical D to spin}  For $\g = D_n^{(1)}$, there is no \emph{injective} homomorphism
\begin{align*}
\Vkm{1^m}_{(-q)^{-n+2}}  &\otimes \Vkm{(n-2)^{m}}_{(-q)}  \\
& \rightarrowtail
\Vkm{(n-2)^{m-1}} \otimes \Vkm{1^{m-1}}_{(-q)^{-n+2}}  \otimes \Vkm{n}_{(-q)^{m-1}} \otimes \Vkm{n-1}_{(-q)^{m-1}}  
\end{align*}
and no \emph{isomorphism}
\begin{align*}
\Vkm{1^m}_{(-q)^{-n+2}} \otimes \Vkm{(n-2)^{m}}_{(-q)}  \isoto    \Vkm{n^{m}} \otimes \Vkm{(n-1)^{m}}. 
\end{align*}

\item \label{it: no surjective or injective C,D}
For $\g = C_n^{(1)}$, there is no \emph{injective} homomorphism 
\begin{align*}
\Vkm{n^m}_{(-\qs)^{-3}} \otimes \Vkm{n^{m}}_{(-\qs)^3}   
\rightarrowtail \Vkm{n^{m-1}}_{\qs} \otimes \Vkm{n^{m-1}}_{\qs^{-5}} \otimes \Vkm{(n-2)^2}_{(-1)^{m+1}\qs^{2m-2}}
\end{align*}
and no \emph{isomorphism}
\begin{align*}
\Vkm{n^m}_{(-\qs)^{-3}} \otimes \Vkm{n^{m}}_{(-\qs)^3}  \isoto    \Vkm{(n-2)^{2m}}_{(-1)^{m+1}}. 
\end{align*}  
\item For $\g = D_n^{(1)}$,  there is no \emph{injective} homomorphism
\begin{align*}
\Vkm{n^m}_{(-q)^{-2}} \otimes \Vkm{(n-1)^m}_{(-q)^{2}}   
\rightarrowtail \Vkm{(n-1)^{m-1}}_{(-q)} \otimes \Vkm{n^{m-1}}_{(-q)^{-3}} \otimes \Vkm{n-3}_{(-q)^{m-1}}, 
\end{align*}
and no \emph{isomorphism}
\begin{align*}
\Vkm{n^m}_{(-q)^{-2}} \otimes \Vkm{(n-1)^m}_{(-q)^{2}}   \isoto \Vkm{(n-3)^m}.
\end{align*}

\item \label{it: no surjective or injective folded B}
For $\g$ of type $B_n^{(1)}$ and  $n < k < 2n-1$, there is no \emph{injective} homomorphism
\begin{align*}
\Wkm{1^m}_{(-q)^{-k+1}} \otimes \Wkm{k^m}_{-(-q)}  
\rightarrowtail    \Wkm{k^{m-1}}_{(-1)} \otimes \Wkm{1^{m-1}}_{(-q)^{-k}}  \otimes \Wkm{k+1}_{-(-q)^{m-1}} 
\end{align*}
and no \emph{isomorphism}
\begin{align*}
\Wkm{1^m}_{(-q)^{-k+1}} \otimes \Wkm{k^m}_{-(-q)}  \isoto \Wkm{(k+1)^m}_{(-1)}
\end{align*}
\ee
\end{lemma}

\begin{proof} 
We prove the claims by looking at the branching into $U_q(\g_0)$-modules.
We show that the $U_q(\g_0)$-module $V(\lambda)$, where $\lambda$ is the dominant extremal weight (considered in $\wl$) does not appear as a factor in the codomain.
In all cases, this will follow from the fact that the dominant extremal weight $\mu$ in the codomain is dominated by $\lambda$ (and clearly $\mu \neq \lambda$).
These computations can be reduced down to the case $m = 1$ as either it is a multiple of this case or all other terms cancel.
Therefore, it is a straightforward (and well-known) computation using the root system data.
\end{proof}
 
\begin{lemma} [{cf.~Proposition~\ref{prop: no intersection}}] \label{lem: no intersection}
Let $M_1,M_2$ be KR modules with 
$$\rch(M_k) = \range{a_k,b_k} \qtq \clr(M_k)=\im_k \quad (k=1,2).$$ 
\bnum
\item \label{it: far enough1} If  there exists
a $\rmQ$-datum $\calQ=\Qdatum$ of $\g$ such that $(\im_1,a_1)\in \Gamma^\calQ_0$ and $(\im_2,b_2) \in  (\Gamma^\calQ[-s])_0$ for some $s \ge 1$,
\begin{align} \label{eq: DM1 M2 0}
\de(\scrD M_1,M_2)=0.     
\end{align}
\item \label{it: far enough2} In particular,~\eqref{eq: DM1 M2 0} holds when $b_2<a_1$.
\ee
\end{lemma}

\begin{proof} \eqref{it: far enough1} For any $\rmQ$-datum $\calQ$ and fundamental module $F_k$ $(k=1,2)$, it is known that if $F_k \in \mC_\calQ[m_k]$ with $|m_1-m_2| >1$, $\de(F_1,F_2)=0$ (see \cite[Appendix]{Oh14R} and \cite{OhS19,OhS19Add} for precise information).

By taking $\calQ$ in the statement, we have $M_1 \in \mC_{\calQ,\ge 0}$, $M_2  \in \mC_{\calQ,< 0}$, and hence $\scrD M_1 \in \mC_{\calQ,\ge 1}$.
Since 
\bitem
\item $M_1 = \head( \calV(\im_1,b_1) \otimes \calV(\im_1,b_1-d_{\im_1}) \otimes \cdots \otimes \calV(\im_1,a_1) )$,
\item $M_2 = \head( \calV(\im_2,b_2) \otimes \calV(\im_2,b_2-d_{\im_2}) \otimes \cdots \otimes \calV(\im_2,a_2) )$,
\eitem
for $\clr(M_k)=\im_k$, the first assertion follows from Proposition~\ref{prop: de less than equal to}. 

\medskip\noindent
\eqref{it: far enough2} Note that there exists a $\rmQ$-datum $\calQ$ whose height function $\xi$ has values in (1) $[a_1,a_1+1]$ for $\g=A_{n-1}^{(t)}$,
$D_n^{(t)}$, $E_n^{(t)}$, (2)  $[a_1,a_1+2]$ for $\g=B_n^{(1)}$, $C_n^{(1)}$, $F_n^{(1)}$, and (3) $[a_1,a_1+3]$ for $\g=F_4^{(1)}$ (see~\cite[\S 4.5]{FHOO} for more detail). Thus the second assertion holds. 
\end{proof} 

\begin{corollary} \label{cor: unmixed}
Let $\uqpg$ be of non-exceptional untwisted affine classical type and $m \in \Z_{\ge 1}$.
For $1 \le k <n$, we have the following:
\[
\de(\scrD \Vkm{(k+1)}_{(-\chq_{k+1})^{m}}, \Vkm{k^{m}} ) = 0  \text{ if $k+1 < n - \delta(\g = D_n^{(1)})$.}
\]
Additionally, we have
\bnum
\item $\de(\scrD \Vkm{n}_{(-\qs)^{m}}, \Vkm{(n-1)^{m}} ) = 0$ for $\g=C_n^{(1)}$,
\item $\de (\scrD \Vkm{n^2}_{(-1)^{m+1}\qs^{2m}},\Vkm{(n-1)^{m}}) =0$  for $\g=B_n^{(1)}$,
\item $\de(\scrD \Vkm{n}_{(-q)^{m}}, \Vkm{(n-2)^{m}} ) = \de(\scrD\Vkm{n-1}_{(-q)^{m}}, \Vkm{(n-2)^{m}} ) = 0$  for $\g=D_n^{(1)}$.
\ee    
\end{corollary}

\begin{proof}
Since the proofs are similar, we assume $\g=A_{n-1}^{(1)}$. Then we have
$\rch(\Vkm{k+1}_{(-q)^m}=\range{m}$ and $\rch(\Vkm{k^m}=\range{1-m,m-1}$. Then the assertion follows from Lemma~\ref{lem: no intersection}.
The particular cases (i)$\sim$(iii) also follows from the same argument. 
\end{proof}

\begin{proposition} \label{prop: k lm}
For classical untwisted affine types $\g=A_{n}^{(1)}$, $B_n^{(1)}$, $C_n^{(1)}$, $D_n^{(1)}$, $m \in \Z_{\ge1}$ and
$1\le b < l \le n'  \seteq n- \delta(\g = D_n^{(1)} )$, we have 
\begin{align}\label{eq: 1 k m 2}
\de(\Vkm{(l-b)^m}_{(-\chq_{l-b})^{-b}},\Vkm{(l-1)^m}_{(-\chq_{l-1})}) \le 1.
\end{align}    
In particular, by taking $l=b+1$, we have
\begin{align}\label{it: 1 k m}
 \de(\Vkm{1^m}_{(-\chq_1)^{-b}}, \Vkm{b^m}_{(-\chq_b)} )\le 1 \quad \text{ if $b+1=l \le n - \delta(\g = D_n^{(1)})$}.    
\end{align}
\end{proposition}

\begin{proof}
Since the proofs are similar, we give a proof for $A_{n}^{(1)}$.
Note that 
$$\Vkm{(l-b)^m}_{(-q)^{-b}} \iso  \Vkm{(l-b)^{m-1}}_{(-q)^{-b+1} } \hconv \Vkm{l-b}_{(-q)^{1-m-b}}.$$
Since
\begin{align*}
\rch(\Vkm{(l-b)^{m-1}}_{(-q)^{-b+1}}) & = \range{3-m-b,m-b-1} \text{ and }
\\ \exrch_{l-b}(\Vkm{(l-1)^m}_{(-q)}) & = \range{1-m-b,m+b+1},
\end{align*}
we have
$$
\de(\Vkm{(l-1)^m}_{(-q)},  \Vkm{(l-b)^{m-1}}_{(-q)^{-b+1} } \hconv \Vkm{l-b}_{(-q)^{1-m-b}}) \le \de(\Vkm{(l-1)^m}_{(-q)}, \Vkm{l-b}_{(-q)^{1-m-b}}) ).
$$
Thus now let us focus on $\de(\Vkm{(l-1)^m}_{(-q)}, \Vkm{l-b}_{(-q)^{1-m-b}})$. Note that
$$\Vkm{(l-1)^m}_{(-q)} \hconv \Vkm{(l-1)}_{(-q)^{-m}} \iso \Vkm{(l-1)^{m+1}} $$
and $\exrch(\Vkm{(l-1)^{m+1}})=\range{-1-m-k,m+k+1}$. Thus $\de( \Vkm{(l-1)^{m+1}},\Vkm{l-b}_{(-q)^{1-m-b}})=0$.
On the other hand, Lemma~\ref{Lem: MNDM} and Proposition~\ref{prop: de less than equal to} imply 
\begin{align*}
& 0 \le \de(\Vkm{(l-1)^m}_{(-q)}, \Vkm{l-b}_{(-q)^{1-m-b}}) \\
& \hspace{7ex}\le \de(\Vkm{(l-1)^{m+1}}, \Vkm{l-b}_{(-q)^{1-m-b}})  + \de( \scrD^{-1}\Vkm{(l-1)}_{(-q)^{-m}}  ,\Vkm{(l-b)}_{(-q)^{1-m-b}}), \\
& \hspace{14ex} = \de( \scrD^{-1}\Vkm{(l-1)}_{(-q)^{-m}}  ,\Vkm{(l-b)}_{(-q)^{1-m-b}}).
\end{align*}
It is known that $\de( \scrD^{-1}\Vkm{(l-1)}_{(-q)^{-m}}, \Vkm{(l-b)}_{(-q)^{1-m-b}})=1$ (see Appendix~\ref{subsec: fundamental deno} below).
Hence the assertion follows.
\end{proof}

\begin{proposition}  \label{prop: BW de le 1}
Let $\g$ of type $B_n^{(1)}$ and take $n < l \le 2n-1$ and $b$ with $1+b \le l$ and $l-b < n < l-1$.
Then we have
\begin{align} \label{eq: folded de <1}
\de( \Wkm{(l-b)^m}_{(-q)^{-b+1}} \otimes \Wkm{(l-1)^m}_{-(-q)}  ) \le 1 \quad\text{for any $m \in \Z_{\ge1}$}.
\end{align}
In particular, when $l=b+1$
\begin{align} \label{eq: folded de <1 speical}
\de( \Wkm{1^m}_{(-q)^{-b+1}} \otimes \Wkm{b^m}_{-(-q)}  ) \le 1.    
\end{align}
\end{proposition}
\begin{proof}
Note that   
$$\Wkm{(l-1)^m}_{-(-q)}  \iso \Wkm{(l-1)}_{-(-q)^{m}} \hconv \Wkm{(l-1)^{m-1}}_{(-1)}.$$   
Since 
\begin{equation}
\begin{aligned} \label{eq: exrch folded B}
 \exrch_{(l-1)}(\Wkm{(l-b)^m}_{(-q)^{-b+1}}) & = \range{-2m-4b+8,2m} \qtq \\
  \rch(\Wkm{(l-1)^{m-1}}_{-(-q)}) &= \range{4-2m,2m-4},
\end{aligned}
\end{equation}
we have
$$
\de(\Wkm{(l-b)^m}_{(-q)^{-b+1}}, \Wkm{(l-1)}_{-(-q)^{m}} \hconv \Wkm{(l-1)^{m-1}}_{(-1)}) \le\de(\Wkm{(l-b)^m}_{(-q)^{-b+1}}, \Wkm{(l-1)}_{-(-q)^{m}}).
$$
Thus let us focus on $\de(\Wkm{(l-b)^m}_{(-q)^{-b+1}}, \Wkm{(l-1)}_{-(-q)^{m}})$. Note that
\[
\Wkm{(l-b)}_{(-q)^{m-b+2}} \hconv  \Wkm{(l-b)^m}_{(-q)^{-b+1}}  \iso \Wkm{(l-b)^{m+1}}_{(-q)^{-b+2}}
\]
and $\exrch_k(\Wkm{(l-b)^{m+1}}_{(-q)^{-b+2}}) = \range{4-2m-4b,2m+4}$.
Therefore,
\[
\de(\Wkm{(l-b)^{m+1}}_{(-q)^{-b+2}} \Wkm{(l-1)}_{-(-q)^{m}}) = 0.
\]
On the other hand, Lemma~\ref{Lem: MNDM} and Proposition~\ref{prop: de less than equal to}, say that 
\begin{align*}
& 0 \le \de(\Wkm{(l-b)^m}_{(-q)^{-b+1}}, \Wkm{(l-1)}_{-(-q)^{m}}) \\
& \hspace{7ex}\le \de(\Wkm{(l-b)^{m+1}}_{(-q)^{-b+2}} \Wkm{(l-1)}_{-(-q)^{m}})  
  + \de( \scrD\Wkm{(l-b)}_{(-q)^{m-b+2}}  ,\Wkm{(l-1)}_{-(-q)^{m}}), \\
& \hspace{14ex} = \de( \scrD\Wkm{(l-b)}_{(-q)^{m-b+2}}  ,\Wkm{(l-1)}_{-(-q)^{m}})
\end{align*}
It is known that $\de( \scrD\Wkm{(l-b)}_{(-q)^{m-b+2}}  ,\Wkm{(l-1)}_{-(-q)^{m}})=1$ (see Appendix~\ref{subsec: fundamental deno} below). Hence the assertion follows.
\end{proof}

\begin{proposition} \label{prop: Cn nm nm -3 comp2}
Let $m \in \Z_{\geq 1}$.
\bna
\item \label{it: C comp2} For $\g=C_{n}^{(1)}$, we have $\de(\Vkm{n^m}_{(-\qs)^{-3}} , \Vkm{n^m}_{(-\qs)^{3}}) \le 1$. 
\item \label{it: D comp2} For $\g=D_{n}^{(1)}$, we have $\de(\Vkm{(n-1)^{m}}_{(-\qs)^{-2}} ,\Vkm{n^m}_{(-\qs)^{2}}) \le 1$. 
\ee
\end{proposition}

\begin{proof} Since the proofs are similar, we only give a proof of~\eqref{it: C comp2}. 
Note that $\clr(\Vkm{n^m}_{(-\qs)^{-3}} ) \ne \clr(\Vkm{n^m}_{(-\qs)^{3}})$ for a fixed skeleton category containing them.
Without loss of generality, we set 
$$n= \clr(\Vkm{n^m}_{(-\qs)^{-3}} ) \qtq n+1 = \clr(\Vkm{n^m}_{(-\qs)^{3}} ).$$
Note also that
$$
\Vkm{n^m}_{(-\qs)^{3}} \iso \Vkm{n}_{(-\qs)^{2m+1}} \hconv \Vkm{n^{m-1}}_{(-\qs)}.
$$
Since
$$
\exrch_{n+1}(\Vkm{n^m}_{(-\qs)^{-3}} )  = \range{-2m-7,2m+1} \qtq \rch(\Vkm{n^{m-1}}_{(-\qs)}) = \range{5-2m,2m-3},
$$
we have $\de(\Vkm{n^m}_{(-\qs)^{-3}},\Vkm{n^{m-1}}_{(-\qs)})=0$. By Proposition~\ref{prop: de less than equal to}, we have
\begin{align*}
\de(\Vkm{n^m}_{(-\qs)^{-3}} , \Vkm{n^m}_{(-\qs)^{3}})  & \le \de(\Vkm{n^m}_{(-\qs)^{-3}} , \Vkm{n^{m-1}}_{(-\qs)}) +
\de(\Vkm{n^m}_{(-\qs)^{-3}} , \Vkm{n}_{(-\qs)^{2m+1}}) \\
& = \de(\Vkm{n^m}_{(-\qs)^{-3}} , \Vkm{n}_{(-\qs)^{2m+1}}).
\end{align*}
Note that
$$
\Vkm{n^m}_{(-\qs)^{-3}} \hconv \Vkm{n}_{(-\qs)^{2m-1}} \iso \Vkm{n^{m+1}}_{(-\qs)^{-1}}
$$
and $\exrch_{n+1}(\Vkm{n^{m+1}}_{(-\qs)^{-1}}) = \range{-7-2m,2m+5}$. Hence 
$$
\de(\Vkm{n^{m+1}}_{(-\qs)^{-1}},\Vkm{n}_{(-\qs)^{2m+1}} )=0.
$$

On the other hand, Lemma~\ref{Lem: MNDM} and Proposition~\ref{prop: de less than equal to}, say that 
\begin{align*}
& 0 \le \de(\Vkm{n^m}_{(-\qs)^{-3}} , \Vkm{n}_{(-\qs)^{2m+1}}) \\
& \hspace{7ex}\le \de(\Vkm{n^{m+1}}_{(-\qs)^{-1}},\Vkm{n}_{(-\qs)^{2m+1}})+\de(\scrD \Vkm{n}_{(-\qs)^{2m-1}},\Vkm{n}_{(-\qs)^{2m+1}}), \\
& \hspace{14ex} = \de( \scrD \Vkm{n}_{(-\qs)^{2m-1}},\Vkm{n}_{(-\qs)^{2m+1}})
\end{align*}
It is known that $\de(\scrD \Vkm{n}_{(-\qs)^{2m-1}},\Vkm{n}_{(-\qs)^{2m+1}})=1$ (see Appendix~\ref{subsec: fundamental deno} below). Hence the assertion follows by Proposition~\ref{prop: length 2}.
\end{proof}

\begin{remark}
In the above three propositions, we have inequalities of the form $0 \le \de(M,N) \le 1$. 
If we can prove that $M \otimes N$ have a simple head not isomorphic to itself, we con conclude that $\de(M,N) = 1$. 
In the successive subsection, we will refine the inequalities in the above three propositions to $\de(M,N) = 1$.  
\end{remark}

The following proposition is an alternative proof of Proposition~\ref{prop: i-box d-value}\eqref{it: d=1 1 a-b a b+}:

\begin{proposition}
For every quantum affine algebra $\uqpg$ and its T-system, we have
$$
\de(\frakR[a^+,b], \frakR[a,b^-])=1.
$$
\end{proposition}

\begin{proof} 
Since the proofs are similar, we give a proof for $\g=C_n^{(1)}$, $\im \seteq \clr(\frakR[a^+,b]) = \clr(\frakR[a,b^-]) \in \{ n,n+1 \}$.
In this case, T-system is described as follows (see~Appendix~\ref{Appendix: T-system} also): 
\begin{align} \label{eq: T-system C}
0 \to  \Vkm{n^{m-1}} \tens   \Vkm{n^{m+1}}  \to  \Vkm{n^m}_{(-q)^{-1}} \tens \Vkm{n^m}_{(-q)}   \to \Vkm{(n-1)^{2m}}_{(-1)^{1+m}}  \to 0,     
\end{align}
where $\Vkm{n^m}_{(-q)^{\pm 1}}$ corresponds to $\frakR[a^+,b]$ (resp. $\frakR[a,b^-]$).
Note that $\Vkm{n^m}_{(-q)^{-1}} \iso \Vkm{n^{m-1}} \hconv \Vkm{n}_{(-q)^{-m}}$.
Since 
$$
\exrch_\im(\Vkm{n^m}_{(-q)})=\range{-2m,2m+4} \qtq \rch(\Vkm{n^{m-1}} ) = \range{4-2m,2m-4},
$$
we have
$$
0 \le \de(\Vkm{n^m}_{(-q)},  \Vkm{n^m}_{(-q)^{-1}}) \le \de(\Vkm{n^m}_{(-q)},\Vkm{n}_{(-q)^{-m}}).
$$
Since $\Vkm{n^m}_{(-q)} \iso   \scrD^{-1} \Vkm{n}_{(-q)^{-m}} \hconv \Vkm{n^{m+1}}$ and 
$$
\exrch_\im(\Vkm{n^{m+1}}) = \range{-2m-4,2m+4} \qtq \rch(\Vkm{n}_{(-q)^{-m}}) = \range{-2m}
$$
we have
\begin{align*}
0 \le \de(\Vkm{n^m}_{(-q)},  \Vkm{n^m}_{(-q)^{-1}}) & \le \de(\Vkm{n^m}_{(-q)},\Vkm{n}_{(-q)^{-m}}) \\
&\le  \de( \scrD^{-1} \Vkm{n}_{(-q)^{-m}},\Vkm{n}_{(-q)^{-m}}) \overset{*}{=}1,  
\end{align*}
where $\overset{*}{=}$ holds by Proposition~\ref{prop: i-box d-value}~\eqref{it: d=1 Da a}, as the particular case $a=b$. 

On the other hand, T-system implies $\de(\Vkm{n^m}_{(-q)},  \Vkm{n^m}_{(-q)^{-1}}) >0$, since $\Vkm{n^m}_{(-q)^{-1}} \otimes  \Vkm{n^m}_{(-q)}$ is not simple. 
Hence the assertion follows as desired.
\end{proof}

\subsection{Higher Dorey's rule with restrictions and a generalized T-system}
 
In this section, we first consider $\uqpg$ of classical untwisted affine type $A^{(1)}_{n-1}$, $B^{(1)}_{n}$, $C^{(1)}_{n}$ or $D^{(1)}_{n}$. We call the homomorphisms will be investigated in this subsection the \defn{higher Dorey's rule} for untwisted affine type.

\begin{theorem}[Higher Dorey's rule I]
\label{thm: Higher Dorey I}  
For $\uqpg$ of non-exceptional untwisted affine classical type, let $m \in \Z_{\ge 1}$ and set $\epsilon \seteq \delta(m \equiv_2 0)$.
\begin{enumerate}[{\rm (1)}]
\begin{subequations} \label{eq: higher homo}
\item \label{thmitem:higher_Dorey_general} 
For $1 \le k, l < n$, let us assume $k + l < n - \delta(\g = D_n^{(1)})$ and 
$$\min(k,l)=1.$$
Then, we have  
\begin{align} \label{eq: k+l<n homo}
\Vkm{l^m}_{(-\check{q}_l)^{-k}} \hconv \Vkm{k^m}_{(-\check{q}_k)^{l}} \iso \Vkm{(k+l)^m}.     
\end{align}
In particular, if $\g \ne A_{n-1}^{(1)}$ and $k+l =n- \delta(\g = D_n^{(1)})$, we have 
\begin{align} \label{eq: k+l=n homo}
\bc
\Vkm{k^m}_{(-1)^{m+l}q^{-l}} \hconv \Vkm{l^m}_{(-1)^{k+m}q^k} \iso \Vkm{n^{2m}} & \text{ if } \g = B_n^{(1)}, \\
\Vkm{l^m}_{(-\qs)^{-k}} \hconv \Vkm{k^m}_{(-\qs)^{l}} \iso  \Vkm{n^{\lceil m/2 \rceil}}_{(-\qs)^{-\epsilon} }\otimes \Vkm{n^{\lfloor m/2 \rfloor}}_{(-\qs)^{\epsilon}}
& \text{ if } \g = C_n^{(1)}, \\
\Vkm{l^m}_{(-q)^{-k}} \hconv \Vkm{k^m}_{(-q)^{l}}  \iso \Vkm{(n-1)^{m}} \otimes \Vkm{n^m} & \text{ if } \g = D_n^{(1)}.
\ec
\end{align} 
\item  
For $\g$ of type $B_n^{(1)}$ and  $1 \le l < n < k <2n-1$ with $k+l \le 2n-1$ and $$\min(k,l)=1,$$
we have 
\begin{equation}  \label{eq: B folded homo} 
\begin{aligned}
& \Wkm{l^m}_{(-q)^{-k+1}} \hconv \Wkm{k^m}_{-(-q)^{l}}  \iso     \Wkm{(k+l)^m} \\
& \hspace{7ex} \iff \Vkm{\overline{l}^m}_{(-q)^{k-1}} \hconv \Vkm{\overline{k}^m}_{-(-q)^{l}}  \iso     \Vkm{(\overline{k+l})^m}_{(-1)}.
\end{aligned} 
\end{equation}
\item \label{thmitem:higher_Dorey_C}
For $\g$ of type $C_n^{(1)}$ and  $1 \le l <n$, we have 
\begin{align} \label{eq: spin C homo} 
\Vkm{n^m}_{(-1)^{m+1}(-\qs)^{-1-n+l}} \hconv \Vkm{n^m}_{(-1)^{m+1}(-\qs)^{n+1-l}} \iso     \Vkm{l^{2m}}.
\end{align} 
\item \label{thmitem:higher_Dorey_D}
For $\g$ of type $D_n^{(1)}$, $1 \le l < n-1$
and $n',n'' \in \{n-1,n\}$ such that $n'-n'' \equiv n-l \ {\rm mod} \ 2$, we have 
\begin{align} \label{eq: spin D homo} 
\Vkm{{n'}^m}_{(-q)^{-n+l+1}} \hconv \Vkm{{n''}^m}_{(-q)^{n-l-1}} \iso \Vkm{l^{m}}.
\end{align} 
 \end{subequations}
\end{enumerate}
\end{theorem}

\begin{remark}
In Theorem~\ref{thm: Higher Dorey I}, there are restrictions, $\min(k,l)=1$ for~\eqref{eq: k+l<n homo},~\eqref{eq: k+l=n homo} and~\eqref{eq: B folded homo}, while the others, namely~\eqref{eq: spin C homo} and~\eqref{eq: spin D homo}, do not have such restrictions.
In Theorem~\ref{thm: Higher Dorey II} below, the restrictions for~\eqref{eq: k+l<n homo},~\eqref{eq: k+l=n homo} and~\eqref{eq: B folded homo} will be removed.
\end{remark}

Note that the case when $m=1$, Theorem~\ref{thm: Higher Dorey I} recovers many cases of the classical Dorey's rule (Theorem~\ref{thm: Dorey}).
In the rest of this subsection, we will prove the Theorem~\ref{thm: Higher Dorey I} for $m \ge 2$.
We prove each rule separately, although the approach used in all cases/types is essentially uniform.
Moreover, the proof in different types will essentially only depend on the root data.

\begin{proof} [Proof of~\eqref{eq: k+l<n homo}] 

Since the proofs are similar, we shall consider the case when $\g=A_{n-1}^{(1)}$.
Without loss of generality, let us assume that $1 = l \le k$. We  prove the assertion by using induction on $m$.

Note that $\de(\Vkm{1^m}_{(-q)^{-k}},\Vkm{k^{m-1}})=0$, since
$$
\exrch_k(\Vkm{1^m}_{(-q)^{-k}}) = \range{-m-2k,m} \qtq 
\rch(\Vkm{k^{m-1}}) = \range{2-m,m-2}.
$$
First we can see that the composition
$$
c_1 \colon \Vkm{k^{m-1}} \otimes \Vkm{1^m}_{(-q)^{-k}} \otimes  \Vkm{k}_{(-q)^m}
\to \Vkm{k^{m-1}} \otimes \Vkm{1^{m-1}}_{(-q)^{-k-1}}  \otimes \Vkm{k+1}_{(-q)^{m-1}},
$$
obtained from 
\begin{align*}
& \Vkm{1^m}_{(-q)^{-k}} \otimes \Vkm{k^{m-1}} \otimes \Vkm{k}_{(-q)^m}
\iso  \Vkm{k^{m-1}} \otimes \Vkm{1^m}_{(-q)^{-k}} \otimes  \Vkm{k}_{(-q)^m}  \\
& \hspace{5ex} \rightarrowtail 
\Vkm{k^{m-1}} \otimes \Vkm{1^{m-1}}_{(-q)^{-k-1}} \otimes \Vkm{1}_{(-q)^{m-k-1}}  \otimes  \Vkm{k}_{(-q)^m} \\
& \hspace{10ex} \twoheadrightarrow
\Vkm{k^{m-1}} \otimes \Vkm{1^{m-1}}_{(-q)^{-k-1}}  \otimes \Vkm{k+1}_{(-q)^{m-1}},
\end{align*}
is non-zero by Proposition~\ref{prop: non-van}.
Note that the sequence of real simple modules  
$$
 (\Vkm{k+1}_{(-q)^{m-1}}, \Vkm{1^{m-1}}_{(-q)^{-k-1}} , \Vkm{k^{m-1}}) \text{ is normal},
$$
by Lemma~\ref{lem: normal seq d}(a-iii) and Corollary~\ref{cor: unmixed}.
Then we have
\begin{align*}
&\soc(\Vkm{k^{m-1}} \otimes \Vkm{1^{m-1}}_{(-q)^{-k-1}}  \otimes \Vkm{k+1}_{(-q)^{m-1}}) \\ 
& \hspace{7ex} \iso 
\head( \Vkm{k+1}_{(-q)^{m-1}} \otimes \Vkm{1^{m-1}}_{(-q)^{-k-1}}  \otimes \Vkm{k^{m-1}}) 
\ (\because \text{Lemma~\ref{lem: normal property}~\eqref{eq: head socle}})
\\  
& \hspace{14ex} \iso 
\head( \Vkm{k+1}_{(-q)^{m-1}} \otimes  \Vkm{(k+1)^{m-1}}_{(-q)^{-1}}) \ (\because \text{an induction on $m$} ) \\
& \hspace{21ex} \iso \Vkm{(k+1)^m}. 
\end{align*} 
Thus the image of $\Vkm{k^{m-1}} \otimes \Vkm{1^m}_{(-q)^{-k}} \otimes  \Vkm{k}_{(-q)^m}$ under $c_1$ contains $\Vkm{(k+1)^m} $.
Then the composition
\begin{equation} \label{eq: c2 classical}
\begin{aligned} 
c_2 \colon & \Vkm{1^m}_{(-q)^{-k}} \otimes \Vkm{k^{m}}_{(-q)}   
\rightarrowtail \Vkm{k^{m-1}} \otimes \Vkm{1^m}_{(-q)^{-k}} \otimes  \Vkm{k}_{(-q)^m} \\
 & \hspace{25ex} \overset{c_1}{\longrightarrow} \Vkm{k^{m-1}} \otimes \Vkm{1^{m-1}}_{(-q)^{-k-1}}  \otimes \Vkm{k+1}_{(-q)^{m-1}}
\end{aligned}
\end{equation}
does not vanish by Proposition~\ref{prop: q-character uniquely appear}.
Hence we can conclude that
\[
\Vkm{(k+1)^m} \subseteq \Image(c_2).
\]

On the other hand, we have  
$$
\de(\Vkm{k^{m}} , \Vkm{1^{m}}_{(-q)^{-k-1}}) \le 1
$$
by~\eqref{it: 1 k m} of Proposition~\ref{prop: k lm}.
Moreover, by the induction hypothesis on $m \ge 2$,  we have $$\de(\Vkm{k^{m-1}} , \Vkm{1^{m-1}}_{(-q)^{-k-1}}) =1.$$   
From Lemma~\ref{lem: no surjective or injective}\eqref{it: no homom classical},
$c_2$ in not injective and   
$$\de(\Vkm{k^{m}} , \Vkm{1^{m}}_{(-q)^{-k-1}}) =1.$$
Hence the composition length 
of $\Vkm{1^m}_{(-q)^{-k}} \otimes \Vkm{k^{m}}_{(-q)}$ is $2$.
Thus we have
\[
\Vkm{(k+1)^m} \iso \Image(c_2) \text{ indeed}
\]
and 
$$
\Vkm{1^m}_{(-q)^{-k}} \hconv \Vkm{k^{m}}_{(-q)} \iso \Vkm{(k+1)^m},
$$ 
which completes the assertion. 
\end{proof}

\begin{corollary} \label{cor: de=1 Dorey w res}
For $1 \le k  < n$ with $k + 1 < n - \delta(\g = D_n^{(1)})$, we have 
$$\de(\Vkm{k^{m}}_{(-\chq_k)} , \Vkm{1^{m}}_{(-\chq_1)^{-k}}) =1, \quad \de(\Vkm{k^m}_{(-\chq_k)}, \Vkm{1}_{(-\chq_1)^{1-m-k}})=1,$$    
and hence there exists a short exact sequence 
\begin{align*}
0 \to    \Vkm{1^{m}}_{(-\chq_1)^{-k}} \sconv \Vkm{k^{m}}_{(-\chq_k)} \to \Vkm{1^{m}}_{(-\chq_1)^{-k}} \tens \Vkm{k^{m}}_{(-\chq_k)} 
\to \Vkm{(k+1)^m} \to 0,
\end{align*}
where the left and right terms are real simple modules by {\rm Proposition~\ref{prop: real head}}. 
\end{corollary}

\begin{proof} [Proof of~\eqref{eq: k+l=n homo}] Since the proofs are similar, we only give a proof for $C_n^{(1)}$.
As~\eqref{eq: k+l<n homo}, let us assume $l=1$ without loss of generality.

Note that we can see that the composition
\begin{align*}
c_1 \colon \Vkm{(n-1)^{m-1}} \otimes & \Vkm{1^m}_{(-\qs)^{-n+1}} \otimes  \Vkm{(n-1)}_{(-\qs)^m} \\
 & \to \Vkm{(n-1)^{m-1}} \otimes \Vkm{1^{m-1}}_{(-\qs)^{-k-1}}  \otimes \Vkm{n}_{(-\qs)^{m-1}},    
\end{align*}
obtained from 
\begin{align*}
& \Vkm{1^m}_{(-\qs)^{-n+1}} \otimes \Vkm{(n-1)^{m-1}} \otimes \Vkm{(n-1)}_{(-\qs)^m} \\
& \hspace{3ex}\iso  \Vkm{(n-1)^{m-1}} \otimes \Vkm{1^m}_{(-q)^{-n+1}} \otimes  \Vkm{(n-1)}_{(-\qs)^m}  \\
& \hspace{6ex} \rightarrowtail 
\Vkm{(n-1)^{m-1}} \otimes \Vkm{1^{m-1}}_{(-\qs)^{-n}} \otimes \Vkm{1}_{(-\qs)^{m-n}}  \otimes  \Vkm{(n-1)}_{(-\qs)^m} \\
& \hspace{9ex} \twoheadrightarrow
\Vkm{(n-1)^{m-1}} \otimes \Vkm{1^{m-1}}_{(-\qs)^{-n}}  \otimes \Vkm{n}_{(-\qs)^{m-1}},
\end{align*}
is non-zero by Proposition~\ref{prop: non-van}.
Note that the sequence of real simple modules 
$$
 (\Vkm{n}_{(-\qs)^{m-1}}, \Vkm{1^{m-1}}_{(-\qs)^{-n}} , \Vkm{(n-1)^{m-1}}) \text{ is normal},
$$
by Lemma~\ref{lem: normal seq d}(a-iii) and Corollary~\ref{cor: unmixed}.
Then we have
\begin{align*}
&\soc(\Vkm{(n-1)^{m-1}} \otimes \Vkm{1^{m-1}}_{(-\qs)^{-n}}  \otimes \Vkm{n}_{(-\qs)^{m-1}}) \\ 
& \hspace{3ex} \iso 
\head\left(\Vkm{n}_{(-\qs)^{m-1}} \otimes \Vkm{1^{m-1}}_{(-\qs)^{-n}}  \otimes \Vkm{(n-1)^{m-1}}\right) 
\ (\because \text{Lemma~\ref{lem: normal property}~\eqref{eq: head socle}})
\\ 
& \hspace{3ex} \iso 
\head\left( \Vkm{n}_{(-\qs)^{m-1}} \otimes  \Vkm{n^{\lceil (m-1)/2 \rceil}}_{(-\qs)^{-\epsilon'-1} }\otimes \Vkm{n^{\lfloor (m-1)/2 \rfloor}}_{(-\qs)^{\epsilon'-1}} \right)
\ (\because \text{an induction on $m$} ) \\
& \hspace{3ex} \iso \Vkm{n^{\lceil m/2 \rceil}}_{(-\qs)^{-\epsilon} }\otimes \Vkm{n^{\lfloor m/2 \rfloor}}_{(-\qs)^{\epsilon}},
\end{align*} 
where $\epsilon' = \delta(m-1\equiv_2 0)$. 
Thus the image of $\Vkm{(n-1)^{m-1}} \otimes \Vkm{1^m}_{(-q)^{-n+1}} \otimes  \Vkm{k}_{(-\qs)^m}$ under $c_1$ contains $\Vkm{n^{\lceil m/2 \rceil}}_{(-\qs)^{-\epsilon} }\otimes \Vkm{n^{\lfloor m/2 \rfloor}}_{(-\qs)^{\epsilon}}$, which is simple.
Then the composition
\begin{equation} \label{eq: c2 spin C}
\begin{aligned} 
c_2 \colon & \Vkm{1^m}_{(-\qs)^{-n+1}} \otimes \Vkm{(n-1)^{m}}_{(-\qs)}   
\rightarrowtail \Vkm{(n-1)^{m-1}} \otimes \Vkm{1^m}_{(-\qs)^{-n+1}} \otimes  \Vkm{n}_{(-\qs)^m} \\
 & \hspace{25ex} \overset{c_1}{\longrightarrow} \Vkm{(n-1)^{m-1}} \otimes \Vkm{1^{m-1}}_{(-\qs)^{-n}}  \otimes \Vkm{n}_{(-\qs)^{m-1}}
\end{aligned}
\end{equation}
does not vanish by Proposition~\ref{prop: q-character uniquely appear}.
Hence we can conclude that
\[
\Vkm{n^{\lceil m/2 \rceil}}_{(-\qs)^{-\epsilon} }\otimes \Vkm{n^{\lfloor m/2 \rfloor}}_{(-\qs)^{\epsilon}} \subseteq \Image(c_2).
\]

On the other hand, we have  
$$
\de(\Vkm{(n-1)^{m}} , \Vkm{1^{m}}_{(-q)^{-n}}) \le 1
$$
by~\eqref{it: 1 k m} of Proposition~\ref{prop: k lm}.
Moreover, by the induction hypothesis on $m \ge 2$, we have
\[
\de(\Vkm{(n-1)^{m-1}} , \Vkm{1^{m-1}}_{(-q)^{-n}}) =1.
\]
From Lemma~\ref{lem: no surjective or injective}~\eqref{it: no homo classical C to spin},
$c_2$ in not injective and hence 
\[
\de(\Vkm{(n-1)^{m}} , \Vkm{1^{m}}_{(-\qs)^{-n}}) =1
\]
and the composition length of $\Vkm{1^m}_{(-\qs)^{-n+1}} \otimes \Vkm{(n-1)^{m}}_{(-\qs)}$ is $2$.
Thus we have
\[
\Vkm{n^{\lceil m/2 \rceil}}_{(-\qs)^{-\epsilon} }\otimes \Vkm{n^{\lfloor m/2 \rfloor}}_{(-\qs)^{\epsilon}} \iso \Image(c_2) \text{ indeed}
\]
and 
\[
\Vkm{1^m}_{(-\qs)^{-n+1}} \hconv \Vkm{(n-1)^{m}}_{(-\qs)} \iso 
\Vkm{n^{\lceil m/2 \rceil}}_{(-\qs)^{-\epsilon} }\otimes \Vkm{n^{\lfloor m/2 \rfloor}}_{(-\qs)^{\epsilon}},
\] 
which completes the assertion.  
\end{proof}

\begin{corollary}
For $\g=B_n^{(1)},C_n^{(1)}$ and $D_n^{(1)}$, we have
$$
\bc
\de(\Vkm{1^m}_{(-1)^{m+(n-1)}q^{-n+1}} , \Vkm{(n-1)^m}_{(-1)^{1+m}q})=1   & \text{ if } \g = B_n^{(1)}, \\
\de(\Vkm{1^m}_{(-\qs)^{-n+1}} , \Vkm{(n-1)^m}_{(-\qs)})=1   & \text{ if } \g = C_n^{(1)}, \\
\de(\Vkm{1^m}_{(-q)^{-n+2}} , \Vkm{(n-2)^m}_{(-q)})=1 & \text{ if } \g = D_n^{(1)},
\ec 
$$
and hence we have the corresponding short exact sequences:
\begin{align*}
& 0 \to \bc
\Vkm{1^m}_{(-1)^{m+(n-1)}q^{-n+1}} \sconv \Vkm{(n-1)^m}_{(-1)^{1+m}q}\\
\Vkm{1^m}_{(-\qs)^{-n+1}} \sconv \Vkm{(n-1)^m}_{(-\qs)} \\
\Vkm{1^m}_{(-q)^{-n+2}} \sconv \Vkm{(n-2)^m}_{(-q)}, 
\ec  \\
& \hspace{18ex} \to 
\bc
\Vkm{1^m}_{(-1)^{m+(n-1)}q^{-n+1}} \tens \Vkm{(n-1)^m}_{(-1)^{1+m}q}\\
\Vkm{1^m}_{(-\qs)^{-n+1}} \tens \Vkm{(n-1)^m}_{(-\qs)} \\
\Vkm{1^m}_{(-q)^{-n+2}} \tens \Vkm{(n-2)^m}_{(-q)}, 
\ec  \\
& \hspace{36ex} \to 
\bc
\Vkm{n^{2m}}  \\
 \Vkm{n^{\lceil m/2 \rceil}}_{(-\qs)^{-\epsilon} }\otimes \Vkm{n^{\lfloor m/2 \rfloor}}_{(-\qs)^{\epsilon}}\\
 \Vkm{(n-1)^{m}} \otimes \Vkm{n^m}
\ec  \to 0,
\end{align*}
where the left and right terms are real simple modules by {\rm Proposition~\ref{prop: real head}}. 
\end{corollary}
 
\begin{proof} [Proof of ~\eqref{eq: B folded homo}] 
We prove the assertion for $l=1$ by using induction on $m$, as previous cases.

First we can see that the composition
$$
c_1 \colon \Wkm{k^{m-1}}_{(-1)} \otimes \Wkm{1^m}_{(-q)^{-k+1}} \otimes  \Wkm{k}_{-(-q)^m}
\to \Wkm{k^{m-1}}_{(-1)} \otimes \Wkm{1^{m-1}}_{(-q)^{-k}}  \otimes \Wkm{k+1}_{-(-q)^{m-1}},
$$
obtained from  
\begin{align*}
& \Wkm{1^m}_{(-q)^{-k+1}} \otimes \Wkm{k^{m-1}}_{(-1)} \otimes   \Wkm{k}_{-(-q)^m}
\overset{*}{\iso} \Wkm{k^{m-1}}_{(-1)} \otimes \Wkm{1^m}_{(-q)^{-k+1}} \otimes  \Wkm{k}_{-(-q)^m}  \\
& \hspace{5ex} \rightarrowtail 
\Wkm{k^{m-1}}_{(-1)} \otimes \Wkm{1^{m-1}}_{(-q)^{-k}} \otimes  \Wkm{1}_{(-q)^{m-k}} \otimes  \Wkm{k}_{-(-q)^m} \\
& \hspace{10ex} \twoheadrightarrow
\Wkm{k^{m-1}}_{(-1)} \otimes \Wkm{1^{m-1}}_{(-q)^{-k}}  \otimes \Wkm{k+1}_{-(-q)^{m-1}},
\end{align*}
is non-zero by Proposition~\ref{prop: non-van}. Here $\overset{*}{\iso}$ holds by~\eqref{eq: exrch folded B}. 
Note that the sequence of real simple modules  
$$
 ( \Wkm{k+1}_{-(-q)^{m-1}}, \Wkm{1^{m-1}}_{(-q)^{-k}},\Wkm{k^{m-1}}_{(-1)}) \text{ is normal},
$$
by Lemma~\ref{lem: no intersection}. 
Then we have
\begin{align*}
&\soc(\Wkm{k^{m-1}}_{(-1)} \otimes \Wkm{1^{m-1}}_{(-q)^{-k}}  \otimes \Wkm{k+1}_{-(-q)^{m-1}}) \\ 
& \hspace{7ex} \iso 
\head( \Wkm{k+1}_{-(-q)^{m-1}}\otimes \Wkm{1^{m-1}}_{(-q)^{-k}}  \otimes \Wkm{k^{m-1}}_{(-1)} ) 
\ (\because \text{Lemma~\ref{lem: normal property}~\eqref{eq: head socle}})
\\  
& \hspace{14ex} \iso 
\head( \Wkm{k+1}_{-(-q)^{m-1}}\otimes \Wkm{(k+1)^{m-1}}_{-(-q)^{-1}}) \ (\because \text{an induction on $m$} ) \\
& \hspace{21ex} \iso \Wkm{(k+1)^m}_{(-1)}. 
\end{align*} 
Thus the image of $\Wkm{k^{m-1}}_{(-1)} \otimes \Wkm{1^m}_{(-q)^{-k+1}} \otimes \Wkm{k}_{-(-q)^m}$ under $c_1$ contains $\Wkm{(k+1)^m}_{(-1)}$.
Then the composition 
\begin{equation} \label{eq: c2 B folded}
\begin{aligned} 
c_2 \colon & \Wkm{1^m}_{(-q)^{-k+1}} \otimes \Wkm{k^m}_{-(-q)}  
\rightarrowtail \Wkm{k^{m-1}}_{(-1)} \otimes \Wkm{1^m}_{(-q)^{-k+1}} \otimes  \Wkm{k}_{-(-q)^m} \\
 & \hspace{25ex} \overset{c_1}{\longrightarrow} \Wkm{k^{m-1}}_{(-1)} \otimes \Wkm{1^{m-1}}_{(-q)^{-k}}  \otimes \Wkm{k+1}_{-(-q)^{m-1}} 
\end{aligned}
\end{equation} 
does not vanish by Proposition~\ref{prop: q-character uniquely appear folded B}.
Hence we can conclude that
\[
\Wkm{(k+1)^m}_{(-1)} \subseteq \Image(c_2).
\]

On the other hand, we have  
\[
\de( \Wkm{1^m}_{(-q)^{-k+1}}, \Wkm{k^m}_{-(-q)}  ) \le 1
\]
by~\eqref{eq: folded de <1 speical} of Proposition~\ref{prop: BW de le 1}.
Moreover, by the induction hypothesis on $m \ge 2$,  we have 
\[
\de( \Wkm{1^{m-1}}_{(-q)^{-k+1}}, \Wkm{k^{m-1}}_{-(-q)}  ) =1.
\]
From Lemma~\ref{lem: no surjective or injective}(\ref{it: no surjective or injective folded B}), $c_2$ in not injective and hence 
\[
\de( \Wkm{1^m}_{(-q)^{-k+1}}, \Wkm{k^m}_{-(-q)}  ) =1
\]
and the composition length of $\Wkm{1^m}_{(-q)^{-k+1}} \otimes \Wkm{k^m}_{-(-q)}$ is $2$.
Thus we have
\[
\Wkm{(k+1)^m}_{(-1)} \iso \Image(c_2)
\quad \text{ and } \quad
\Wkm{1^m}_{(-q)^{-k+1}} \hconv \Wkm{k^m}_{-(-q)} \iso  \Wkm{(k+1)^m}_{(-1)},
\]
which completes the assertion. 
\end{proof}
 
\begin{corollary}
For $\g$ of type $B_n^{(1)}$ and  $n < k <2n-1$, we have    
$$
\de( \Wkm{1^m}_{(-q)^{-k+1}}, \Wkm{k^m}_{-(-q)}  ) = 1 \quad \de(\Wkm{1^m}_{(-q)^{-k+1}}, \Wkm{k}_{-(-q)^{m}}) = 1,
$$
and hence the short exact sequence
$$
0 \to \Wkm{1^m}_{(-q)^{-k+1}} \sconv \Wkm{k^m}_{-(-q)} \to \Wkm{1^m}_{(-q)^{-k+1}} \otimes \Wkm{k^m}_{-(-q)} \to \Wkm{(k+1)^m}_{-1} \to 0,
$$
where the left and right terms are real simple modules by Proposition~\textup{\ref{prop: real head}}. 
\end{corollary}

Now we have a generalization of~\eqref{eq: mesh type Dorey's rule} by using~\eqref{eq: k+l<n homo} and~\eqref{eq: k+l=n homo}:
\begin{theorem}[Higher Dorey's rule of mesh type I] \label{thm: higher mesh}
For classical untwisted affine types $\g=A_{n-1}^{(1)}$, $B_n^{(1)}$, $C_n^{(1)}$, $D_n^{(1)}$, $m \in \Z_{\ge1}$ and
$l < n'  \seteq n- \delta(\g = D_n^{(1)} )$, assume 
\[
\text{$1\le a,b < n'$ with $a+b < l$ and $\min(a,b)=1$.}
\]
Then we have 
\begin{subequations} \label{eq: higher homo mesh}
\begin{align}  \label{eq: mesh type Dorey's rule higher}
\Vkm{(l-b)^m}_{(-\chq_{l-b})^{-b}} \hconv \Vkm{(l-a)^m}_{(-\chq_{l-a})^{a}} \iso \Vkm{l^m} \otimes \Vkm{(l-a-b)^m}_{(-\chq_{l-a-b})^{a-b}}.
\end{align}    
\bna
\item In particular, if $l=n'$, we have
\begin{equation} \label{eq: mesh BCD}
\begin{aligned}
&\Vkm{(n'-b)^m}_{(-\chq_{n'-b})^{-b}} \hconv \Vkm{(n'-a)^m}_{(-\chq_{n'-a})^{a}}  \\  & \hspace{5ex}
\iso
\bc
\Vkm{n^{2m}}_{(-1)^m} \otimes \Vkm{(n'-a-b)^m}_{(-\chq_{n'-a-b})^{a-b}} & \text{ if } \g =B_n^{(1)}, \\
\Vkm{n^{\lceil m/2 \rceil}}_{(-\qs)^{-\epsilon} }\otimes \Vkm{n^{\lfloor m/2 \rfloor}}_{(-\qs)^{\epsilon}} \otimes \Vkm{(n'-a-b)^m}_{(-\qs)^{a-b}} & \text{ if } \g =C_n^{(1)}, \\
\Vkm{n^{m}} \otimes \Vkm{(n-1)^{m}}  \otimes  \Vkm{(n'-a-b)^m}_{(-q)^{a-b}} & \text{ if } \g =D_n^{(1)},
\ec
\end{aligned}
\end{equation}    
for $a,b < n'$ with $\min(a,b)=1$. Here we understand $\Vkm{(l-a-b)}$ as the trivial module $\mathbf{1}$ when $l-a-b = 0$. 
\item For $\g=B_n^{(1)}$ and $l > n$, take $b \in [1,2n-1]$ with $l-1 >n$, $l-b <n$ and $1+b < l$ if there are. Then we have
\begin{align} \label{eq: folded homo higher}
\Wkm{(l-b)^m}_{(-q)^{-b+1}} \hconv \Wkm{(l-1)^m}_{-(-q)} \iso \Wkm{l^m}_{(-1)} \otimes 
\Wkm{(l-1-b)^m}_{(-q)^{-b+2}}. 
\end{align}
\ee
\end{subequations}
\end{theorem}

\begin{proof} Since the proof is similar, we give only the proof of~\eqref{eq: mesh type Dorey's rule higher}. 
Without loss of generality, we set $a=1$.
Note that we have
$$
\Vkm{(l-b)^m}_{(-\chq_{l-b})^{-b}} \rightarrowtail  \Vkm{(l-1-b)^m}_{(-\chq_{l-1-b})^{-b+1}} \otimes \Vkm{1^m}_{(-\chq_{1})^{1-l}},
$$
and
$$
\Vkm{1^m}_{(-\chq_{1})^{1-l}} \otimes \Vkm{(l-1)^m}_{(-\chq_{l-1})} \twoheadrightarrow \Vkm{l^m}
$$
by~\eqref{eq: k+l<n homo}. 
Hence we have the sequence of homomorphism
\begin{align*}
& \Vkm{(l-b)^m}_{(-\chq_{l-b})^{-b}}  \otimes \Vkm{(l-a)^m}_{(-\chq_{l-1})} \\    
&\hspace{5ex}\rightarrowtail \Vkm{(l-1-b)^m}_{(-\chq_{l-1-b})^{-b+1}} \otimes \Vkm{1^m}_{(-\chq_{1})^{1-l}} \otimes \Vkm{(l-1)^m}_{(-\chq_{l-1})}\\
& \hspace{10ex} \twoheadrightarrow \Vkm{(l-1-b)^m}_{(-\chq_{l-1-b})^{-b+1}} \otimes \Vkm{l^m}
\end{align*}
whose composition does not vanish by Proposition~\ref{prop: non-van}. Since $\Vkm{(l-1-b)^m}_{(-\chq_{l-1-b})^{-b+1}} \otimes \Vkm{l^m}$
is simple by the $i$-box argument, the assertion follows. 
\end{proof}

By~\eqref{eq: 1 k m 2} from Proposition~\ref{prop: k lm} and~\eqref{eq: folded de <1} from Proposition~\ref{prop: BW de le 1}, we have a generalization of the T-system, which recovers the usual T-system~\eqref{eq: T-system} as a particular case at $b=1$.

\begin{theorem}[Generalized T-system] \label{thm: generalization of T-system b>1}
Consider a classical untwisted affine type $\g=A_{n-1}^{(1)}$, $B_n^{(1)}$, $C_n^{(1)}$, $D_n^{(1)}$, and assume $m \in \Z_{\ge1}$, $l < n' \seteq n - \delta(\g = D_n^{(1)} )$, and $1 < b + 1 \leq l$.  
Then we have 
\begin{subequations}
\label{eq: T higher homo mesh 1}
\begin{equation}  \label{eq: d-value min(a,b)=1}
\begin{aligned}     
&\de(\Vkm{(l-b)^m}_{(-\chq_{l-b})^{-b}},\Vkm{(l-1)^m}_{(-\chq_{l-1})})=1, \\
& \de(\Vkm{(l-1)^m}_{(-\chq_{l-1})}, \Vkm{l-b}_{(-\chq_{l-b})^{1-m-b}})=1,
\end{aligned}   
\end{equation}
and hence short exact sequences as follows: For $l < n'$, we have
\begin{gather}
\begin{aligned}  \label{eq: T mesh type Dorey's rule higher ses 1}
& 0 \to \Vkm{(l-b)^m}_{(-\chq_{l-b})^{-b}} \sconv \Vkm{(l-1)^m}_{(-\chq_{l-1})}  \\
& \hspace{10ex} \to  \Vkm{(l-b)^m}_{(-\chq_{l-b})^{-b}} \otimes \Vkm{(l-1)^m}_{(-\chq_{l-1})} \\
& \hspace{25ex} \to \Vkm{l^m} \otimes \Vkm{(l-1-b)^m}_{(-\chq_{l-1-b})^{1-b}} \to 0,
\end{aligned}   
\allowdisplaybreaks\\
\begin{aligned}  \label{eq: T mesh type Dorey's rule higher ses 2}
& 0 \to \Vkm{(n'-b)^m}_{(-\chq_{n'-b})^{-b}} \sconv \Vkm{(n'-1)^m}_{(-\chq_{n'-1})} \\
& \hspace{5ex} \to  \Vkm{(n'-b)^m}_{(-\chq_{n'-b})^{-b}} \otimes \Vkm{(n'-1)^m}_{(-\chq_{n'-1})}  \\
& \hspace{10ex} \to 
\bc
\Vkm{n^{2m}}_{(-1)^m} \otimes \Vkm{(n-1-b)^m}_{(-\chq_{n'-1-b})^{1-b}} & \text{ if } \g =B_n^{(1)}, \\
\Vkm{n^{\lceil m/2 \rceil}}_{(-\qs)^{-\epsilon} }\otimes \Vkm{n^{\lfloor m/2 \rfloor}}_{(-\qs)^{\epsilon}} \otimes \Vkm{(n-1-b)^m}_{(-\qs)^{1-b}} & \text{ if } \g =C_n^{(1)}, \\
\Vkm{n^{m}} \otimes \Vkm{(n-1)^{m}}  \otimes  \Vkm{(n-2-b)^m}_{(-q)^{1-b}} & \text{ if } \g =D_n^{(1)},
\ec \\
& \hspace{20ex} \to 0,
\end{aligned}
\end{gather}    
where the left and right terms are real simple modules.  
Additionally, for $\g=B_n^{(1)}$, we have
\begin{align*}  
&\de(\Wkm{(l-b)^m}_{(-\chq_{l-b})^{-b}},\Wkm{(l-1)^m}_{-(-\chq_{l-1})})=1, \\
& \de(\Wkm{(l-1)^m}_{-(-\chq_{l-1})}, \Wkm{l-b}_{(-\chq_{l-b})^{1-m-b}})=1,
\end{align*}  
and
\begin{equation}
\begin{aligned}  \label{eq: T mesh type Dorey's rule higher ses Bw}
& 0 \to \Wkm{(l-b)^m}_{(-q)^{-b+1}} \sconv \Wkm{(l-1)^m}_{-(-q)}  \\
& \hspace{10ex} \to  \Wkm{(l-b)^m}_{(-q)^{-b+1}} \otimes \Wkm{(l-1)^m}_{-(-q)} \\
& \hspace{25ex} \to \Wkm{l^m}_{(-1)} \otimes \Wkm{(l-1-b)^m}_{(-q)^{1-b}} \to 0
\end{aligned}   
\end{equation}
for $n < l \le 2n-1$ and $b$ with $1+b \le l$ and $l-b < n < l-1$, where the left and right terms are real simple modules. 
\end{subequations}
\end{theorem}

We note that the left and right terms being real simple modules in the short exact sequences above follow from Proposition~\ref{prop: real head}.

The socle of the classical T-system (equivalently $b=1$) is not prime if and only if $m > 1$ (see Appendix~\ref{Appendix: T-system}), but we can show that the socle is always a prime simple module when $b > 1$.
For this proof, we note that each simple module $M$ in $\scrC_\g$ is labeled by a dominant monomial $m$ in $\Z[Y_{i,a}]_{i \in I_0, a \in \bfk^\times}$; i.e., $M \simeq L(m)$ (we refer~\cite[\S 2.4]{FHOO} for the notation $L(m)$; see also Example~\ref{ex:prime_T_socle}).
For instance, when $\g=A_n^{(1)}$, $\Vkm{(l-1)^m}_{(-q)} \simeq L(m)$ where $m = \prod_{s=0}^{m-1} Y_{l-1,(-q)^{2-m+2s}}$.

\begin{corollary} \label{cor: socle prime}
The socle in~\eqref{eq: T higher homo mesh 1} is not prime if and only if $b = 1$ and $m > 1$.
\end{corollary}

\begin{proof}
From the above discussion, we only need to show that if $b > 1$, then the socle is prime.
In this case, one can see that the socle in~\eqref{eq: T mesh type Dorey's rule higher ses 1} can be represented as follows:
\[
\Vkm{(l-b)^m}_{(-q)^{-b}} \sconv \Vkm{(l-1)^m}_{(-q)} \iso L\left( \prod_{s=0}^{m-1} Y_{l-1,(-q_{l-1})^{2-m+2s}}Y_{l-b,(-q_{l-b})^{1-b-m+2s}} \right).
\]
Therefore, the claim follows from Corollary~\ref{cor:two_factor_prime}.
\end{proof}

\begin{example} \label{ex:prime_T_socle}
Consider type $\g=A_3^{(1)}$ with $l=3$, $m=2$, and $b=2$.
We can see that the socle $\Vkm{1^2}_{(-q)^{-2}} \sconv \Vkm{2^2}_{(-q)}$ in~\eqref{eq: T higher homo mesh 1}, which in this case is
\begin{equation}
\label{eq:T_higher_mesh_prime_ex}
0 \to\Vkm{1^2}_{(-q)^{-2}} \sconv \Vkm{2^2}_{(-q)}  \to \Vkm{1^2}_{(-q)^{-2}} \tens \Vkm{2^2}_{(-q)} \to \Vkm{3^2} \to 0,
\end{equation}
is prime.
Indeed, the proof of Corollary~\ref{cor:two_factor_prime} says that if this was not prime, then we would have
$\Vkm{1^2}_{(-q)^{-2}} \sconv \Vkm{2^2}_{(-q)} \iso \Vkm{1^2}_{(-q)^{-2}} \otimes \Vkm{2^2}_{(-q)}$.
Yet this is clearly impossible since the dimensions are not equal.
In particular, the short exact sequence~\eqref{eq:T_higher_mesh_prime_ex} implies that
\[
\Vkm{1^2}_{(-q)^{-2}} \sconv \Vkm{2^2}_{(-q)} \iso V(2\Lambda_1 + 2\Lambda_2) + V(\Lambda_1 + \Lambda_2 + \Lambda_3)
\]
as $U_q(\g_0)$-modules.
\end{example}

\begin{proof} [Proof of~\eqref{eq: spin C homo}] {\rm (i)} Note that the case when $l=n-1$ or $n=2$ coincides with T-system in~\eqref{eq: T-system C}.
We first consider the case where $l=n-2$.
Note that
we have the non-zero composition 
\begin{align*} 
&c_1 \colon \Vkm{n^m}_{(-\qs)^{-3}} \otimes \Vkm{n^{m-1}}_{\qs} \otimes \Vkm{n}_{(-1)^{m}\qs^{2m+1}} \\
& \hspace{3ex} \iso \Vkm{n^{m-1}}_{\qs} \otimes \Vkm{n^m}_{(-\qs)^{-3}} \otimes \Vkm{n}_{(-1)^{m}\qs^{2m+1}} \\
& \hspace{6ex} \rightarrowtail \Vkm{n^{m-1}}_{\qs}  \otimes \Vkm{n^{m-1}}_{\qs^{-5}} \otimes \Vkm{n}_{(-1)^{m}\qs^{2m-5}} \otimes \Vkm{n}_{(-1)^{m}\qs^{2m+1}} \\
& \hspace{9ex} \twoheadrightarrow  \Vkm{n^{m-1}}_{\qs} \otimes \Vkm{n^{m-1}}_{\qs^{-5}} \otimes \Vkm{(n-2)^2}_{(-1)^{m+1}\qs^{2m-2}} 
\end{align*} 
by Proposition~\ref{prop: non-van}. 
Note that,  
$$
\de(\scrD \Vkm{(n-2)^2}_{(-1)^{m+1}\qs^{2m-2}}, \Vkm{n^{m-1}}_{(-\qs)})=0
$$
by Lemma~\ref{lem: no intersection}, since $n \ge 2$. 
Thus the sequence of real simple modules
$$
(\Vkm{(n-2)^2}_{(-1)^{m+1}\qs^{2m-2}} , \Vkm{n^m}_{\qs^{-5}} ,\Vkm{n^{m-1}}_{\qs}) \text{ is normal},
$$
by  Lemma~\ref{lem: normal seq d}~(a-iii). 
Then we have  
\begin{align*}
&\soc( \Vkm{n^{m-1}}_{\qs} \otimes \Vkm{n^{m-1}}_{\qs^{-5}} \otimes \Vkm{(n-2)^2}_{(-1)^{m+1} \qs^{2m-2}}  ) \\ 
& \hspace{3ex} \iso 
\head(\Vkm{(n-2)^2}_{(-1)^{m+1} \qs^{2m-2}} \otimes \Vkm{n^{m-1}}_{\qs^{-5}} \otimes \Vkm{n^{m-1}}_{\qs} ) 
\ (\because \text{Lemma~\ref{lem: normal property}~\eqref{eq: head socle}})
\\  
& \hspace{3ex} \iso 
\head( \Vkm{(n-2)^2}_{(-1)^{m+1} \qs^{2m-2}}  \otimes  \Vkm{(n-2)^{2m-2}}_{(-1)^{m+1}\qs^{-2}}
\ (\because \text{an induction on $m$} ) \\
& \hspace{3ex} \iso \Vkm{(n-2)^{2m}}_{(-1)^{m+1}}.
\end{align*}  

Thus the image of $ \Vkm{n^{m-1}}_{\qs} \otimes \Vkm{n^m}_{(-\qs)^{-3}} \otimes \Vkm{n}_{(-1)^{m+1}\qs^{2m+1}}$ under $c_1$ contains $\Vkm{(n-2)^{2m}}_{(-1)^{m}}$.
Then the composition
\begin{equation} \label{eq: c2 spin C2}
\begin{aligned} 
c_2 \colon & \Vkm{n^m}_{(-\qs)^{-3}} \otimes \Vkm{n^{m}}_{(-\qs)^3}   
\rightarrowtail \Vkm{n^{m-1}}_{\qs} \otimes \Vkm{n^m}_{(-\qs)^{-3}} \otimes \Vkm{n}_{(-1)^{m}\qs^{2m+1}}  \\
 & \hspace{25ex} \overset{c_1}{\longrightarrow} \Vkm{n^{m-1}}_{\qs} \otimes \Vkm{n^{m-1}}_{\qs^{-5}} \otimes \Vkm{(n-2)^2}_{(-1)^{m+1}\qs^{2m-2}} 
\end{aligned}
\end{equation}
does not vanish by Proposition~\ref{prop: q-character uniquely appear spin C and D}.
Hence we can conclude that
\[
\Vkm{(n-2)^{2m}}_{(-1)^{m+1}} \subseteq \Image(c_2).
\]

On the other hand, we have  
$$
\de(\Vkm{n^m}_{(-\qs)^{-3}} , \Vkm{n^{m}}_{(-\qs)^3}) \le 1
$$
by Proposition~\ref{prop: Cn nm nm -3 comp2}(\ref{it: C comp2}).
Moreover, by the induction hypothesis on $m \ge 2$, we have
\[
\de(\Vkm{n^{m-1}}_{(-\qs)^{-3}} , \Vkm{n^{m-1}}_{(-\qs)^3}) = 1.
\]
From Lemma~\ref{lem: no surjective or injective}(\ref{it: no surjective or injective C,D}), $c_2$ in not injective and hence 
\[
\de(\Vkm{n^m}_{(-\qs)^{-3}} , \Vkm{n^{m}}_{(-\qs)^3}) =1
\]
and the composition length of $ \Vkm{n^m}_{(-\qs)^{-3}} \otimes \Vkm{n^{m}}_{(-\qs)^3} $ is $2$.
Thus we have
\[
\Vkm{(n-2)^{2m}}_{(-1)^{m+1}} \iso \Image(c_2) \text{ indeed}
\]
and 
\[
\Vkm{n^m}_{(-\qs)^{-3}} \hconv \Vkm{n^{m}}_{(-\qs)^3} \iso \Vkm{(n-2)^{2m}}_{(-1)^{m+1}},
\] 
which completes the assertion for $l=n-2$.

\medskip\noindent
{\rm (ii)} Now let us prove the assertion for $l=1$. Note that 
we now have
\begin{align*}
\Vkm{1^{2m}}_{(-\qs)^{-n+2}}    \otimes \Vkm{(n-2)^{2m}}_{(-\qs)}  & \twoheadrightarrow    \Vkm{(n-1)^{2m}}, \allowdisplaybreaks\\
\Vkm{(n-1)^{2m}} \otimes \Vkm{n^m}_{(-\qs)^{2n}}  & \twoheadrightarrow  \Vkm{n^m}_{(-\qs)^2},  \allowdisplaybreaks\\
\Vkm{(n-2)^{2m}} \otimes \Vkm{n^m}_{(-\qs)^{2n-1}} & \twoheadrightarrow   \Vkm{n^m}_{(-\qs)^{3}}.
\end{align*}
Now we claim that the sequence of real simple modules 
\begin{align} \label{eq: normal claim}
(  \Vkm{1^{2m}}_{(-\qs)^{-n+2}}  , \Vkm{(n-2)^{2m}}_{(-\qs)} , \Vkm{n^m}_{(-\qs)^{2n}} ) \text{ is  normal.}    
\end{align}

Note that
\bitem
\item $\rch(\Vkm{1^{2m}}_{(-\qs)^{n+4}}) = \range{-2m+n+5,2m+n+3}$,
\item $ \exrch_1(\Vkm{n^m}_{(-\qs)^{2n}}) = \range{-2m+n+1,2m+3n-3}$.
\eitem 
Since $n>3$, we have 
$$
\de(\scrD  \Vkm{1^{2m}}_{(-\qs)^{-n+2}} ,  \Vkm{n^m}_{(-\qs)^{2n}} ) = \de( \Vkm{1^{2m}}_{(-\qs)^{n+4}} ,  \Vkm{n^m}_{(-\qs)^{2n}} ) = 0,
$$
which implies~\eqref{eq: normal claim} by Lemma~\ref{lem: normal seq d}. 

Hence we have the following commutative diagram:
\begin{align}\label{Diagram: normal C}
 \scalebox{0.8}{\raisebox{3em}{\xymatrix@R=3.5ex{
\Vkm{1^{2m}}_{(-\qs)^{-n+2}}  \otimes \Vkm{(n-2)^{2m}}_{(-\qs)} \otimes \Vkm{n^m}_{(-\qs)^{2n}}   \ar@{->>}[dr] \ar@{->>}[r] 
&   \Vkm{(n-1)^{2m}}  \ar@{->>}[r] \otimes \Vkm{n^m}_{(-\qs)^{2n}}  & \Vkm{n^m}_{(-\qs)^{2}}   \\
&   \Vkm{1^{2m}}_{(-\qs)^{-n+2}} \otimes \Vkm{n^m}_{(-\qs)^{4}}   \ar@{->>}[ur]
}}}
\end{align}
Thus we have
\begin{align*}
  &\Vkm{1^{2m}}_{(-\qs)^{-n+2}} \otimes \Vkm{n^m}_{(-\qs)^{4}} \twoheadrightarrow \Vkm{n^m}_{(-\qs)^{2}}  \\
  & \iff \Vkm{n^m}_{(-\qs)^{-2n+2}}  \otimes  \Vkm{n^m}_{(-\qs)^{2}} \twoheadrightarrow   \Vkm{1^{2m}}_{(-\qs)^{-n+2}} \\
& \iff \Vkm{n^m}_{(-\qs)^{-n}}  \otimes  \Vkm{n^m}_{(-\qs)^{n}} \twoheadrightarrow   \Vkm{1^{2m}},
\end{align*}
as we desired. 

\mnoi
{\rm (iii)} Now let us apply an induction on $l \ge 1$. 
Note that we have the homomorphism $(d=-n+m-l+1)$
$$
\Vkm{(l-1)^{2m}} \otimes \Vkm{n^m}_{(-1)^{d}\qs^{n+l}}   \twoheadrightarrow \Vkm{n^m}_{(-1)^{d}\qs^{n+2-l}} 
$$
by the induction.
Note that the sequence $(\Vkm{1^{2m}}_{(-\qs)^{-l}}, \Vkm{(l-1)^{2m}} , \Vkm{n^m}_{(-1)^{d}\qs^{n+l}} )$ is normal as  
\bitem
\item $\exrch_{n'}(\scrD(\Vkm{1^{2m}}_{(-\qs)^{-l}} ))
= \range{-2m+n-l,2m+4n-l+2}$,
\item $\rch(\Vkm{n^m}_{(-1)^{d}\qs^{n+l}} )  = \range{2-2m+n+l,2m-2+n+l},
$
\eitem 
where $n'=\clr(\Vkm{n^m}_{(-1)^{d}\qs^{n+l}} ) \in \{n,n+1\}$ for a fixed skeleton category containing them.

Hence we have the following commutative diagram:
\begin{align*} 
 \scalebox{0.8}{\raisebox{3em}{\xymatrix@R=3.5ex{
\Vkm{1^{2m}}_{(-\qs)^{-l}} \otimes \Vkm{(l-1)^{2m}} \otimes \Vkm{n^m}_{(-1)^{d}\qs^{n+l}}   \ar@{->>}[dr] \ar@{->>}[r] 
& \Vkm{1^{2m}}_{(-\qs)^{-l}} \otimes \Vkm{n^m}_{(-1)^{d}\qs^{n+2-l}} \ar@{->>}[r]   & \Vkm{n^m}_{(-1)^d(\qs)^{n-l}} \\
&  \Vkm{l^{2m}}_{(-\qs)^{-1}}  \otimes \Vkm{n^m}_{(-1)^{d}\qs^{n+l}} \ar@{->>}[ur]
}}}
\end{align*}
by using the homomorphism in~\eqref{eq: k+l<n homo}.  
Hence we have
$$
\Vkm{l^{2m}}_{(-\qs)^{-1}}  \otimes \Vkm{n^m}_{(-1)^{d}\qs^{n+l}} \twoheadrightarrow \Vkm{n^m}_{(-1)^d(\qs)^{n-l}}
$$
which is equivalent to 
$$  
 \Vkm{n^m}_{(-1)^{d+1}(\qs)^{-n-1+l}} \otimes  \Vkm{n^m}_{(-1)^{d+1}\qs^{n+1-l}}  \twoheadrightarrow     \Vkm{l^{2m}}
$$
as we desired. Thus the assertion follows.
\end{proof}

The following corollary can be also understood as a generalization of T-system.

\begin{corollary} \label{cor: length 2 C}
For $\g=C_n^{(1)}$, we have 
\[
\de(\Vkm{n^m}_{(-\qs)^{-3}}, \Vkm{n^{m}}_{(-\qs)^3}) = 1,
\qquad\qquad
\de(\Vkm{n^m}_{(-\qs)^{-3}} , \Vkm{n}_{(-\qs)^{2m+1}}) = 1,
\]
and hence we have the short exact sequence  
\begin{equation}
\label{eq:Cn_gen_T_system_SES}
0 \to  \Vkm{n^m}_{(-\qs)^{-3}} \sconv \Vkm{n^{m}}_{(-\qs)^3}   \to \Vkm{n^m}_{(-\qs)^{-3}} \otimes \Vkm{n^{m}}_{(-\qs)^3}
\to \Vkm{(n-2)^{2m}} \to 0,
\end{equation}
where the left and right terms are real simple modules. 
\end{corollary}

The fact that the left and right terms in~\eqref{eq:Cn_gen_T_system_SES} are due to Proposition~\ref{prop: real head}.

\begin{proof}[Proof for~\eqref{eq: spin D homo}]
The assertion for $n-2$ coincides with T-system in~\eqref{eq: T-system}.
Thus now we consider when $k=n-3$. 
Note that we have the non-zero composition 
\begin{align*} 
&c_1 \colon \Vkm{n^m}_{(-q)^{-2}} \otimes \Vkm{(n-1)^{m-1}}_{(-q)} \otimes \Vkm{(n-1)}_{(-q)^{m+1}} \\
& \hspace{3ex} \iso    \Vkm{(n-1)^{m-1}}_{(-q)} \otimes \Vkm{n^m}_{(-q)^{-2}} \otimes \Vkm{(n-1)}_{(-q)^{m+1}}\\
& \hspace{6ex} \rightarrowtail  \Vkm{(n-1)^{m-1}}_{(-q)} \otimes \Vkm{n^{m-1}}_{(-q)^{-3}} \otimes \Vkm{n}_{(-q)^{m-3}} \otimes \Vkm{(n-1)}_{(-q)^{m+1}} \\
& \hspace{9ex} \twoheadrightarrow  \Vkm{(n-1)^{m-1}}_{(-q)} \otimes \Vkm{n^{m-1}}_{(-q)^{-3}} \otimes \Vkm{n-3}_{(-q)^{m-1}} 
\end{align*} 
by Proposition~\ref{prop: non-van}. 
Note that,  
$$
\de(\scrD  \Vkm{n-3}_{(-q)^{m-1}} , \Vkm{(n-1)^{m-1}}_{(-q)})=0
$$
by Lemma~\ref{lem: no intersection}.
Thus the sequence of real simple modules
$$
( \Vkm{n-3}_{(-q)^{m-1}} ,\Vkm{n^{m-1}}_{(-q)^{-3}}, \Vkm{(n-1)^{m-1}}_{(-q)} ) \text{ is normal},
$$
by  Lemma~\ref{lem: normal seq d}~(a-iii). 
Then we have  
\begin{align*}
&\soc( \Vkm{(n-1)^{m-1}}_{(-q)} \otimes \Vkm{n^{m-1}}_{(-q)^{-3}} \otimes \Vkm{n-3}_{(-q)^{m-1}} ) \\ 
& \hspace{3ex} \iso 
\head(\Vkm{n-3}_{(-q)^{m-1}} \otimes \Vkm{n^{m-1}}_{(-q)^{-3}}  \otimes  \Vkm{(n-1)^{m-1}}_{(-q)} ) 
\ (\because \text{Lemma~\ref{lem: normal property}~\eqref{eq: head socle}})
\\  
& \hspace{3ex} \iso 
\head( \Vkm{n-3}_{(-q)^{m-1}} \otimes \  \otimes \Vkm{(n-3)^{m-1}}_{(-q)^{-1}})
\ (\because \text{an induction on $m$} ) \\
& \hspace{3ex} \iso \Vkm{(n-3)^{m}}.
\end{align*}  

Thus
the image of $\Vkm{n^m}_{(-q)^{-2}} \otimes \Vkm{(n-1)^{m-1}}_{(-q)} \otimes \Vkm{(n-1)}_{(-q)^{m+1}}$ under $c_1$
contains $\Vkm{(n-3)^{m}}$ Then the composition
\begin{equation} \label{eq: c2 spin D}
\begin{aligned} 
c_2 \colon & \Vkm{n^m}_{(-q)^{-2}} \otimes \Vkm{(n-1)^m}_{(-q)^{2}}   \\
 & \hspace{12.5ex} \rightarrowtail \Vkm{(n-1)^{m-1}}_{(-q)} \otimes \Vkm{n^m}_{(-q)^{-2}} \otimes \Vkm{(n-1)}_{(-q)^{m+1}} \\
 & \hspace{25ex} \overset{c_1}{\longrightarrow} \Vkm{(n-1)^{m-1}}_{(-q)} \otimes \Vkm{n^{m-1}}_{(-q)^{-3}} \otimes \Vkm{n-3}_{(-q)^{m-1}} 
\end{aligned}
\end{equation}
does not vanish by Proposition~\ref{prop: q-character uniquely appear spin C and D}.
Hence we can conclude that
\[
\Vkm{(n-3)^{m}} \subseteq \Image(c_2).
\]

On the other hand, we have  
$$
\de(\Vkm{n^m}_{(-q)^{-2}} , \Vkm{(n-1)^m}_{(-q)^{2}}  ) \le 1
$$
by Proposition~\ref{prop: Cn nm nm -3 comp2}~\eqref{it: D comp2}. Moreover, by the induction hypothesis on $m \ge 2$,  we have $$\de(\Vkm{n^{m-1}}_{(-q)^{-2}} , \Vkm{(n-1)^{m-1}}_{(-q)^{2}} ) =1.$$   
From Lemma~\ref{lem: no surjective or injective}(\ref{it: no surjective or injective C,D}), $c_2$ in not injective and hence 
\[
\de(\Vkm{n^m}_{(-q)^{-2}} , \Vkm{(n-1)^m}_{(-q)^{2}} ) =1
\]
and the composition length of $\Vkm{n^m}_{(-q)^{-2}} \otimes \Vkm{(n-1)^m}_{(-q)^{2}} $ is $2$.
Thus we have
\[
\Vkm{(n-3)^{m}}  \iso \Image(c_2) \text{ indeed}
\]
and 
\[
\Vkm{n^m}_{(-q)^{-2}} \hconv \Vkm{(n-1)^m}_{(-q)^{2}} \iso \Vkm{(n-3)^{m}} ,
\]
which completes the assertion for $l=n-3$.

Based on the assertions for $l \in \{ n-2,n-3 \}$, we can obtain the assertion for $l=1$ as the proof of~\eqref{eq: spin C homo}~{\rm (ii)}. Then for $l \ge 1$, we can apply an induction as the proof of~\eqref{eq: spin C homo}~{\rm (iii)} with the consideration on the parity of $n-l$.  
\end{proof}

\begin{corollary} \label{cor: length 2 D} For $\g=D_n^{(1)}$, we have 
\[
\de(\Vkm{n^m}_{(-q)^{-2}} , \Vkm{(n-1)^m}_{(-q)^{2}} ) =1 \quad \de(\Vkm{(n-1)^m}_{(-\qs)^{2}} , \Vkm{n}_{(-\qs)^{m+1}})=1,
\]
and hence we have the short exact sequence
\[
0 \to \Vkm{n^m}_{(-q)^{-2}} \sconv \Vkm{(n-1)^m}_{(-q)^{2}} \to \Vkm{n^m}_{(-q)^{-2}} \tens \Vkm{(n-1)^m}_{(-q)^{2}} \to \Vkm{(n-3)^{m}} \to 0,
\]
where the left and right terms are real simple modules by {\rm Proposition~\ref{prop: real head}}. 
\end{corollary}

\subsection{Twisted classical affine types}

In this subsection, we briefly explain how we can obtain the higher Dorey's rule for twisted affine type from the ones of untwisted affine type. 

In~\cite{KKKOIII}, the exact monoidal functor $\calF^0_{A}$, called the exact generalized Schur--Weyl duality functor, between 
\[
\calF^0_{A} \colon \scrC^0_{A_n^{(1)}} \to \scrC^0_{A_n^{(2)}}
\]
sending $\calV^{(1)}(\im,p)$ to $\calV^{(2)}(\im,p)$ and simple modules to simple modules bijectively.
In particular, it sends KR modules $\frakR[a,b]$ in $\scrC^0_{A_n^{(1)}}$ to KR modules $\frakR[a,b]$ in $\scrC^0_{A_n^{(2)}}$.
Thus the higher Dorey's rule for untwisted affine type $A_n^{(1)}$ implies the one for $A_n^{(2)}$. 
 
On the other hand, there currently is no known such functor between (i) $\scrC^0_{D_{n+1}^{(1)}}$ and $\scrC^0_{D_{n+1}^{(2)}}$; (ii)
$E_{6}^{(1)}$ and $E_{6}^{(2)}$; and (iii) $D_{4}^{(1)}$ and $D_{4}^{(3)}$.
However, for each $\rmQ$-datum $\calQ$ of $\calU_q(\g^{(1)})$ and $\calU_q(\g^{(2)})$,
there exists an exact monoidal functor $\calF^{1,t}_\calQ$ in \cite{KKKOIV,OhS19} (see also~\cite{Fuj22b,Naoi21})
between $\scrC^{(1)}_\calQ$  and $\scrC^{(t)}_\calQ$ $(t=2,3)$ sending
$\calV^{(1)}(\im,p)$ in $\scrC^{(1)}_\calQ$ to $\calV^{(2)}(\im,p)$ in $\scrC^{(2)}_\calQ$, and
simple modules to simple modules in their categories bijectively.  
It also sends KR modules $\frakR^{(1)}[a,b]$ in $\scrC^{(1)}_\calQ$ to KR modules $\frakR^{(t)}[a,b]$ in $\scrC^{(t)}_\calQ$ $(t=2,3)$.
Moreover,   
\bitem 
\item $\de(\calV^{(1)}(\im,p),\calV^{(1)}(\jm,s) )=\de(\calV^{(t)}(\im,p),\calV^{(t)}(\jm,s) )$ $(t=2,3)$~\cite{Oh14R};
\item the (twisted) $q$-characters of KR modules for $\g^{(r)}$ can be defined in terms of the corresponding $q$-characters of $\g^{(1)}$~\cite{Her10}.
\eitem 
Then, using the same framework of untwisted affine type, we can obtain the following results:

\begin{proposition} \label{prop: q-character uniquely appear twisted D}  
Let $\uqpg$ be a quantum affine algebra of twisted type $D_{n+1}^{(2)}$.
Consider $m \in \Z_{\ge 1}$ and $1 \le k, l < n$ such that $k + l < n$.
Then the highest weight $U_q(\g_0)$-module $V(m \La_{k+l})$ appears exactly once in the $U_q(\g_0)$ restriction of the following modules:
\bna
\item $\Vkm{l^m}_{(-\chq_l)^{-k}} \otimes \Vkm{k^m}_{(-\chq_k)^{l}}$,
\item $  \Vkm{k}_{(-\chq_k)^{l+m-1}}  \otimes  \Vkm{k^{m-1}}_{(-\chq_k)^{l-1}} \otimes  \Vkm{l^{m}}_{(-\chq_l)^{-k}}$,
\item $  \Vkm{k^{m-1}}_{(-\chq_k)^{l-1}} \otimes  \Vkm{l^{m-1}}_{(-\chq_l)^{-k-1}} \otimes \Vkm{k+l}_{(-\chq_{k+1})^{m-1}}$.
\item $\Vkm{(k+l)^m}$
\ee
In particular $k+l =n$, then $V(2m \La_n)$ appears exactly once in the $U_q(\g_0)$ restriction of the modules {\rm (a)}$\sim${\rm (c)} and
$\Vkm{n^m} \otimes \Vkm{n^m}_{\sqrt{-1}}$.
\end{proposition}

\begin{proof}
The proof is analogous to the proof of Proposition~\ref{prop: q-character uniquely appear}.
\end{proof}

\begin{lemma} \label{lem: no surjective or injective twsited}
Let $m \in \Z_{\ge 2}$ and $\g=D_{n+1}^{(2)}$
\bnum
\item \label{it: no homom classical twisted}
For $1 \le k < n$ with $k + 1 < n$, there is  no \emph{injective} homomorphism
\begin{align*}
\Vkm{1^m}_{(-q)^{-k}} \otimes \Vkm{k^{m}}_{(-q)} 
\rightarrowtail
\Vkm{k^{m-1}} \otimes \Vkm{1^{m-1}}_{(-q)^{-k-1}}  \otimes \Vkm{k+1}_{(-q)^{m-1}} 
\end{align*}
and no \emph{isomorphism} 
\begin{align*}
 \Vkm{1^m}_{(-q)^{-k}} \otimes \Vkm{k^{m}}_{(-q)} 
\isoto 
\Vkm{(k+1)^m}.  
\end{align*}
\item  \label{it: no homo classical D to spin twisted}  For $\g = D_{n+1}^{(2)}$, there is no \emph{injective} homomorphism
\begin{align*}
\Vkm{1^m}_{(-q)^{-n+2}}  &\otimes \Vkm{(n-1)^{m}}_{(-q)}  \\
& \rightarrowtail
\Vkm{(n-1)^{m-1}} \otimes \Vkm{1^{m-1}}_{(-q)^{-n+2}}  \otimes \Vkm{n}_{(-q)^{m-1}} \otimes \Vkm{n}_{\sqrt{-1}(-q)^{m-1}}  
\end{align*}
and no \emph{isomorphism}
\begin{align*}
\Vkm{1^m}_{(-q)^{-n+2}} \otimes \Vkm{(n-1)^{m}}_{(-q)}  \isoto    \Vkm{n^{m}} \otimes \Vkm{n^{m}}_{\sqrt{-1}}. 
\end{align*}  
\item For $\g = D_{n+1}^{(2)}$,  there is no \emph{injective} homomorphism
\begin{align*}
\Vkm{n^m}_{(-q)^{-2}} \otimes \Vkm{n^m}_{\sqrt{-1}(-q)^{2}}   
\rightarrowtail \Vkm{n^{m-1}}_{\sqrt{-1}(-q)} \otimes \Vkm{n^{m-1}}_{(-q)^{-3}} \otimes \Vkm{n-2}_{(-q)^{m-1}}, 
\end{align*}
and no \emph{isomorphism}
\begin{align*}
\Vkm{n^m}_{(-q)^{-2}} \otimes \Vkm{n^m}_{\sqrt{-1}(-q)^{2}}   \isoto \Vkm{(n-2)^m}.
\end{align*}
\ee
\end{lemma}

\begin{proof}
The proof is analogous to the proof of Lemma~\ref{lem: no surjective or injective}.
\end{proof}

Let us fix a compatible reading $\frakR$ of $\hbDynkin$.

\begin{theorem}[Higher Dorey's rule twisted types]
\label{thm: Higher Dorey twisted}
The higher Dorey's for classical twisted affine type holds. Namely, if there exists a homomorphism
$$
\frakR^{(1)}[a_1,b_1] \otimes \frakR^{(1)}[a_2,b_2] \twoheadrightarrow \dtens_{j \in J} \frakR^{(1)}[c_j,d_j] \quad\text{ in } \scrC^0_{\g^{(1)}}
$$
for KR modules $\frakR^{(1)}[a_u,b_u]$ $(u=1,2)$ and $\frakR^{(1)}[c_j,d_j]$ $(j \in J)$, there exists a homomorphism
$$
\frakR^{(t)}[a_1,b_1] \otimes \frakR^{(t)}[a_2,b_2] \twoheadrightarrow \dtens_{j \in J} \frakR^{(t)}[c_j,d_j] \quad\text{ in } \scrC^0_{\g^{(t)}} \quad \text{for $t=2,3$,}
$$
where $\frakR^{(t)}[a_u,b_u]$ $(u=1,2)$ and $\frakR^{(t)}[c_j,d_j]$ $(j \in J)$ are corresponding KR modules. 
\end{theorem}

\section{Universal coefficient formulas and denominator formulas}
\label{sec:denominator_formulas}

From this section, we shall compute the denominator formulas between KR modules.
In this section, we first collect the universal coefficient formulas to prepare the computations.
Then we will present the denominator formulas, which will be proved in the next section. 

\subsection{Universal coefficient formulas}\label{subsec: UCF}

In this subsection, we will compute $a_{V,W}(z)$'s when $V$ and $W$ are KR modules as we will need them to show our denominator formulas.
As we only need them for types $A^{(1)}_{n-1}$, $B^{(1)}_{n}$, $C^{(1)}_{n}$, and $D^{(1)}_{n}$, we restrict ourselves to these cases.
For the twisted types, one can compute them in the similar way, and we leave the computation to the interested reader.

As mentioned in~\eqref{eq: aimjl}, the universal coefficient $a_{M,N}$ for KR module $M,N$ can be obtained from $d_{M,N}(z)$ and $d_{\scrD N,M}(z)$.
However, as we need $a_{M,N}$ to prove $d_{M,N}(z)$, we instead use \emph{only} the fusion rule~\eqref{eq:KR-surj} (\textit{cf.}~\eqref{eq:fusion_rule}) to compute $a_{M,N}$ here.

For $k \in \Z$, we will use the following notation
\begin{align*}
[k] & \seteq ((-\qs)^kz; p^{*2})_\infty,
&
\PA{k} & \seteq (-(-q)^{k}z ; p^{*2})_\infty,
\\
{}_s[k]_{(t)} & \seteq ((-1)^t \qs^{k}z ; p^{*2})_\infty, & {}_s[k] & \seteq ((-\qs)^{k}z ; p^{*2})_\infty.
\end{align*}

\begin{proposition} \label{prop:universal_coeff_A}
Let $\g$ be of type $A_{n-1}^{(1)}$.
For $k,l \in I_0$ and $p,m \in \Z_{\ge 1}$, set
$\mu_1 \seteq  \min(k,l,n-k,n-l)$ and  $\mu_2 \seteq \min(p,m)-1$.
We have
\begin{align}\label{eq: akl}
a_{l^p,k^m}(z) & = 
\prod_{s=1}^{\mu_1 }  \prod_{t=0}^{\mu_2}  \dfrac{[n+|n-k-l|+\abs{p-m}+2(s+t)][n-|n-k-l|-\abs{p-m}-2(s+t)] }{[|k-l|+\abs{p-m}+2(s+t)][2n-|k-l|-\abs{p-m}-2(s+t)]}.
\end{align}
\end{proposition}

\begin{proof}
For simplicity, set
\[
\mu_1 \seteq  \min(k,l,n-k,n-l) \quad \text{ and } \quad  \mu_2 = \min(p,m) - 1.  
\]
Recall
\begin{align}\label{eq: akl dkl An-1}
d_{k,l}(z)= \displaystyle\prod_{s=1}^{\mu_1} \big(z-(-q)^{2s+|k-l|}\big), \ \
a_{k,l}(z) \equiv \dfrac{[|k-l|][2n-|k-l|]}{[k+l][2n-k-l]}
\end{align}
(see~\ref{subsec: fundamental deno} below). We prove the claim by using Theorem~\ref{Thm: basic properties},~\eqref{eq: akl dkl An-1} and a sequence of inductions. Indeed, we first show that
\[
a_{1,1^m}(z) = \dfrac{[1-m][2n+m-1]}{ [m+1] [2n-(m+1)] }
\]
by induction on $m$.
Note that
\begin{align}\label{Diagram: Runiv1}
\raisebox{4.5em}{\diagramone{1}{1^{m-1}}{1}{1}{1-m}{1^m}}
\end{align}
and hence
\begin{align}\label{Diagram: Runiv2}
\raisebox{3em}{\xymatrix@R=3.5ex{
u_{1} \otimes u_{1^{m-1}} \otimes u_{1}     \ar@{|-_{>}}[r] \ar@{|-_{>}}[d] &     u_1 \otimes u_{1^m} \ar@{|-_{>}}[dd]\\
a_{1,1^{m-1}}((-q)z) u_{1^{m-1}} \otimes u_{1} \otimes u_{1}  \ar@{|-_{>}}[d] &\\
*++{a_{1,1^{m-1}}((-q)z) a_{1,1}((-q)^{l-m}z)  u_{1^{m-1}} \otimes u_1 \otimes u_1} \ar@{|-_{>}}[r] &    a_{1,1}(z) u_{1^m} \otimes u_1     }}
\end{align}
where $u_*$ denotes the dominant extremal weight vector of $\Vkm{*}$. Then, by induction hypothesis, we have
\begin{equation} \label{eq: 11m univ}
\begin{aligned}
a_{1,1^m}(z) & =   a_{1,1^{m-1}}((-q)z)a_{1,1}((-q)^{1-m}z)  \\
& =  \dfrac{[3-m][2n+m-1]}{ [m+1] [2n-m+1] } \times \dfrac{[1-m][2n+1-m]}{[3-m][2n-1-m]} = \dfrac{[1-m][2n+m-1]}{ [m+1] [2n-m-1] }.
\end{aligned}
\end{equation}
We can then apply the same arguments as in~\eqref{Diagram: Runiv1} and~\eqref{Diagram: Runiv2} and a straightforward induction to obtain
\begin{gather*}
a_{1,(n-1)^{m}}(z) = \dfrac{[n+m+1][n-m-1]}{[n+m-1][n-m+1]}, \\
a_{k,1^m} = \dfrac{[k-m][2n-k+m]}{[k+m][2n-k-m]},
\quad \text{ and } \quad
a_{k,(n-1)^{m}}(z) = \dfrac{[n+m+k][n-m-k]}{[n-m+k][n+m-k]}.
\end{gather*}
Next, by a similar argument we compute
\[
a_{1,k^m}(z) = \dfrac{[k-m][2n+m-k]}{ [m+k] [2n-k-m] }
\quad \text{ and } \quad
a_{(n-1),k^m}(z)  =   \dfrac{[n+m+k][n-m-k]}{[n-m+k][n+m-k]}.
\]
By again using the same arguments as in~\eqref{Diagram: Runiv1} and~\eqref{Diagram: Runiv2}, we have
\begin{align} \label{eq: alkm type A}
a_{l,k^m}(z) = \prod_{s=1}^{\mu_1} \dfrac{[n+\abs{n-l-k}+m-1+2s][n-\abs{n-l-k}-m+1-2s]}{[\abs{k-l}+m-1+2s][2n-\abs{k-l}-m+1-2s]}
\end{align}
by assuming $l \leq k$ without loss of generality and an induction on $\min(l,k,n-k,n-l)$, where we split the induction step into the cases $l \leq n - k$ and $l > n - k$.
Note that our assertion for $m=1$ holds by~\eqref{eq: aij} and~\eqref{eq: akl dkl An-1}.
Finally, we have
\begin{align*}
\prod_{s=1}^{\mu_1} \prod_{t=0}^{\mu_2}  \dfrac{[n+|n-k-l|+\abs{p-m}+2(s+t)][n-|n-k-l|-\abs{p-m}-2(s+t)] }{[|k-l|+\abs{p-m}+2(s+t)][2n-|k-l|-\abs{p-m}-2(s+t)]}
\end{align*}
by a straightforward induction on $p$. Note that we can assume that $\min(p,m) = p$ since $d_{l^p,k^m}(z)=d_{k^m,l^p}(z)$ (implying $a_{l^p,k^m}(z)= a_{k^m,l^p}(z)$) and splitting the induction step into cases when $l \leq n - k$ and $l > n - k$.
\end{proof}

The sequence of steps we used to prove Proposition~\ref{prop:universal_coeff_A} will be the same those used to compute the universal coefficients between KR modules. This will hold for all types and all proofs will be similar to the proof of Proposition~\ref{prop:universal_coeff_A}. Therefore, we omit the proofs for the remaining cases
by using the denominator formulas between fundamental modules and universal coefficients, which can be found in~\cite{AK97,KKK15,Oh14R,OhS19} (see Appendix~\ref{subsec: fundamental deno} below for the the denominator formulas and Appendix~\ref{sec: univ table} for universal coefficients). 

\begin{proposition} \label{prop:universal_coeff_B1}
Let $\g$ be of type $B_n^{(1)}$.
For $1 \le k,l <n$ and $m,p \ge 1$, set $\mu_1 \seteq \min(p,m) - 1$, $\mu_2 \seteq \min(2p,m)  - 1$, and $d \seteq m+n+l+p$. Then we have
\begin{align*}
a_{l^p,k^m}(z) & = \hspace{-1ex}\prod_{s=1}^{\min(k,l)}  \prod_{t=0}^{\mu_1} \hspace{-1ex} \left( \dfrac{[k+l-\abs{m-p}-2(s+t)]\PA{2n-|k-l|-\abs{m-p}-1-2(s+t)}}{[|k-l|+\abs{m-p}+2(s+t)]\PA{2n+k+l-\abs{m-p}-1-2(s+t)}} \right.  \\
 & \hspace{5ex} \left. \times \dfrac{\PA{2n+|k-l|+\abs{m-p}-1+2(s+t)}[4n-k-l+\abs{m-p}-2+2(s+t)]}{\PA{2n-k-l+\abs{m-p}-1+2(s+t)}[4n-|k-l|-\abs{m-p}-2-2(s+t)]} \right),
\allowdisplaybreaks \\
a_{l^p,n^m}(z) & = \prod_{s=1}^{l} \prod_{t=0}^{\mu_2} \dfrac{\PNZ{2n+2l-|2p-m|-4s-2t}{d}\PNZ{6n-2l-4+|2p-m|+4s+2t}{d}}{\PNZ{2n-2l-2+|2p-m|+4s+2t}{d}\PNZ{6n-2+2l-|2p-m|-4s-2t}{d}},
\allowdisplaybreaks \\
a_{n^p,n^m}(z) & = \prod_{s=1}^n \prod_{t=0}^{\mu_1} \dfrac{\PNZ{4n+4s+2t-4+|m-p|}{|m-p|}\PNZ{4n-4s-2t-|m-p|}{|m-p|}}{\PNZ{4s+2t-2+|m-p|}{|m-p|}\PNZ{8n-2-4s-2t-|m-p|}{|m-p|}}.
\end{align*}
\end{proposition}

\begin{proof}
The proof of $a_{l^p,k^m}(z)$ and $a_{n^p,n^m}(z)$ is similar to the proof of Proposition~\ref{prop:universal_coeff_A}, but we split the induction step into the cases $k - l \le 0$ or not and $m - p \le 0$ or not.

The computation of $a_{k^m,n^p}(z)$ begins a little different than the previous calculations.
Our assertion for $a_{1,n^2}(z)$ is a direct consequence of the computation
$
a_{1,n^m}(z)  = a_{1,n}((-\qs)z) a_{1,n}((-\qs)^{-1}z).
$
Therefore, our assertion for $a_{1,n^m}(z)$, for $m \ge 3$, can be obtained by the usual induction
$
a_{1,n^m}(z)  = a_{1,n^{m-1}}((-\qs)z) a_{1,n}((-\qs)^{1-m}z).
$
Next, we compute $a_{l,n^m}(z)$ from induction on $l$ with $a_{l,n^m}(z) = a_{l-1,n^{m}}((-q)^{-1}z) a_{1,n^m}((-q)^{l-1}z)$.
Similarly, we compute $a_{^p,n}(z)$ from induction on $p$ with $a_{l^p,n}(z)=a_{l^{p-1},n}((-q)z) a_{l,n}((-q)^{1-p}z)$. Based on the previous computation, we can obtain our desired formula by  the induction steps into cases  $|2p-m| \le 0$ and $|2p-m| > 0$.
\end{proof}

\begin{proposition} \label{prop:universal_coeff_C1}
Let $\g$ be of type $C_n^{(1)}$.
For $1 \le k,l <n$ and $m,p \ge 1$, set $\mu_1 \seteq \min(p,m) - 1$, $\mu_2 \seteq \min(p,2m) - 1$, and $d \seteq m+n+l+p$. Then we have
\begin{align*}
a_{l^p,k^m}(z) & = \hspace{-2ex} \prod_{s=1}^{\min(k,l)} \hspace{-.5ex} \prod_{t=0}^{\mu_1}  \hspace{-1ex}
\left( \dfrac{{}_s[k+l-\abs{m-p}-2s-2t]{}_s[4n+4-k-l+\abs{m-p}+2s+2t]}{{}_s[|k-l|+\abs{m-p}+2s+2t]{}_s[4n+4-|k-l|-\abs{m-p}-2s-2t]} \right. \\
 & \hspace{7ex} \left. \times\dfrac{{}_s[2n+2+|k-l|+\abs{m-p}+2s+2t]{}_s[2n+2-|k-l|-\abs{m-p}-2s-2t]}{{}_s[2n+2-k-l+\abs{m-p}+2s+2t]{}_s[2n+2+k+l-\abs{m-p}-2s-2t]} \right)\!,
\allowdisplaybreaks \\
a_{l^p,n^m}(z) & = \prod_{s=1}^{l} \prod_{t=0}^{\mu_2} \dfrac{{}_s[n+1+l-|2m-p|-2s-2t]_{(d)}{}_s[3n+3-l+|2m-p|+2s+2t]_{(d)}}
{{}_s[n+1-l+|2m-p|+2s+2t]_{(d)}{}_s[3n+3+l-|2m-p|-2s-2t]_{(d)}},
\allowdisplaybreaks \\
a_{n^p,n^m}(z)  
&= \prod_{s=1}^{n} \prod_{t=0}^{\mu_1} \dfrac{{}_s[2n+4+|2m-2p|+2s+4t]_{(m+p)}{}_s[2n-|2m-2p|-2s-4t]_{(m+p)}}{{}_s[2+|2m-2p|+2s+4t]_{(m+p)}{}_s[4n+2-|2m-2p|-2s-4t]_{(m+p)}}.
\end{align*}
\end{proposition}

\begin{proposition}
Let $\g$ be of type $D_n^{(1)}$.
For $1 \le k, l <n-1$ and $m,p \ge 1$, set $\mu \seteq \min(p,m)-1$, and we have
\begin{align*}
a_{l^p,k^m}(z)  & = \hspace{-1ex} \prod_{s=1}^{\min(k,l)} \hspace{-.3ex}  \prod_{t=0}^{\mu} \hspace{-1ex} \left( \dfrac{[k+l-\abs{m-p}-2(s+t))][2n-2+ |k-l|+\abs{m-p}+2(s+t)]}{[|k-l|+\abs{m-p}+2(s+t)][2n-2+k+l-\abs{m-p}-2(s+t)]  } \right.  \\
 & \hspace{6ex} \left. \times \dfrac{[2n-2-|k-l|-\abs{m-p}-2(s+t)][4n-k-l+\abs{m-p}-4+2(s+t))]}{[2n-k-l+\abs{m-p}-2+2(s+t)][4n-4-|k-l|-\abs{m-p}-2(s+t)]} \right),
\allowdisplaybreaks\\
 a_{l^p,n^m}(z) & =  \prod_{s=1}^{l} \prod_{t=0}^{\mu}  \dfrac{[3n-l-3+\abs{p-m}+2(s+t)][n-1+l-\abs{p-m}-2(s+t)]}{[n-l-1+\abs{p-m}+2(s+t)][3n-3+l-\abs{p-m}-2(s+t)]}, \\
 a_{n^p,n^m}(z) &= a_{(n-1)^p,(n-1)^m}(z) = \prod_{t=0}^{p-1}\prod_{s=1}^{\lfloor n/2 \rfloor} \dfrac{[2n+4s+2t-4+|m-p|][2n-4s-2t-|m-p|]}{[4s+2t-2+|m-p|][4n-2-4s-2t-|m-p|]}, \\
 a_{(n-1)^p,n^m}(z)&=\prod_{t=0}^{p-1}\prod_{s=1}^{\lfloor (n-1)/2 \rfloor} \dfrac{[2n+4s+2t+|m-p|-2][2n-4s-2t-|m-p|-2]}{[4s+2t+|m-p|][4n-4s-2t-|m-p|-4]}.
\end{align*}
\end{proposition}

\subsection{Denominator formulas} \label{subsec: denom result}

In this subsection, we will present our main result on the denominator formulas $d_{k^m,l^p}(z)$.
Note that when $\max(m,p)=1$, the result is known in~\cite{AK97,DO94,KKK15,Oh14R,OhS19}.
For the rest of this paper, we will deal with the quantum affine algebras.  

\begin{theorem}
\label{thm:denominators_untwisted}
For untwisted affine types $A_n^{(1)}$, $B_n^{(1)}$, and $D_n^{(1)}$, we have
\begin{align}\label{eq: normal form}
    d_{k^m,l^p}(z) = \prod_{t=0}^{ \min(m,p) -1 } d_{k,l} \bigl( (-q)^{-\abs{p-m}-2t}z \bigr) \hspace{7em} (k,l \in I_0)
\end{align}
unless $\g=B^{(1)}_{n}$ and $\max(k,l)=n$.
In that case, take $l < k = n$, and we have
\begin{align*}
d^{B^{(1)}_{n}}_{l^p,n^m}(z)&= \displaystyle\prod_{t=0}^{\min(2p,m)-1}\prod_{s=1}^{l} (z-(-1)^{n+l+p+m}(\qs)^{2n-2l-2+|2p-m|+4s+2t}), 
\allowdisplaybreaks \\
d^{B^{(1)}_{n}}_{n^p,n^m}(z) & = \prod_{t=0}^{ \min(m,p) -1 } d_{n,n} \bigl( (-\qs)^{-\abs{p-m}-2t}z \bigr).
\end{align*}
For affine type $C_n^{(1)}$, $\max(m,p)>1$, and $1 \leq k,l < n$, we have
\begin{align*}
d^{C^{(1)}_{n}}_{k^{2m},l^p}(z) & = \hspace{-2ex}  \displaystyle  \prod_{t=0}^{\min(2m,p)-1} \prod_{s=1}^{\min(k,l)}  (z-(-\qs)^{|k-l|+|2m-p|+2s+2t})(z-(-\qs)^{2n+2-k-l+|2m-p|+2s+2t} ), \allowdisplaybreaks \\
d^{C^{(1)}_{n}}_{l^p,n^m}(z) & = d^{C^{(1)}_{n}}_{n^m,l^p}(z) =   \displaystyle\prod_{t=0}^{\min(p,2m)-1}  \prod_{s=1}^{l} (z-(-1)^{n+p+l+m}\qs^{n+1-l+|2m-p|+2s+2t}), \allowdisplaybreaks \\
d^{C^{(1)}_{n}}_{n^p,n^m}(z) & =  \prod_{t=0}^{\min(p,m)-1} \prod_{s=1}^{n} (z-(-1)^{m+p}\qs^{2+|2m-2p|+2s+4t}).
\end{align*}
while we have
\begin{align*}
d^{C^{(1)}_{n}}_{k,l}(z)=   \hspace{-2ex} \displaystyle  \prod_{s=1}^{ \min(k,l,n-k,n-l)}  \hspace{-3ex}
\big(z-(-\qs)^{|k-l|+2s}\big) \prod_{s=1}^{ \min(k,l)}  \big(z-(-\qs)^{2n+2-k-l+2s}\big) \qquad (k,l \in I_0).
\end{align*}
Furthermore, if {\rm (i)} $k \ne l <n$ or {\rm (ii)} $\min(k,l)=1$ and $k,l<n$, we have
$$d^{C^{(1)}_{n}}_{k^m,l^p}(z) = \displaystyle \prod_{t=0}^{\min(m,p)-1} \prod_{s=1}^{\min(k,l)}  (z-(-\qs)^{|k-l|+\abs{m-p}+2(s+t)})(z-(-\qs)^{2n+2-k-l+\abs{m-p}+2(s+t)} ).$$ 
\end{theorem}

\begin{conjecture} \label{conj: denom BC}
Let $m,p \in 2\Z_{\ge1}+1$,  
$k  <n$ and $\g=C^{(1)}_{n}$. Then we have the following:
$$d_{k^m,k^p}(z) = \displaystyle \prod_{t=0}^{\min(m,p)-1} \prod_{s=1}^{k}  (z-(-\qs)^{\abs{m-p}+2(s+t)})(z-(-\qs)^{2n+2-2k+\abs{m-p}+2(s+t)} ).$$ 
\end{conjecture}

\begin{theorem} \label{thm:denominators_twisted}
For untwisted affine types $A_{N-1}^{(2)}$, $D_{n+1}^{(2)}$, and $D_4^{(3)}$, we have
$$d_{k^m,l^p}(z) = \prod_{t=0}^{ \min(m,p) -1 } d_{k,l} \bigl( (-q)^{-\abs{p-m}-2t}z \bigr)  \hspace{7em} (k,l \in I_0)$$
unless $\g=D^{(2)}_{n+1}$ and $\max(k,l)=n$.
In that case, take $l < k = n$, and we have
\begin{align*}
d^{D^{(2)}_{n+1}}_{l^p,n^m}(z) & =  d^{D^{(2)}_{n+1}}_{n^m,l^p}(z) = \prod_{t=0}^{\min(p,m)-1}\prod_{s=1}^{l} (z^2+(-1)^{n+l+p+m}q^{n-l+\abs{p-m}+2(s+t)}), \allowdisplaybreaks \\
d^{D^{(2)}_{n+1}}_{n^p,n^m}(z) & = \prod_{t=0}^{\min(p,m)-1} \prod_{s=1}^{n} (z+(-1)^{s+t+p+m}q^{2s+2t+\abs{p-m}} ).
\end{align*}
\end{theorem}

\begin{conjecture} \label{conj: denom E}
Equation~\eqref{eq: normal form} holds for types $E_{6,7,8}^{(1)}$ and $E_6^{(2)}$.
\end{conjecture}

\begin{remark} \hfill
\bna
\item
Using~\eqref{eq: aimjl} and the denominator formulas in this subsection, we can also obtain
the universal coefficients in the previous subsection, in which the denominator formulas are conjectural at this moment.
\item  Comparing the denominator formulas $A_{n}^{(1)}$ and $A_{n}^{(2)}$ (resp. $D_{n+1}^{(1)}$ and $D_{n+1}^{(2)}$, $D_{4}^{(1)}$ and $D_{4}^{(3)}$),
one can see that $\de( \frakR^{(1)}[a_1,b_1],\frakR^{(1)}[a_2,b_2])=\de( \frakR^{(t)}[a_1,b_1],\frakR^{(t)}[a_2,b_2])$
for $i$-boxes $[a_u,b_u]$ $(u=1,2)$.
\ee
\end{remark} 

Note that our one of the main tools will be Proposition~\ref{prop: aMN}, which means we will be computing formulas of the form
\[
\frac{f(z)}{d_{M,N}(z)} \in \ko[z^{\pm1}],
\qquad \text{ and } \qquad
\frac{d_{M,N}(z)}{g(z)} \in \ko[z^{\pm1}],
\]
where $f,g \in \ko[z]$ split (i.e., equal a product of linear factors), to determine $d_{M,N}(z)$.
If $\rho$ is a root of $f(z)$ (resp.\ $g(z)$) with multiplicity $m_{\rho}$, then the multiplicity of the root $\rho$ in $d_{M,N}(z)$ is bounded above (resp.\ below) by $m_{\rho}$.
We say there is an \defn{ambiguity} for the root $\rho$ if the multiplicity of $\rho$ has not been uniquely determined.
Our proof of Theorem~\ref{thm:denominators_twisted} relies on the generalized Schur--Weyl duality functors which will be explained.

Note that the denominator formula $d_{l^p,k^m}(z)$ for all simply-laced affine types (among others) is presented uniformly and based solely on the $d_{l,k}(z)$ values.

The next section is devoted to proving these results. 
We will give full details in type $A_{n-1}^{(1)}$, and then we will proceed by skipping some the intermediate computations as the additional calculations, while lengthy and slightly technical, are all similar and straightforward.

\begin{remark}
\label{rem:small_case_C}  
We remark here that, for Conjecture~\ref{conj: denom BC}, it is hard to remove the ambiguities when $\abs{m-p}$ is relatively small, while we do not have such ambiguity when $\abs{m-p}$ is relatively big (see Proposition~\ref{prop: dlpkm C}). 
\end{remark}

\section{Proofs of denominator formulas}
\label{sec:denom_proofs}

In this section, we will prove the denominator formulas.
Throughout this section, we will compute them for each untwisted classical affine type and apply the framework of \cite{AK97,KKK15,Oh14R}.  

\subsection{Proof for type \texorpdfstring{$A^{(1)}_{n-1}$}{An-1(1)}} \label{subsec: A}

Recall the denominator formulas and universal coefficients in~\eqref{eq: akl dkl An-1}
and we have set 
\[
\mu_1 \seteq  \min(k,l,n-k,n-l) \quad \text{ and } \quad  \mu_2 = \min(p,m) - 1.  
\]

Note that, as special cases of Dorey's rule in Theorem~\ref{thm: Dorey}, we have 
\[
\Vkm{k-1}_{(-q)^{-1}} \otimes  \Vkm{1}_{(-q)^{k-1}} \twoheadrightarrow \Vkm{k} \quad  \text{ and } \quad
\Vkm{1}_{(-q)^{1-l}} \otimes \Vkm{l-1}_{-q} \twoheadrightarrow \Vkm{l}
\]
for $2 \le k , l \le n-1$.

We start with the lemma below on $d_{1,k^2}(z) \seteq d_{\Vkm{1},\Vkm{k^2}}$ and $d_{n-1,k^2}(z) \seteq d_{\Vkm{n-1},\Vkm{k^2}}$, for $k \in I_0$ as we require this special case by using $\de$-invariants in \S\ref{subsec: de-theory}:

\begin{lemma} \label{lemma:d1k2_typeA} For $k \in I_0$, we have
\[
d_{1,k^2}(z) = (z - (-q)^{k+2}) \qquad \text{ and } \qquad d_{n-1,k^2}(z) = (z - (-q)^{n-k+2}).
\]
\end{lemma}

\begin{proof}
We show $d_{1,k^2}(z) = (z - (-q)^{k+2})$ as the proof of $d_{n-1,k^2}(z) = (z - (-q)^{n-k+2})$ is similar under the Dynkin diagram automorphism $\vee: i \leftrightarrow n - i$.
Since $d_{1,k}(z)=(z - (-q)^{k+1})$, it is enough to see 
\bnum
\item \label{eq: A 1 k2 case1} $\de(\Vkm{1}, \Vkm{k^2}_{(-q)^{k+2}})  = \de(\Vkm{1}, \Vkm{k}_{(-q)^{k+3}}  \hconv \Vkm{k}_{(-q)^{k+1}})$ and 
\item \label{eq: A 1 k2 case2} $\de(\Vkm{1}, \Vkm{k^2}_{(-q)^{k}})  = \de(\Vkm{1}, \Vkm{k}_{(-q)^{k+1}} \hconv \Vkm{k}_{(-q)^{k-1}})$,
\ee
by Proposition~\ref{prop: de less than equal to}. 

\noindent
\eqref{eq: A 1 k2 case1} Note that $\scrD^{-1} \Vkm{1} \iso \Vkm{n-1}_{(-q)^{-n}}$, 
$$\de(\Vkm{1},\Vkm{k}_{(-q)^{k+3}}) =0 \qtq  \de(\Vkm{n-1}_{(-q)^{-n}},\Vkm{k}_{(-q)^{k+3}}) =0 $$
which can be derived from $d_{k,l}(z)$ in~\eqref{eq: akl dkl An-1}. Hence Lemma~\ref{lem: de=de}~\eqref{it: de=de 2} implies the first assertion  
$$\de(\Vkm{1}, \Vkm{k^2}_{(-q)^{k+2}}) =1.$$
Similarly, we obtain $\de(\Vkm{n-1}, \Vkm{k^2}_{(-q)^{n-k+2}}) =1$.

\noindent
\eqref{eq: A 1 k2 case2}  Note that 
$$
\rch(\Vkm{1})=\range{0} \qtq \exrch_1(\Vkm{k^2}_{(-q)^k})=\range{-2,2k+2}. 
$$
Hence Theorem~\ref{thm: i-box commute} and~\eqref{eq: range commute} say that
$$
\de(\Vkm{1}, \Vkm{k^2}_{(-q)^{k}}) =0
$$
as we desired. Similarly, $\de(\Vkm{n-1}, \Vkm{k^2}_{(-q)^{n-k}}) =0$.
\end{proof}

In proving the following lemmas, we will often invoke Proposition~\ref{prop: aMN} without reference.

\begin{lemma} \label{lem: d11m A}
For any $m \in \Z_{\ge 1}$, we have
\begin{align}\label{eq: d11m A}
 d_{1,1^m}(z) = (z-(-q)^{m+1}) \quad \text{ and } \quad d_{1,(n-1)^m}(z) = (z-(-q)^{n+m-1}).
\end{align}
\end{lemma}

\begin{proof}
By fusion rules,
\[
\Vkm{1^{m-1}}_{(-q)} \otimes \Vkm{1}_{(-q)^{1-m}} \twoheadrightarrow \Vkm{1^{m}} \qtq \Vkm{1^{m}} \otimes  \Vkm{n-1}_{(-q)^{n+m-1}} \twoheadrightarrow \Vkm{1^{m-1}}_{(-q)^{-1}}
\]
\eqref{eq: AK eqns} in Proposition~\ref{prop: aMN} says that
\begin{subequations}
\begin{align}   
\dfrac{ d_{1,1^{m-1}}((-q)z)d_{1,1}((-q)^{1-m}z)}{d_{1,1^m}(z)} & \equiv \dfrac{(z - (-q)^{m-1})(z-(-q)^{m+1})}{d_{1,1^m}(z)} \in \ko[z^{\pm 1}], \label{eq: a11mstep1} \\
\dfrac{ d_{1,1}((-q)^{m-1}z)d_{1,1^{m-1}}((-q)^{-1}z)}{d_{1,1^m}(z)} &\equiv \dfrac{(z - (-q)^{3-m})(z-(-q)^{m+1})}{d_{1,1^m}(z)} \in \ko[z^{\pm 1}]
\label{eq: a11mstep1p}
\end{align}
\end{subequations}
by an induction on $m$. Here we remark that the part in~\eqref{eq: AK eqns} consisting of universal coefficients arising from fusion rules are $1$ as in ~\eqref{eq: 11m univ}.\footnote{Throughout this paper, we use this fact frequently without any other mention.}
Then an ambiguity happens at $m=2$ for $(-q)$ by comparing~\eqref{eq: a11mstep1} and~\eqref{eq: a11mstep1p}.
However, we already proved the case in Lemma~\ref{lemma:d1k2_typeA}.\footnote{Throughout this paper, we have to remove such ambiguities, which will be the main obstacle to overcome.}

On the other hand, the homomorphism
\[ 
\Vkm{n-1}_{(-q)^{-n-m+1}}  \otimes \Vkm{1^{m}}  \twoheadrightarrow \Vkm{1^{m-1}}_{(-q)},
\]
obtained from fusion rule by taking dual of $\Vkm{1}_{(-q)^{1-m}}$, says that we have
\[
\dfrac{ d_{1,1^{m}}(z)d_{1,(n-1)}((-q)^{-n-m+1}z)}{d_{1,1^{m-1}}((-q)z)} \times \dfrac{a_{1,1^{m-1}}((-q)z)}{ a_{1,1^{m}}(z)a_{1,(n-1)}((-q)^{-n-m+1}z)} \in \ko[z^{\pm 1}],
\]
by Lemma~\ref{lem:denominators_symmetric} and~\eqref{eq: AK2} in Proposition~\ref{prop: aMN}. Applying the induction, we have
\[
\dfrac{ d_{1,1^{m}}(z)d_{1,(n-1)}((-q)^{-n-m+1}z)}{d_{1,1^{m-1}}((-q)z)} \equiv \dfrac{ d_{1,1^{m}}(z)  (z-(-q)^{2n+m-1})  }{  (z-(-q)^{m-1}) },
\]
and   
\begin{align*}
 \dfrac{a_{1,1^{m-1}}((-q)z)}{ a_{1,1^{m}}(z)a_{1,(n-1)}((-q)^{-n-m+1}z)}
&= \dfrac{[2n-m-1][1-m]}{[-m-1][2n-m+1]} = \dfrac{(z-(-q)^{m-1})}{(z-(-q)^{m+1})},
\end{align*}
by direct computation. 
Thus, we have
\begin{align}\label{eq: a11mstep2}
\dfrac{ d_{1,1^{m}}(z)  (z-(-q)^{2n+m-1})  }{  (z-(-q)^{m-1}) }  \times  \dfrac{(z-(-q)^{m-1})}{(z-(-q)^{m+1})}   =  \dfrac{ d_{1,1^{m}}(z)  (z-(-q)^{2n+m-1})  }{ (z-(-q)^{m+1}) }   \in \ko[z^{\pm 1}].
\end{align}
Hence the first assertion follows from~\eqref{eq: a11mstep1},~\eqref{eq: a11mstep1p} and~\eqref{eq: a11mstep2}. Similarly, the second assertion follows.
\end{proof}

\begin{lemma} \label{lem: k1m A}
For any $1 < k < n-1$ and  $m \in \Z_{\ge 1}$, we have
\begin{align}\label{eq: dklm A}
 d_{k,1^m}(z) = (z-(-q)^{m+k})  \quad \text{ and }  \quad d_{k,(n-1)^m}(z) = (z-(-q)^{n-k+m}) .
\end{align}
\end{lemma}

\begin{proof}
Our assertion for $k=1$ or $n-1$ already is shown.

By an induction on $m$, fusion rule
\[
\Vkm{1^{m-1}}_{(-q)} \otimes \Vkm{1}_{(-q)^{1-m}} \twoheadrightarrow \Vkm{1^{m}},
\]
implies that we have
\begin{align}\label{eq: d1mk step1}
\dfrac{ d_{k,1^{m-1}}((-q)z)d_{k,1}((-q)^{1-m}z)}{d_{k,1^m}(z)} \equiv \dfrac{ (z-(-q)^{m+k-2}) (z-(-q)^{m+k})    }{d_{k,1^m}(z)} \in \ko[z^{\pm 1}].
\end{align}
By Dorey's rule, we have  
\[
\Vkm{k-1}_{(-q)^{-1}} \otimes  \Vkm{1}_{(-q)^{k-1}} \twoheadrightarrow \Vkm{k} \qtq \Vkm{1}_{(-q)^{1-k}} \otimes  \Vkm{k-1}_{(-q)} \twoheadrightarrow \Vkm{k},
\]
Equation~\eqref{eq: AK eqns} in Proposition~\ref{prop: aMN} implies that
\begin{align}   \label{eq: ak1mstep1 AR}
\dfrac{(z-(-q)^{m+k})(z-(-q)^{-m-k+2})}{d_{k,1^m}(z)}  \qtq 
\dfrac{ (z-(-q)^{m+k}) (z-(-q)^{k-m-2}) }{d_{k,1^m}(z)} \in \ko[z^{\pm 1}],
\end{align}
by using induction on $k$ and the universal coefficient formulas. 
Here we use  
\begin{align*}
\dfrac{a_{k,1^m}(z)}{ a_{k-1,1^m}((-q)^{-1}z)a_{1,1^m}((-q)^{k-1}z)}&=\dfrac{[k-2+m][2n-m+k-2]}{[2n+m+k-2][k-2-m]} = \dfrac{(z-(-q)^{-m-k+2})}{(z-(-q)^{m-k+2})},
\allowdisplaybreaks \\ 
\dfrac{a_{k,1^m}(z)}{ a_{1,1^m}((-q)^{1-k}z)a_{k-1,1^m}((-q)z)} &= \dfrac{ [m-k+2][2n-k-m+2]}{[2n-k+m+2][2-m-k]} =\dfrac{(z-(-q)^{k-m-2})}{(z-(-q)^{k+m-2})}.
\end{align*}
By~\eqref{eq: ak1mstep1 AR}, an ambiguity happens at $k = 2$ for $(-q)^{-m}$. However, $(-q)^{-m}$ cannot be a root by  \eqref{eq: d1mk step1}.

On the other hand, the homomorphism
\[
\Vkm{1^{m}} \otimes  \Vkm{n-1}_{(-q)^{n+m-1}} \twoheadrightarrow \Vkm{1^{m-1}}_{(-q)^{-1}},
\]
implies that we have
\begin{equation}\label{eq: ak1mstep2}
\begin{aligned}
&\dfrac{ d_{k,1^{m}}(z)d_{k,(n-1)}((-q)^{-n-m+1}z)}{d_{k,1^{m-1}}((-q)z)} \times \dfrac{a_{k,1^{m-1}}((-q)z)}{ a_{k,1^{m}}(z)a_{k,(n-1)}((-q)^{-n-m+1}z)} \\ 
& \hspace{40ex} = \dfrac{ d_{k,1^{m}}(z)(z-(-q)^{2n+m-k})}{z-(-q)^{m+k}}   \in \ko[z^{\pm 1}],
\end{aligned}
\end{equation}
by direct computation. Thus,~\eqref{eq: ak1mstep1 AR} and~\eqref{eq: ak1mstep2} imply the first assertion. Similarly, the second assertion follows.
\end{proof}

The following lemma can be proved by using the similar argument of previous lemma. 
\begin{lemma}
For any $m \in \ZZ_{>0}$ and $1 \le k \le n-1$, we have
\begin{subequations}
\begin{align}
 \label{eq: d1km type A} d_{1,k^m}(z) & = (z-(-q)^{m+k}), \\
 \label{eq: d1n-km type A} d_{1,(n-k)^m}(z) & = d_{n-1,k^m}(z) = (z-(-q)^{n+m-k}).
\end{align}
\end{subequations}
\end{lemma}

\begin{lemma} \label{lem: lkm A}
For any $m \in \ZZ_{>0}$ and $1 \le k,l \le n-1$, we have
\[
 d_{l,k^m}(z) =  d_{n-l,(n-k)^m}(z) = \prod_{s=1}^{\min(k,l,n-k,n-l)}  ( z - (-q)^{|k-l|+m-1+2s} ).
 \]
\end{lemma}

\begin{proof} 
Since $ d_{l,k^m}(z) =  d_{n-l,(n-k)^m}(z)$, we assume that $l \le k$ without loss of generality.

\noindent
(a) Assume $l \le n-k$.
By the homomorphisms,
\[
{\rm (i)} \ \  \Vkm{1}_{(-q)^{1-l}} \otimes \Vkm{l-1}_{-q} \twoheadrightarrow \Vkm{l} \quad \text{ and } \quad
{\rm (ii)} \ \ \Vkm{l} \otimes \Vkm{n-1}_{(-q)^{n-l+1}} \twoheadrightarrow \Vkm{l-1}_{(-q)}
\]
we have
\begin{align*}
&{\rm (i)'} \ \dfrac{ \displaystyle  \prod_{s=1}^{l}  \left( z - (-q)^{k-l+m-1+2s} \right) }{d_{l,k^m}(z)}   
\text{ and }
{\rm (ii)'} \  \dfrac{ d_{l,k^{m}}(z) \times (z-(-q)^{2n+m-l+1-k})}{ \displaystyle\prod_{s=1}^{l}(z - (-q)^{k-l-1+m+2s})} \times (z-(-q)^{k-l-m+1})
\end{align*}
are contained in $\ko[z^{\pm 1}]$, by direct calculations. For ${\rm (i)'}$, we have used the calculation 
$$
 \dfrac{a_{l,k^m}(z)}{a_{1,k^{m}}((-q)^{1-l}z)a_{l-1,k^{m}}((-q)z)} =1.
$$
Then ${\rm (i)'}$ and ${\rm (ii)'}$ imply our assertion since
\[
2n+m-l+1-k \neq k-l-1+m+2s \quad \text{ and } \quad k-l-1+m+2s \neq k-l-m+1
\]
for each $1 \le s \le l \le n-k$.

\noindent
(b) We assume that $l > n-k$.
By the homomorphisms,
\[
{\rm (iii)} \quad \Vkm{1}_{(-q)^{1-l}} \otimes \Vkm{l-1}_{-q} \twoheadrightarrow \Vkm{l} \quad \text{ and } \quad 
{\rm (iv)} \quad \Vkm{l-1}_{(-q)^{-1}} \otimes \Vkm{1}_{-q^{l-1}} \twoheadrightarrow \Vkm{l},
\]
we have
\begin{align*}
& {\rm (iii)'} \quad \dfrac{ (z-(-q)^{m+l-1+k}) \times \displaystyle \prod_{s=1}^{n-k} ( z - (-q)^{k-l+m-1+2s})     }{d_{l,k^m}(z)} \in \ko[z^{\pm 1}],   \\
& {\rm (iv)'} \quad \dfrac{  (z-(-q)^{k-l+m+1})  \times \displaystyle \prod_{s=1}^{n-k} ( z - (-q)^{k-l+m-1+2s})     }{d_{l,k^m}(z)} \in \ko[z^{\pm 1}].
\end{align*} 
Here we have used the following calculations:
$$
\dfrac{a_{l,k^m}(z)}{a_{1,k^{m}}((-q)^{1-l}z)a_{l-1,k^{m}}((-q)z)} = \dfrac{a_{l,k^m}(z)}{a_{1,k^{m}}((-q)^{l-1}z)a_{l-1,k^{m}}((-q)^{-1}z)} =1.
$$
The only ambiguity that can occur by the above equations is at $l=1$, which would contradict our assumption that $l > n - k$.
Finally, the homomorphism
\[
{\rm (v)} \quad \Vkm{k^m} \otimes  \Vkm{n-k}_{(-q)^{n+m-1}} \twoheadrightarrow \Vkm{k^{m-1}}_{(-q)^{-1}}
\]
implies that we have
\begin{align*}
\dfrac{a_{l,k^{m-1}}((-q)z)}{a_{l,k^{m}}(z)a_{l,(n-k)}((-q)^{-n-m+1}z)} =    \dfrac{(z-(-q)^{m+k-l-1})}{(z-(-q)^{2n-k-l+m-1})}
\end{align*}
and hence 
\begin{align*}
{\rm (v)'} \quad  \dfrac{ d_{l,k^{m}}(z) \displaystyle  \prod_{s=1}^{n-k} (z-(-q)^{m+k+l-1+2s})}{\displaystyle \prod_{s=1}^{n-k}(z-(-q)^{k-l+m-1+2s})} \in \ko[z^{\pm1}],
\end{align*}
which implies our assertion since
\[
m+k+l-1+2s' \neq k-l+m-1+2s  
\]
for any $1 \le s,s' \le n-k < l$.
\end{proof}

\begin{lemma} \label{lemma: division A}
For $1\leq k,l \le n-1$ and $m,p \ge 1$,
\[
\text{$d_{l^p,k^m}(z)$ divides $\displaystyle\prod_{s=1}^{\min(k,l,n-k,n-l)}  \prod_{t=0}^{\min(p,m)-1} (z-(-q)^{|k-l|+\abs{p-m}+2(s+t)})$.}
\]
\end{lemma}

\begin{proof}
We assume that $p < m$.
Since $ d_{l^p,k^m}(z) =  d_{(n-l)^p,(n-k)^m}(z)$, we assume that $l \le k$ without loss of generality.
By the homomorphism
\[
\Vkm{l^{p-1}}_{(-q)} \otimes  \Vkm{l}_{(-q)^{1-p}} \twoheadrightarrow \Vkm{l^p},
\]
we have
\[
\dfrac{
d_{l^{p-1},k^m}((-q)z)d_{l,k^m}((-q)^{1-p}z)}{d_{l^p,k^m}(z)}
\in \ko[z^{\pm 1}].
\] 
By an induction on $p$, we have
\[
d_{l^{p-1},k^m}((-q)z) \text{ divides }  \displaystyle \prod_{s=1}^{\min(l,n-k)}  \prod_{t=0}^{p-2} (z-(-q)^{k-l+m-p+2(s+t)}),
\]
and we know
\[
d_{l,k^m}((-q)^{1-p}z) = \prod_{s=1}^{\min(l,n-k)}  \left( z - (-q)^{k-l+m+p-2+2s} \right).
\]
Therefore, our assertion holds for this case. The remaining case can be proved in a similar way.
\end{proof}

\begin{lemma} \label{lemma: dlpkm A}
Let $\mu_1 = \min(l,k,n-k,n-l)$.
For $1\leq k,l \le n-1$, $\min(m,p)>1$ and $\abs{m-p} \geq \mu_1-1$, we have
\begin{align}\label{eq: dlpkm A}
 d_{l^p,k^m}(z) =\prod_{s=1}^{\mu_1}  \prod_{t=0}^{\min(p,m)-1} (z-(-q)^{|k-l|+\abs{p-m}+2(s+t)}).
\end{align}
In particular, there is no ambiguity when $\mu_1=1$. 
\end{lemma}

\begin{proof}  
Since $d_{l^p,k^m}(z) =  d_{(n-l)^p,(n-k)^m}(z)$, we assume that $l \leq k$ without loss of generality.
Note that we can assume that $\min(p,m)=p$ since $d_{l^p,k^m}(z)=d_{k^m,l^p}(z)$.

\medskip\noindent
(a) Now we assume that $ l \le n-k$ and $m-p \ge  l-1$.  
By the homomorphism
\[
\Vkm{l^p}  \otimes  \Vkm{n-l}_{(-q)^{n+p-1}} \twoheadrightarrow \Vkm{l^{p-1}}_{(-q)^{-1}},
\]
we have
\begin{align}   \label{eq: a1pkmstepr1}
\dfrac{
d_{l^p,k^{m}}(z)d_{(n-l),k^{m}}((-q)^{-n-p+1}z)}{d_{l^{p-1},k^m}((-q)z)}
\times \dfrac{a_{l^{p-1},k^m}((-q)z)}{
a_{l^p,k^{m}}(z)a_{(n-l),k^{m}}((-q)^{-n-p+1}z)} \in \ko[z^{\pm 1}].
\end{align}
By the induction on $p$, we have
\[
\dfrac{
d_{l^p,k^{m}}(z)d_{(n-l),k^{m}}((-q)^{-n-p+1}z)}{d_{l^{p-1},k^m}((-q)z)}
= \dfrac{ d_{l^p,k^{m}}(z) \displaystyle \prod_{s=1}^{l}(z -
(-q)^{2n-k-l+m+p-2+2s}) }{  \displaystyle \prod_{s=1}^{l}
\prod_{t=0}^{p-2} (z-(-q)^{k-l+m-p+2(s+t)}) }
\]
and, by direct computation, we have
\begin{align*}
& \dfrac{a_{l^{p-1},k^m}((-q)z)}{ a_{l^p,k^{m}}(z)a_{(n-l),k^{m}}((-q)^{-n-p+1}z)}  = \prod_{s=1}^{l} \dfrac{(z-(-q)^{k+l-m+p-2s})}{(z-(-q)^{k-l+m+p-2+2s})}.
\end{align*}
Hence~\eqref{eq: a1pkmstepr1} can be re-written as follows:
\begin{equation} \label{eq: m-p ge l A}
\dfrac{ d_{l^p,k^{m}}(z) \times \displaystyle \prod_{s=1}^{l}(z - (-q)^{2n-k-l+m+p-2+2s})(z-(-q)^{k-l-m+p-2+2s}) }{  \displaystyle \prod_{s=1}^{l}  \prod_{t=0}^{p-1} (z-(-q)^{k-l+m-p+2(s+t)}) } \in \ko[z^{\pm 1}].
\end{equation}
Note that
\[
n - k + p - 1 > s - s' + t \qquad \text{ and } \qquad p - m < s - s' + t
\]
since $-l < (s-s') < l \leq n-k$, the assumption $m - p  \geq  \mu_1 -1 = l-1$, and $0 \leq t \le p - 1$.
Thus~\eqref{eq: m-p ge l A} implies that
\begin{align*}
\dfrac{ d_{l^p,k^{m}}(z)}{ \displaystyle \prod_{s=1}^{l}  \prod_{t=0}^{p-1} (z-(-q)^{k-l+m-p+2(s+t)}) } \in \ko[z^{\pm 1}],
\end{align*}
which implies our assertion with Lemma~\ref{lemma: division A}.

\medskip\noindent
(b) Now we assume that $ l > n-k$.
By the homomorphism
\[
\Vkm{l^p}  \otimes  \Vkm{n-l}_{(-q)^{n+p-1}} \twoheadrightarrow \Vkm{l^{p-1}}_{(-q)^{-1}},
\]
we have
\begin{align}   \label{eq: a1pkmstepr1b}
\dfrac{ d_{l^p,k^{m}}(z)d_{(n-l),k^{m}}((-q)^{-n-p+1}z)}{d_{l^{p-1},k^m}((-q)z)} \times
\dfrac{a_{l^{p-1},k^m}((-q)z)}{ a_{l^p,k^{m}}(z)a_{(n-l),k^{m}}((-q)^{-n-p+1}z)} \in \ko[z^{\pm 1}].
\end{align}
By the induction on $p$, we have
\[
\dfrac{ d_{l^p,k^{m}}(z)d_{(n-l),k^{m}}((-q)^{-n-p+1}z)}{d_{l^{p-1},k^m}((-q)z)} =
\dfrac{ d_{l^p,k^{m}}(z) \times \displaystyle \prod_{s=1}^{n-k} (z - (-q)^{l+k+p+m-2+2s})  }{  \displaystyle \prod_{s=1}^{n-k}  \prod_{t=0}^{p-2} (z-(-q)^{k-l+m-p+2(s+t)}) }
\]
and, by direct computation, we have
\begin{align*}
& \dfrac{a_{l^{p-1},k^m}((-q)z)}{ a_{l^p,k^{m}}(z)a_{(n-l),k^{m}}((-q)^{-n-p+1}z)} = \prod_{s=1}^{n-k} \dfrac{(z-(-q)^{2n-k-l+p-m-2s})}{(z-(-q)^{k-l+m+p-2+2s})}.
\end{align*}
Hence \eqref{eq: a1pkmstepr1b} can be re-written as follows:
\begin{equation}\label{eq: |m-p| ge n-k}
\dfrac{ d_{l^p,k^{m}}(z) \times \displaystyle \prod_{s=1}^{n-k} (z - (-q)^{l+k+p+m-2+2s}) \times (z-(-q)^{2n-k-l+p-m-2s}) }{  \displaystyle \prod_{s=1}^{n-k}  \prod_{t=0}^{p-1} (z-(-q)^{k-l+m-p+2(s+t)}) } \in \ko[z^{\pm 1}].
\end{equation}
By writing $m =p+ n-k +u$ for some $u \ge -1$, we have
\begin{equation}\label{eq: |m-p| ge n-k again}
\begin{aligned}
& \dfrac{ d_{l^p,k^{m}}(z) \times \displaystyle \prod_{s=1}^{n-k} (z - (-q)^{l+2p+n+u-2+2s}) \times (z-(-q)^{2k-l-n-u-2+2s}) }{  \displaystyle \prod_{s=1}^{n-k}  \prod_{t=0}^{p-1} (z-(-q)^{-l+n+u+2(s+t)}) }
\in \ko[z^{\pm 1}]
\end{aligned}
\end{equation}
Since 
\bitem
\item $ l+2p+n+u-2+2s' \ne -l+n+u+2(s+t) \iff l+p-1  \ne  s-s'+t$, 
\item $2k-l-n-u-2+2s' \ne -l+n+u+2(s+t) \iff k -n-u-1  \ne    s-s'+t$
\eitem 
for $1 \le s,s' \le n-k <l$,  $u \ge -1$  and $0 \le t \le p-1$,  \eqref{eq: |m-p| ge n-k} is equivalent to
\[
 \dfrac{ d_{l^p,k^{m}}(z)  }{  \displaystyle \prod_{s=1}^{n-k}  \prod_{t=0}^{p-1} (z-(-q)^{k-l+m-p+2(s+t)}) },
\]
which implies our assertion with Lemma~\ref{lemma: division A}.
\end{proof}

\begin{proof}[Proof of {\rm Theorem~\ref{thm:denominators_untwisted}} for type $A_n^{(1)}$]
Without loss of generality, assume $l \leq k$ and $p \leq m$. Also, by Lemma~\ref{lemma: dlpkm A}, it suffices to consider when
$m-p < \min(k,l,n-k,n-l)-1$. We also assume that $\ell \le n - k$. Now we shall prove
$$
d_{l^p,k^m}(z) =\prod_{s=1}^{\mu_1}  \prod_{t=0}^{p-1} (z-(-q)^{k-l+m-p+2(s+t)})
$$
without restriction, as we desired.

From the homomorphism in Theorem~\ref{thm: Higher Dorey I} with the restriction $\min(k,l)=1$\footnote{In this section, we would like to emphasize the restriction, which will be \emph{removed} in the later section with the results in this section.} 
\[
\Vkm{1^m}_{(-q)^{-(k+1)}} \otimes \Vkm{k^m} \twoheadrightarrow \Vkm{(k+1)^m}_{(-q)^{-1}},
\]
we have
\[
\dfrac{d_{l^p,1^m}( (-q)^{k+1} z ) d_{l^p,k^m}(z)}{d_{l^p,(k+1)^m}((-q) z)} \times \dfrac{a_{l^p,(k+1)^m}((-q) z)}{a_{l^p,1^m}( (-q)^{k+1} z ) a_{l^p,k^m}(z)} \in \ko[z^{\pm1}].
\]
By direct computation, we have
\begin{align*}
\dfrac{a_{l^p,(k+1)^m}((-q) z)}{a_{l^p,1^m}( (-q)^{k+1} z ) a_{l^p,k^m}(z)} = \prod_{t=0}^{p-1} \dfrac{(z - (-q)^{l-k-m+p-2t-2})}{(z - (-q)^{l-k+m-p+2t})}.
\end{align*}
Furthermore, by descending induction on $k$, we have
\[
\dfrac{d_{l^p,1^m}((-q)^{k+1} z) d_{l^p,k^m}(z)}{d_{l^p,(k+1)^m}((-q) z)}
 = \dfrac{d_{l^p,k^m}(z) \displaystyle \prod_{t=0}^{p-1} (z-(-q)^{l-k+m-p+2t})}{\displaystyle \prod_{s=1}^l \prod_{t=0}^{p-1} (z-(-q)^{k-l+m-p+2(s+t)})},
\]
since we know $d_{l^p,(k+1)^m}(z)$ and  $d_{l^p,1^m}(z)$ completely. Therefore, we have
\begin{align*}
\dfrac{d_{l^p,k^m}(z)}{\displaystyle \prod_{s=1}^{l} \prod_{t=0}^{p-1} (z-(-q)^{k-l+m-p+2(s+t)})} \times \prod_{t=0}^{p-1} (z - (-q)^{l-k-m+p-2t-2}) \in \ko[z^{\pm1}].
\end{align*}
We note that
\[
k-l+m-p+2(s+t) \ne  l-k-m+p-2t'-2
\]
can never occur since $k \geq l$ and $m \geq p$ and $s \geq 1$ and $t,t' \geq 0$. Therefore, our assertion claim holds by Lemma~\ref{lemma: division A}.
The remaining cases also hold by the similar argument.
\end{proof}

\begin{corollary} \label{cor: root module A}
For  $m < n$, the KR modules $\Vkm{1^m}_x$ and $\Vkm{(n-1)^m}_x$ $(x \in \bfk^{\times})$ are root modules.
\end{corollary}

\begin{definition}
We call a KR module $M$ is \defn{plain} if $\de(M,N) \le 1$ for every KR module $N$ in $\Cg$.   
\end{definition}

\begin{corollary} \label{cor: plain A} 
Let $x \in \bfk^{\times}$.
\bna
\item For \emph{any} $m \ge 1$,  the KR module $M = \Vkm{1^m}_x$ $($resp. $\Vkm{(n-1)^m}_x)$  is plain. 
\item For any $1\le k <n$, the fundamental module $M = \Vkm{k}_x$ is plain.
\ee 
Hence the composition length of $M \tens N$ for any KR module $N$ in $\Cg$ is less than or equal to $2$.
\end{corollary}

\begin{remark} \label{rmk: non KR root}
Combining Corollary~\ref{cor: root module A} with the denominator formulas, we can construct other root modules that, in general, are not KR modules.
For instance, the denominator formula 
\[
d^{A_4^{(1)}}_{1^4,2}(z)=d^{A_4^{(1)}}_{4^4,2}(z) = (z-q^6)  
\]
says that 
\[
\de\big(\Vkm{1^4}, \Vkm{2}_{(-q)^6}\big) =1 \qtq \de\big(\scrD^{\pm k} \Vkm{1^4}, \Vkm{2}_{(-q)^6}\big) = 0 \quad (k >0).
\]
Thus $\big(\Vkm{1^4}, \Vkm{2}_{(-q)^6}\big)$ is a $\mathfrak{sl}_3$-pair.
Hence $\Vkm{1^4} \hconv \Vkm{2}_{(-q)^6}$ and $\Vkm{2}_{(-q)^6} \hconv \Vkm{1^4}$ are root modules by Lemma~\ref{Lem: two root modules}.    
Similar constructions are available for other types as well (see Lemmas~\ref{cor: root module B}, \ref{cor: root module C}, and \ref{cor: root module D} below). 
\end{remark}

\subsection{Proof for type \texorpdfstring{$B_n^{(1)}$}{Bn(1)}}

In this subsection, we first prove our assertion on $d_{k^m,l^p}(z)$ for $1 \le k, l <n$, which we can apply the
similar argument in Section~\ref{subsec: A}.
Then we will prove our assertion on $d_{l^m,n^p}(z)$ for $1 \le  l \le n$. 
Recall that $q_n=\qs$ and $\qs^2=q$.

\subsubsection{$d_{k^m,l^p}(z)$ for $1 \le k, l <n$} 
Recall that
\begin{equation}
\begin{aligned}\label{eq: B akl dkl}
d_{k,l}(z) &=  \displaystyle \prod_{s=1}^{\min (k,l)} \big(z-(-q)^{|k-l|+2s}\big)\big(z+(-q)^{2n-k-l-1+2s}\big), \\
a_{k,l}(z) & \equiv  \dfrac{  [|k-l|]\PA{2n+k+l-1}\PA{2n-k-l-1}[4n-|k-l|-2]}{[k+l]\PA{2n+k-l-1}\PA{2n-k+l-1}[4n-k-l-2]}.
\end{aligned}
\end{equation}

Note also that we have surjective homomorphisms
\[
\Vkm{k-1}_{(-q)^{-1}} \otimes  \Vkm{1}_{(-q)^{k-1}} \twoheadrightarrow \Vkm{k} \quad  \text{ and } \quad
\Vkm{1}_{(-q)^{1-l}} \otimes \Vkm{l-1}_{-q} \twoheadrightarrow \Vkm{l}
\]
for all $1< k,l <n$, as special cases in Theorem~\ref{thm: Dorey}.
 
As Lemma~\ref{lemma:d1k2_typeA}, we prove the lemma below by using $\de$-invariants in \S\ref{subsec: de-theory}.

\begin{lemma} \label{lem: d1k2 B}
For $1 \le k \le n-1$, we have
\begin{align}\label{eq: d1k2 B}
  d_{1,k^2}(z)=(z-(-q)^{k+2}) (z+(-q)^{2n-k+1}).
\end{align}
\end{lemma}

\begin{proof}
From~\eqref{eq: B akl dkl},  we know $ d_{1,k}(z) =  (z-(-q)^{k+1}) (z+(-q)^{2n-k})$. 
Hence it is enough to show that
\bna
\item \label{it: ak B} $\de(\Vkm{1}, \Vkm{k^2}_{(-q)^k}) = \de(\Vkm{1}, \Vkm{k}_{(-q)^{k+1}} \hconv  \Vkm{k}_{(-q)^{k-1}})=0$, 
\item \label{it: bk B} $\de(\Vkm{1}, \Vkm{k^2}_{(-q)^{k+2}}) =\de(\Vkm{1}, \Vkm{k}_{(-q)^{k+3}} \hconv  \Vkm{k}_{(-q)^{k+1}})=1$,
\item \label{it: ck B}  $\de(\Vkm{1}, \Vkm{k^2}_{-(-q)^{2n-k-1}}) = \de(\Vkm{1}, \Vkm{k}_{-(-q)^{2n-k}} \hconv  \Vkm{k}_{-(-q)^{2n-k-2}} )=0$, 
\item \label{it: dk B}  $\de(\Vkm{1}, \Vkm{k^2}_{-(-q)^{2n-k+1}})=\de(\Vkm{1}, \Vkm{k}_{-(-q)^{2n-k+2}} \hconv  \Vkm{k}_{-(-q)^{2n-k}} )=1$,
\ee
by Proposition~\ref{prop: de less than equal to}. 
Note that  
$$\clr(\Vkm{k^2}_{(-q)^s}) \ne \clr(\Vkm{k^2}_{-(-q)^s}) = 2n-k$$ and we assume $\clr(\Vkm{k^2}_{(-q)^s}) = k$ and $\clr(\Vkm{k^2}_{-(-q)^s}) = 2n-k$ for $s \in \Z$.

\medskip \noindent
\eqref{it: ak B} Note that
$$
\rch(\Vkm{1}) = \range{0} \qtq \exrch_1(\Vkm{k^2}_{(-q)^k}) = \range{-4,4k+4}, 
$$
since $d_{i}=2$ in~\eqref{eq: D and di} for $i <n$. Thus~\eqref{eq: range commute} implies~\eqref{it: ak B}. 
Similarly,~\eqref{it: ck B} holds, since $\clr(\Vkm{k^2}_{-(-q)^{2n-k-1}}) = 2n-k$.

\medskip \noindent
\eqref{it: bk B}  Note that the sequences 
$$(\Vkm{1}, \Vkm{k}_{(-q)^{k+3}} , \Vkm{k}_{(-q)^{k+1}}) \qtq (\Vkm{k}_{(-q)^{k+3}} , \Vkm{k}_{(-q)^{k+1}},\Vkm{1})$$
are normal, since 
$$\de(\Vkm{1}, \Vkm{k}_{(-q)^{k+3}} )=0 \qtq \de(\Vkm{k}_{(-q)^{k+3}}, \scrD^{-1}\Vkm{1}) =0,$$ 
respectively. Hence
\begin{align*}
\de(\Vkm{1}, \Vkm{k}_{(-q)^{k+3}} \hconv  \Vkm{k}_{(-q)^{k+1}} ) =\de(\Vkm{1}, \Vkm{k}_{(-q)^{k+3}} ) + \de(\Vkm{1},    \Vkm{k}_{(-q)^{k+1}} ) =1,    
\end{align*}
by Lemma~\ref{lem: normal seq d} {\rm (v)}. By the same argument,  \eqref{it: dk B} also holds.
\end{proof}

\begin{lemma} \label{lem: d11m B}
For any $m \in \ZZ_{>0}$, we have
\begin{align}\label{eq: d11m B}
 d_{1,1^m}(z) =  (z-(-q)^{m+1}) (z+(-q)^{2n+m-2}).
\end{align}
\end{lemma}

\begin{proof}
By fusion rule,
\[
\Vkm{1^{m-1}}_{(-q)} \otimes \Vkm{1}_{(-q)^{1-m}} \twoheadrightarrow \Vkm{1^m} \quad \text{ and } \quad \Vkm{1}_{(-q)^{m-1}} \otimes \Vkm{1^{m-1}}_{(-q)^{-1}} \twoheadrightarrow \Vkm{1^m}
\]
we have
\begin{subequations} \label{eq: B11 step} 
\begin{align} 
&\dfrac{(z-(-q)^{m-1}) (z+(-q)^{2n+m-4})  (z-(-q)^{m+1}) (z+(-q)^{2n+m-2}) }{d_{1,1^m}(z)} \in \ko[z^{\pm 1}],  \\
& \dfrac{   (z-(-q)^{3-m}) (z+(-q)^{2n-m})   (z-(-q)^{m+1}) (z+(-q)^{2n+m-2}) }{d_{1,1^m}(z)} \in \ko[z^{\pm 1}].  
\end{align}
\end{subequations}
Then an ambiguity happens at $m=2$ for $(-q)$ and $-(-q)^{2n-2}$ by comparing equations in~\eqref{eq: B11 step}.
However, the ambiguity is already covered by Lemma~\ref{lem: d1k2 B}.

On the other hand, the homomorphism
\[
\Vkm{1^m} \otimes \Vkm{1}_{-(-q)^{2n+m-2}} \twoheadrightarrow \Vkm{1^{m-1}}_{(-q)^{-1}},
\]
yields
\begin{align} \label{eq: B11 step2} 
\dfrac{ d_{1,1^{m}}(z) (z-(-q)^{4n+m-3})(z+(-q)^{2n+m})}{(z-(-q)^{m+1})(z+(-q)^{2n+m-2})}  \in \ko[z^{\pm 1}].
\end{align}
by induction and direct computations.
Then our assertion follows with~\eqref{eq: B11 step} and~\eqref{eq: B11 step2}.
\end{proof}

\begin{lemma}\label{lem: dk1m B}  
For any $1 < k < n$  and  $m \in \ZZ_{>0}$, we have
\[
 d_{k,1^m}(z) =  (z-(-q)^{k+m})(z+(-q)^{2n-k+m-1}) .
\]
\end{lemma}

\begin{proof}
By Dorey's rule in \S\ref{subsec: lower Dorey}
\[
\Vkm{k-1}_{(-q)^{-1}} \otimes  \Vkm{1}_{(-q)^{k-1}} \twoheadrightarrow \Vkm{k} \qtq \Vkm{1}_{(-q)^{1-k}} \otimes  \Vkm{k-1}_{(-q)} \twoheadrightarrow \Vkm{k}
\]
we have
\begin{subequations} \label{eq: ak1mstep1 B}
\begin{align}
&\dfrac{ (z-(-q)^{k+m})(z+(-q)^{2n-k+m+1})  (z-(-q)^{-k-m+2}) (z+(-q)^{2n+m-k-1})}{d_{k,1^m}(z)} \in \ko[z^{\pm 1}], \\
&\dfrac{ (z-(-q)^{m+k}) (z+(-q)^{2n+m+k-3})  (z-(-q)^{k-m-2})(z+(-q)^{2n-k+m-1}) }{d_{k,1^m}(z)} \in \ko[z^{\pm 1}],
\end{align}
\end{subequations}
by the induction on $m$ and direct computations. 
By~\eqref{eq: ak1mstep1 B}, the ambiguity happens at $k=2$. In that case, {\rm (i)} $(-q)^{-m}$ and {\rm (ii)} $-(-q)^{2n+m-1}$ might be roots of $d_{k,1^m}(z)$.
However, {\rm (i)}   
Note that the fusion rule 
\[
\Vkm{1^{m-1}}_{(-q)} \otimes \Vkm{1}_{(-q)^{1-m}} \twoheadrightarrow \Vkm{1^m} \quad \text{ and } \quad \Vkm{1}_{(-q)^{m-1}} \otimes \Vkm{1^{m-1}}_{(-q)^{-1}} \twoheadrightarrow \Vkm{1^m}
\]
induces 
\begin{align*}
&\dfrac{  (z-(-q)^{m})(z+(-q)^{2n+m-5})  (z-(-q)^{m+2})(z+(-q)^{2n+m-3}) }{d_{k,1^m}(z)}   \in \bfk[z^{\pm1}], \\
&\dfrac{ (z-(-q)^{k+m})(z+(-q)^{2n-k+m-1}) (z-(-q)^{4-m})(z+(-q)^{2n+m-1})   }{d_{k,1^m}(z)}    \in \bfk[z^{\pm1}],
\end{align*}
by the induction hypothesis on $m$. 
Thus $(-q)^{-m}$ can not be a root of $ d_{2,1^m}(z)$.
{\rm (ii)} Since
$$
\Vkm{1^m}_{-(-q)^{2n+m-1}} \iso \head(\Vkm{1}_{-(-q)^{2n+2m-2}} \tens \cdots \tens \Vkm{1}_{-(-q)^{2n}})
$$
and 
$$\de(\Vkm{2},\Vkm{1}_{-(-q)^{2n+2s}}) =0 \text{ for } 0 \le s \le m-1 \text{ as }d_{1,2}(z)=(z-(-q)^3)(z+(-q)^{2n-2}),$$
$-(-q)^{2n+m-1}$ can not be a root of $d_{2,1^m}(z)$ by Proposition~\ref{prop: de less than equal to}, either.

On the other hand, the homomorphism
\[
 \Vkm{1^m} \otimes \Vkm{1}_{-(-q)^{2n+m-2}} \twoheadrightarrow \Vkm{1^{m-1}}_{(-q)^{-1}}
\]
we obtain
\begin{align} \label{eq: ak1mstep2 B}
\dfrac{ d_{k,1^{m}}(z) (z-(-q)^{4n+m-k-2})(z+(-q)^{2n+m+k-1})  }{(z-(-q)^{m+k})(z+(-q)^{2n+m-k-1})}  \in \ko[z^{\pm 1}],
\end{align}
by induction and direct computations. Thus our assertion follows from~\eqref{eq: ak1mstep1 B} and~\eqref{eq: ak1mstep2 B}
\end{proof}

\begin{lemma}\label{lem: d1km B}
For any $m \in \ZZ_{>0}$ and $1 \le k \le n-1$, we have
\begin{align*} 
 d_{1,k^m}(z)=(z-(-q)^{k+m}) (z+(-q)^{2n-k+m-1}).
\end{align*}
\end{lemma}

\begin{proof} It is enough to consider when $k \ge 2$. 
By fusion rule
\[
\Vkm{k^{m-1}}_{(-q)} \otimes  \Vkm{k}_{(-q)^{1-m}} \twoheadrightarrow \Vkm{k^m} \text{ and }
\Vkm{k}_{(-q)^{m-1}} \otimes  \Vkm{k^{m-1}}_{(-q)^{-1}} \twoheadrightarrow \Vkm{k^m},
\]
we have
\begin{subequations}
\label{eq: a1kmstep1 BA}
\begin{align}
&  \dfrac{   (z-(-q)^{k+m-2}) (z+(-q)^{2n-k+m-3}) (z-(-q)^{k+m}) (z+(-q)^{2n-k+m-1}) }{d_{1,k^m}(z)} \in \ko[z^{\pm1}],  \\
& \dfrac{  (z-(-q)^{k-m+2}) (z+(-q)^{2n-m-k+1}) (z-(-q)^{k+m}) (z+(-q)^{2n-k+m-1})   }{d_{1,k^m}(z)} \in \ko[z^{\pm1}].  
\end{align}
\end{subequations}
Then an ambiguity happens at $m=2$ for $(-q)^k$ and $-(-q)^{2n-k-1}$ by comparing equations in~\eqref{eq: a1kmstep1 BA}.
However, such ambiguity does not happen by Lemma~\ref{lem: d1k2 B}. 

On the other hand, the homomorphism
\[
\Vkm{k^m} \otimes \Vkm{k}_{-(-q)^{2n+m-2}} \twoheadrightarrow \Vkm{k^{m-1}}_{(-q)^{-1}},
\]
yields
\begin{align} \label{eq: a1kmstep1 BA2}
 \dfrac{ d_{1,k^{m}}(z)(z-(-q)^{4n+m-k-2})(z+(-q)^{2n+m+k-1}) }{(z-(-q)^{m+k})(z+(-q)^{2n+m-k-1})} \in \ko[z^{\pm 1}].
\end{align}
by induction and direct computations. Hence our assertion follows from~\eqref{eq: a1kmstep1 BA} and~\eqref{eq: a1kmstep1 BA2}.
\end{proof}

\begin{proposition} \label{prop: B dlkm}
For any $m \in \Z_{\ge 1}$ and $1 \le k,l \le n-1$, we have
\begin{align}\label{eq: dlkm}
 d_{l,k^m}(z) =   \prod_{s=1}^{\min(k,l)} (z-(-q)^{|k-l|+m-1+2s)}) (z+(-q)^{2n-k-l+m-2+2s)}).
\end{align}
\end{proposition}

\begin{proof}
(a) Let us assume first that $l \le k$.
By the homomorphism
\[
\Vkm{l-1}_{(-q)^{-1}} \otimes  \Vkm{1}_{(-q)^{l-1}} \twoheadrightarrow \Vkm{l},
\]
we have
\begin{equation} \label{eq: alkmstep1 B}
\begin{aligned}   
&   \dfrac{\displaystyle  \prod_{s=1}^{l} (z-(-q)^{k-l+m-1+2s}) (z+(-q)^{2n-k-l+m-2+2s}) }{d_{l,k^m}(z)}  \in \bfk[z^{\pm1}].
\end{aligned}
\end{equation}
Here we have used
$$ \dfrac{a_{l,k^m}(z)}{a_{l-1,k^m}( (-q)^{-1} z ) a_{1,k^m}( (-q)^{l-1}z)}=1.
$$

By the homomorphism
\[
 \Vkm{k^m} \otimes \Vkm{k}_{-(-q)^{2n+m-2}} \twoheadrightarrow \Vkm{k^{m-1}}_{(-q)^{-1}}
\]
we have
\begin{align*}
& \dfrac{ d_{l,k^m}(z)d_{l,k}(-(-q)^{-2n-m+2}z)}{d_{l,k^{m-1}}((-q)z)} \times \dfrac{a_{l,k^{m-1}}((-q)z)}{a_{l,k^m}(z)a_{l,k}(-(-q)^{-2n-m+2}z)} \in \ko[z^{\pm 1}].
\end{align*}

Note that
\begin{align*}
& \dfrac{ d_{l,k^m}(z)d_{l,k}(-(-q)^{-2n-m+2}z)}{d_{l,k^{m-1}}((-q)z)} = \dfrac{d_{l,k^m}(z) \times \displaystyle \prod_{s=1}^{l} (z-(-q)^{4n+m-k-l-3+2s} )(z+(-q)^{2n+m+k-l-2+2s})  }
{ \displaystyle \prod_{s=1}^{l} (z-(-q)^{k-l+m-3+2s)}) (z+(-q)^{2n-k-l+m-4+2s)}) }
\end{align*}
and
\begin{align*}
& \dfrac{a_{l,k^{m-1}}((-q)z)}{a_{l,k^m}(z)a_{l,k}(-(-q)^{-2n-m+2}z)}  =  \dfrac{(z-(-q)^{m+k-l-1})(z+(-q)^{2n+m-k-l+2})}{(z-(-q)^{m+k+l-1})(z+(-q)^{2n+m-k+l-2})}
\end{align*}
by direct calculation. Thus we have
\begin{align} \label{eq: alkmstepot1p B}
\dfrac{d_{l,k^m}(z) \times \displaystyle \prod_{s=1}^{l}(z-(-q)^{4n+m-k-l-3+2s} )(z+(-q)^{2n+m+k-l-2+2s})  }
{ \displaystyle \prod_{s=1}^{l} (z-(-q)^{k-l+m-1+2s)}) (z+(-q)^{2n-k-l+m-2+2s)}) } \in \bfk[z^{\pm1}] 
\end{align}
which implies our assertion with~\eqref{eq: alkmstep1 B}, since
$$
4n+m-k-l-3+2s' \ne k-l+m-1+2s \qtq 2n+m+k-l-2+2s' \ne 2n-k-l+m-2+2s
$$
for $1 \le s,s' \le l \le k$.

\mnoi
(b) Now we assume that $  l >k $.
From the homomorphisms in Theorem~\ref{thm: Higher Dorey I}~\eqref{eq: k+l<n homo} with the restriction $\min(k,l)=1$ 
\[
\Vkm{1^m}_{(-q)^{-k+1}} \otimes \Vkm{(k-1)^m}_{(-q)} \twoheadrightarrow \Vkm{k^m},
\]
we have 
\begin{align*}
\dfrac{ d_{l,1^{m}}((-q)^{-k+1}z)d_{l,(k-1)^m}((-q)z)}{d_{l,k^{m}}(z)} \times \dfrac{a_{l,k^{m}}(z)}{a_{l,1^{m}}((-q)^{-k+1}z)a_{l,(k-1)^m}((-q)z)}.    
\end{align*}
By the induction hypothesis on $k$ and the direct calculation, we have
\begin{align} \label{eq: alkmstepot1 B}
\dfrac{ d_{l,1^{m}}((-q)^{-k+1}z)d_{l,(k-1)^m}((-q)z)}{d_{l,k^{m}}(z)}  = 
\dfrac{\displaystyle\prod_{s=1}^{k} (z-(-q)^{l-k+m-1+2s)}) (z+(-q)^{2n-k-l+m-2+2s)}) }{d_{l,k^{m}}(z)} ,
\end{align}
and 
\begin{align*}
& \dfrac{a_{l,k^{m}}(z)}{a_{l,1^{m}}((-q)^{-k+1}z)a_{l,(k-1)^m}((-q)z)} 
=1.
\end{align*}

By the homomorphism
\[
 \Vkm{k^m}  \otimes \Vkm{k}_{-(-q)^{2n+m-2}} \twoheadrightarrow \Vkm{k^{m-1}}_{(-q)^{-1}}
\]
we have
\begin{align*}    
& \dfrac{ d_{l,k^m}(z)d_{l,k}(-(-q)^{-2n-m+2}z)}{d_{l,k^{m-1}}((-q)z)} \times \dfrac{a_{l,k^{m-1}}((-q)z)}{a_{l,k^m}(z)a_{l,k}(-(-q)^{-2n-m+2}z)} \in \ko[z^{\pm 1}].
\end{align*}

Note that
\begin{align*}
\dfrac{ d_{l,k^m}(z)d_{l,k}(-(-q)^{-2n-m+2}z)}{d_{l,k^{m-1}}((-q)z)} = \dfrac{ d_{l,k^m}(z) \displaystyle \prod_{s=1}^{k} (z-(-q)^{4n+m-k-l-3+2s} )(z+(-q)^{2n+m+l-k-2+2s}) }
{ \displaystyle \prod_{s=1}^{k} (z-(-q)^{l-k+m-3+2s)}) (z+(-q)^{2n-k-l+m-4+2s)}) }
\end{align*}
and
\begin{align*}
= \dfrac{(z-(-q)^{m-k+l-1})(z+(-q)^{2n+m-2-k-l})}{(z-(-q)^{m+k+l-1})(z+(-q)^{2n+m-l+k-2})},
\end{align*}
by direct calculations. Thus we have
\begin{align} \label{eq: alkmstepot1p Bp}
\dfrac{ d_{l,k^m}(z) \displaystyle \prod_{s=1}^{k} (z-(-q)^{4n+m-k-l-3+2s} )(z+(-q)^{2n+m+l-k-2+2s}) }
{ \displaystyle \prod_{s=1}^{k} (z-(-q)^{l-k+m-1+2s}) (z+(-q)^{2n-k-l+m-2+2s}) } \in \bfk[z^{\pm1}].
\end{align}
Thus we have the assertion with~\eqref{eq: alkmstepot1 B} and~\eqref{eq: alkmstepot1p Bp}, since
$$
4n+m-k-l-3+2s' \ne l-k+m-1+2s \qtq 2n+m+l-k-2+2s' \ne 2n-k-l+m-2+2s
$$
for $1 \le s,s' \le k<l$.
\end{proof}

\begin{lemma}\label{lemma: division B}
For $1\leq k,l \le n-1$ and $m,p \ge 1$, we have
\[
\text{$d_{l^p,k^m}(z)$ divides $\prod_{t=0}^{\min(m,p)-1}\prod_{s=1}^{\min(k,l)} (z-(-q)^{|k-l|+\abs{m-p}+2(s+t)}) (z+(-q)^{2n-k-l+\abs{m-p}-1+2(s+t)})$.}
\]
\end{lemma}

\begin{proof}
Since $d_{l^p,k^m}(z) =  d_{k^m,l^p}(z)$, we assume that $l \le k$ without loss of generality. We assume that $\min(p,m)=p$. By fusion rule
\[
\Vkm{l^{p-1}}_{(-q)} \otimes  \Vkm{l}_{(-q)^{1-p}} \twoheadrightarrow \Vkm{l^p},
\]
we have
\[
 \dfrac{d_{l^{p-1},k^m}((-q)z)d_{l,k^m}((-q)^{1-p}z)}{d_{l^p,k^m}(z)} \in \ko[z^{\pm 1}].
\]
By induction hypothesis,
\[
d_{l^{p-1},k^m}((-q)z) \text{ divides }  \prod_{t=0}^{p-2}\prod_{s=1}^{l} (z-(-q)^{k-l+m-p+2(s+t)}) (z+(-q)^{2n-k-l+m-p-1+2(s+t)})
\]
and we know 
\[
d_{l,k^m}((-q)^{1-p}z) =   \prod_{s=1}^{l} (z-(-q)^{k-l+m+p-2+2s)}) (z+(-q)^{2n-k-l+m+p-3+2s)}).
\]
A direct computation using induction on $p$ shows that our assertion holds for this case. The other case can be proved in a similar way.
\end{proof}

With the previous statements and the analogous proofs, we have the analog of Lemma~\ref{lemma: dlpkm A}.
  
\begin{lemma} \label{Thm: dlpkm B}
For $1\leq k,l \le n-1$ and $\abs{m-p}\ge \min(k,l)-1$, we have
\begin{align}\label{eq: dlpkm B}
d_{l^p,k^m}(z)=  \prod_{t=0}^{\min(m,p)-1}\prod_{s=  1}^{\min(k,l)} (z-(-q)^{|k-l|+\abs{m-p}+2(s+t)}) (z+(-q)^{2n-k-l+\abs{m-p}-1+2(s+t)})
\end{align}
In particular, there is no ambiguity when $\min(k,l)=1$. 
\end{lemma}

\begin{proof}
Since $ d_{l^p,k^m}(z) =  d_{k^m,l^p}(z)$, we assume that $l \le k$
without loss of generality. We assume that $\min(p,m)=p$ and $m-p \ge l$.
The homomorphism
\[
 V(l^p) \otimes V(l)_{-(-q)^{2n+p-2}} \twoheadrightarrow V(l^{p-1})_{(-q)^{-1}}
\]
tells that 
we have
\begin{align*}
\dfrac{ d_{l^p,k^{m}}(z)d_{l,k^{m}}(-(-q)^{-2n-p+2}z)}{d_{l^{p-1},k^m}((-q)z)} =
\dfrac{ d_{l^p,k^{m}}(z)\displaystyle\prod_{s=1}^{l} (z-(-q)^{4n+p-k-l+m-4+2s})(z+(-q)^{2n+p+k-l+m-3+2s}) }{ \displaystyle \prod_{t=0}^{p-2}\prod_{s=1}^{l} (z-(-q)^{k-l+m-p+2(s+t)}) (z+(-q)^{2n-k-l+m-p-1+2(s+t)}) }
\end{align*}
by the induction on $p$,
\begin{align*}
& \dfrac{a_{l^{p-1},k^m}((-q)z)}{a_{l^p,k^{m}}(z)a_{l,k^{m}}(-(-q)^{-2n-p+2}z)}   
=  \prod_{s=1}^{l} \dfrac{(z-(-q)^{p+k+l-m-2s})(z+(-q)^{2n+p-k+l-m-1-2s})}{(z-(-q)^{p+k-l+m-2+2s})(z+(-q)^{2n+p-k-l+m-3+2s})}
\end{align*}
and hence 
\begin{align*}
& \dfrac{ d_{l^p,k^{m}}(z)\displaystyle\prod_{s=1}^{l} (z-(-q)^{4n+p-k-l+m-4+2s})(z+(-q)^{2n+p+k-l+m-3+2s}) }{ \displaystyle \prod_{t=0}^{p-1}\prod_{s=1}^{l} (z-(-q)^{k-l+m-p+2(s+t)}) (z+(-q)^{2n-k-l+m-p-1+2(s+t)}) } \allowdisplaybreaks\\
& \hspace{25ex} \times \prod_{s=1}^{l} (z-(-q)^{p+k+l-m-2s})(z+(-q)^{2n+p-k+l-m-1-2s}) \in \ko[z^{\pm 1}].
\end{align*}
Since\footnote{Hereafter, we frequently skip the computation like~\eqref{eq: no cancel}}
\begin{equation} \label{eq: no cancel}
\begin{aligned}
&  4n+p-k-l+m-4+2s' = k-l+m-p+2(s+t)\text{ and } \\
& 2n+p+k-l+m-3+2s' \ne 2n-k-l+m-p-1+2(s+t)
\end{aligned}
\end{equation}
for $1 \le s,s' \le l \le k$ and $0 \le t \le p-1$, 
the above equation can be refined as follows:
\begin{align}\label{eq: m-p ge l B}
\dfrac{d_{l^p,k^{m}}(z)\times\displaystyle\prod_{s=1}^{l} (z-(-q)^{p+k+l-m-2s})(z+(-q)^{2n+p-k+l-m-1-2s})}{\displaystyle\prod_{t=0}^{p-1}\prod_{s=1}^{l}(z-(-q)^{k-l+m-p+2(s+t)})(z+(-q)^{2n-k-l+m-p-1+2(s+t)}) } \in \ko[z^{\pm 1}].
\end{align}

By writing $m=p+l+u$ for some $u \ge  -1$,~\eqref{eq: m-p ge l B} can be re-written as follows:
\begin{align*}
\dfrac{d_{l^p,k^{m}}(z)\times\displaystyle\prod_{s=1}^{l} (z-(-q)^{k-u-2s})(z+(-q)^{2n-k-u-1-2s})}
{\displaystyle\prod_{t=0}^{p-1}\prod_{s=1}^{l}(z-(-q)^{k+u+2(s+t)})(z+(-q)^{2n-k+u-1+2(s+t)}) } \in \ko[z^{\pm 1}].
\end{align*}

Note that
$$k-u-2s'  \ne k+u+2(s+t) \text{ and } 2n-k-u-1-2s' \ne 2n-k+u-1+2(s+t)$$
for $(s-s') < l$,  $u \ge -1$  and $t \le p-1$. Thus~\eqref{eq: m-p ge l B} is equivalent to
\begin{align*}
\dfrac{d_{l^p,k^{m}}(z)}{\displaystyle\prod_{t=0}^{p-1}\prod_{s=1}^{l}(z-(-q)^{k-l+m-p+2(s+t)})(z+(-q)^{2n-k-l+m-p-1+2(s+t)}) } \in \ko[z^{\pm 1}],
\end{align*}
which implies our assertion with Proposition~\ref{lemma: division B}.
The remained case $(\min(m,p)=m)$ can be proved in a similar way. 
\end{proof}

\begin{proof}[Proof of {\rm Theorem~\ref{thm:denominators_untwisted}} for $\max(l,k) < n$ in type $B_n^{(1)}$]
Without loss of generality, we assume that $p \leq m$ and $m-p < \min(k,l)-1$.
We also assume $l \leq k$.

From the homomorphism in Theorem~\ref{thm: Higher Dorey I}~\eqref{eq: k+l<n homo} with the restriction $\min(k,l)=1$, we obtain 
\[
\Vkm{l^p}  \otimes \Vkm{1^p}_{-(-q)^{2n-l}} \twoheadrightarrow  \Vkm{(l-1)^p}_{(-q)}.
\]
From the above homomorphism, we obtain 
\[
\dfrac{d_{l^p,k^m}( z ) d_{1^p,k^m}(-(-q)^{2n-l}z)}{d_{(l-1)^p,k^m}((-q)z)} \times \dfrac{a_{(l-1)^p,k^m}((-q)z)}{a_{l^p,k^m}( z ) a_{1^p,k^m}(-(-q)^{2n-l}z)} \in \ko[z^{\pm1}].
\] 
By direct computation,
\begin{align*}
& \dfrac{a_{(l-1)^p,k^m}((-q)z)}{a_{l^p,k^m}( z ) a_{1^p,k^m}(-(-q)^{2n-l}z)}  = \prod_{t=1}^{p-1} \dfrac{(z-(-q)^{-k+l-m+p-2-2t})(z+(-q)^{-2n+k+l-m+p-1-2t})}{(z-(-q)^{l-k+m-p+2t})(z+(-q)^{-2n+k+l+m-p+1+2t})}.
\end{align*}
Furthermore, by the induction on $l$, we have
\begin{align*}
& \dfrac{d_{l^p,k^m}( z ) d_{1^p,k^m}(-(-q)^{2n-l}z)}{d_{(l-1)^p,k^m}((-q)z)} =
\dfrac{d_{l^p,k^m}( z )\displaystyle \prod_{t=0}^{p-1} (z-(-q)^{-k+l+m-p+2t}) (z+(-q)^{-2n+k+l+m-p+1+2t}) }{\displaystyle \prod_{t=0}^{p-1}\prod_{s=1}^{l-1} (z-(-q)^{k-l+m-p+2(s+t)}) (z+(-q)^{2n-k-l+m-p-1+2(s+t)}) }.
\end{align*}
Therefore, we have
\begin{align}
& \dfrac{d_{l^p,k^m}( z )\displaystyle \prod_{t=0}^{p-1} (z-(-q)^{-k+l+m-p+2t}) (z+(-q)^{-2n+k+l+m-p+1+2t}) }{\displaystyle \prod_{t=0}^{p-1}\prod_{s=1}^{l-1} (z-(-q)^{k-l+m-p+2(s+t)}) (z+(-q)^{2n-k-l+m-p-1+2(s+t)}) }
\nonumber\allowdisplaybreaks\\
& \hspace{15ex} \times  \prod_{t=1}^{p-1} \dfrac{(z-(-q)^{-k+l-m+p-2-2t})(z+(-q)^{-2n+k+l-m+p-1-2t})}{(z-(-q)^{l-k+m-p+2t})(z+(-q)^{-2n+k+l+m-p+1+2t})} \nonumber\allowdisplaybreaks\\
& \hspace{10ex} = \dfrac{d_{l^p,k^m}( z )\displaystyle  \prod_{t=1}^{p-1} (z-(-q)^{-k+l-m+p-2-2t})(z+(-q)^{-2n+k+l-m+p-1-2t}) }{\displaystyle \prod_{t=0}^{p-1}\prod_{s=1}^{l-1} (z-(-q)^{k-l+m-p+2(s+t)}) (z+(-q)^{2n-k-l+m-p-1+2(s+t)}) } \in
\ko[z^{\pm1}] \nonumber \allowdisplaybreaks\\
& \hspace{10ex} \Rightarrow \dfrac{d_{l^p,k^m}( z ) }{\displaystyle \prod_{t=0}^{p-1}\prod_{s=1}^{l-1} (z-(-q)^{k-l+m-p+2(s+t)}) (z+(-q)^{2n-k-l+m-p-1+2(s+t)}) } \in
\ko[z^{\pm1}]. \label{eq: final step1 B}
\end{align}
On the other hand, fusion rule
\[
\Vkm{k^m} \otimes \Vkm{k}_{(-q)^{-m-1}} \twoheadrightarrow \Vkm{k^{m+1}}_{(-q)^{-1}},
\]
we have
\begin{align}
 \dfrac{d_{l^p,k^m}(z) d_{l^p,k}((-q)^{-m-1} z)}{d_{l^p,k^{m+1}}((-q)^{-1} z)} & =
\dfrac{d_{l^p,k^m}(z) \displaystyle \prod_{s=1}^{l} (z-(-q)^{k-l+m+p+2s}) (z+(-q)^{2n-k-l+m+p+2s-1})}{\displaystyle \prod_{t=0}^{p-1}\prod_{s=1}^{l} (z-(-q)^{k-l+(m-p)+2(s+t)+2}) (z+(-q)^{2n-k-l+(m-p)+2(s+t)+1})}
 \allowdisplaybreaks \nonumber \\
& \hspace{-15ex}
= \dfrac{d_{l^p,k^m}(z)}{\displaystyle \prod_{t=1}^{p-1}\prod_{s=1}^{l} (z-(-q)^{k-l+(m-p)+2(s+t)}) (z+(-q)^{2n-k-l+(m-p)+2(s+t)-1})} \in \ko[z^{\pm1}], \label{eq: final step2 B}
\end{align}
which follows from the descending induction on $m-p$. 

Then, unless $t=0$ and $s=l$, our assertion holds by comparing~\eqref{eq: final step1 B} and~\eqref{eq: final step2 B} since
\[
k+l+m-p+2t \ne k-l+m-p+2s,  \ \text{ and } \ 2n-k-l+m-p+2s-1 = 2n-k+l+m-p+2t-1.
\]
Thus it suffices to prove that the factor at $t=0$ and $s=l$ appears in $d_{l^p,k^{m}}(z)$
as many times as we want.

As we did, the homomorphism
\[
 \Vkm{l^p} \otimes \Vkm{l}_{-(-q)^{2n+p-2}} \twoheadrightarrow \Vkm{l^{p-1}}_{(-q)^{-1}}
\]
yields 
\begin{align}\label{eq: m-p ge l}
\dfrac{d_{l^p,k^{m}}(z)\times\displaystyle\prod_{s=1}^{l} (z-(-q)^{p+k+l-m-2s})(z+(-q)^{2n+p-k+l-m-1-2s})}{\displaystyle\prod_{t=0}^{p-1}\prod_{s=1}^{l}(z-(-q)^{k-l+m-p+2(s+t)})(z+(-q)^{2n-k-l+m-p-1+2(s+t)}) } \in \ko[z^{\pm 1}],
\end{align}
which guarantees that the factor at $t=0$ and $s=l$ appears in $d_{l^p,k^{m}}(z)$
as many times as we want. Thus the assertion for this case is completed.

The remaining cases can be proved in a similar way.
\end{proof}

\begin{corollary} \label{cor: root module B} Let $x \in \bfk^{\times}$.
\bna
\item For $m \le n$, the KR module $\Vkm{1^m}_x$ is a root module.
\item For \emph{any} $m \ge 1$,  the KR module $M = \Vkm{1^m}_x$ $(x \in \bfk^{\times})$ is plain. 
\item For any $1\le k \le n$, the fundamental module $M = \Vkm{k}_x$ is plain.
\ee
Hence the composition length of $M \tens N$ for any KR module $N$ in $\Cg$ is less than or equal to $2$.  
\end{corollary}

\subsubsection{$d_{l^m,n^p}(z)$ for $1 \le l <n$}

Note that
\begin{equation}
\begin{aligned}\label{eq: B ank dnk}
d_{n,l}(z) = \displaystyle  \prod_{s=1}^{l}\big(z-(-1)^{n+l}\qs^{2n-2l-1+4s}\big) \ \
a_{n,l}(z) \equiv   \dfrac{\PSN{2n-2l-1}\PSN{6n+2l-3}}{\PSN{2n+2l-1}\PSN{6n-2l-3}},
\end{aligned}
\end{equation}

Recall that we have surjective homomorphisms
\begin{align} \label{eq: square Dorey B}
 \Vkm{1}_{(-1)^{n}\qs^{-2n+2}} \otimes \Vkm{n-1}_{\qs^2} \twoheadrightarrow \Vkm{n^2},    
\end{align} 
as a special case in~\eqref{eq: new homom B}.

\begin{lemma}
For $m \ge 1$, we have
\[
d_{1,n^m}(z)=  \prod_{t=0}^{\min(2,m)-1}  (z-(-1)^{n+m}(\qs)^{2n+|2-m|+2t}),
\]
\end{lemma}

\begin{proof}
We first prove our assertion for $m=2$.  By the homomorphisms obtained from~\eqref{eq: new homom C}
\[
\Vkm{n}_{(-\qs)} \otimes \Vkm{n}_{(-\qs)^{-1}} \twoheadrightarrow \Vkm{n^2} \quad \text{ and } \quad \Vkm{n^2}\otimes \Vkm{n}_{-\qs^{4n-1}} \twoheadrightarrow  \Vkm{n}_{(-\qs)^{-1}},
\]
we have
\begin{align*}
&\dfrac{ (z-(-1)^{n+2}(\qs)^{2n})(z-(-1)^{n+2}(\qs)^{2n+2})}{d_{1,n^2}(z)} \in \ko[z^{\pm 1}], \allowdisplaybreaks\\
& \dfrac{ d_{1,n^{2}}(z)  (z-(-1)^{n+2}\qs^{6n})  }{ (z-(-1)^{n+2}\qs^{2n})   }  \times   \dfrac{(z-(-1)^{n+2}\qs^{2n-2})}{(z-(-1)^{n+2}\qs^{2n+2})}   \in \ko[z^{\pm 1}],
\end{align*}
which implies our assertion for $m=2$.

Now let us assume that $m \ge 3$. By fusion rule
\[
\Vkm{n^{m-1}}_{(-\qs)} \otimes \Vkm{n}_{(-\qs)^{1-m}} \twoheadrightarrow \Vkm{n^m}  \text{ and } \Vkm{n}_{(-\qs)^{m-1}} \otimes \Vkm{n^{m-1}}_{(-\qs)^{-1}} \twoheadrightarrow \Vkm{n^m}
\]
we have
\begin{subequations} \label{eq: a1nmstep1 B}
\begin{align}
&\dfrac{    (z-(-1)^{n+m}(\qs)^{2n+m-4})(z-(-1)^{n+m}(\qs)^{2n+m-2}) (z-(-1)^{n+m}(\qs)^{2n+m})  }{d_{1,n^m}(z)} \in \ko[z^{\pm 1}],  \\ 
&\dfrac{     (z-(-1)^{n+m}(\qs)^{2n-m}) (z-(-1)^{n+m}(\qs)^{2n+m-2})(z-(-1)^{n+m}(\qs)^{2n+m})   }{d_{1,n^m}(z)} \in \ko[z^{\pm 1}]. 
\end{align}
\end{subequations}
By equations in~\eqref{eq: a1nmstep1 B}, the ambiguity happens at $m=2$ which we already cover.

On the other hand, the homomorphism
\[
\Vkm{n^m}  \otimes \Vkm{n}_{(-\qs)^{4n+m-3}} \twoheadrightarrow \Vkm{n^{m-1}}_{(-\qs)^{-1}}.
\]
implies that we have
\begin{align} \label{eq: a1nmstep1 BN}
\dfrac{ d_{1,n^{m}}(z) (z-(-1)^{n+m}\qs^{6n+m-2}) }{ (z-(-1)^{n+m}(\qs)^{2n+m-2})(z-(-1)^{n+m}\qs^{2n+m})} \in \ko[z^{\pm 1}],
\end{align}
which implies our assertion for $m\geq 3$ with~\eqref{eq: a1nmstep1 B} and~\eqref{eq: a1nmstep1 BN}.
\end{proof}

\begin{lemma}
We have
$$
d_{1^2,n}(z)
= (z-(-1)^{n}(\qs)^{2n+3})
$$
\end{lemma}

\begin{proof}
Note that $d_{1,n}(z) = \displaystyle  \big(z-(-1)^{n+1}\qs^{2n+1}\big)$. By Proposition~\ref{prop: de less than equal to}, it suffices to consider
\bna
\item \label{it: B 12n i}
$\de(\Vkm{n},\Vkm{1}_{(-1)^{n+1}\qs^{2n+1}} \hconv \Vkm{1}_{(-1)^{n+1}\qs^{2n-3}})=\de(\Vkm{n},\Vkm{1^2}_{(-1)^{n}\qs^{2n-1}}).$
\item  \label{it: B 12n ii}
$\de(\Vkm{n},\Vkm{1}_{(-1)^{n+1}\qs^{2n+5}} \hconv \Vkm{1}_{(-1)^{n+1}\qs^{2n+1}})=\de(\Vkm{n},\Vkm{1^2}_{(-1)^{n}\qs^{2n+3}}).$
\ee

\mnoi
\eqref{it: B 12n i} Note that 
$$
\exrch(\Vkm{1^2}_{(-1)^{n}\qs^{2n-1}}) = \range{-8,4n} \qtq \rch(\Vkm{n})\range{0}.
$$
Thus $\de(\Vkm{n},\Vkm{1^2}_{(-1)^{n}\qs^{2n-1}})=0$.

\mnoi
\eqref{it: B 12n ii} Since $$\de(\Vkm{n},\Vkm{1}_{(-1)^{n+1}\qs^{2n+5}}) = \de(\scrD^{-1}\Vkm{n},\Vkm{1}_{(-1)^{n+1}\qs^{2n+5}})=0,$$
Lemma~\ref{lem: de=de}~\eqref{it: de=de 2} implies $\de(\Vkm{n},\Vkm{1^2}_{(-1)^{n}\qs^{2n+3}})=1$ as we desired.
\end{proof}

\begin{lemma}
For $m \ge 1$, we have
$$
d_{1^m,n}(z)
= (z-(-1)^{n+m}(\qs)^{2n+2m-1})
$$
\end{lemma}

\begin{proof}
Since the cases when $m<3$ is already covered, let us assume that $m \ge3$. 
By fusion rule,
\[
\Vkm{1^{m-1}}_{(-q)} \otimes \Vkm{1}_{(-q)^{1-m}} \twoheadrightarrow \Vkm{1^m} \quad \text{ and } \quad \Vkm{1}_{(-q)^{m-1}} \otimes V(1^{m-1})_{(-q)^{-1}} \twoheadrightarrow V(1^m)
\]
we have
\begin{align*}
&\dfrac{ (z-(-1)^{n+m}(\qs)^{2n+2m-5}) (z-(-1)^{n+m}(\qs)^{2n+2m-1}) }{d_{1^m,n}(z)} \in \ko[z^{\pm1}], \\
&\dfrac{ (z-(-1)^{n+m}(\qs)^{2n-2m+3}) (z-(-1)^{n+m}(\qs)^{2n+2m-1}) }{d_{1^m,n}(z)}  \in \ko[z^{\pm1}], 
\end{align*}
whose ambiguity do not happen for $m \ge 3$. 

By the homomorphism
\[
\Vkm{1^m} \otimes \Vkm{1}_{-(-q)^{2n+m-2}} \twoheadrightarrow \Vkm{1^{m-1}}_{(-q)^{-1}},
\]
we have
\[
\dfrac{d_{1^m,n}(z) (z-(-1)^{n+m}(\qs)^{-2n-2m+1})  }{ (z-(-1)^{n+m}(\qs)^{2n+2m-1})}
\in \bfk[z^{\pm1}],
\]
which implies the assertion. 
\end{proof}

\begin{lemma}
For $m \ge 1$ and $1 \le l \le n-1$, we have
\[
d_{l,n^m}(z)=  \prod_{t=0}^{\min(2,m)-1}\prod_{s=1}^{l} (z-(-1)^{n+l+1+m}(\qs)^{2n-2l-2+|2-m|+4s+2t}).
\]
\end{lemma}

\begin{proof}
By the homomorphism
\[
\Vkm{l-1}_{(-q)^{-1}} \otimes \Vkm{1}_{(-q)^{l-1}} \twoheadrightarrow \Vkm{l},
\]
we have
\begin{equation*}
\dfrac{ d_{l-1,n^m}((-q)^{-1}z)d_{1,n^m}((-q)^{l-1}z)}{d_{l,n^m}(z)} \equiv \dfrac{\displaystyle  \prod_{t=0}^{1}\prod_{s=1}^{l} (z-(-1)^{n+l+1+m}(\qs)^{2n-2l+m-4+4s+2t})}{d_{l,n^m}(z)} \in \ko[z^{\pm 1}].
\end{equation*}
Next, by the homomorphism
\[ 
 \Vkm{l} \otimes \Vkm{1}_{-(-q)^{2n+l-2}} \twoheadrightarrow \Vkm{l-1}_{(-q)},
\]
we have
\[
\dfrac{ d_{l,n^{m}}(z)d_{1,n^{m}}(-(-q)^{-2n-l+2}z)}{d_{l-1,n^m}((-q)z)} \times \dfrac{a_{l-1,n^m}((-q)z)}{a_{l,n^{m}}(z)a_{1,n^{m}}(-(-q)^{-2n-l+2}z)} \in \ko[z^{\pm 1}].
\]

Let us focus on $m=2$ first. By an induction hypothesis on $l$, we have  
\begin{align*}
\dfrac{ d_{l,n^{2}}(z)  \times \displaystyle \prod_{t=0}^{1}  (z-(-1)^{n+1+l}(\qs)^{6n+2l-4+2t})  }
{\displaystyle\prod_{t=0}^{1}\prod_{s=1}^{l-1} (z-(-1)^{n+l+1}(\qs)^{2n-2l-2+4s+2t}) } \times     
\dfrac{(z-(-1)^{n+1+l}(\qs)^{2n+2l-4})  }{(z-(-1)^{n+1+l}(\qs)^{2n+2l-2}) } \in \ko[z^{\pm1}]
\end{align*}
by direct computations. 
Thus it suffices to prove that $(-1)^{n+1+l}(\qs)^{2n+2l-4}$ and $(-1)^{n+1+l}(\qs)^{2n+2l}$ are factors of $d_{l,n^{2}}(z)$.

Note that there exists a homomorphism
\begin{align*}
\quad \Vkm{n^2} \otimes \Vkm{1}_{(-1)^{n}\qs^{2n}} & \twoheadrightarrow  \Vkm{n-1}_{\qs^2}.
\end{align*}
obtained from~\eqref{eq: square Dorey B}. Then we have
\begin{align*}
\dfrac{ d_{l,n^{2}}(z) \times (z-(-1)^{n+l}\qs^{6n-2l}) \times (z-(-1)^{n+l+1}\qs^{2n-2l-2})   }{ \displaystyle   \prod_{t=0}^{1}\prod_{s=1}^{l} (z-(-1)^{n+l+1+m}(\qs)^{2n-2l-2+4s+2t})  } \in \ko[z^{\pm 1}],
\end{align*}
which completes the assertion for $m = 2$.

Now we assume that $m \ge 3$.
Note that we have a homomorphism
\begin{align*}
\Vkm{n^m}  \otimes \Vkm{n}_{(-\qs)^{4n+m-3}} \twoheadrightarrow \Vkm{n^{m-1}}_{(-\qs)^{-1}}.
\end{align*}
Then we have
\begin{align*}
\dfrac{ d_{l,n^{m}}(z)  \displaystyle \prod_{s=1}^{l} (z-(-1)^{n+l+m+1}(\qs)^{6n-2l+m-4+4s}) }{ \displaystyle \prod_{t=0}^{1}\prod_{s=1}^{l} (z-(-1)^{n+l+1+m}(\qs)^{2n-2l+m-4+4s+2t})   }
\in \ko[z^{\pm1}],
\end{align*}
which yields our assertion since 
\begin{align*}
& 6n-2l+m-4+4s' \ne  2n-2l+m-4+4s+2t  \iff 2n  \ne  2(s-s')+t
\end{align*}
for $1 \le s,s' < n$ and $t \in \{ 0,1 \}$.
\end{proof}

\begin{lemma} \label{lemma: dlpn2m C}
For $1\leq l \le n-1$ and $p>1$, we have
\[
d_{l^p,n^2}(z) =  \prod_{s=1}^{l} (z-(-1)^{n+l+p}(\qs)^{2n-2l+2p-4+4s})(z-(-1)^{n+l+p}(\qs)^{2n-2l+2p-2+4s}).
\]
\end{lemma}

\begin{proof}
By~\eqref{eq: square Dorey B}  
\begin{align*}
&\Vkm{1}_{(-1)^{n}\qs^{-2n+2}} \otimes \Vkm{n-1}_{\qs^2} \twoheadrightarrow \Vkm{n^2}, \\
&\Vkm{n-1}_{ \qs^{-2} } \otimes \Vkm{1}_{(-1)^{n}\qs^{2n-2}} \twoheadrightarrow \Vkm{n^2},
\end{align*}
we have
\begin{equation} \label{eq:dlpn2step1}
\begin{aligned}
& \dfrac{\displaystyle \prod_{s=1}^{l} (z-(-1)^{n+l+p}q^{2n-2l+2p-4+4s}) (z-(-1)^{n+l+p}q^{2n-2l+2p-2+4s})   }{d_{l^p,n^2}(z)} \allowdisplaybreaks\\
& \hspace{20ex} \times  (z-(-1)^{l+p+n}\qs^{2n-2l-2p-2})(z-(-1)^{l+p+n}\qs^{6n-2l+2p-4}) \in \ko[z^{\pm 1}],
\end{aligned}
\end{equation}
and
\begin{equation} \label{eq:dlpn2step1p}
\begin{aligned}
& \dfrac{  \displaystyle\prod_{s=1}^{l} (z-(-1)^{l+p+n}\qs^{2n-2l+2p-4+4s)}) (z-(-1)^{l+p+n}\qs^{2n-2l+2p-2+4s)}) }{d_{l^p,n^2}(z)} \allowdisplaybreaks\\
& \hspace{20ex} \times (z-(-1)^{l+p+n}\qs^{-2n+2l-2p+2}) (z-(-1)^{l+p+n}\qs^{2n+2l+2p})  \in \ko[z^{\pm 1}],
\end{aligned}
\end{equation}
as in the previous cases. 

By~\eqref{eq:dlpn2step1} and~\eqref{eq:dlpn2step1p}, we have an ambiguity happens at $l=n-1$. 
\noindent
In that case, 
(i) Since
$$
\Vkm{(n-1)^p}_{(-\qs)^{4n+2p-2}} =\head(\Vkm{n-1}_{(-\qs)^{6n+2p-4}} \otimes \cdots \Vkm{n-1}_{(-\qs)^{4n}} )
$$
and $$\de(\Vkm{n^2},\Vkm{n-1}_{(-\qs)^{4n+4r}})=0  \quad \text{ for }  r \ge 0 \text{ as } d_{l,n^2}(z)=\displaystyle\prod_{t=0}^1\prod_{s=1}^{n-1} (z-(\qs)^{4s+2t}),$$ 
$(-\qs)^{4n+2p-2}$ can not be a root by Proposition~\ref{prop: de less than equal to}.
(ii) Since the universal formula arising from $(-\qs)^{-2p}$ doe not vanish, $(-\qs)^{-2p}$ can not be a root by 
the universal coefficient formula in Proposition~\ref{prop:universal_coeff_B1}.

From the homomorphism,
$$
\Vkm{n^2} \otimes \Vkm{1}_{(-1)^{n}\qs^{2n}}  \twoheadrightarrow  \Vkm{n-1}_{\qs^2},
$$
we obtain
\begin{align*}
\dfrac{ d_{l^p,n^{2}}(z) (z-(-1)^{l+p+n}\qs^{2n-2l-2p}) (z-(-1)^{l+p+n}\qs^{6n-2l+2p-2})   }{ \displaystyle \prod_{s=1}^{l} (z-(-1)^{l+p+n}\qs^{2n-2l+2p-4+4s}) (z-(-1)^{l+p+n}\qs^{2n-2l+2p-2+4s}) } \in \ko[z^{\pm 1}],
\end{align*}
as in the previous cases.
Since
$$
\{ 6n-2l+2p-2, 2n-2l-2p \} \cap \{ 2n-2l+2p-4+4s, 2n-2l+2p-2+4s \ | \ 1 \le s \le l\} =\emptyset,
$$
our assertion follows from~\eqref{eq:dlpn2step1} and ~\eqref{eq:dlpn2step1p}.
\end{proof}

\begin{lemma}
\label{lemma: dlpnm}
For $1\leq l \le n-1$ and $p>1$, we have
\[
d_{l^p,n}(z)=   \prod_{s=1}^{l} (z-(-1)^{n+l+p+1}(\qs)^{2n-2l+2p-3+4s})
\]
\end{lemma}

\begin{proof}
From the homomorphism in Theorem~\ref{thm: Higher Dorey I}~\eqref{eq: k+l<n homo} with the restriction $\min(k,l)=1$
$$
\Vkm{1^p}_{(-q)^{-l+1}} \otimes \Vkm{(l-1)^p}_{(-q)} \twoheadrightarrow \Vkm{l^p},
$$
we have
\begin{align*}
 \dfrac{\displaystyle\prod_{s=1}^{l}  (z-(-1)^{n+l+p+1}(\qs)^{2n-2l+2p-3+4s})  }{d_{l^p,n}(z)} \in \bfk[z^{\pm1}],
\end{align*}
since 
$\dfrac{a_{l^p,n}(z)}{a_{(l-1)^p,n}((-q)z)a_{1^p,n}((-q)^{-l+1}z)} = 1$.

On the other hand, the homomorphism
\begin{align*}
&\Vkm{l^p}  \otimes  \Vkm{1^p}_{-(-q)^{2n-l}}   \twoheadrightarrow \Vkm{(l-1)^p}_{(-q)}
\end{align*}
yields
\begin{align*}
& \dfrac{ d_{l^p,n}(z) \displaystyle\prod_{s=1}^{l}(z-(-1)^{n+l+p+1}\qs^{6n-2p-2l+3+4s})}{  \displaystyle\prod_{s=1}^{l} (z-(-1)^{n+l+p+1}(\qs)^{2n-2l+2p-3+4s}) }
\in \ko[z^{\pm 1}].
\end{align*} 
Note that  
\begin{align*}
& 6n-2p-2l+3+4s' \ne 2n-2l+2p-3+4s \iff 2n -2p -2l-2(s'-s) \ne -3
\end{align*}
for $1 \le s,s' \le l$, which implies the assertion.
\end{proof}

\begin{lemma}\label{lemma: divisopn B l,n}
For $1\leq l \le n-1$ and $m,p \ge 1$,
\[
d_{l^p,n^m}(z) \text{ divides }  \prod_{t=0}^{\min(2p,m)-1}\prod_{s=1}^{l} (z-(-1)^{n+l+p+m}(\qs)^{2n-2l-2+|2p-m|+4s+2t}).
\]
\end{lemma}

\begin{proof}
We assume that $\min(2p,m)=m$.  By fusion rule
\[
\Vkm{n^{m-1}}_{(-\qs)} \otimes \Vkm{n}_{(-\qs)^{1-m}} \twoheadrightarrow \Vkm{n^m},
\]
we have
\begin{align*}
& \dfrac{ d_{l^p,n^{m-1}}((-\qs)z)d_{l^p,n}((-\qs)^{1-m}z)}{d_{l^p,n^m}(z)} \in \ko[z^{\pm 1}].
\end{align*}
By the induction, we have
\[
d_{l^p,n^{m-1}}((-\qs)z)  \text{ divides }  \displaystyle \prod_{t=0}^{m-2}\prod_{s=1}^{l} (z-(-1)^{n+l+p+m}(\qs)^{2n-2l-2+2p-m+4s+2t}),
\]
and we know
\[
d_{l^p,n}((-q)^{1-m}z) =  \prod_{s=1}^{l} (z-(-1)^{n+l+p+1}(\qs)^{2n-2l+2p+m-4+4s}).
\]
Thus our assertion holds for this case. The remaining case can be proved in a similar way.
\end{proof}

\begin{lemma} \label{lemma: dlpnm B}
For $1\leq l \le n-1$, $\min(m,p) \ge 2$ and $2p-m\ge 2(l+1)$ or $m-2p \ge (l-1)$, we have
\begin{align}\label{eq: dlpnm}
d_{l^p,n^m}(z) =  \prod_{t=0}^{\min(2p,m)-1}\prod_{s=1}^{l} (z-(-1)^{n+l+p+m}(\qs)^{2n-2l-2+|2p-m|+4s+2t}).
\end{align}
\end{lemma}

\begin{proof}
(a) Let us assume that $2p - m \ge l+2$.
By the homomorphism
\[
 \Vkm{n^m}  \otimes  \Vkm{n}_{(-\qs)^{4n+m-3}}  \twoheadrightarrow \Vkm{n^{m-1}}_{(-\qs)^{-1}},
\]
we have
\begin{align*}
\dfrac{ d_{l^p,n^m}(z)d_{l^p,n}((-\qs)^{-4n-m+3}z)}{d_{l^p,n^{m-1}}((-\qs)z)} \times
\dfrac{a_{l^p,n^{m-1}}((-\qs)z)}{ a_{l^p,n^m}(z)a_{l^p,n}((-\qs)^{-4n-m+3}z)} \in \ko[z^{\pm 1}].
\end{align*}

Note that
\begin{align*}
& \dfrac{ d_{l^p,n^m}(z)d_{l^p,n}((-\qs)^{-4n-m+3}z)}{d_{l^p,n^{m-1}}((-\qs)z)} =
\dfrac{ d_{l^p,n^m}(z) \times \displaystyle\prod_{s=1}^{l} (z-(-1)^{n+l+p+m}(\qs)^{6n-2l+2p+m-6+4s})  }{ \displaystyle\prod_{t=0}^{m-2}\prod_{s=1}^{l} (z-(-1)^{n+l+p+m}(\qs)^{2n-2l-2+2p-m+4s+2t})  }
\end{align*}
by the induction on $m$, and the direct computation yields
\begin{align*}
& \dfrac{a_{l^p,n^{m-1}}((-q)z)}{ a_{l^p,n^m}(z)a_{l^p,n}((-\qs)^{-4n-m+3}z)}  = \prod_{s=1}^{l} \dfrac{(z-(-1)^{n+l+p+m}\qs^{2n+2l-2p+m-4s})}{(z-(-1)^{n+l+p+m}\qs^{2n-2l+2p+m-4+4s})}.
\end{align*}
Thus we have
\begin{align*} 
\dfrac{ d_{l^p,n^m}(z) \times \displaystyle\prod_{s=1}^{l} (z-(-1)^{d}(\qs)^{6n-2l+2p+m-6+4s})  (z-(-1)^{d}\qs^{2n-2l-2p+4+m+4s}) }{ \displaystyle\prod_{t=0}^{m-1}\prod_{s=1}^{l} (z-(-1)^{d}(\qs)^{2n-2l-2+2p-m+4s+2t})  }
\in \ko[z^{\pm 1}],
\end{align*}
where $d = n+l+p+m$. Note that 
\begin{align*}
& 6n-2l+2p+m-6+4s' \ne 2n-2l-2+2p-m+4s+2t, \\
& \iff  2n+m-2 \ne 2(s-s')+t
\end{align*}
for $1 \le s,s' \le l$ and $0 \le t \le m-1$. Thus the above equation tells that
\begin{align}\label{eq: 2p-m ge l}
 \dfrac{ d_{l^p,n^m}(z) \times \displaystyle \prod_{s=1}^{l}  (z-(-1)^{d}\qs^{2n-2l-2p+4+m+4s}) }{ \displaystyle\prod_{t=0}^{m-1}\prod_{s=1}^{l} (z-(-1)^{d}(\qs)^{2n-2l-2+2p-m+4s+2t})  }
    \in \ko[z^{\pm 1}].
\end{align}
Thus we see that there is no cancellation  
\begin{align*}
2n-2l-2p+4+m+4s \ne 2n-2l-2+2p-m+4s'+2t' \iff 2(s-s')-t' \ne u-3
\end{align*}
for $1 \le s,s' \le l$. Hence our assertion follows with Lemma~\ref{lemma: divisopn B l,n}.

\noindent
(b) Let us assume that $2p < m$. By the homomorphism
\[ 
 \Vkm{l^p} \otimes \Vkm{l}_{-(-q)^{2n+p-2}} \twoheadrightarrow \Vkm{l^{p-1}}_{(-q)^{-1}},
\]
we have
\begin{align*}
\dfrac{ d_{l^p,n^m}(z)d_{l,n^m}(-(-q)^{-2n-p+2}z)}{d_{l^{p-1},n^m}((-q)z)} \times
\dfrac{a_{l^{p-1},n^m}((-q)z)}{ a_{l^p,n^m}(z)d_{l,n^m}(-(-q)^{-2n-p+2}z)} \in \ko[z^{\pm 1}].
\end{align*}

Note that
\begin{align*}
& \dfrac{ d_{l^p,n^m}(z)d_{l,n}(-(-q)^{-2n-p+2}z)}{d_{l^{p-1},n^m}((-q)z)} =
\dfrac{ d_{l^p,n^m}(z) \displaystyle \prod_{t=0}^{1}\prod_{s=1}^{l} (z-(-1)^{n+l+p+m}(\qs)^{6n-2l+2p+2+m-2+4s+2t}) }{ \displaystyle \prod_{t=0}^{2p-3}\prod_{s=1}^{l} (z-(-1)^{n+l+p+m}(\qs)^{2n-2l-2+m-2p+4s+2t})}
\end{align*}
by the induction on $p$, and $(d =m+n+l+p)$ the direct computation yields 
\begin{align*}
& \dfrac{a_{l^{p-1},n^m}((-q)z)}{ a_{l^p,n^m}(z)d_{l,n^m}(-(-q)^{-2n-p+2}z)}  = \prod_{t=0}^{1}\prod_{s=1}^{l}  \dfrac{(z-(-1)^d\qs^{2n+2l+2p-m-4s-2t})}{(z-(-1)^d\qs^{2n-2l+m+2p-6+4s+2t})}
\end{align*}
Thus we have
\begin{align*}
\dfrac{ d_{l^p,n^m}(z) \displaystyle \prod_{t=0}^{1}\prod_{s=1}^{l} (z-(-1)^{d}\qs^{6n-2l+2p+2+m-2+4s+2t}) }{ \displaystyle \prod_{t=0}^{2p-1}\prod_{s=1}^{l} (z-(-1)^{d}\qs^{2n-2l-2+m-2p+4s+2t})}
\times  \prod_{t=0}^{1}\prod_{s=1}^{l}  (z-(-1)^d\qs^{2n+2l+2p-m-4s-2t}) \in \ko[z^{\pm 1}].
\end{align*}
Since
\begin{align*}
& 6n-2l+2p+2+m-2+4s'+2t' \ne  2n-2l-2+m-2p+4s+2t    \\
& \iff 2n+2p \ne  -1+2(s-s')+(t-t') 
\end{align*} 
for $1 \le s,s' \le l$, $0 \le t' \le 1$ and $0 \le t \le 2p-1$, we have 
\begin{align*}
\dfrac{ d_{l^p,n^m}(z) \displaystyle \prod_{t=0}^{1}\prod_{s=1}^{l}  (z-(-1)^d\qs^{2n+2l+2p-m-4s-2t}) }{ \displaystyle \prod_{t=0}^{2p-1}\prod_{s=1}^{l} (z-(-1)^{d}\qs^{2n-2l-2+m-2p+4s+2t})}  \in \ko[z^{\pm 1}].
\end{align*}
Then our assertion under the assumption follows with Lemma~\ref{lemma: divisopn B l,n}:  
Namely, we have $(u\seteq m-2p \ge 2(l-1))$
\begin{align*}
&2n+2l+2p-m-4s'-2t' \ne  2n-2l-2+m-2p+4s+2t \\
&\iff 2l-2(s+s')-(t'+t) \ne  u-1. 
\end{align*}
for $1 \le s,s' \le l$, $0 \le t' \le 1$ and $0 \le t \le 2p-1$. 
\end{proof}

\begin{proposition} \label{prop: dlpn2m general B even}
For $1\leq l \le n-1$ and $m,p \in \Z_{>0}$, we have
\[
d_{l^p,n^{2m}}(z) = \prod_{t=0}^{\min(2p,2m)-1}\prod_{s=1}^{l} (z-(-1)^{n+l+p}q^{n-l+\abs{p-m}+2s+t-1}).
\]    
\end{proposition} 

\begin{proof}  
We first assume that $p \le m$.
From the homomorphism obtained from Theorem~\ref{thm: Higher Dorey I}~\eqref{eq: k+l=n homo} with the restriction $\min(k,l)=1$
\[
\Vkm{n^{2m}}   \otimes \Vkm{1^m}_{(-1)^{n+m-1}q^{n}}  \twoheadrightarrow \Vkm{(n-1)^m}_{(-1)^{1+m}q},
\]
we have 
\[
\dfrac{d_{l^p,n^{2m}}(z) d_{l^p,1^m}((-1)^{n+m-1}q^{n}z)}{d_{l^p,(n-1)^{m}}((-1)^{1+m}q)} \times \dfrac{a_{l^p,(n-1)^{m}}((-1)^{1+m}q)}{a_{l^p,n^{2m}}(z) a_{l^p,1^m}((-1)^{n+m-1}q^{n}z)} \in \ko[z^{\pm1}].
\]
By direct computation, we have (set $d=l+n+p$)
\begin{align*}
 \dfrac{a_{l^p,(n-1)^{m}}((-1)^{1+m}q)}{a_{l^p,n^{2m}}(z) a_{l^p,1^m}((-1)^{n+m-1}q^{n}z)} = \prod_{t=0}^{p-1} \dfrac{(z-(-1)^{d}q^{l-n-1-m+p-2t} )}{(z-(-1)^{d}q^{l-n+m-p+1+2t} )}
\end{align*}
and  
\begin{align*}
& \dfrac{d_{l^p,n^{2m}}(z) d_{l^p,1^m}((-1)^{n+m-1}q^{n}z)}{d_{l^p,(n-1)^{m}}((-1)^{1+m}q)}
=\dfrac{d_{l^p,n^{2m}}(z) \displaystyle \prod_{t=0}^{p-1}  (z-(-1)^{d} q^{l-n+1+m-p+2t}) (z-(-1)^{d}q^{n-l+m-p+2t})}
{\displaystyle  \prod_{t=0}^{p-1}\prod_{s=1}^{l} (z-(-1)^{d}q^{n-2-l+m-p+2(s+t)}) (z-(-1)^{d}q^{n-l+m-p-1+2(s+t)}) }
\end{align*}

Then we have
\begin{align*}
&\dfrac{d_{l^p,n^{2m}}(z) \displaystyle \prod_{t=0}^{p-1} (z-(-1)^{d}q^{l-n-1-m+p-2t} ) }{\displaystyle  \prod_{t=0}^{p-1}\left( \prod_{s=2}^{l} (z-(-1)^{d}q^{n-2-l+m-p+2(s+t)})
\prod_{s=1}^{l} (z-(-1)^{d}q^{n-l+m-p-1+2(s+t)}) \right) }  \in \ko[z^{\pm1}]\allowdisplaybreaks \\
& \Rightarrow  \dfrac{d_{l^p,n^{2m}}(z) }{\displaystyle  \prod_{t=0}^{p-1} \left(  \prod_{s=1}^{l-1} (z-(-1)^{d}q^{n-l+m-p+2(s+t)}) \prod_{s=1}^{l}  (z-(-1)^{d}q^{n-l+m-p-1+2(s+t)}) \right) }  \in \ko[z^{\pm1}],
\end{align*}
which is equivalent to
\begin{align} \label{eq: final step1 B nl}
\dfrac{d_{l^p,n^{2m}}(z)  \displaystyle \prod_{t=0}^{p-1}  (z-(-1)^{d}q^{n+l+m-p+2t}) }{\displaystyle   \prod_{t=0}^{2p-1}  \prod_{s=1}^{l} (z-(-1)^{d}q^{n-l+m-p+2s+t-1})  } \in \ko[z^{\pm1}].
\end{align}

On the other hand, fusion rule  
\[
\Vkm{n^{2m}} \otimes \Vkm{n^2}_{(-\qs)^{-2m-2}} \twoheadrightarrow \Vkm{n^{2m+2}}_{(-\qs)^{-2}},
\]
implies that
\begin{align}
&\dfrac{d_{l^p,n^{2m}}(z) \displaystyle \prod_{s=1}^{l} (z-(-1)^{n+l+p}q^{n-l+p+m+2s-1})(z-(-1)^{n+l+p}q^{n-l+p+m+2s})  } 
{\displaystyle\prod_{t=0}^{2p-1}\prod_{s=1}^{l} (z-(-1)^{d}q^{n-l+m-p+2s+t+1})}   \in \ko[z^{\pm1}]  \nonumber\\
&\iff \dfrac{d_{l^p,n^{2m}}(z)}{\displaystyle\prod_{t=2}^{2p-1}\prod_{s=1}^{l} (z-(-1)^{d}q^{n-l+m-p+2s+t-1})}   \in \ko[z^{\pm1}]. \label{eq: final step2 B nl}
\end{align}
which follows from the descending induction on $|2m-p|$. Then our assertion holds by comparing~\eqref{eq: final step1 B nl} and~\eqref{eq: final step2 B nl} as
\[
n+l+m-p+2t \ne n-l+m-p+2s-1, n-l+m-p+2s' 
\]
for $0 \le t \le 2p-1$ and $1 \le s,s' \le l$, unless $t=0,s'=l$.
Then by the universal coefficient formula, our assertion follows. 
The remaining cases can be proved in a similar way.
\end{proof}

\begin{proof}[Proof of {\rm Theorem~\ref{thm:denominators_untwisted}} for $1 \leq l < n = k$ in type $B_n^{(1)}$]
It suffices to show that when $m$ is odd. Thus shall consider only $d_{l^p,n^{2m+1}}(z)$ and hence we have to prove that
\[
d_{l^p,n^{2m+1}}(z) = \prod_{t=0}^{\min(2p,2m+1)-1}\prod_{s=1}^{l} (z-(-1)^{n+l+p+1}(\qs)^{2n-2l+\abs{2p-2m-1}+4s+2t-2}).
\]
Assume first that $2p \le 2m+1$. 
%
By considering fusion rule
\[
\Vkm{n^{2m+1}} \otimes \Vkm{n}_{(-\qs)^{-2m-2}} \twoheadrightarrow \Vkm{n^{2m+2}}_{(-\qs)^{-1}},
\]
first, we have 
\begin{align}
\dfrac{d_{l^p,n^{2m+1}}(z)}{\displaystyle \prod_{t=1}^{2p-1}\prod_{s=1}^{l} (z-(-1)^{n+l+p+1}(\qs)^{2n-2l+2m-2p+4s+2t-1})} \in \ko[z^{\pm 1}]. \label{eq: B lpnm step 1}
\end{align}
Thus we can obtain 
\begin{align*}
 d_{l^p,n^{2m+1}}(z) &=    \prod_{t=1}^{2p-1}\prod_{s=1}^{l} (z-(-1)^{n+l+p+1}(\qs)^{2n-2l+2m-2p+4s+2t-1}) \\
 & \hspace{15ex}\times   \prod_{s=1}^{l} (z-(-1)^{n+l+p+1}(\qs)^{2n-2l+2m-2p+4s-1})^{\epsilon_s}, 
\end{align*}
for $\epsilon_s \in \{0,1\}$.
By classical T-system
$$
\Vkm{n^{2m+1}}_{(-\qs)^{-1}} \otimes \Vkm{n^{2m+1}}_{(-\qs)} \twoheadrightarrow \Vkm{(n-1)^{m+1}}_{(-1)^m} \otimes \Vkm{(n-1)^{m}}_{(-1)^{m+1}},
$$
we have
\[
\dfrac{d_{l^p,n^{2m+1}}((-\qs)^{-1}z) d_{l^p,n^{2m+1}}((-\qs)z)}{d_{l^p,(n-1)^{m+1}}((-1)^{m}z)d_{l^p,(n-1)^{m}}((-1)^{m+1}z)} \times 
\dfrac{a_{l^p,(n-1)^{m+1}}((-1)^{m}z)a_{l^p,(n-1)^{m}}((-1)^{m+1}z)}{a_{l^p,n^{2m+1}}(z) a_{l^p,n^{2m+1}}((-\qs)z)} \in \ko[z^{\pm1}].
\]
Since
$$
\dfrac{a_{l^p,(n-1)^{m+1}}((-1)^{m}z)a_{l^p,(n-1)^{m}}((-1)^{m+1}z)}{a_{l^p,n^{2m+1}}(z) a_{l^p,n^{2m+1}}((-\qs)z)} =1,
$$
we have $(d\seteq n+l+p+1)$  
\begin{align*}
& \dfrac{d_{l^p,n^{2m+1}}((-\qs)^{-1}z) d_{l^p,n^{2m+1}}((-\qs)z)}{d_{l^p,(n-1)^{m+1}}((-1)^{m}z)d_{l^p,(n-1)^{m}}((-1)^{m+1}z)} \\
& =   \prod_{t=1}^{2p-1}\prod_{s=1}^{l} (z-(-1)^{d}(\qs)^{2n-2l+2m-2p+4s+2t}) \times   \prod_{s=1}^{l} (z-(-1)^{d}(\qs)^{2n-2l+2m-2p+4s})^{\epsilon_s} \\
& \times  \prod_{t=1}^{2p-1}\prod_{s=1}^{l} (z-(-1)^{d}(\qs)^{2n-2l+2m-2p+4s+2t-2})  \times   \prod_{s=1}^{l} (z-(-1)^{d}(\qs)^{2n-2l+2m-2p+4s-2})^{\epsilon_s} \\
& \times \dfrac{1}
{ \displaystyle \prod_{t=0}^{p-1}\prod_{s=1}^{l} (z-(-1)^d\qs^{2n-2l+2m+2-2p+4(s+t)}) (z-(-1)^d\qs^{2n+4-2l+2m-2p+4(s+t)})}  \\
& \times \dfrac{1}{\displaystyle\prod_{t=0}^{p-1}\prod_{s=1}^{l} (z-(-1)^d\qs^{2n-2l+2m-2-2p+4(s+t)}) (z-(-1)^d\qs^{2n-2l+2m-2p+4(s+t)})}
\end{align*}
which implies $\epsilon_s=1$ for all $1\le s \le l$.

The remaining cases can be proved in a similar way.
\end{proof}

\subsubsection{$d_{n^m,n^p}(z)$}
Recall
\begin{align}
d_{n,n}(z)=\displaystyle \prod_{s=1}^{n} \big(z-(\qs)^{4s-2}\big), \quad
a_{n,n}(z) = \prod_{s=1}^n \dfrac{\PNZ{4n+4s-4}{0}\PNZ{4n-4s}{0}}{\PNZ{4s-2}{0}\PNZ{8n-4s-2}{0}}.
\end{align}

The following lemma can be easily proved by using homomorphisms
\[
\Vkm{1}_{(-1)^{n}\qs^{-2n+2}} \otimes \Vkm{n-1}_{\qs^2} \twoheadrightarrow \Vkm{n^2} \quad \text{ and } \quad
\Vkm{n^2} \otimes \Vkm{1}_{(-1)^{n}\qs^{2n}}  \twoheadrightarrow  \Vkm{n-1}_{\qs^2}.
\]

\begin{lemma}
For any $m \in \ZZ_{>0}$, we have
\[
d_{n,n^2}(z) =  \prod_{s=1}^{n} (z-(-\qs)^{4s-1} ).
\]
\end{lemma}

\begin{lemma}
For any $p \in \Z_{\ge 2}$, we have
\[
d_{n^p,n^2}(z) =  \prod_{s=1}^{n} (z-(-\qs)^{4s+p-4} )(z-(-\qs)^{4s+p-2} ).
\]
\end{lemma}

\begin{proof}  
By the homomorphisms
\[
\Vkm{1}_{(-1)^{n}\qs^{-2n+2}} \otimes \Vkm{n-1}_{\qs^2} \twoheadrightarrow \Vkm{n^2} \quad \text{ and } \quad
\Vkm{n^2} \otimes \Vkm{1}_{(-1)^{n}\qs^{2n}} \twoheadrightarrow  \Vkm{n-1}_{\qs^2},
\]
the induction hypothesis says that we have 
\begin{align*}
& 
\dfrac{  \displaystyle \prod_{s=1}^{n} (z-(-\qs)^{p-4+4s})(z-(-\qs)^{p-2+4s})    }{d_{n^p,n^2}(z)} \in \ko[z^{\pm 1}], \\
& \dfrac{ d_{n^p,n^{2}}(z) (z-(-\qs)^{4n+p}) }{ \displaystyle\prod_{s=1}^{n} (z-(-\qs)^{p+4s-4})\prod_{s=1}^{n-1}(z-(-\qs)^{p+4s-2}) }
 \times (z-(-\qs)^{-p})(z-(-\qs)^{-p+2})  \in \ko[z^{\pm 1}],
\end{align*}
respectively. Note that
\begin{align*}
\Vkm{n^p}_{(-\qs)^{p+4n-2}} & \iso \head(\Vkm{n}_{(-\qs)^{2p+4n-3}} \otimes \cdots \otimes \Vkm{n}_{(-\qs)^{4n+1}} \otimes \Vkm{n}_{(-\qs)^{4n-1}} ) \\
& \iso \head( \Vkm{n^{p-1}}_{(-\qs)^{p+4n-1}} \otimes \Vkm{n}_{(-\qs)^{4n-1}} ).
\end{align*}
Then \bnum
\item $\de(\scrD^{s} \Vkm{n^2},\Vkm{n^{p-1}}_{(-\qs)^{p+4n-1}}) =0$ $(s=0,-1)$
by Proposition~\ref{prop: de less than equal to},  
\item $\de(\Vkm{n^2},\Vkm{n}_{(-\qs)^{4n-1}})=1$, 
\ee
as $d_{n,n^2}(z) =  \prod_{s=1}^{n} (z-(-\qs)^{4s-1} )$. 
Hence $$\de(\Vkm{n^2},\Vkm{n^p}_{(-\qs)^{p+4n-2}})=1$$ by Proposition~\ref{lem: de=de}~\eqref{it: de=de 2}, which completes the assertion. 
\end{proof}

\begin{lemma}
For any $m \in \Z_{\ge 3}$, we have
\[
d_{n,n^m}(z) =  \prod_{s=1}^{n} (z-(-\qs)^{4s+m-3} ).
\]
\end{lemma}

\begin{proof}
By fusion rules
\[
\Vkm{n^{m-1}}_{(-\qs)} \otimes \Vkm{n}_{(-\qs)^{1-m}} \twoheadrightarrow \Vkm{n^m} \quad \text{and} \quad
\Vkm{n}_{(-\qs)^{m-1}} \otimes \Vkm{n^{m-1}}_{(-\qs)^{-1}} \twoheadrightarrow \Vkm{n^m},
\]
we have
\begin{align*}
\dfrac{ d_{n,n^{m-1}}((-\qs)z)d_{n,n}((-\qs)^{1-m}z)}{d_{n,n^m}(z)} & \equiv
\dfrac{\displaystyle \prod_{s=1}^{n} (z-(-\qs)^{4s+m-5} ) \times \prod_{s=1}^{n} (z-(-\qs)^{4s+m-3} ) }{d_{n,n^m}(z)} \in \ko[z^{\pm 1}], \allowdisplaybreaks \\
\dfrac{ d_{n,n}((-\qs)^{m-1}z)d_{n,n^{m-1}}((-\qs)^{-1}z)}{d_{n,n^m}(z)} & \equiv
\dfrac{\displaystyle \prod_{s=1}^{n} (z-(-\qs)^{4s-m-1} ) \times \prod_{s=1}^{n} (z-(-\qs)^{4s+m-3} ) }{d_{n,n^m}(z)} \in \ko[z^{\pm 1}].
\end{align*}
Note there are double roots in the numerator only at $m \equiv_2 0$.

On the other hand, the  fusion homomorphism
\[
\Vkm{n^{m-2}}_{(-\qs)^2} \otimes \Vkm{n^2}_{(-\qs)^{2-m}} \twoheadrightarrow \Vkm{n^m},
\]
implies that we have
\[
\dfrac{ d_{n^{m-2},n}((-\qs)^2z)d_{n^2,n}(-\qs^{2-m}z)}{d_{n^m,n}(z)} =
\dfrac{    \displaystyle\prod_{s=1}^{n} (z-(-\qs)^{4s+m-7} ) \prod_{s=1}^{n} (z-(-\qs)^{4s+m-3} ) }{d_{n^m,n}(z)} \in \ko[z^{\pm 1}],
\]
and note there are double roots in the numerator only at $m \equiv_2 1$. 
Therefore, every root in $d_{n,n^m}(z) = d_{n^m,n}(z)$ can appear with multiplicity at most $1$.
Finally, by the homomorphism
\[
 \Vkm{n^m}  \otimes  \Vkm{n}_{(-\qs)^{4n+m-3}}  \twoheadrightarrow \Vkm{n^{m-1}}_{(-\qs)^{-1}}
\]
we have
\[
\dfrac{ d_{n,n^{m}}(z) \displaystyle \prod_{s=1}^{n}  (z-(\qs)^{4n+m-5+4s})  }{  \displaystyle \prod_{s=1}^{n} (z-(-\qs)^{m-3+4s})}    \in \ko[z^{\pm 1}],
\]
and thus our assertion follows.
\end{proof}

As Lemma~\ref{lemma: division A}, we have the following:

\begin{lemma}\label{lemma: division B n}
For $1\leq k,l \le n-1$ and $m,p \ge 1$, we have
\[
\text{$d_{n^p,n^m}(z)$ divides $\prod_{t=0}^{\min(p,m)-1}\prod_{s=1}^{n} (z-(-\qs)^{4s+2t-2+\abs{p-m}} )$.}
\]
\end{lemma}

\begin{lemma} \label{lemma: dnpnmB}
For any $\abs{m-p} \ge 2n-2$, we have
\[
d_{n^p,n^m}(z) =  \prod_{t=0}^{\min(p,m)-1}\prod_{s=1}^{n} (z-(-\qs)^{4s+2t-2+\abs{p-m}} ).
\]
\end{lemma}

\begin{proof}
Without loss of generality, we can assume that $2 \le p \le m$.
Then our assertion can be re-written as follows:
\[
d_{n^p,n^m}(z) =  \prod_{t=0}^{p-1}\prod_{s=1}^{n} (z-(-\qs)^{4s+2t-2+m-p} ).
\]

By the homomorphism
\[
 \Vkm{n^p}  \otimes  \Vkm{n}_{(-\qs)^{4n+p-3}}  \twoheadrightarrow \Vkm{n^{p-1}}_{(-\qs)^{-1}}
\]
and descending induction on $m-p$, we have
\begin{align*}
\dfrac{ d_{n^p,n^m}(z)d_{n,n^m}((-\qs)^{-4n-p+3}z)}{d_{n^{p-1},n^{m}}((-\qs)z)} \times
\dfrac{a_{n^{p-1},n^{m}}((-\qs)z)}{ a_{n^p,n^m}(z)a_{n,n^m}((-\qs)^{-4n-p+3}z)} \in \ko[z^{\pm 1}].
\end{align*}

By descending induction on $m-p$, we have
\begin{align*}
& \dfrac{ d_{n^p,n^m}(z)d_{n,n^m}((-\qs)^{-4n-p+3}z)}{d_{n^{p-1},n^{m}}((-\qs)z)}
= \dfrac{ d_{n^p,n^m}(z) \displaystyle \prod_{s=1}^{n} (z-(-\qs)^{4n+4s+m+p-6} ) }{ \displaystyle\prod_{t=0}^{p-2}\prod_{s=1}^{n} (z-(-\qs)^{4s+2t-2+m-p} )  }
\end{align*}
and, the direct computation yields 
\begin{align*}
&\dfrac{a_{n^{p-1},n^{m}}((-\qs)z)}{ a_{n^p,n^m}(z)a_{n,n^m}((-\qs)^{-4n-p+3}z)} =  \prod_{s=1}^n \dfrac{(z-(-1)^{m-p}\qs^{p-m-4+4s})}{(z-(-1)^{m-p}\qs^{p+m-4+4s})}.
\end{align*}
Thus we have
\begin{align*}
\dfrac{ d_{n^p,n^m}(z) \displaystyle \prod_{s=1}^{n} (z-(-\qs)^{4n+m+p-6+4s} ) }{ \displaystyle\prod_{t=0}^{p-1}\prod_{s=1}^{n} (z-(-\qs)^{m-p-2+4s+2t} )  }
\times \prod_{s=1}^n (z-(-\qs)^{-m+p-4+4s})\in \ko[z^{\pm1}].
\end{align*}
Since $4n+m+p-6+4s' \ne m-p-2+4s+2t$,  
the above equation can be written as
\begin{align*}
\dfrac{ d_{n^p,n^m}(z) \displaystyle \prod_{s=1}^n (z-(-\qs)^{-m+p-4+4s}) }{ \displaystyle\prod_{t=0}^{p-1}\prod_{s=1}^{n} (z-(-\qs)^{m-p-2+4s+2t} )  }
\in \ko[z^{\pm1}].
\end{align*}
Since $-m+p-4+4s' \ne m-p-2+4s+2t$ under our assumption, our claim follows with Lemma~\ref{lemma: division B n}.  
\end{proof}

\begin{proposition}
For $m,p \in \Z_{\ge1}$, we have
$$
d_{n^p,n^{2m}}(z) = \prod_{t=0}^{\min(p,2m)-1} \prod_{s=1}^n (z-(-\qs)^{|p-2m|+4s+2t-2})
$$
\end{proposition}

\begin{proof}
We assume that $2 \le p \le 2m$.
From the homomorphism obtained from Theorem~\ref{thm: Higher Dorey I}~\eqref{eq: k+l=n homo} with the restriction $\min(k,l)=1$
\[
\Vkm{n^{2m}}   \otimes \Vkm{1^m}_{(-1)^{n+m-1}q^{n}}  \twoheadrightarrow \Vkm{(n-1)^m}_{(-1)^{1+m}q},
\]
we have
\begin{align}
&\dfrac{d_{n^{p},n^{2m}}(z)  \displaystyle\prod_{t=0}^{p-1} (z-(-\qs)^{4n+2m-p+2t})(z-(-\qs)^{2p-2m-2t}) }{ \displaystyle\prod_{t=0}^{p-1}\prod_{s=1}^{n} (z-(-\qs)^{2m-p-2+4s+2t}) } \in \ko[z^{\pm1}]
\allowdisplaybreaks \nonumber\\
&\Rightarrow
\dfrac{d_{n^{p},n^{2m}}(z)  \displaystyle\prod_{t=0}^{p-1} (z-(-\qs)^{4n+2m-p+2t}) }{ \displaystyle\prod_{t=0}^{p-1}\prod_{s=1}^{n} (z-(-\qs)^{2m-p-2+4s+2t}) } \in \ko[z^{\pm1}], \label{eq: B nn even step1}
\end{align}
which guarantees that $\displaystyle\prod_{s=1}^{n} (z-(-\qs)^{2m-p-2+4s})(z-(-\qs)^{2m-p-2+4s+2})$ appears in $d_{n^{p},n^{2m}}(z)$ as many as we want.  

On the other hand, fusion rule
\[
\Vkm{n^{2m}} \otimes \Vkm{n^2}_{(-\qs)^{-2m-2}} \twoheadrightarrow \Vkm{n^{2m+2}}_{(-\qs)^{-2}},
\]
\begin{align}
& \dfrac{d_{n^{p},n^{2m}}(z)  }
{\displaystyle \prod_{t=2}^{p-1}\prod_{s=1}^{n} (z-(-\qs)^{2m-2-p+4s+2t} ) } \in \ko[z^{\pm1}]. \label{eq: B nn even step2}
\end{align}
Here we apply descending induction on $2m-p$. Thus our assertion follows from~\eqref{eq: B nn even step1} and~\eqref{eq: B nn even step2}.

The other case is proved in a similar way.
\end{proof}

\begin{proof}[Proof of {\rm Theorem~\ref{thm:denominators_untwisted}} for $l = k = n$ in type $B_n^{(1)}$]
It suffices to prove that both of $m$ and $p$ are odd; i.e., we have to prove that
\[
 d_{n^p,n^{2m+1}}(z) = \prod_{t=0}^{\min(p,2m+1)-1}\prod_{s=1}^{n} (z-(-\qs)^{|2m+1-p|+4s+2t-2} ),
\]
when $p$ is also odd. Let us assume that $p \le 2m+1$.
From fusion rule
\[
\Vkm{n^{2m+1}} \otimes \Vkm{n}_{(-\qs)^{-2m-2}} \twoheadrightarrow \Vkm{n^{2m+2}}_{(-\qs)^{-1}},
\]
we obtain
\begin{align} \label{eq: final B nn step 1}
& \dfrac{d_{n^p,n^{2m+1}}(z)}{\displaystyle \prod_{t=1}^{p-1}\prod_{s=1}^{n} (z-(-\qs)^{2m-p+4s+2t-1} )  } \in \ko[z^{\pm1}]. 
\end{align} 
On the other hand, the homomorphism
\[
 \Vkm{n^p}  \otimes  \Vkm{n}_{(-\qs)^{4n+p-3}}  \twoheadrightarrow \Vkm{n^{p-1}}_{(-\qs)^{-1}}
\]
implies that
\begin{align*}
\dfrac{ d_{n^p,n^{2m+1}}(z)d_{n,n^{2m+1}}((-\qs)^{-4n-p+3}z)}{d_{n^{p-1},n^{2m+1}}((-\qs)z)} \times
\dfrac{a_{n^{p-1},n^{2m+1}}((-\qs)z)}{ a_{n^p,n^{2m+1}}(z)a_{n,n^{2m+1}}((-\qs)^{-4n-p+3}z)} \in \ko[z^{\pm 1}].
\end{align*}
By direct computation, we have
\begin{align*}
& \dfrac{a_{n^{p-1},n^{2m+1}}((-\qs)z)}{ a_{n^p,n^{2m+1}}(z)a_{n,n^{2m+1}}((-\qs)^{-4n-p+3}z)}
 = \prod_{s=1}^n \dfrac{(z-(-1)^{p+1}(\qs)^{4n-2m+p-1-4s}  )}{(z-(-1)^{p+1}(\qs)^{2m+p-3+4s}  )}
\end{align*}
and
\begin{align*}
& \dfrac{ d_{n^p,n^{2m+1}}(z)d_{n,n^{2m+1}}((-\qs)^{-4n-p+3}z)}{d_{n^{p-1},n^{2m+1}}((-\qs)z)}
 = \dfrac{ d_{n^p,n^{2m+1}}(z) \displaystyle \prod_{s=1}^{n} (z-(-\qs)^{4n+4s+2m+p-5} ) }{ \displaystyle\prod_{t=0}^{p-2}\prod_{s=1}^{n} (z-(-\qs)^{2m-p+4s+2t-1} ) },
\end{align*}
which follows from the descending induction. 
Thus we have
\begin{align}
& 
\dfrac{ d_{n^p,n^{2m+1}}(z) \displaystyle \prod_{s=1}^{n} (z-(-\qs)^{4n+4s+2m+p-5} )(z-(-1)^{p+1}(\qs)^{4n-2m+p-1-4s}  ) }{ \displaystyle\prod_{t=0}^{p-1}\prod_{s=1}^{n} (z-(-\qs)^{2m-p+4s+2t-1} ) } \in \ko[z^{\pm1}]
\nonumber \allowdisplaybreaks\\
& \Rightarrow
 \dfrac{ d_{n^p,n^{2m+1}}(z) \displaystyle \prod_{s=1}^{n}(z-(-1)^{p+1}(\qs)^{4n-2m+p-1-4s}  ) }{ \displaystyle\prod_{t=0}^{p-1}\prod_{s=1}^{n} (z-(-\qs)^{2m-p+4s+2t-1} ) } \in \ko[z^{\pm1}].
 \label{eq: final B nn step 2}
\end{align}
Thus our assertion hold for odd $p$, since
$
2m-p+4s'-1 \ne 4n-2m+p-1-4s
$
by comparing~\eqref{eq: final B nn step 1} and~\eqref{eq: final B nn step 2}.
\end{proof}

\subsection{Proof for type \texorpdfstring{$C_n^{(1)}$}{Cn(1)}}

In this subsection, we first prove our assertion on $d_{k^m,n^p}(z)$ for $1 \le  k < n$ and on $d_{n^m,n^p}(z)$.
Then we will prove our assertion on $d_{k^{m},l^p}(z)$ for $1 \le k, l <n$, by applying the similar argument in Section~\ref{subsec: A}.

Recall, for $1 \le k ,l \le n$, we have
\begin{equation}
\begin{aligned}\label{eq: C ank dnk}
&d_{k,l}(z)= \displaystyle \prod_{s=1}^{ \min(k,l,n-k,n-l)}
\big(z-(-\qs)^{|k-l|+2s}\big)\prod_{i=1}^{ \min(k,l)}\big(z-(-\qs)^{2n+2-k-l+2s}\big), \\
& a_{k,l}(z) \equiv
\dfrac{\PS{|k-l|}\PS{2n+2-k-l}\PS{2n+2+k+l}\PS{4n+4-|k-l|}}{\PS{k+l}\PS{2n+2-k+l}\PS{2n+2+k-l}\PS{4n+4-k-l}},
\end{aligned}
\end{equation}
Note also that we have surjective homomorphisms
\[
\Vkm{k-1}_{(-\qs)^{-1}} \otimes  \Vkm{1}_{(-\qs)^{k-1}} \twoheadrightarrow \Vkm{k} \quad  \text{ and } \quad
\Vkm{1}_{(-\qs)^{1-l}} \otimes \Vkm{l-1}_{-\qs} \twoheadrightarrow \Vkm{l}
\]
for all $1< k,l \le n$, as Dorey's rule in Theorem~\ref{thm: Dorey}, and surjective homomorphisms
\[
\Vkm{n}_{(-\qs)^{-1-n+k}} \otimes \Vkm{n}_{(-\qs)^{n+1-k}} \twoheadrightarrow \Vkm{k^2}
\]
given in~\eqref{eq: new homom C} for $1 \le k <n$.

\subsubsection{$d_{l^m,n^p}(z)$ for $1 \le l <n$}
As Lemma~\ref{lem: d1k2 B}, we can prove the following lemma:

\begin{lemma}
\label{lemma: d1n2 type C}
We have
$
d_{1,n^2}(z) = (z + (-\qs)^{n+5}).
$
\end{lemma}

\begin{proof}
By~\eqref{eq: C ank dnk}, we have $d_{1,n}(z) = \big(z-(-\qs)^{n+3}\big)$.
Hence it is enough to consider
\begin{align*}
{\rm (i)} & \ \  \de( \Vkm{1}, \Vkm{n^2}_{-(-\qs)^{n+5}}) =\de( \Vkm{1}, \Vkm{n}_{(-\qs)^{n+3}} \hconv \Vkm{n}_{(-\qs)^{n+7}} )    \\
{\rm (ii)} & \ \  \de( \Vkm{1}, \Vkm{n^2}_{-(-\qs)^{n+1}}) = \de( \Vkm{1}, \Vkm{n}_{(-\qs)^{n-1}} \hconv \Vkm{n}_{(-\qs)^{n+3}} )  
\end{align*}

\noindent
{\rm (i)} Note that the sequences of root modules 
$$( \Vkm{1}, \Vkm{n}_{(-\qs)^{n+3}} , \Vkm{n}_{(-\qs)^{n+7}} ) \qtq ( \Vkm{n}_{(-\qs)^{n+3}} , \Vkm{n}_{(-\qs)^{n+7}},\Vkm{1} )$$ are normal, since 
$\de( \scrD \Vkm{1}, \Vkm{n}_{(-\qs)^{n+7}})=0$
and $\de( \Vkm{1}, \Vkm{n}_{(-\qs)^{n+7}})=0$, respectively. Hence Lemma~\ref{lem: normal seq d}~\eqref{it: de +}~{\rm (v)} says that
$$
\de( \Vkm{1}, \Vkm{n}_{(-\qs)^{n+3}} \hconv \Vkm{n}_{(-\qs)^{n+7}} ) = \de( \Vkm{1}, \Vkm{n}_{(-\qs)^{n+3}}) + \de( \Vkm{1}, \Vkm{n}_{(-\qs)^{n+7}} ) =1.
$$

\noindent
{\rm (ii)} Note that
$$
\rch(\Vkm{1})=\range{0} \qtq \exrch_1( \Vkm{n^2}_{(-\qs)^{n+1}}) = \range{-2,2n+4}. 
$$
Hence Theorem~\ref{thm: i-box commute} implies 
$$
\de( \Vkm{1}, \Vkm{n^2}_{-(-\qs)^{n+1}} )  =0,
$$
which completes our assertion.
\end{proof}

With Lemma~\ref{lemma: d1n2 type C}, we can prove the following lemma.

\begin{lemma}
For $m \ge 1$ and $1 \le l <n$, we have
\[
d_{l,n^m}(z) = \prod_{s=1}^{l} (z-(-1)^{(n+m+l+1)}\qs^{n-l+2m+2s}).
\]
\end{lemma}

\begin{lemma}
For $p \ge 1$ and $1 \le l <n$, we have
\begin{equation} \label{eq: dlpn type C}
d_{l^p,n}(z) =  \prod_{t=0}^{\min(p,2)-1} \prod_{s=1}^{l} (z-(-1)^{(n+p+l+1)}\qs^{n+1-l+|2-p|+2s+2t}).
\end{equation}
\end{lemma}

\begin{proof}
Our assertion is known for $p=1$. Therefore, we assume $p \ge 2$.
By fusion rule,
\[
\Vkm{l}_{(-\qs)} \otimes \Vkm{l}_{(-\qs)^{-1}} \twoheadrightarrow \Vkm{l^2},
\]
we have
\begin{align*}
\dfrac{ d_{l,n}((-\qs)z) d_{l,n}((-\qs)^{-1}z) }{d_{l^2,n}(z)}
= \dfrac{ \displaystyle\prod_{t=0}^{1} \prod_{s=1}^{l} (z-(-1)^{n+l}\qs^{n+1-l+2s+2t}).  }{d_{l^2,n}((-\qs)z)} \in \ko[z^{\pm1}].
\end{align*}
On the other hand, the homomorphism
\[
\Vkm{n} \otimes \Vkm{1}_{(-\qs)^{n+3}} \twoheadrightarrow  \Vkm{n-1}_{(-\qs)}
\]
implies that we have  
\begin{align*}
& \dfrac{ d_{l^2,n}(z) (z-(-1)^{n+l}\qs^{n-l+1}) (z-(-1)^{n+l}\qs^{3n-l+7})  }
{ \displaystyle\prod_{s=1}^{l}  (z-(-1)^{n+l}\qs^{n-l+1+2s})(z-(-1)^{n+l}\qs^{n-l+3+2s} ) }
\in \ko[z^{\pm 1}],
\end{align*}
which implies our assertion when $p=2$.
For $p \ge 3$, we can apply the similar argument.
\end{proof}

As Lemma~\ref{lemma: division A} and Lemma~\ref{lemma: division B}, we have the following analogous statement by splitting it into the cases when $p \leq 2m$ and $p > 2m$.

\begin{lemma}\label{lemma: division C l n}
For $m,p \ge 1$ and $1 \le l <n$, we have
\[
d_{l^p,n^m}(z) \text{ divides} \prod_{t=0}^{\min(p,2m)-1} \prod_{s=1}^{l} (z-(-1)^{n+p+l+m}\qs^{n+1-l+|2m-p|+2s+2t}).
\]
\end{lemma}

\begin{proof}
(a) Let us assume that $p \le 2m$.
By fusion rule
\[
\Vkm{l^{p-1}}_{(-\qs)} \otimes \Vkm{l}_{(-\qs)^{1-p}} \twoheadrightarrow \Vkm{l^p},
\]
we have
\[
\dfrac{d_{l^{p-1},n^m}((-\qs)z)d_{l,n^m}((-\qs)^{1-p}z)}{d_{l^p,n^m}} \in \ko[z^{\pm 1}].
\] 
By induction,
\[
d_{l^{p-1},n^m}((-\qs)z) \text{ divides } \prod_{t=0}^{p-2} \prod_{s=1}^{l} (z-(-1)^{n+p+l+m}\qs^{n+1-l+2m-p+2s+2t})
\]
and we know
\[
d_{l,n^m}((-\qs)^{1-p}z) = \prod_{s=1}^{l} (z-(-1)^{n+m+l+1}\qs^{n-l+2m+p-1+2s}).
\]
Thus our assertion follows. The remaining case can be proved in a similar way.
\end{proof}

\begin{lemma} \label{lemma: dlpnm C}
For $m ,p \ge 1$  with $|2m - p| \ge l -1$ and $1 \le l <n$,  set $d=n+p+l+m$. Then
we have
\begin{align*}
d_{l^p,n^m}(z) = \prod_{t=0}^{\min(p,2m)-1} \prod_{s=1}^{l} (z-(-1)^d\qs^{n+1-l+|2m-p|+2s+2t}).
\end{align*}
\end{lemma}

\begin{proof}
Note that we have proved when $\min(m,p)=1$. Thus we will consider when $\min(m,p)>1$.

\noindent
(a) Let us assume that $p \le 2m$.
By the homomorphism
\[
 \Vkm{l^p} \otimes \Vkm{l}_{(-\qs)^{2n+p+1}} \twoheadrightarrow \Vkm{l^{p-1}}_{(-\qs)^{-1}},
\]
we have
\begin{align*}
\dfrac{ d_{l^p,n^{m}}(z)d_{l,n^{m}}((-\qs)^{-2n-p-1}z)}{d_{l^{p-1},n^m}((-\qs)z)} \times
\dfrac{a_{l^{p-1},n^m}((-\qs)z)}{a_{l^p,n^{m}}(z)a_{l,n^{m}}((-\qs)^{-2n-p-1}z)} \in \ko[z^{\pm 1}].
\end{align*}

By the induction on $p$, we have
\begin{align*}
&\dfrac{ d_{l^p,n^{m}}(z)d_{l,n^{m}}((-\qs)^{-2n-p-1}z)}{d_{l^{p-1},n^m}((-\qs)z)}
= \dfrac{ d_{l^p,n^{m}}(z) \displaystyle\prod_{s=1}^{l} (z-(-1)^d\qs^{3n-l+2m+p+1+2s}) }{ \displaystyle\prod_{t=0}^{p-2} \prod_{s=1}^{l} (z-(-1)^d\qs^{n+1-l+2m-p+2s+2t})}
\end{align*}
and,  direct computation yields 
\begin{align*}
&\dfrac{a_{l^{p-1},n^m}((-\qs)z)}{a_{l^p,n^{m}}(z)a_{l,n^m}((-\qs)^{-2n-p-1}z)} = \prod_{s=1}^{l} \dfrac{(z-(-1)^d\qs^{n-l+p-1-2m+2s})}{(z-(-1)^d\qs^{n-l+p-1+2m+2s})}.
\end{align*}

Thus we have
\begin{align*}
& \dfrac{ d_{l^p,n^{m}}(z) \displaystyle\prod_{s=1}^{l} (z-(-1)^d\qs^{3n-l+2m+p+1+2s}) }{ \displaystyle\prod_{t=0}^{p-1} \prod_{s=1}^{l} (z-(-1)^d\qs^{n+1-l+2m-p+2s+2t})} \times \prod_{s=1}^{l} (z-(-1)^d\qs^{n-l+p-1-2m+2s}) \in \ko[z^{\pm1}].
\end{align*}
Since 
\begin{align*}
3n-l+2m+p+1+2s'  \ne n+1-l+2m-p+2s+2t \iff 2n+2p & \ne 2(s-s')+2t  
\end{align*}
for $1\le s,s' \le l$ and $0\le t \le p-1$, the above implies
\begin{align}\label{eq: 2m-p ge l}
\dfrac{ d_{l^p,n^{m}}(z) \displaystyle \prod_{s=1}^{l} (z-(-1)^d\qs^{n-l+p-1-2m+2s}) }
{ \displaystyle\prod_{t=0}^{p-1} \prod_{s=1}^{l} (z-(-1)^d\qs^{n-l+2m-p+1+2s+2t})} \in \ko[z^{\pm1}],
\end{align}
which proves our assertion under the assumption with Lemma~\ref{lemma: division C l n}, since $(u\seteq 2m-p \ge l-1)$
$$
n-l-1-u+2s' \ne n-l+u+1+2s+2t \iff 2(s'-s)-2t \ne 2u+2
$$
for $1\le s,s' \le l$ and $0 \le t \le p-1$.  

\noindent
(b) Assume that $2m < p$. By the homomorphism
\[
\Vkm{n^m} \otimes \Vkm{n}_{(-1)^{1-m}\qs^{2n+2m}} \twoheadrightarrow   \Vkm{n^{m-1}}_{(-\qs^2)^{-1}},
\]
we have by the induction on $m$
\begin{align*}
\dfrac{ d_{l^p,n^m}(z)d_{l^p,n}((-1)^{1-m}\qs^{-2n-2m}z)}{d_{l^p,n^{m-1}}((-\qs^2)z)} \times
\dfrac{a_{l^p,n^{m-1}}((-\qs^2)z)}{ a_{l^p,n^m}(z)a_{l^p,n}((-1)^{1-m}\qs^{-2n-2m}z)} \in \ko[z^{\pm 1}].
\end{align*}

By the descending induction on $p-2m$, we have
\begin{align*}
\dfrac{ d_{l^p,n^m}(z)d_{l^p,n}((-1)^{1-m}\qs^{-2n-2m}z)}{d_{l^p,n^{m-1}}((-\qs^2)z)}
= \dfrac{ d_{l^p,n^m}(z) \displaystyle \prod_{t=0}^{p-1} \prod_{s=1}^{l} (z-(-1)^{n+p+l+m}\qs^{3n+1-l+p+2m-2+2s+2t})}{\displaystyle\prod_{t=0}^{2m-3} \prod_{s=1}^{l} (z-(-1)^{n+p+l+m}\qs^{n+1-l+p-2m+2s+2t})}
\end{align*}
and the direct computation yields 
\begin{align*}
& \dfrac{a_{l^p,n^{m-1}}((-\qs^2)z)}{ a_{l^p,n^m}(z)a_{l^p,n}((-1)^{1-m}\qs^{-2n-2m}z)}= \prod_{t=0}^{1} \prod_{s=1}^{l} \dfrac{(z-(-1)^d\qs^{n+l-p+2m+1-2s-2t})}{(z-(-1)^d\qs^{n-l+p+2m-3+2s+2t})}.
\end{align*}
Thus we have
\begin{align*}
& \dfrac{ d_{l^p,n^m}(z) \displaystyle \prod_{t=0}^{p-1} \prod_{s=1}^{l} (z-(-1)^{d}\qs^{3n+1-l+p+2m-2+2s+2t})}{\displaystyle\prod_{t=0}^{2m-1} \prod_{s=1}^{l} (z-(-1)^{d}\qs^{n+1-l+p-2m+2s+2t})}
\times \prod_{t=0}^{1} \prod_{s=1}^{l} (z-(-1)^d\qs^{n+l-p+2m+1-2s-2t}) \in \ko[z^{\pm1}].
\end{align*}
Since  
\begin{align*}
& 3n+1-l+p+2m-2+2s+2t \ne n+1-l+p-2m+2s'+2t'    \\
\iff & 2n+4m \ne 2(s'-s)+2(t'-t)+2
\end{align*}
for $1 \le s,s' \le l$, $0\le t' \le 2m-1$ and $0 \le t \le p- 1$, we have
\begin{align}\label{eq: p-2m ge l}
\dfrac{ d_{l^p,n^m}(z)\displaystyle \prod_{t=0}^{1} \prod_{s=1}^{l} (z-(-1)^d\qs^{n+l-p+2m+1-2s-2t})}{\displaystyle\prod_{t=0}^{2m-1} \prod_{s=1}^{l} (z-(-1)^{d}\qs^{n+1-l+p-2m+2s+2t})}\in \ko[z^{\pm1}],
\end{align}
which implies our assertion under our assumption with Lemma~\ref{lemma: division C l n}, since $(u\seteq p-2m \ge l-1)$
$$
n+l-p+2m+1-2s-2t \ne n+1-l+p-2m+2s'+2t' \iff 2l-2(s+s')-2(t+t') \ne 2u,
$$
for $1 \le s,s'\le l$, $0\le t \le 1$ and $0\le t' \le 2m-1$.  
\end{proof}

\begin{proof}[Proof of {\rm Theorem~\ref{thm:denominators_untwisted}} for $1 \leq l < n = k$ in type $C_n^{(1)}$]
We first assume that $p \le 2m$.
From the homomorphism in obtained from Theorem~\ref{thm: Higher Dorey I}~\eqref{eq: k+l<n homo} with the restriction $\min(k,l)=1$
\[
 \Vkm{l^p}  \otimes \Vkm{1^p}_{(-\qs)^{2n-l+3}} \twoheadrightarrow  \Vkm{(l-1)^p}_{(-\qs)},
\]
we obtain 
\[
\dfrac{d_{l^p,n^m}( z ) d_{1^p,n^m}((-\qs)^{2n-l+3}z)}{d_{(l-1)^p,n^m}((-\qs)z)} \times \dfrac{a_{(l-1)^p,n^m}((-\qs)z)}{a_{l^p,n^m}( z ) a_{1^p,n^m}((-\qs)^{2n-l+3}z)} \in \ko[z^{\pm1}].
\]
By direct computation, (set $d=n+p+l+m$)
\begin{align*}
\dfrac{a_{(l-1)^p,n^m}((-\qs)z)}{a_{l^p,n^m}( z ) a_{1^p,n^m}((-\qs)^{2n-l+3}z)} = \prod_{t=0}^{p-1} \dfrac{(z-(-1)^d\qs^{-n+l-2m+p-3-2t})}{(z-(-1)^d\qs^{l-n+2m-p-1+2t})}.
\end{align*}
Furthermore, by an induction on $l$, we have 
\begin{align*}
& \dfrac{d_{l^p,n^m}( z ) d_{1^p,n^m}((-\qs)^{2n-l+3}z)}{d_{(l-1)^p,n^m}((-\qs)z)}
= \dfrac{d_{l^p,n^m}( z )\displaystyle \prod_{t=0}^{p-1} (z-(-1)^d\qs^{-n+2m-p+l-1+2t}) }{\displaystyle \prod_{t=0}^{p-1} \prod_{s=1}^{l-1} (z-(-1)^d\qs^{n+1-l+2m-p+2s+2t}) }
\end{align*}
Thus we have
\begin{align}
& \dfrac{d_{l^p,n^m}( z )\displaystyle \prod_{t=0}^{p-1}(z-(-1)^d\qs^{-n+l-2m+p-3-2t})   }{\displaystyle \prod_{t=0}^{p-1} \prod_{s=1}^{l-1} (z-(-1)^d\qs^{n+1-l+2m-p+2s+2t}) } \in \ko[z^{\pm1}] \nonumber\allowdisplaybreaks\\
& \Rightarrow \dfrac{d_{l^p,n^m}( z )    }{\displaystyle \prod_{t=0}^{p-1} \prod_{s=1}^{l-1} (z-(-1)^d\qs^{n+1-l+2m-p+2s+2t}) } \in \ko[z^{\pm1}], \label{eq: C dlpnm step1}
\end{align}
by induction on $l$. Here we use  
$$
-n+l-2m+p-3-2t' \ne n+1-l+2m-p+2s+2t \iff l \ne n+2+(2m-p)+s+(t+t')
$$
for $1 \le s \le l-1$ and $0 \le t,t' \le p-1$. 

On the other hand, fusion rule
\[
\Vkm{n^m} \otimes  \Vkm{n}_{(-\qs^2)^{-m-1}} \twoheadrightarrow \Vkm{n^{m+1}}_{(-\qs^2)^{-1}}
\]
we have
\begin{align}
&
\dfrac{d_{l^p,n^m}(z) \displaystyle    }{ \displaystyle\prod_{t=2}^{p-1} \prod_{s=1}^{l} (z-(-1)^d\qs^{n-l+2m-p+1+2s+2t}) } \in \ko[z^{\pm1}], \label{eq: C dlpnm step2}
\end{align}
which follows from the descending induction on $2m-p$.
Then we have
\[
n-l+2m-p+1+2s+ 2t' \ne n+1+l+2m-p+2t \quad \text{ for } 0 \le t' \le 1, \ 1 \le s \le l, \ 0 \le t \le p-1,
\]
unless $(t',t,s)=(0,0,l)$, $(1,0,l-1)$ or $(1,1,l)$ when we compare~\eqref{eq: C dlpnm step1} and~\eqref{eq: C dlpnm step2}.
Similarly, 
\[
\Vkm{n}_{(-\qs^2)^{m+1}}  \otimes \Vkm{n^m}  \twoheadrightarrow \Vkm{n^{m+1}}_{(-\qs^2)}
\]
yields
$$
\dfrac{d_{l^p,n^m}(z) \displaystyle  \prod_{t=0}^{1} \prod_{s=1}^{l} (z-(-1)^{d}\qs^{n-l+p-2m-3+2s+2t})   }
{ \displaystyle \prod_{t=0}^{p-1}  \prod_{s=1}^{l} (z-(-1)^{d}\qs^{n+1-l+2m-p+2s+2t}) } \in \ko[z^{\pm1}]
$$
which completes our assertion for this case, 

The remaining case is proved similarly.
\end{proof}

\subsubsection{$d_{n^p,n^m}(z)$} 
\begin{lemma}
For $m \geq 1$, we have
\[
d_{n,n^m}(z)= \prod_{s=1}^{n} (z-(-1)^{m+1}\qs^{2+|2m-2|+2s}).
\]
\end{lemma}

\begin{proof}
The case for $m = 1$ is already known. Thus, we assume $m \ge 2$.

By Dorey's rule in Theorem~\ref{thm: Dorey}, 
\[
\Vkm{1}_{(-\qs)^{1-n}} \otimes \Vkm{n-1}_{(-\qs)} \twoheadrightarrow \Vkm{n},
\]
we have
\[
\dfrac{d_{1,n^m}((-\qs)^{1-n}z)d_{n-1,n^m}((-\qs)z)}{d_{n,n^m}(z)} = 
\dfrac{\displaystyle\prod_{s=1}^{n} (z-(-1)^{m+1}\qs^{2m+2s})   }{d_{n,n^m}(z)} \in \ko[z^{\pm 1}].
\]
Next the homomorphism
\[
\Vkm{n} \otimes \Vkm{1}_{(-\qs)^{n+3}} \twoheadrightarrow  \Vkm{n-1}_{(-\qs)}
\]
implies that we have 
\begin{align*}
& \dfrac{ d_{n,n^{m}}(z)  (z-(-1)^{m+1}\qs^{2n+2m+4})(z-(-\qs)^{2-2m}) }{\displaystyle \prod_{s=1}^{n} (z-(-1)^{m+1}\qs^{2+2m+2s})}  \in \ko[z^{\pm 1}].
\end{align*}
Hence our claim follows.
\end{proof}
\noindent

Following the proof of, e.g., Lemma~\ref{lemma: division A}, we have the following analogous statement.

\begin{lemma} \label{lemma: division C n}
For $m, p \ge 1$, we have
\[
\text{ $d_{n^p,n^m}(z)$ divides $\prod_{t=0}^{\min(p,m)-1} \prod_{s=1}^{n} (z-(-1)^{m+p}\qs^{2+|2m-2p|+2s+4t}).$}
\]
\end{lemma}

\begin{lemma}  \label{lemma: dnpnm C}
For $|m - p| \ge \lceil n/2 \rceil$, we have
\begin{align*}
d_{n^p,n^m}(z) = \prod_{t=0}^{\min(p,m)-1} \prod_{s=1}^{n} (z-(-1)^{m+p}\qs^{2+|2m-2p|+2s+4t}).
\end{align*}
\end{lemma}

\begin{proof}
Without loss of generality, we assume that $p \le m$. The case when $\min(m,p) = 1$ was proven above, hence let us consider when $\min(m,p) \geq 2$.
By the homomorphism
\[
\Vkm{n^p} \otimes \Vkm{n}_{(-1)^{1-m}\qs^{2n+2p}}  \twoheadrightarrow   \Vkm{n^{p-1}}_{(-\qs^2)^{-1}},
\]
we have
\begin{align*}
\dfrac{ d_{n^p,n^{m}}(z)d_{n,n^{m}}((-1)^{1-m}\qs^{-2n-2p}z)}{d_{n^{p-1},n^m}((-\qs)^{2}z)} \times
\dfrac{a_{n^{p-1},n^m}((-\qs)^{2}z)}{a_{n^p,n^{m}}(z)a_{n,n^{m}}((-1)^{1-m}\qs^{-2n-2p}z)} \in \ko[z^{\pm 1}].
\end{align*}
By descending induction on $m-p$, we have
\begin{align*}
& \dfrac{ d_{n^p,n^{m}}(z)d_{n,n^{m}}((-1)^{1-m}\qs^{-2n-2p}z)}{d_{n^{p-1},n^m}((-\qs^2)z)}
= \dfrac{ d_{n^p,n^{m}}(z) \displaystyle \prod_{s=1}^{n} (z-(-1)^d\qs^{2n+2m+2p+2s}) }{ \displaystyle \prod_{t=0}^{p-2} \prod_{s=1}^{n} (z-(-1)^d\qs^{2+2m-2p+2s+4t})}
\end{align*}
and the direct computation yields 
\begin{align*}
&\dfrac{a_{n^{p-1},n^m}((-\qs^{2})z)}{a_{n^p,n^{m}}(z)a_{n,n^{m}}((-1)^{1-m}\qs^{-2n-2p}z)}= \prod_{s=1}^n \dfrac{(z-(-1)^d\qs^{-2m+2p-2+2s})}{(z-(-1)^d\qs^{2m+2p-2+2s})}
\end{align*}
Thus we have
\begin{align*}
& \dfrac{ d_{n^p,n^{m}}(z) \displaystyle \prod_{s=1}^{n} (z-(-1)^d\qs^{2n+2m+2p+2s}) }{ \displaystyle \prod_{t=0}^{p-1} \prod_{s=1}^{n} (z-(-1)^d\qs^{2+2m-2p+2s+4t})}
 \times \prod_{s=1}^n (z-(-1)^d\qs^{-2m+2p-2+2s})  \in \ko[z^{\pm 1}],
\end{align*}
which implies
\begin{align*}
\dfrac{ d_{n^p,n^{m}}(z) \displaystyle \prod_{s=1}^n (z-(-1)^d\qs^{-2m+2p-2+2s})  }{ \displaystyle \prod_{t=0}^{p-1} \prod_{s=1}^{n} (z-(-1)^d\qs^{2+2m-2p+2s+4t})} \in \ko[z^{\pm 1}].
\end{align*}
Here we used  
\begin{align*}
2n+2m+2p+2s' \ne  2+2m-2p+2s+4t \iff   n+2p-1 \ne (s-s')+2t
\end{align*}
for $1 \le s \le n$ and $0 \le t \le p-1$.
Thus our assertion under the assumption follows, since   
\begin{align*}
-2m+2p-2+2s' \ne  2+2m-2p+2s+4t  \iff (s'-s)-2t \ne  2+2u \quad (u \seteq m-p \ge \ceil{n/2})
\end{align*}
for $1 \le s,s' \le n$ and $0 \le t \le p-1$. 
\end{proof}

\begin{proof}[Proof of {\rm Theorem~\ref{thm:denominators_untwisted}} for $l = k = n$ in type $C_n^{(1)}$]
We assume without loss of generality that $p \le m$.
From the homomorphism
\[
\Vkm{n^m} \otimes  \Vkm{n}_{(-\qs^2)^{-m-1}} \twoheadrightarrow \Vkm{n^{m+1}}_{(-\qs^2)^{-1}}
\]
and the descending induction on $\abs{m-p}$, we obtain
\[
\dfrac{d_{n^p,n^m}(z) d_{n^p,n}((-\qs^2)^{-m-1} z)}{d_{n^p,n^{m+1}}((-\qs^2)^{-1} z)} \times \dfrac{a_{n^p,n^{m+1}}((-\qs^2)^{-1} z)}{a_{n^p,n^m}(z) a_{n^p,n}((-\qs^2)^{-m-1} z)} \in \ko[z^{\pm1}].
\]
By direct computation, we have
\begin{align*}
\dfrac{a_{n^p,n^{m+1}}((-\qs^2)^{-1} z)}{a_{n^p,n^m}(z) a_{n^p,n}((-\qs^2)^{-m-1} z)} & = 1
\end{align*}
and
\begin{align*}
&\dfrac{d_{n^p,n^m}(z) d_{n^p,n}((-\qs^2)^{-m-1} z)}{d_{n^p,n^{m+1}}((-\qs^2)^{-1} z)}
= \dfrac{d_{n^p,n^m}(z)\displaystyle \prod_{s=1}^{n} (z-(-1)^{p+m}\qs^{2p+2m+2+2s})  }{ \displaystyle\prod_{t=0}^{p-1} \prod_{s=1}^{n} (z-(-1)^{m+p}\qs^{2m+6-2p+2s+4t}) } \\
& =
\dfrac{d_{n^p,n^m}(z)  }{ \displaystyle\prod_{t=1}^{p-1} \prod_{s=1}^{n} (z-(-1)^{m+p}\qs^{2m+2-2p+2s+4t}) } \in \ko[z^{\pm1}].
\end{align*}
Thus we conclude
\[
d_{n^p,n^m}(z) = \prod_{t=1}^{p-1} \prod_{s=1}^{n} (z-(-1)^{m+p}\qs^{2+2m-2p+2s+4t}) \times \prod_{s=1}^{n} (z-(-1)^{m+p}\qs^{2+2m-2p+2s})^{\epsilon_s}
\]
for some $\epsilon_s \in \{0,1\}$.

From the homomorphism obtained from Theorem~\ref{thm: Higher Dorey I}~\eqref{eq: spin C homo} \emph{without} the restriction on $l$
\begin{align*}
& \Vkm{1^{2m}}_{(-1)^{-n-m+1}\qs^{-n-2}} \otimes \Vkm{n^m}  \twoheadrightarrow \Vkm{n^m}_{(-\qs)^{-2}},
\end{align*}
we have
\[
\dfrac{d_{n^p,1^{2m}}((-1)^{-n-m+1}\qs^{-n-2}z) d_{n^p,n^m}(z)}{d_{n^p,n^{m}}((-\qs)^{-2}z)} \times
\dfrac{a_{n^p,n^{m}}((-\qs)^{-2}z)}{a_{n^p,1^{2m}}((-1)^{-n-m+1}\qs^{-n-2}z) a_{n^p,n^m}(z)} \in \ko[z^{\pm1}].
\]

By direct calculation,
we have
\begin{align*}
& \dfrac{  \displaystyle\prod_{s=1}^{n} (z-(-1)^{m+p}\qs^{2+2m-2p+2s})^{\epsilon_s}}{\displaystyle \prod_{s=1}^{n} (z-(-1)^{m+p}\qs^{4+2m-2p+2s})^{\epsilon_s} }
\times  \prod_{t=0}^{p-1} (z-(-1)^{m+p} \qs^{2n+2m-2p+6+4t})(z-(-1)^{m+p}\qs^{2p-2m-4t}) \allowdisplaybreaks\\
& \hspace{15ex} \times  \dfrac{(z-(-1)^{m+p} \qs^{2n+2m-2p+4})}{(z-(-1)^{m+p}\qs^{2m-2p+4})} \in \ko[z^{\pm1}]\allowdisplaybreaks \\
& \Rightarrow \dfrac{  \displaystyle\prod_{s=1}^{n} (z-(-1)^{m+p}\qs^{2+2m-2p+2s})^{\epsilon_s}}{\displaystyle \prod_{s=1}^{n} (z-(-1)^{m+p}\qs^{4+2m-2p+2s})^{\epsilon_s} }
\times  \dfrac{(z-(-1)^{m+p} \qs^{2n+2m-2p+4})}{(z-(-1)^{m+p}\qs^{2m-2p+4})} \in \ko[z^{\pm1}].
\end{align*}
Thus $\epsilon_1$ must be $1$ and hence $\epsilon_k=1$ for all $1 \le s \le n$, which implies our assertion.  
\end{proof}

\subsubsection{$d_{k^m,l^p}(z)$ for $1 \le k, l <n$}

\begin{lemma} \label{lem: d11m C}
For $m \ge 1$, we have
\[
d_{1,1^m}(z)= (z-(-\qs)^{m+1}) (z-(-q)^{2n+m+1}).
\]
\end{lemma}

\begin{proof}
By fusion rules
\[
\Vkm{1^{m-1}}_{(-\qs)} \otimes \Vkm{1}_{(-\qs)^{1-m}} \twoheadrightarrow \Vkm{1^m} \quad \text{ and } \quad
\Vkm{1}_{(-\qs)^{m-1}} \otimes \Vkm{1^{m-1}}_{(-\qs)^{-1}} \twoheadrightarrow \Vkm{1^m},
\]
we have by induction on $m$
\begin{align*}
&  \dfrac{(z-(-\qs)^{m-1}) (z-(-\qs)^{2n+m-1}) \times (z-(-\qs)^{m+1}) (z-(-\qs)^{2n+m+1})  }{d_{1,1^m}(z)} \in \ko[z^{\pm 1}], \allowdisplaybreaks \\
&  \dfrac{ (z-(-\qs)^{-m+3}) (z-(-\qs)^{2n-m+3}) \times (z-(-\qs)^{m+1}) (z-(-\qs)^{2n+m+1})  }{d_{1,1^m}(z)} \in \ko[z^{\pm 1}].
\end{align*}
Therefore, there is an ambiguity when at $m = 2$ or $m = n + 2$.
For the case of $m=2$, the surjective morphism~\eqref{eq: new homom C}
\[
\Vkm{n}_{(-\qs)^{-n}} \otimes \Vkm{n}_{(-\qs)^{n}} \twoheadrightarrow \Vkm{1^2},
\]
which resolves the ambiguity since
\[
\dfrac{ d_{1,n}((-\qs)^{-n}z)d_{1,n}((-\qs)^{n}z)}{d_{1,k^2}(z)} =  \dfrac{ (z-(-\qs)^{2n+3})(z-(-\qs)^{3}) }{d_{1,k^2}(z)} \in \ko[z^{\pm1}].
\]
For the case of $m = n + 2$, fusion rule
\[
\Vkm{1^2}_{(-\qs)^{n}} \otimes \Vkm{1^n}_{(-\qs)^{-2}} \twoheadrightarrow \Vkm{1^{n+2}}
\]
resolves the ambiguity since it implies
\begin{align*}
& \dfrac{ (z-(-\qs)^{3-n}) (z-(-\qs)^{n+3}) \times (z-(-\qs)^{n+3}) (z-(-\qs)^{3n+3})  }{d_{1,1^{n+2}}(z)} \in \ko[z^{\pm 1}].
\end{align*}
On the other hand, the homomorphism
\[
\Vkm{1^m}  \otimes  \Vkm{1}_{(-\qs)^{2n+m+1}} \twoheadrightarrow  \Vkm{1^{m-1}}_{(-\qs)^{-1}}
\]
implies that we have  
\begin{align*}
& \dfrac{ d_{1,1^{m}}(z)   (z-(-\qs)^{2n+m+3}) (z-(-\qs)^{4n+m+3}) }{(z-(-\qs)^{2n+m+1})(z-(-\qs)^{m+1})} \in \ko[z^{\pm1}],
\end{align*}
from which our assertion follows.
\end{proof}

With Lemma~\ref{lem: d11m C}, one can check the following lemma.

\begin{lemma}
For  $1 \le k \le n-1$ and $m \ge 1$, we have
\[
d_{k,1^m} = (z-(-\qs)^{m+k}) (z-(-\qs)^{2n+m-k+2}).
\]
\end{lemma}

\begin{lemma}
For  $1 \le k \le n-1$ and $m \ge 1$, we have
\[
d_{1,k^m} = (z-(-\qs)^{m+k}) (z-(-\qs)^{2n+m-k+2}).
\]
\end{lemma}

\begin{proof}
By fusion rules
\[
\Vkm{k^{m-1}}_{(-\qs)} \otimes  \Vkm{k}_{(-\qs)^{1-m}} \twoheadrightarrow \Vkm{k^m} \quad \text{ and }
\quad
\Vkm{k}_{(-\qs)^{m-1}} \otimes  \Vkm{k^{m-1}}_{(-\qs)^{-1}} \twoheadrightarrow \Vkm{k^m},
\]
we have
\begin{subequations} 
\begin{align} 
& \dfrac{  (z-(-\qs)^{m+k-2}) (z-(-\qs)^{2n+m-k})  (z-(-\qs)^{m+k}) (z-(-\qs)^{2n+m-k+2})  }{d_{1,k^m}(z)} \in \ko[z^{\pm1}],
\label{eq: a1kmstep1 type C} \\ 
&   \dfrac{  (z-(-\qs)^{2+k-m}) (z-(-\qs)^{2n-m-k+4}) (z-(-\qs)^{m+k}) (z-(-\qs)^{2n+m-k+2})    }{d_{1,k^m}(z)} \in \ko[z^{\pm1}]. \label{eq: a1kmstep1p type C}
\end{align}
\end{subequations}
By~\eqref{eq: a1kmstep1 type C} and~\eqref{eq: a1kmstep1p type C}, there is an ambiguity when $m = 2$.
In the case of $m = 2$, the surjective morphism~\eqref{eq: new homom C}
\[
\Vkm{n}_{(-\qs)^{-n+k-1}} \otimes \Vkm{n}_{(-\qs)^{n-k+1}} \twoheadrightarrow \Vkm{k^2}
\]
implies
\[
\dfrac{ d_{1,n}((-\qs)^{-n+k-1}z)d_{1,n}((-\qs)^{n-k+1}z)}{d_{1,k^2}(z)} =  \dfrac{ (z-(-\qs)^{2n-k+4})(z-(-\qs)^{k+2}) }{d_{1,k^2}(z)} \in \ko[z^{\pm1}],
\]
resolving the ambiguity.
Now we consider $m \geq 3$, and by the homomorphism
\[
 \Vkm{k^m} \otimes \Vkm{k}_{(-\qs)^{2n+m+1}} \twoheadrightarrow \Vkm{k^{m-1}}_{(-\qs)^{-1}},
\]
our insertion follows since we have  
\begin{align*}
& \dfrac{ d_{1,k^m}(z)    (z-(-\qs)^{2n+m+2+k}) (z-(-\qs)^{4n+m-k+4})  }{  (z-(-\qs)^{m+k})(z-(-\qs)^{2n+m-k+2})  } \in \ko[z^{\pm1}]. \qedhere
\end{align*}
\end{proof}
 
As Proposition~\ref{prop: B dlkm}, we can obtain the following proposition.

\begin{proposition}\label{prop: dlkm C}  
For any $m \in \Z_{\ge 2}$ and  $1 \le k,l < n$, we have
\begin{align}\label{eq: dlkm C}
 d_{l,k^m}(z) =   \prod_{s=1}^{\min(k,l)}  (z-(-\qs)^{|k-l|+m-1+2s})(z-(-\qs)^{2n-k-l+m+1+2s} )
\end{align}
\end{proposition}

As Lemma~\ref{lemma: division A} and Lemma~\ref{lemma: division B}, we have the following analogous statement.

\begin{lemma}\label{lemma: division C}  
For $1\leq k,l \le n-1$ and $\max(m,p) >  1$, we have
\[
\text{$d_{l^p,k^m}(z)$ divides $  \prod_{t=0}^{\min(m,p)-1} \prod_{s=1}^{\min(k,l)}  (z-(-\qs)^{|k-l|+\abs{m-p}+2s+2t})(z-(-\qs)^{2n+2-k-l+\abs{m-p}+2s+2t} )$.}
\]
\end{lemma}

\begin{proposition} \label{prop: dlpkm C}   
For any $m,p \in \Z_{\ge 1}$ with $\max(m, p)\ge 2$, $1 \le k,l < n$ and
\[
\abs{m-p} \ge \mu \seteq  \max(n-\max(k,l),\min(k,l)-1),
\]
we have
\[
d_{l^p,k^m}(z) =   \prod_{t=0}^{\min(m,p)-1} \prod_{s=1}^{\min(k,l)}  (z-(-\qs)^{|k-l|+\abs{m-p}+2s+2t})(z-(-\qs)^{2n+2-k-l+\abs{m-p}+2s+2t} ).
\]
\end{proposition}

\begin{proposition} [{\rm Theorem~\ref{thm:denominators_untwisted}} for $\max(l,k) < n$, and $m \equiv_2 0$ or $p \equiv_2 0$] \label{prop: dlpk2m C}
For $k,l<n$ and $m,p \in \Z_{\ge1}$, we have
$$
d_{k^{2m},l^p}(z) = \prod_{t=0}^{\min(2m,p)-1} \prod_{s=1}^{\min(k,l)}  (z-(-\qs)^{|k-l|+|2m-p|+2s+2t})(z-(-\qs)^{2n+2-k-l+|2m-p|+2s+2t} )
$$    
\end{proposition} 

\begin{proof}[Proof of {\rm Theorem~\ref{thm:denominators_untwisted}} for $\max(l,k) < n$, and $m \equiv_2 0$ or $p \equiv_2 0$ 
in type $C_n^{(1)}$] 
Recall that we have a homomorphism in~\eqref{eq: spin C homo}
\[
\Vkm{n^m}_{(-1)^{-n-m+k}\qs^{-1-n+k}} \otimes \Vkm{n^m}_{(-1)^{n+m-k}\qs^{n+1-k}} \twoheadrightarrow \Vkm{k^{2m}},
\]
which yields the homomorphism
\begin{align*}
& \Vkm{k^{2m}} \otimes \Vkm{n^m}_{(-1)^{-n-m+k}\qs^{n+1+k}}  \twoheadrightarrow \Vkm{n^m}_{(-1)^{n+m-k}\qs^{n+1-k}}
\end{align*}
by duality. Then we have
\begin{align*}
\dfrac{d_{l^p,k^{2m}}(z) d_{l^p,n^m}((-1)^{-n-m+k}\qs^{n+1+k} z)}{d_{l^p,n^{m}}((-1)^{n+m-k}\qs^{n+1-k}z)} \times
\dfrac{a_{l^p,n^{m}}((-1)^{n+m-k}\qs^{n+1-k}z) }{ a_{l^p,k^{2m}}(z) a_{l^p,n^m}((-1)^{-n-m+k}\qs^{n+1+k} z)}\in \bfk[ z^{\pm 1} ].
\end{align*}

We assume that $l \le k$ and $p \le 2m$. By a direct calculation, we have
\begin{align*}
& \dfrac{a_{l^p,n^{m}}((-1)^{n+m-k}\qs^{n+1-k}z)}{ a_{l^p,k^{2m}}(z) a_{l^p,n^m}((-1)^{-n-m+k}\qs^{n+1+k} z)}
 = \prod_{t=0}^{p-1}  \prod_{s=1}^{l}  \dfrac{(z-(-\qs)^{l-k-2m+p-2s-2t})}{(z-(-\qs)^{-l-k+2m-p+2s+2t})}
\end{align*}
and
\begin{align*}
&\dfrac{d_{l^p,k^{2m}}(z) d_{l^p,n^m}((-1)^{-n-m+k}\qs^{n+1+k} z)}{d_{l^p,n^{m}}((-1)^{n+m-k}\qs^{n+1-k}z)}
= \dfrac{d_{l^p,k^{2m}}(z)\displaystyle \prod_{t=0}^{p-1} \prod_{s=1}^{l} (z-(-\qs)^{-l-k+2m-p+2s+2t})}{\displaystyle\prod_{t=0}^{p-1} \prod_{s=1}^{l} (z-(-\qs)^{-l+k+2m-p+2s+2t})}.
\end{align*}

Thus we have 
\begin{align}
& \dfrac{d_{l^p,k^{2m}}(z)\displaystyle \prod_{t=0}^{p-1} \prod_{s=1}^{l} (z-(-\qs)^{-l-k+2m-p+2s+2t})}{\displaystyle\prod_{t=0}^{p-1} \prod_{s=1}^{l} (z-(-\qs)^{-l+k+2m-p+2s+2t})}
\times \prod_{t=0}^{p-1}  \prod_{s=1}^{l}  \dfrac{(z-(-\qs)^{l-k-2m+p-2s-2t})}{(z-(-\qs)^{-l-k+2m-p+2s+2t})} \nonumber\allowdisplaybreaks\\
& = \dfrac{d_{l^p,k^{2m}}(z)\displaystyle \prod_{t=0}^{p-1} \prod_{s=1}^{l} (z-(-\qs)^{l-k-2m+p-2s-2t})}{\displaystyle\prod_{t=0}^{p-1} \prod_{s=1}^{l} (z-(-\qs)^{-l+k+2m-p+2s+2t})} \in \ko[z^{\pm1}] \nonumber\allowdisplaybreaks\\
& \Rightarrow \dfrac{d_{l^p,k^{2m}}(z) }{\displaystyle\prod_{t=0}^{p-1} \prod_{s=1}^{l} (z-(-\qs)^{-l+k+2m-p+2s+2t})} \in \ko[z^{\pm1}]. \label{eq: factor 1}
\end{align}
Then our assertion for this case follows from Lemma~\ref{lemma: division C}. 
The remaining cases can be proved in a similar way.  
\end{proof}

\begin{lemma} \label{lem: C refining 1}
For $\g=C_n^{(1)}$, $k,l < n$,  
we have
\begin{align*}
&\de(\Vkm{l^{p}}, \Vkm{k^{m}}_{(-\qs)^{2n+2-|k-l|+|m-p|+2t}}) \\
& \hspace{7ex} = \bc
\de(\Vkm{l^p}, \Vkm{k^{m+1}}_{(-\qs)^{2n+2-|k-l|+|m-p|+1+2t}})     & \text{ if } m \ge p, \\
\de(\Vkm{l^{p+1}}_{(-\qs)^{-1}}, \Vkm{k^{m}}_{(-\qs)^{2n+2-|k-l|+|m-p|+2t}})     & \text{ if } m <p,
\ec
\end{align*}
for $0 \le  t \le \min(m,p)-1$. 
\end{lemma}

\begin{proof} We set $\max(k,l)=a$ and $\min(k,l)=b$.

\noi
(i) We first assume $m \ge p$. 
Note that 
$$
\Vkm{k^{m+1}}_{(-\qs)^{2n+3-a+b+m-p+2t}} \iso \Vkm{k}_{(-\qs)^{2n+3-a+b+2m-p+2t}} \hconv \Vkm{k^{m}}_{(-\qs)^{2n+2-a+b+m-p+2t}}
$$
and
$$
\de( \Vkm{l^p},\Vkm{k}_{(-\qs)^{2n+3-a+b+2m-p+2t}}) 
= \de( \Vkm{l^p}_{(-\qs)^{-2n-2}},\Vkm{k}_{(-\qs)^{2n+3-a+b+2m-p+2t}}) =0
$$
by
$$
d_{k,l^p}(z) =   
\prod_{s=1}^{b}  (z-(-\qs)^{a-b+p-1+2s})(z-(-\qs)^{2n-a-b+p+1+2s} ).
$$
Namely,
\bna
\item $ 2n+3-a+b+2m-p+2t \ne a-b+p-1+2s \iff  (n-a)+2+(m-p)+(b-s) > 0 \ge -t$,
\item $2n+3-a+b+2m-p+2t \ne 2n-a-b+p+1+2s \iff 2+(m-p)+(b-s) >0 \ge -t$,
\item $4n+5-a+b+2m-p+2t \ne a-b+p-1+2s \iff n+3+(n-a)+(m-p)+(b-s) > 0 \ge -t$,
\item $4n+5-a+b+2m-p+2t \ne 2n-a-b+p+1+2s \iff n+2+(m-p)+(b-s) >0 \ge -t$,
\ee
for $1 \le s \le b$ and $0 \le t \le p-1$.
Thus we have the assertion by Lemma~\ref{lem: de=de}~\eqref{it: de=de 2}.

\noi
(ii) Now let us assume that $m<p$.
Note that 
$$\Vkm{l^{p+1}}_{(-\qs)^{-1}} \iso \Vkm{l^{p}}\hconv  \Vkm{l}_{(-\qs)^{-1-p}}.$$
and
$$
\de( \Vkm{k^{m}}_{(-\qs)^{2n+2-a+b+p-m+2t}},\Vkm{l}_{(-\qs)^{-1-p}}) 
= \de(  \Vkm{k^{m}}_{(-\qs)^{4n+4-a+b+p-m+2t}},\Vkm{l}_{(-\qs)^{-1-p}}) =0
$$
by
$$
d_{l,k^m}(z) =   
\prod_{s=1}^{\min(k,l)}  (z-(-\qs)^{a-b+m-1+2s})(z-(-\qs)^{2n-a-b+m+1+2s} ).
$$
Namely,
\bna
\item $2n+3-a+b+2p-m+2t  \ne a-b+m-1+2s \iff (n-a)+2+(p-m)+(b-s) >0  \ge -t  $,
\item $2n+3-a+b+2p-m+2t   \ne 2n-a-b+m+1+2s \iff 2+(p-m)+(b-s) >0 \ge -t $,
\item $4n+5-a+b+2p-m+2t  \ne a-b+m-1+2s \iff n+3+(n-a)+(p-m)+(b-s) >0  \ge -t$,
\item $4n+5-a+b+2p-m+2t  \ne 2n-a-b+m+1+2s \iff n+2+(p-m)+(b-s) >0 \ge -t$,
\ee
for $1 \le s \le b$ and $0 \le t \le m-1$. Thus we have the assertion by Lemma~\ref{lem: de=de}~\eqref{it: de=de 1}.
\end{proof}

\begin{lemma} \label{lem: C refining roots 1}
For $\g=C_n^{(1)}$, $k,l < n$ and $m, p \in 2\Z_{\ge1}+1$  
we assume that $$\min(k,l)=1.$$ Then 
\ben  
\item we have
\begin{align*}
&\de(\Vkm{l^{p}}, \Vkm{k^{m}}_{(-\qs)^{|k-l|+|m-p|+2}}) \\
& \hspace{7ex} = \bc
\de(\Vkm{l^p}, \Vkm{k^{m+1}}_{(-\qs)^{|k-l|+p-m+3}})     & \text{ if } m < p, \\
\de(\Vkm{l^{p+1}}_{(-\qs)^{-1}}, \Vkm{k^{m}}_{(-\qs)^{|k-l|+m-p+2}})     & \text{ if } m \ge p,
\ec
\end{align*}
\item we have
\begin{align*}
&\de(\Vkm{l^{p}}, \Vkm{k^{m}}_{(-\qs)^{2n+2+|k-l|+|m-p|+2}}) \\
& \hspace{7ex} = \bc
\de(\Vkm{l^p}, \Vkm{k^{m+1}}_{(-\qs)^{2n+2+|k-l|+p-m+3}})     & \text{ if } m < p, \\
\de(\Vkm{l^{p+1}}_{(-\qs)^{-1}}, \Vkm{k^{m}}_{(-\qs)^{2n+2+|k-l|+m-p+2}})     & \text{ if } m \ge p.
\ec
\end{align*}
\ee 
\end{lemma}

\begin{proof} We set $a \seteq \max(k,l)$. 

\noi
(1-i) We first assume $m < p$. 
Note that 
$$
\Vkm{k^{m+1}}_{(-\qs)^{a+p-m+2}} \iso \Vkm{k}_{(-\qs)^{a+p+2 }} \hconv \Vkm{k^{m}}_{(-\qs)^{a+p-m+1}}
$$
and
$$
\de( \Vkm{l^p},\Vkm{k}_{(-\qs)^{a+p+2}}) 
= \de( \Vkm{l^p}_{(-\qs)^{-2n-2}},\Vkm{k}_{(-\qs)^{a+p+2}}) =0
$$
by
$$
d_{k,l^p}(z) = (z-(-\qs)^{a+p})(z-(-\qs)^{2n-a+p+2} ).
$$
Namely, 
\begin{align*}
& a+p+2 \ne a+p, &&    a+p+2 \ne 2n-a+p+2 \iff  n-a \ne 0, \\
& 2n+2+a+p+2 \ne a+p, && 2n+a+p+4 \ne 2n+2-a+p \iff a+1 \ne 0.
\end{align*}
Thus we have the assertion by Lemma~\ref{lem: de=de}~\eqref{it: de=de 2};

\noi
(1-ii) Now let us assume that $m \ge p$.
Note that 
$$\Vkm{l^{p+1}}_{(-\qs)^{-1}} \iso \Vkm{l^{p}}\hconv  \Vkm{l}_{(-\qs)^{-1-p}}.$$
Since 
$$
\de( \Vkm{k^{m}}_{(-\qs)^{a+m-p+1}},\Vkm{l}_{(-\qs)^{-1-p}}) 
= \de(  \Vkm{k^{m}}_{(-\qs)^{2n+a+m-p+3}},\Vkm{l}_{(-\qs)^{-1-p}}) =0
$$
$$
d_{l,k^m}(z) = (z-(-\qs)^{a+m})(z-(-\qs)^{2n-a+m+2s} ),
$$
as in the previous case, we have the assertion by Lemma~\ref{lem: de=de}~\eqref{it: de=de 1}.

\noi
(2-i) We first assume $m < p$. 
Note that 
$$
\Vkm{k^{m+1}}_{(-\qs)^{a+p-m+2n+4}} \iso \Vkm{k}_{(-\qs)^{a+p+2n+4 }} \hconv \Vkm{k^{m}}_{(-\qs)^{a+p-m+2n+3}}
$$
and 
$$
\de( \Vkm{l^p},\Vkm{k}_{(-\qs)^{2n+4+a+p}}) 
= \de( \Vkm{l^p}_{(-\qs)^{-2n-2}},\Vkm{k}_{(-\qs)^{a+p+2n+4}}) =0
$$
by
$$
d_{k,l^p}(z) = (z-(-\qs)^{a+p})(z-(-\qs)^{2n-a+p+2} ).
$$
Namely, 
\begin{align*}
& a+p+2n+4 \ne a+p, &&    a+p+2n+4 \ne 2n+2-a+p \iff  a+1 \ne 0, \\
& 4n+6+a+p  \ne a+p, && 4n+6+a+p \ne 2n+2-a+p \iff n+a+2 \ne 0.
\end{align*}
Thus we have the assertion by Lemma~\ref{lem: de=de}~\eqref{it: de=de 2};

\noi
(2-ii) Now let us assume that $m \ge p$.
Note that 
$$\Vkm{l^{p+1}}_{(-\qs)^{-1}} \iso \Vkm{l^{p}}\hconv  \Vkm{l}_{(-\qs)^{-1-p}}.$$
Since 
$$
\de( \Vkm{k^{m}}_{(-\qs)^{a+m-p+2n+3}},\Vkm{l}_{(-\qs)^{-1-p}}) 
= \de(  \Vkm{k^{m}}_{(-\qs)^{2n+a+m-p+2n+5}},\Vkm{l}_{(-\qs)^{-1-p}}) =0
$$
$$
d_{l,k^m}(z) = (z-(-\qs)^{a+m})(z-(-\qs)^{2n-a+m+2s} ),
$$
as in the previous case, we have the assertion by Lemma~\ref{lem: de=de}~\eqref{it: de=de 1}.
\end{proof}

\begin{proof}[Proof of {\rm Theorem~\ref{thm:denominators_untwisted}} for {\rm (i)} $k \ne l <n$ or {\rm (ii)} $\min(k,l)=1$ in type $C_n^{(1)}$]  
Since the case $\min(p,m) = 1$ is already covered, let us  assume $\min(p,m) > 1$.
Also the case when $\min(k,l)=1$ is covered by Proposition~\ref{prop: dlpkm C} and hence we shall assume $\min(k,l)>1$.
By the previous proof, it suffices to show that when   $m,p$ are odd. Thus we shall consider only $d_{l^p,k^{2m+1}}(z)$ when $p$ is odd.
Note that we have to prove that $d_{l^p,k^{2m+1}}(z) $ coincides with
\begin{align} \label{eq: formual C final}
D_{l^p,k^{2m+1}}(z) \seteq \prod_{t=0}^{p-1}\prod_{s=1}^{l}  (z-(-\qs)^{k-l+2m+1-p+2s+2t})(z-(-\qs)^{2n+2-k-l+2m+1-p+2s+2t} ).    
\end{align}
We assume further that $k \ge l$ and $2m+1 \ge p$.

From fusion rule
\[
\Vkm{k^{2m+1}} \otimes \Vkm{k}_{(-\qs)^{-2m-2}} \twoheadrightarrow \Vkm{k^{2m+2}}_{(-\qs)^{-1}},
\]
we obtain
\[
\dfrac{d_{l^p,k^{2m+1}}(z) d_{l^p,k}((-\qs)^{-2m-2} z)}{d_{l^p,k^{2m+2}}((-\qs)^{-1} z)} \times \dfrac{a_{l^p,k^{2m+2}}((-\qs)^{-1} z)}{a_{l^p,k^{2m+1}}(z) a_{l^p,k}((-\qs)^{-2m-2} z)} \in \ko[z^{\pm1}].
\]
By direct computation, we have
\begin{align*}
& \dfrac{d_{l^p,k^{2m+1}}(z) d_{l^p,k}((-\qs)^{-2m-2} z)}{d_{l^p,k^{2m+2}}((-\qs)^{-1} z)}  \allowdisplaybreaks\\
& \hspace{10ex} = \dfrac{d_{l^p,k^{2m+1}}(z) \displaystyle\prod_{s=1}^{l}  (z-(-\qs)^{k-l+p+2m+1+2s})(z-(-\qs)^{2n+3-k-l+p+2m+2s+2t} )  }{\displaystyle \prod_{t=0}^{p-1} \prod_{s=1}^{l}  (z-(-\qs)^{k-l+2m+3-p+2s+2t})(z-(-\qs)^{2n-k-l+2m+5-p+2s+2t} ) }
 \\
& \hspace{10ex} =
\dfrac{d_{l^p,k^{2m+1}}(z)   }{\displaystyle \prod_{t=1}^{p-1} \prod_{s=1}^{l}  (z-(-\qs)^{k-l+2m+1-p+2s+2t})(z-(-\qs)^{2n-k-l+2m+3-p+2s+2t} ) } \in \ko[z^{\pm1}], 
\end{align*}
by applying the denominator formulas $d_{l^p,k^{2m+2}}(z)$ and $d_{l,k^{2m+1}}(z)$. 
Hence Lemma~\ref{lemma: division C} implies that 
\begin{equation} \label{eq: factor C 1'}
\begin{aligned}
d_{l^p,k^{2m+1}}(z) & =  \prod_{t=1}^{p-1} \prod_{s=1}^{l}  (z-(-\qs)^{k-l+2m-p+1+2s+2t})(z-(-\qs)^{2n-k-l+2m-p+3+2s+2t} )     \\
& \times \prod_{s=1}^{l}  (z-(-\qs)^{k-l+2m-p+1+2s})^{\epsilon_s}(z-(-\qs)^{2n-k-l+2m-p+3+2s} )^{\epsilon'_s}
\end{aligned}
\end{equation}
for $\epsilon_s,\epsilon'_s \in \{0,1\}$ for $1\le s \le l$

By Lemma~\ref{lem: C refining roots 1}, 
\begin{align} \label{eq: no ambiguity}
\text{we have no ambiguity when $l=1$; i.e $\epsilon_1=\epsilon'_1=1$.}    
\end{align}

From the homomorphism  obtained from Theorem~\ref{thm: Higher Dorey I}~\eqref{eq: k+l<n homo} with the restriction $\min(k,l)=1$
\[
 \Vkm{l^p}  \otimes \Vkm{1^p}_{(-\qs)^{2n-l+3}} \twoheadrightarrow  \Vkm{(l-1)^p}_{(-\qs)},
\]
we obtain
\[
\dfrac{d_{l^p,k^{2m+1}}( z ) d_{1^p,k^{2m+1}}((-\qs)^{2n-l+3}z)}{d_{(l-1)^p,k^{2m+1}}((-\qs)z)} \times \dfrac{a_{(l-1)^p,k^{2m+1}}((-\qs)z)}{a_{l^p,k^{2m+1}}( z ) a_{1^p,k^{2m+1}}((-\qs)^{2n-l+3}z)} \in \ko[z^{\pm1}].
\]

By direct calculation, we have
\begin{align*}
\dfrac{a_{(l-1)^p,k^{2m+1}}((-\qs)z)}{a_{l^p,k^{2m+1}}( z ) a_{1^p,k^{2m+1}}((-\qs)^{2n-l+3}z)}
=  \prod_{t=0}^{p-1} \dfrac{(z-(-\qs)^{l-k-2m+p-3-2t})(z-(-\qs)^{-2n+k+l-2m+p-5-2t})}{(z-(-\qs)^{l-k+2m-p+1+2t})(z-(-\qs)^{-2n+k+l+2m-p-1+2t})}
\end{align*}
and the induction on $l$ from~\eqref{eq: no ambiguity} implies 
\begin{align*}
& \dfrac{d_{l^p,k^{2m+1}}( z ) d_{1^p,k^{2m+1}}((-\qs)^{2n-l+3}z)}{d_{(l-1)^p,k^{2m+1}}((-\qs)z)}\allowdisplaybreaks \\
&= \dfrac{d_{l^p,k^{2m+1}}( z ) \displaystyle\prod_{t=0}^{p-1} (z-(-\qs)^{-2n+k+l+2m-p-1+2t})(z-(-\qs)^{-k+l+2m-p+1+2t} ) }
{  \displaystyle\prod_{t=0}^{p-1} \prod_{s=1}^{l-1}  (z-(-\qs)^{k-l+2m-p+1+2s+2t})(z-(-\qs)^{2n-k-l+2m-p+3+2s+2t} )}
\end{align*}
Thus we have  
\begin{align*}
&\dfrac{d_{l^p,k^{2m+1}}( z ) \displaystyle\prod_{t=0}^{p-1}  (z-(-\qs)^{-2n+k+l+2m-p-1+2t})(z-(-\qs)^{-k+l+2m-p+1+2t} ) }
{  \displaystyle\prod_{t=0}^{p-1} \prod_{s=1}^{l-1}  (z-(-\qs)^{k-l+2m-p+1+2s+2t})(z-(-\qs)^{2n-k-l+2m-p+3+2s+2t} )}  \allowdisplaybreaks\\
& \hspace{15ex} \times \prod_{t=0}^{p-1} \dfrac{(z-(-\qs)^{l-k-2m+p-3-2t})(z-(-\qs)^{-2n+k+l-2m+p-5-2t})}{(z-(-\qs)^{l-k+2m-p+1+2t})(z-(-\qs)^{-2n+k+l+2m-p-1+2t})}   \allowdisplaybreaks\\
& \dfrac{d_{l^p,k^{2m+1}}( z ) \displaystyle\prod_{t=0}^{p-1} (z-(-\qs)^{l-k-2m+p-3-2t})(z-(-\qs)^{-2n+k+l-2m+p-5-2t})  }
{  \displaystyle\prod_{t=0}^{p-1} \prod_{s=1}^{l-1}  (z-(-\qs)^{k-l+2m-p+1+2s+2t})(z-(-\qs)^{2n-k-l+2m-p+3+2s+2t} )} \in \ko[z^{\pm1}]   \allowdisplaybreaks\\
& \Rightarrow \dfrac{d_{l^p,k^{2m+1}}( z )}
{  \displaystyle\prod_{t=0}^{p-1} \prod_{s=1}^{l-1}  (z-(-\qs)^{k-l+2m-p+1+2s+2t})(z-(-\qs)^{2n-k-l+2m-p+3+2s+2t} )} \in \ko[z^{\pm1}].  
\end{align*}
Hence Lemma~\ref{lemma: division C} implies that 
\begin{equation} \label{eq: factor C t1}
\begin{aligned}
d_{l^p,k^{2m+1}}(z) & =  \prod_{t=0}^{p-1} \prod_{s=1}^{l-1}  (z-(-\qs)^{k-l+2m-p+1+2s+2t})(z-(-\qs)^{2n-k-l+2m-p+3+2s+2t} )     \\
& \times \prod_{t=0}^{p-1}  (z-(-\qs)^{k+l+2m-p+1+2t})^{\epsilon_t}(z-(-\qs)^{2n-k+l+2m-p+3+2s+2t} )^{\epsilon'_t}
\end{aligned}
\end{equation}
for $\epsilon_t,\epsilon'_t \in \{0,1\}$ for $1\le t \le p-1$

From~\eqref{eq: factor C t1}, we need to show
\begin{eqnarray} &&
\parbox{85ex}{
$\displaystyle  \prod_{t=0}^{p-1} (z-(-\qs)^{k+l+2m-p+1+2t})(z-(-\qs)^{2n+2-k+l+2m+1-p+2t} )$ appears in~$d_{l^p,k^{2m+1}}(z)$ with the same multiplicities of $D_{l^p,k^{2m+1}}(z)$ in~\eqref{eq: formual C final}. 
}\label{eq: what we want C final before}
\end{eqnarray}
By Lemma~\ref{lem: C refining 1} and Proposition~\ref{prop: dlpk2m C},~\eqref{eq: what we want C final before} for $\displaystyle\prod_{t=0}^{p-1} (z-(-\qs)^{2n+2-k+l+2m+1-p})$ follows.
Namely, we have
\begin{equation} \label{eq: factor C 1''}
\begin{aligned}
d_{l^p,k^{2m+1}}(z) & =  \prod_{t=0}^{p-1} \prod_{s=1}^{l-1}  (z-(-\qs)^{k-l+2m-p+1+2s+2t})(z-(-\qs)^{2n-k-l+2m-p+3+2s+2t} )     \\
& \times \prod_{t=0}^{p-1}  (z-(-\qs)^{k+l+2m-p+1+2t})^{\epsilon_t}(z-(-\qs)^{2n-k+l+2m-p+3+2s+2t} )
\end{aligned}
\end{equation}
for $\epsilon_t  \in \{0,1\}$ for $0\le t \le p-1$.

Now we assume further $k>l$.
Note that $l<n-1$ by the assumption and hence we have $n>k>l$.
Then the T-system $(p=2p'+1)$
\[
\Vkm{l^{2p'+1}}_{(-\qs)^{-1}} \otimes \Vkm{l^{2p'+1}}_{(-\qs)} \twoheadrightarrow \Vkm{(l+1)^{2p'+1}}  \otimes \Vkm{(l-1)^{2p'+1}}
\]
and \eqref{eq: factor C 1''} yield 
\begin{align*}
& \dfrac{d_{l^p,k^{2m+1}}((-\qs)z)d_{l^p,k^{2m+1}}((-\qs)^{-1}z)}{d_{(l+1)^p,k^{2m+1}}(z)d_{(l-1)^p,k^{2m+1}}(z) } \allowdisplaybreaks\\
& = \dfrac{\displaystyle\prod_{t=0}^{p-1}  (z-(-\qs)^{k+l+2m-p+2+2t})^{\epsilon_t}  \times \prod_{t=0}^{p-1}  (z-(-\qs)^{k+l+2m-p+2t})^{\epsilon_t} }
{ \displaystyle\prod_{t=0}^{p-1}  (z-(-\qs)^{k+l+2m-p+2+2t})^{\epsilon_t} \prod_{t=0}^{p-1}  (z-(-\qs)^{k+l+2m-p+2t})  }  \in \bfk[z^{\pm1}]
\end{align*} 
by induction hypothesis on $l$ since 
\begin{align} \label{eq: a=1 T-system}
\dfrac{a_{(l+1)^p,k^{2m+1}}(z)a_{(l-1)^p,k^{2m+1}}(z) }{a_{l^p,k^{2m+1}}((-\qs)z)a_{l^p,k^{2m+1}}((-\qs)^{-1}z)}=1.
\end{align}
Then we can conclude that $\epsilon_t=1$ for all $0 \le t \le p-1$, as we desired.

The remaining cases can be proved in a similar way.
\end{proof}

\begin{remark} \label{rmk: determine} \hfill 
\bna 
\item In the proof of the above theorem,~\eqref{eq: a=1 T-system} does not hold when $k=l$. Thus it is not easy to remove the ambiguity. 
\item As we can see in~\eqref{eq: factor C 1''},  we only do not know the multiplicities of 
those roots of $d_{k^p,k^{m}}(z)$ for $m,p \in 2\Z_{\ge1}+1$. Hence we can determine whether $\Vkm{k^m}_x \otimes \Vkm{l^p}_y$
is simple or not completely (see Appendix~\ref{appensec: resolving} below for comparison between the conjectural formulas and~\eqref{eq: factor C 1''}).  
\item In Appendix~\ref{appensec: resolving}, we discuss an auxiliary technique that resolves the ambiguities in multiplicities.  
\ee
\end{remark}

\begin{corollary} \label{cor: root module C}
For $m \le n$, the KR module $\Vkm{1^m}_x$ $(x \in \bfk^{\times})$ is a root module.
\end{corollary}

\subsection{Proof for type \texorpdfstring{$D^{(1)}_{n}$}{Dn(1)}}

In this subsection, we first prove our assertion on $d_{{n'}^m,n^p}(z)$ for $n' \in \{ n-1,n \}$ and on $d_{k^m,{n'}^p}(z)$.
Then we will prove our assertion on $d_{k^{m},l^p}(z)$ for $1 \le k, l <n-1$, by applying the similar argument in Section~\ref{subsec: A}.

\subsubsection{$d_{n^p,n^m}(z)$ and $d_{(n-1)^p,n^m}(z)$} \label{subsubsec: dnpnm D}

Here we will prove our assertion on $d_{n^p,n^m}(z)$ and $d_{(n-1)^p,n^m}(z)$.
We assume that $n$ is even for this case. When $n$ is odd, we can prove in the similar way.
Note that $i^* = i$, and recall that
\begin{align*}
& d_{n,n-1}(z)=\displaystyle \prod_{s=1}^{\lfloor \frac{n-1}{2} \rfloor} \big(z-(-q)^{4s}\big),
\ \ d_{n,n}(z)=d_{n-1,n-1}(z)=\displaystyle \prod_{s=1}^{\lfloor \frac{n}{2} \rfloor} \big(z-(-q)^{4s-2}\big), \\
&  a_{n,n-1}(z)  \equiv \dfrac{\displaystyle\prod_{s=1}^{n-1}[4s-2]}{[2n-2]\displaystyle\prod_{s=1}^{n-2}[4s]}, \ \
a_{n,n}(z)=a_{n-1,n-1}(z) \equiv \dfrac{\displaystyle\prod_{s=0}^{n-1}[4s]}{[2n-2]\displaystyle\prod_{s=1}^{n-1}[4s-2]}.
\end{align*}
Also, by Theorem~\ref{thm: Dorey}~\eqref{eq: new homo}, we have the following:
For $1 \le l \le n - 2$ and $n',n''\in \{n,n-1\}$ such that $n'-n'' \equiv n - l \pmod{2}$, there
exist surjective homomorphisms
\[
\Vkm{n'}_{(-q)^{-n+l+1}} \otimes  \Vkm{n''}_{(-q)^{n-l-1}} \to \Vkm{l} 
\text{ and }
\Vkm{l}_{(-q)^{-n+l+1}}  \otimes  \Vkm{n-l-1}_{(-q)^{l}} \twoheadrightarrow  \Vkm{n-1} \otimes \Vkm{n},
\]
which can be obtained from Theorem~\ref{thm: Dorey} and~\eqref{eq: new homom D}, respectively.

\begin{lemma}
For $m \in \Z_{\ge 1}$, we have  
\[
d_{n,n^m}(z) =   \prod_{s=1}^{\lfloor \frac{n}{2} \rfloor} (z-(-q)^{4s+m-3} ) \quad \text{ and } \quad
d_{n-1,n^m}(z)  = \prod_{s=1}^{\lfloor \frac{n-1}{2} \rfloor} (z-(-q)^{4s+m-1} ).
\]
\end{lemma}

\begin{proof}
\noindent (a-1) For simplicity, write $n'=\dfrac{n}{2}$.
By the homomorphism
\[
\Vkm{n^{m-1}}_{(-q)} \otimes \Vkm{n}_{(-q)^{1-m}} \twoheadrightarrow \Vkm{n^m},
\]
we have
\begin{equation}\label{eq: annmstep1}
\begin{aligned}
& \dfrac{ d_{n,n^{m-1}}((-q)z)d_{n,n}((-q)^{1-m}z)}{d_{n,n^m}(z)} \equiv \dfrac{ \displaystyle \prod_{s=1}^{n'} (z-(-q)^{4s+m-5} )(z-(-q)^{4s+m-3} )    }{d_{n,n^m}(z)} \in \ko[z^{\pm 1}].
\end{aligned}
\end{equation} 
By the homomorphism
\[
\Vkm{n}_{(-q)^{m-1}} \otimes \Vkm{n^{m-1}}_{(-q)^{-1}} \twoheadrightarrow \Vkm{n^m},
\]
we have
\begin{equation}\label{eq: annmstep1p}
\begin{aligned}
& \dfrac{ d_{n,n}((-q)^{m-1}z)d_{n,n^{m-1}}((-q)^{-1}z)}{d_{n,n^m}(z)} \equiv
\dfrac{ \displaystyle \prod_{s=1}^{n'} (z-(-q)^{4s-m-1} ) (z-(-q)^{4s+m-3} )  }{d_{n,n^m}(z)} \in \ko[z^{\pm 1}].
\end{aligned}
\end{equation}

On the other hand, the homomorphism
\[
\Vkm{n^m}  \otimes \Vkm{n}_{(-q)^{2n+m-3}}  \twoheadrightarrow \Vkm{n^{m-1}}_{(-q)^{-1}}
\]
implies that
\[
\dfrac{ d_{n,n^{m}}(z)d_{n,n}((-q)^{-2n-m+3}z)}{d_{n,n^{m-1}}((-q)z)} \times \dfrac{a_{n,n^{m-1}}((-q)z)}{ a_{n,n^{m}}(z)a_{n,n}((-q)^{-2n-m+3}z)} \in \ko[z^{\pm 1}].
\]
By our induction hypothesis and a direct computation, we have
\begin{align*}
\dfrac{a_{n,n^{m-1}}((-q)z)}{ a_{n,n^{m}}(z)a_{n,n}((-q)^{-2n-m+3}z)} & = \prod_{s=1}^{n'}\dfrac{(z-(-q)^{m-5+4s})}{(z-(-q)^{m-3+4s})}, \allowdisplaybreaks \\
\dfrac{ d_{n,n^{m}}(z)d_{n,n}((-q)^{-2n-m+3}z)}{d_{n,n^{m-1}}((-q)z)} & =  \dfrac{ d_{n,n^{m}}(z)\displaystyle \prod_{s=1}^{n'} (z-(-q)^{2n+m-5+4s} ) }{ \displaystyle\prod_{s=1}^{n'} (z-(-q)^{4s+m-5}) }.
\end{align*}
Therefore, we have
\begin{align*}
& \dfrac{ d_{n,n^{m}}(z) \displaystyle \prod_{s=1}^{n'} (z-(-q)^{2n+m-5+4s} ) }{ \displaystyle\prod_{s=1}^{n'} (z-(-q)^{4s+m-5}) } \times  \prod_{s=1}^{n'}\dfrac{(z-(-q)^{m-5+4s})}{(z-(-q)^{m-3+4s})}
=  \dfrac{ d_{n,n^{m}}(z) \displaystyle \prod_{s=1}^{n'} (z-(-q)^{2n+m-5+4s} ) }{ \displaystyle\prod_{s=1}^{n'}(z-(-q)^{m-3+4s})} \in\ko[z^{\pm 1}]
\end{align*}
Thus we can conclude that
\[
d_{n,n^{m}}(z) =  \prod_{s=1}^{n'} (z-(-q)^{4s+m-5} )^{\epsilon_s'} \prod_{s=1}^{n'} (z-(-q)^{4s+m-3} )
\]
for some $\epsilon_s' \in \{0,1\}$.

\noindent (b-1)
For simplicity, write $n''= \left\lfloor \dfrac{n-1}{2} \right\rfloor$.
By the homomorphism
\[
\Vkm{n^{m-1}}_{(-q)} \otimes \Vkm{n}_{(-q)^{1-m}} \twoheadrightarrow \Vkm{n^m},
\]
we have
\begin{equation}\label{eq: an-1nmstep1}
\begin{aligned}
& \dfrac{ d_{n-1,n^{m-1}}((-q)z)d_{n-1,n}((-q)^{1-m}z)}{d_{n-1,n^m}(z)} \equiv \dfrac{ \displaystyle \prod_{s=1}^{n''} (z-(-q)^{4s+m-3} ) (z-(-q)^{4s+m-1})  }{d_{n-1,n^m}(z)} \in \ko[z^{\pm 1}].
\end{aligned}
\end{equation}
By the homomorphism
\[
\Vkm{n}_{(-q)^{m-1}} \otimes \Vkm{n^{m-1}}_{(-q)^{-1}} \twoheadrightarrow \Vkm{n^m},
\]
we have
\begin{equation}\label{eq: an-1nmstep1p}
\begin{aligned}
& \dfrac{ d_{n-1,n}((-q)^{m-1}z)d_{n-1,n^{m-1}}((-q)^{-1}z)}{d_{n-1,n^m}(z)} \equiv \dfrac{ \displaystyle  \prod_{s=1}^{n''} (z-(-q)^{4s-m+1} ) (z-(-q)^{4s+m-1} ) }{d_{n-1,n^m}(z)} \in \ko[z^{\pm 1}].
\end{aligned}
\end{equation}
On the other hand, the homomorphism
\[
\Vkm{n^m}  \otimes \Vkm{n}_{(-q)^{2n+m-3}}  \twoheadrightarrow \Vkm{n^{m-1}}_{(-q)^{-1}}
\]
implies that we have
\[
\dfrac{ d_{n-1,n^{m}}(z)d_{n-1,n}((-q)^{-2n-m+3}z)}{d_{n^{m-1},n-1}((-q)z)} \times \dfrac{a_{(n^{m-1},n-1}((-q)z)}{ a_{n-1,n^{m}}(z)a_{n-1,n}((-q)^{-2n-m+3}z)} \in \ko[z^{\pm 1}].
\]
By induction and a direct computation, we have
\begin{align*}
\dfrac{a_{n^{m-1},n-1}((-q)z)}{ a_{n-1,n^{m}}(z)a_{n-1,n}((-q)^{-2n-m+3}z)} &= \prod_{s=1}^{n''}\dfrac{(z-(-q)^{m-3+4s})}{(z-(-q)^{m-1+4s})}, \allowdisplaybreaks \\
\dfrac{ d_{n-1,n^{m}}(z)d_{n-1,n}((-q)^{-2n-m+3}z)}{d_{(n-1)^{m-1},n}((-q)z)} & =  \dfrac{ d_{n-1,n^{m}}(z) \displaystyle\prod_{s=1}^{n''} (z-(-q)^{2n+m-3+4s} )  }{\displaystyle\prod_{s=1}^{n''} (z-(-q)^{4s+m-3} ) }.
\end{align*}
Thus we have
\begin{align*}
&\dfrac{ d_{n-1,n^{m}}(z) \displaystyle\prod_{s=1}^{n''} (z-(-q)^{2n+m-3+4s} )  }{\displaystyle\prod_{s=1}^{n''} (z-(-q)^{4s+m-3} ) }
\times \prod_{s=1}^{n''}\dfrac{(z-(-q)^{m-3+4s})}{(z-(-q)^{m-1+4s})} \allowdisplaybreaks\\
&=   \dfrac{ d_{n-1,n^{m}}(z) \displaystyle\prod_{s=1}^{n''} (z-(-q)^{2n+m-3+4s} )  }{\displaystyle\prod_{s=1}^{n''} (z-(-q)^{m-1+4s})} \in \ko[z^{\pm1}].
\end{align*}
Note that $2n+m-3+4s' \ne m-1+4s$. 
Thus we can conclude that
\[
d_{n-1,n^m}(z) =  \prod_{s=1}^{n''} (z-(-q)^{4s+m-3} )^{\epsilon_s''} \prod_{s=1}^{n''} (z-(-q)^{4s+m-1})
\]
for some $\epsilon_s'' \in \{0,1\}$.

By~\eqref{eq: annmstep1}, ~\eqref{eq: annmstep1p}, ~\eqref{eq: an-1nmstep1} and~\eqref{eq: an-1nmstep1p},
it is enough to consider when $m \le n$.

\noindent (a,b) Note that we have a surjective homomorphism
\[
\Vkm{1}_{(-q)^{-n+2}}  \otimes  \Vkm{n-2}_{(-q)} \twoheadrightarrow  \Vkm{n-1} \otimes \Vkm{n},
\]
which implies
\begin{align*}
\dfrac{ d_{1,n^m}((-q)^{-n+2}z)d_{n-2,n^m}((-q)z)}{d_{n-1,n^m}(z)d_{n,n^m}(z)}  \in \ko[z^{\pm 1}].
\end{align*}
Note that  
\begin{align*}
&\dfrac{ d_{1,n^m}((-q)^{-n+2}z)d_{n-2,n^m}((-q)z)}{d_{n-1,n^m}(z)d_{n,n^m}(z)} \\
&= \dfrac{ \displaystyle \prod_{s=1}^{n-1} (z-(-q)^{m-1+2s})}
{\displaystyle\prod_{s=1}^{n''} (z-(-q)^{4s+m-3} )^{\epsilon_s''} \prod_{s=1}^{n''} (z-(-q)^{4s+m-1})
\prod_{s=1}^{n'} (z-(-q)^{4s+m-5} )^{\epsilon_s'} \prod_{s=1}^{n'} (z-(-q)^{4s+m-3} ) }\in \ko[z^{\pm 1}].
\end{align*}
Thus we can conclude that $\epsilon_s' = \epsilon_s''=0$ for all cases. Hence our assertion follows.
\end{proof}

As Lemma~\ref{lemma: division B n}, we obtain the following lemma:
\begin{lemma}\label{prop: division n,n}
For $m , p \in \Z_{\ge 1}$, we have
\[
d_{n^p,n^m}(z) \text{ divides } \prod_{t=0}^{\min(p,m)-1}\prod_{s=1}^{\lfloor \frac{n}{2} \rfloor} (z-(-q)^{4s+2t-2+|p-m|} ).
\]
\end{lemma}

\begin{lemma}
For $|m - p| \ge n-2$, we have
\[
d_{n^p,n^m}(z)= \prod_{t=0}^{\min(p,m)-1}\prod_{s=1}^{\lfloor \frac{n}{2} \rfloor} (z-(-q)^{4s+2t-2+|p-m|} ).
\]
\end{lemma}

\begin{proof}
For simplicity, write $n'=\dfrac{n}{2}$. Without loss of generality, we can assume that $2 \le p \le m$.
Then our assertion can be re-written as follows:
\[
d_{n^p,n^m}(z)= \prod_{t=0}^{p-1}\prod_{s=1}^{n'} (z-(-q)^{4s+2t-2+m-p}).
\]
By the homomorphism
\[
\Vkm{n^p}  \otimes \Vkm{n}_{(-q)^{2n+p-3}}  \twoheadrightarrow \Vkm{n^{p-1}}_{(-q)^{-1}}
\]
tells that we have
\[
\dfrac{ d_{n^{p},n^m}(z)d_{n,n^m}((-q)^{-2n-p+3}z)}{d_{n^{p-1},n^m}((-q)z)} \times \dfrac{a_{n^{p-1},n^m}((-q)z)}{ a_{n^{p},n^m}(z)a_{n,n^m}((-q)^{-2n-p+3}z)}\in \ko[z^{\pm 1}].
\]
By induction, we have
\begin{align*}
&\dfrac{ d_{n^{p},n^m}(z)d_{n,n^m}((-q)^{-2n-p+3}z)}{d_{n^{p-1},n^m}((-q)z)}
= \dfrac{ d_{n^{p},n^m}(z)\displaystyle \prod_{s=1}^{n'} (z-(-q)^{2n+p+m-6+4s} ) }{ \displaystyle \prod_{t=0}^{p-2}\prod_{s=1}^{n'} (z-(-q)^{4s+2t-2+m-p} ) }
\end{align*}
On the other hand, we have 
\begin{align*}
&  \dfrac{a_{n^{p-1},n^m}((-q)z)}{ a_{n^{p},n^m}(z)a_{n,n^m}((-q)^{-2n-p+3}z)} = \prod_{s=1}^{n'} \dfrac{(z-(-q)^{p-m-4+4s})}{(z-(-q)^{m+p-4+4s})}
\end{align*}
Thus we have
\begin{align*}
& \dfrac{ d_{n^{p},n^m}(z)\displaystyle \prod_{s=1}^{n'} (z-(-q)^{2n+p+m-6+4s} ) }{ \displaystyle \prod_{t=0}^{p-2}\prod_{s=1}^{n'} (z-(-q)^{4s+2t-2+m-p} ) }\times
 \prod_{s=1}^{n'} \dfrac{(z-(-q)^{p-m-4+4s})}{(z-(-q)^{m+p-4+4s})} \allowdisplaybreaks\\
& = \dfrac{ d_{n^{p},n^m}(z)\displaystyle \prod_{s=1}^{n'} (z-(-q)^{2n+p+m-6+4s} ) }{ \displaystyle \prod_{t=0}^{p-1}\prod_{s=1}^{n'} (z-(-q)^{m-p-2+4s+2t} ) }\times
 \prod_{s=1}^{n'} (z-(-q)^{p-m-4+4s})  \in \ko[z^{\pm 1}].
\end{align*}
Since $2n+p+m-6+4s' \ne m-p-2+4s+2t$,  
the above equation is equivalent to
\begin{align}\label{eq: m-p ge n}
\dfrac{ d_{n^{p},n^m}(z)\displaystyle \prod_{s=1}^{n'} (z-(-q)^{p-m-4+4s}) }{ \displaystyle \prod_{t=0}^{p-1}\prod_{s=1}^{n'} (z-(-q)^{m-p-2+4s+2t} ) }  \in \ko[z^{\pm 1}].
\end{align}
Therefore,  
our assertion holds, since $(u\seteq m-p \ge n-2)$  
\begin{align*}
p-m-4+4s' \ne m-p-2+4s+2t \iff  2(s'-s)-t \ne u+1
\end{align*}
for $1 \le s \le n'$ and $0 \le t \le p-1$. 
The remaining cases can be proved in a similar way.
\end{proof}

\begin{lemma}\label{lemma: division n,n-1}
For $m,p \ge 1$,
\[
d_{(n-1)^p,n^m}(z) \text{ divides } \prod_{t=0}^{\min(p,m)-1}\prod_{s=1}^{\lfloor \frac{n-1}{2} \rfloor} (z-(-q)^{4s+2t+|p-m|} ).
\]
\end{lemma}

\begin{proof}
\noindent (1) We first assume that $2 \le p \le m$.
By the homomorphism
\[
\Vkm{(n-1)^{p-1}}_{(-q)} \otimes \Vkm{n-1}_{(-q)^{1-p}} \twoheadrightarrow \Vkm{(n-1)^p},
\]
we have
\[
\dfrac{ d_{(n-1)^{p-1},n^m}((-q)z)d_{n-1,n^m}((-q)^{1-p}z)}{d_{(n-1)^p,n^m}(z)} \in \ko[z^{\pm 1}]
\]
since
\[
\dfrac{a_{(n-1)^p,n^m}(z)}{ a_{(n-1)^{p-1},n^m}((-q)z)a_{n-1,n^m}((-q)^{1-p}z)} = 1.
\]
By induction, we have
\[
d_{(n-1)^{p-1},n^m}((-q)z) \text{ divides } \prod_{t=0}^{p-2}\prod_{s=1}^{n''} (z-(-q)^{4s+2t+m-p} ),
\]
and we know
\[
d_{n-1,n^m}((-q)^{1-p}z) =  \prod_{s=1}^{\lfloor \frac{n-1}{2} \rfloor} (z-(-q)^{4s+m+p-2}).
\]
Thus our assertion follows.
\end{proof}

\begin{lemma}
For $|m - p| \ge n-2$, we have
\[
d_{(n-1)^p,n^m}(z)=  \prod_{t=0}^{\min(p,m)-1}\prod_{s=1}^{\lfloor \frac{n-1}{2} \rfloor} (z-(-q)^{4s+2t+|p-m|} ).
\]
\end{lemma}

\begin{proof}
For simplicity, write $n''=\lfloor \frac{n-1}{2} \rfloor$.

\noindent (1) We first assume that $2 \le p \le m$. Then our assertion can be re-written as follows:
\[
d_{(n-1)^p,n^m}(z)=  \prod_{t=0}^{p-1}\prod_{s=1}^{n''} (z-(-q)^{4s+2t+m-p} ).
\]
By the homomorphism
\[
\Vkm{(n-1)^p}  \otimes \Vkm{n-1}_{(-q)^{2n+p-3}}  \twoheadrightarrow \Vkm{(n-1)^{p-1}}_{(-q)^{-1}}
\]
tells that we have
\[
\dfrac{ d_{(n-1)^{p},n^m}(z)d_{(n-1),n^m}((-q)^{-2n-p+3}z)}{d_{(n-1)^{p-1},n^m}((-q)z)} \times \dfrac{a_{(n-1)^{p-1},n^m}((-q)z)}{a_{(n-1)^{p},n^m}(z)a_{(n-1),n^m}((-q)^{-2n-p+3}z)} \in \ko[z^{\pm 1}].
\]
By induction, we have
\begin{align*}
& \dfrac{ d_{(n-1)^{p},n^m}(z)d_{(n-1),n^m}((-q)^{-2n-p+3}z)}{d_{n-1^{p-1},n^m}((-q)z)}
=  \dfrac{ d_{(n-1)^{p},n^m}(z) \displaystyle\prod_{s=1}^{n''} (z-(-q)^{2n+p+m-4+4s} )}{\displaystyle\prod_{t=0}^{p-2}\prod_{s=1}^{n''} (z-(-q)^{4s+2t+m-p} ) }.
\end{align*}
A direct computation shows
\begin{align*}
 \dfrac{a_{(n-1)^{p-1},n^m}((-q)z)}{a_{(n-1)^{p},n^m}(z)a_{(n-1),n^m}((-q)^{-2n-p+3}z)}  = \prod_{s=1}^{n''}\dfrac{(z-(-q)^{-m+p-2+4s})}{(z-(-q)^{m+p-2+4s})}.
\end{align*}
Thus we have
\begin{align*}
&\dfrac{ d_{(n-1)^{p},n^m}(z) \displaystyle\prod_{s=1}^{n''} (z-(-q)^{2n+p+m-4+4s}) }{\displaystyle\prod_{t=0}^{p-2}\prod_{s=1}^{n''} (z-(-q)^{4s+2t+m-p} ) }
\times   \prod_{s=1}^{n''}\dfrac{(z-(-q)^{-m+p-2+4s})}{(z-(-q)^{m+p-2+4s})}\allowdisplaybreaks \\
& = \dfrac{ d_{(n-1)^{p},n^m}(z) \displaystyle\prod_{s=1}^{n''} (z-(-q)^{2n+p+m-4+4s}) }{\displaystyle\prod_{t=0}^{p-1}\prod_{s=1}^{n''} (z-(-q)^{4s+2t+m-p} ) }
\times   \prod_{s=1}^{n''} (z-(-q)^{-m+p-2+4s})  \in \ko[z^{\pm 1}].
\end{align*}
Note that 
\begin{align*}
2n+p+m-4+4s' \ne  4s+2t+m-p \iff n+p-2 \ne  2(s-s')+t 
\end{align*}
 for $1 \le s,s' \le n''$ and $0 \le t \le p-1$. Thus the above equation can be written as
\begin{align*}
\dfrac{ d_{(n-1)^{p},n^m}(z)\times \displaystyle   \prod_{s=1}^{n''} (z-(-q)^{-m+p-2+4s})  }{\displaystyle\prod_{t=0}^{p-1}\prod_{s=1}^{n''} (z-(-q)^{4s+2t+m-p} ) }
  \in \ko[z^{\pm 1}].
\end{align*} 
Since $(u\seteq m-p \ge n-2)$  
\begin{align*}
-m+p-2+4s' \ne 4s+2t+m-p \iff 2(s'-s)-t \ne u+1
\end{align*}
for $1 \le s,s' \le n''$ and $0 \le t \le p-1$, the assertion follows with Lemma~\ref{lemma: division n,n-1}.  The remaining case can be proved in a similar way.
\end{proof}
 
\begin{proof}[Proof of {\rm Theorem~\ref{thm:denominators_untwisted}} for $d_{(n-1)^p,n^m}(z)$ and $d_{n^p,n^m}(z)$ in type $D_n^{(1)}$]
Consider the homomorphism
\[
\Vkm{n^m} \otimes  \Vkm{n}_{(-q)^{-m-1}} \twoheadrightarrow \Vkm{n^{m+1}}_{(-q)^{-1}}
\]
then we have
\[
\dfrac{d_{n^p,n^m}(z) d_{n^p,n}((-q)^{-m-1} z)}{d_{n^p,n^{m+1}}((-q)^{-1} z)} \times \dfrac{a_{n^p,n^{m+1}}((-q)^{-1} z)}{a_{n^p,n^m}(z) a_{n^p,n}((-q)^{-m-1} z)} \in \ko[z^{\pm1}].
\]
By direct computation, we have
\begin{align*}
\dfrac{a_{n^p,n^{m+1}}((-q)^{-1} z)}{a_{n^p,n^m}(z) a_{n^p,n}((-q)^{-m-1} z)} & = 1
\end{align*}
and
\begin{align*}
&\dfrac{d_{n^p,n^m}(z) d_{n^p,n}((-q)^{-m-1} z)}{d_{n^p,n^{m+1}}((-q)^{-1} z)}
= \dfrac{d_{n^p,n^m}(z) \displaystyle \prod_{s=1}^{n'} (z-(-q)^{4s+2t+m+p-2} ) }{ \displaystyle \prod_{t=0}^{p-1}\prod_{s=1}^{n'} (z-(-q)^{4s+2t+m-p} ) } \allowdisplaybreaks\\
&= \dfrac{d_{n^p,n^m}(z)}{ \displaystyle \prod_{t=1}^{p-1}\prod_{s=1}^{n'} (z-(-q)^{4s+2t+m-p-2} ) } \in \ko[z^{\pm1}]
\end{align*}
which follows from the descending induction on $|m-p|$. Thus we can conclude that
$$
d_{n^p,n^m} = \prod_{t=1}^{p-1}\prod_{s=1}^{n'} (z-(-q)^{4s+2t+m-p-2} )  \times \prod_{s=1}^{n'} (z-(-q)^{4s+m-p-2} )^{\epsilon_s}
$$
for some $\epsilon_s \in \{0,1\}$.

Similarly, we have
\[
\dfrac{d_{(n-1)^p,n^m}(z) d_{(n-1)^p,n}((-q)^{-m-1} z)}{d_{(n-1)^p,n^{m+1}}((-q)^{-1} z)} \times \dfrac{a_{(n-1)^p,n^{m+1}}((-q)^{-1} z)}{a_{(n-1)^p,n^m}(z) a_{(n-1)^p,n}((-q)^{-m-1} z)} \in \ko[z^{\pm1}].
\]
By direct computation, we have
\begin{align*}
\dfrac{a_{(n-1)^p,n^{m+1}}((-q)^{-1} z)}{a_{(n-1)^p,n^m}(z) a_{(n-1)^p,n}((-q)^{-m-1} z)} & = 1
\end{align*}
and
\begin{align*}
&\dfrac{d_{(n-1)^p,n^m}(z) d_{(n-1)^p,n}((-q)^{-m-1} z)}{d_{(n-1)^p,n^{m+1}}((-q)^{-1} z)}
= \dfrac{d_{(n-1)^p,n^m}(z) \displaystyle \prod_{s=1}^{n''} (z-(-q)^{4s+2t+p+m} ) }{ \displaystyle \prod_{t=0}^{p-1}\prod_{s=1}^{n''} (z-(-q)^{4s+2t+m+2-p} ) } \allowdisplaybreaks\\
&= \dfrac{d_{(n-1)^p,n^m}(z)}{   \displaystyle \prod_{t=1}^{p-1}\prod_{s=1}^{n''} (z-(-q)^{4s+2t+m-p} )  } \in \ko[z^{\pm1}]
\end{align*}
which follows from the descending induction on $|m-p|$. Thus we can conclude that
$$
d_{(n-1)^p,n^m} =\prod_{t=1}^{p-1}\prod_{s=1}^{n''} (z-(-q)^{4s+2t+m-p} ) \times \prod_{s=1}^{n''} (z-(-q)^{4s+m-p} )^{\epsilon'_s}
$$
for some $\epsilon'_s \in \{0,1\}$.

By the homomorphism obtained from Theorem~\ref{thm: Higher Dorey I}~\eqref{eq: k+l=n homo} with the restriction $\min(k,l)=1$
\[
\Vkm{(n-1)^m} \otimes \Vkm{n^m} \otimes \Vkm{1^m}_{(-q)^{n}} \twoheadrightarrow   \Vkm{(n-2)^m}_{(-q)},
\]
we have
\[
\dfrac{d_{n^p,(n-1)^m}(z) d_{n^p,n^m}(z) d_{n^p,1^m}((-q)^{n} z)}{d_{n^p,(n-2)^{m}}((-q) z)} \times
\dfrac{a_{n^p,(n-2)^{m}}((-q) z)}{a_{n^p,(n-1)^m}(z) a_{n^p,n^m}(z) a_{n^p,1^m}((-q)^{n} z)} \in \ko[z^{\pm1}].
\]

By direct computation, we have $(n=2n')$
\begin{align*}
& \dfrac{a_{n^p,(n-2)^{m}}((-q) z)}{a_{n^p,(n-1)^m}(z) a_{n^p,n^m}(z) a_{n^p,1^m}((-q)^{n} z)} = \prod_{t=0}^{p-1}  \dfrac{(z-(-q)^{p-m-2-2t})}{(z-(-q)^{m-p+2t})}
\end{align*}
and
\begin{align*}
& \dfrac{d_{n^p,(n-1)^m}(z) d_{n^p,n^m}(z) d_{n^p,1^m}((-q)^{n} z)}{d_{n^p,(n-2)^{m}}((-q) z)}  \allowdisplaybreaks\\
& = \dfrac{ \displaystyle \prod_{s=1}^{n''} (z-(-q)^{4s+m-p} )^{\epsilon'_s} \times \prod_{s=1}^{n'} (z-(-q)^{4s+m-p-2} )^{\epsilon_s} \times \prod_{t=0}^{p-1} (z-(-q)^{m-p+2t}) }
{\displaystyle \prod_{t=1}^{p-1}\prod_{s=1}^{n-2} (z-(-q)^{m-p+2(s+t)})   \times \prod_{s=1}^{n-2} (z-(-q)^{m-p+2s})} \allowdisplaybreaks\\
& \hspace{15ex} \times \prod_{t=1}^{p-1}\prod_{s=1}^{n-1} (z-(-q)^{2s+2t+m-p} ) \times \prod_{t=0}^{p-1}  \dfrac{(z-(-q)^{p-m-2-2t})}{(z-(-q)^{m-p+2t})} \allowdisplaybreaks\\
& = \dfrac{ \displaystyle \prod_{s=1}^{n''} (z-(-q)^{4s+m-p} )^{\epsilon'_s} \times \prod_{s=1}^{n'} (z-(-q)^{4s+m-p-2} )^{\epsilon_s}}
{\displaystyle \prod_{s=1}^{n-2} (z-(-q)^{m-p+2s})}\allowdisplaybreaks \\
& \hspace{15ex} \times \prod_{t=1}^{p-1} (z-(-q)^{2n-2+2t+m-p} ) \times \prod_{t=0}^{p-1}  (z-(-q)^{p-m-2-2t}) \in \ko[z^{\pm1}] \allowdisplaybreaks\\
& \Rightarrow \dfrac{ \displaystyle \prod_{s=1}^{n''} (z-(-q)^{4s+m-p} )^{\epsilon'_s} \times \prod_{s=1}^{n'} (z-(-q)^{4s+m-p-2} )^{\epsilon_s}}
{\displaystyle \prod_{s=1}^{n-2} (z-(-q)^{m-p+2s})} \in \ko[z^{\pm1}],
\end{align*}
which implies $\epsilon_s = \epsilon'_s=1$ for $1 \le s \le n'-1$. Hence, $\epsilon_{n'} = 1$ by the universal coefficient formula. 
The remaining cases can be proved in a similar way.
\end{proof}

\subsubsection{$d_{l^p,n^m}(z)$ for $1 \leq l < n-1$}
We recall that by the diagram symmetry, we have $d_{l^p,n^m}(z) = d_{l^p,(n-1)^m}(z)$.
For $1 \le  l <n-1$, we have
\begin{equation} \label{eq: basic kn for D}
\begin{aligned} 
& d_{k,n-1}(z)=d_{k,n}(z) = \displaystyle \prod_{s=1}^{k}\big(z-(-q)^{n-k-1+2s}\big),  \\
& a_{k,n-1}(z)=a_{k,n}(z) \equiv  \dfrac{[n-k-1][3n+k-3]}{[n+k-1][3n-k-3]}.
\end{aligned}
\end{equation}

\begin{lemma}
We have
$
d_{1,n^2}(z) =d_{1,(n-1)^2}(z) = (z - (-q)^{n+1}).
$
\end{lemma}

\begin{proof}
By~\eqref{eq: basic kn for D}, we have $d_{1,n}(z) =z-(-q)^{n}$. Hence it is enough to consider
\begin{align*}
{\rm (i)} & \ \ \de(\Vkm{1},\Vkm{n^2}_{(-q)^{n-1}}) = \de(\Vkm{1},\Vkm{n}_{(-q)^n} \hconv \Vkm{n}_{(-q)^{n-2}}),      \\
{\rm (ii)} & \ \ \de(\Vkm{1},\Vkm{n^2}_{(-q)^{n+1}}) = \de(\Vkm{1},\Vkm{n}_{(-q)^{n+2}} \hconv \Vkm{n}_{(-q)^{n}}).       
\end{align*}
By Theorem~\ref{thm: i-box commute}, $\de(\Vkm{1},\Vkm{n^2}_{(-q)^{n-1}})=0$.
Note that the sequences 
$$
(\Vkm{1},\Vkm{n}_{(-q)^{n}}, \Vkm{n}_{(-q)^{n-2}}) \qtq (\Vkm{n}_{(-q)^{n+2}}, \Vkm{n}_{(-q)^{n}},\Vkm{1})
$$ are normal, since 
$$
\de(\scrD \Vkm{1}, \Vkm{n}_{(-q)^{n-2}}) =0 \qtq \de( \Vkm{1}, \Vkm{n}_{(-q)^{n-2}}) =0,
$$
respectively. Thus Lemma~\ref{lem: normal seq d}~\eqref{it: de +}{\rm (v)} tells that
$$
\de(\Vkm{1},\Vkm{n^2}_{(-q)^{n+1}}) =1,
$$
which completes the assertion. 
\end{proof}

\begin{lemma}
For $m \ge 3$, we have
\[
d_{1,n^m}(z)= (z-(-q)^{n+m-1}) \quad \text{ and } \quad d_{1,(n-1)^m}(z)= (z-(-q)^{n+m-1}).
\]
\end{lemma}

\begin{proof}
By the homomorphism
\[
\Vkm{n^{m-1}}_{(-q)} \otimes \Vkm{n}_{(-q)^{1-m}} \twoheadrightarrow \Vkm{n^m},
\]
we have
\begin{equation}\label{eq: a1nmstep1}
\begin{aligned}
& \dfrac{ d_{1,n^{m-1}}((-q)z)d_{1,n}((-q)^{1-m}z)}{d_{1,n^m}(z)} \equiv \dfrac{  (z-(-q)^{n+m-3}) (z-(-q)^{n+m-1})  }{d_{1,n^m}(z)} \in \ko[z^{\pm 1}].
\end{aligned}
\end{equation}
By the homomorphism
\[
\Vkm{n}_{(-q)^{m-1}} \otimes \Vkm{n^{m-1}}_{(-q)^{-1}} \twoheadrightarrow \Vkm{n^m},
\]
we have
\begin{equation}\label{eq: a1nmstep1p}
\begin{aligned}
& \dfrac{ d_{1,n}((-q)^{m-1}z)d_{1,n^{m-1}}((-q)^{-1}z)}{d_{1,n^m}(z)} \equiv \dfrac{  (z-(-q)^{n-m+1}) (z-(-q)^{n+m-1})  }{d_{1,n^m}(z)} \in \ko[z^{\pm 1}].
\end{aligned}
\end{equation}
Thus a possible ambiguity happens at $m=2$ which is already covered in the previous lemma.

On the other hand, the homomorphism
\[
\Vkm{n^m}  \otimes \Vkm{n}_{(-q)^{2n+m-3}}  \twoheadrightarrow \Vkm{n^{m-1}}_{(-q)^{-1}}
\]
implies that
\[
\dfrac{ d_{1,n^{m}}(z)d_{1,n}((-q)^{-2n-m+3}z)}{d_{1,n^{m-1}}((-q)z)} \times \dfrac{a_{1,n^{m-1}}((-q)z)}{a_{1,n^{m}}(z)a_{1,n}((-q)^{-2n-m+3}z)} \in \ko[z^{\pm 1}].
\]
By our induction hypothesis, we have
\begin{align*}
& \dfrac{ d_{1,n^{m}}(z)d_{1,n}((-q)^{-2n-m+3}z)}{d_{1,n^{m-1}}((-q)z)}
= \dfrac{ d_{1,n^{m}}(z) (z-(-q)^{3n+m-3})}{(z-(-q)^{n+m-3})}.
\end{align*}
On the other hand, we have 
\[
\dfrac{a_{1,n^{m-1}}((-q)z)}{a_{1,n^{m}}(z)a_{1,n}((-q)^{-2n-m+3}z)} = \dfrac{(z-(-q)^{n+m-3})}{(z-(-q)^{n+m-1})}.
\]
Thus we have
\begin{equation}\label{eq: a1nmstep2}
\dfrac{ d_{1,n^{m}}(z) (z-(-q)^{3n+m-3})}{(z-(-q)^{n+m-3})}  \times \dfrac{(z-(-q)^{n+m-3})}{(z-(-q)^{n+m-1})}
= \dfrac{ d_{1,n^{m}}(z) (z-(-q)^{3n+m-3})}{(z-(-q)^{n+m-1})}\in \ko[z^{\pm 1}],
\end{equation}
which implies our assertion for $m \ge 2$.
\end{proof}

\begin{lemma}
For $1 \le l \le n-2$ and $m \ge 1$, we have
\[
d_{l,n^m}(z)= d_{l,(n-1)^m}(z)=  \prod_{s=1}^{l} (z-(-q)^{n-l+m-2+2s}).  
\]
\end{lemma}

\begin{proof}
By Dorey's rule in Theorem~\ref{thm: Dorey}
\[
\Vkm{l-1}_{(-q)^{-1}} \otimes \Vkm{1}_{(-q)^{l-1}} \twoheadrightarrow \Vkm{l},
\]
we have
\begin{equation}\label{eq: a1n2step1 D}
 \dfrac{ d_{l-1,n^m}((-q)^{-1}z)d_{1,n^m}((-q)^{l-1}z)}{d_{l,n^m}(z)} \equiv \dfrac{ \displaystyle \prod_{s=1}^{l} (z-(-q)^{n-l+m-2+2s})}{d_{l,n^m}(z)}
\end{equation}
By the homomorphism
\[
\Vkm{n^m}  \otimes \Vkm{n}_{(-q)^{2n+m-3}}  \twoheadrightarrow \Vkm{n^{m-1}}_{(-q)^{-1}},
\]
we have
\[
\dfrac{ d_{l,n^{m}}(z)d_{l,n}((-q)^{-2n-m+3}z)}{d_{l,n^{m-1}}((-q)z)} \times \dfrac{a_{l,n^{m-1}}((-q)z)}{a_{l,n^{m}}(z)a_{l,n}((-q)^{-2n-m+3}z)}\in \ko[z^{\pm 1}].
\]
By induction, we have
\[
\dfrac{ d_{l,n^{m}}(z)d_{l,n}((-q)^{-2n-m+3}z)}{d_{l,n^{m-1}}((-q)z)} \equiv
\dfrac{ d_{l,n^{m}}(z)  \displaystyle  \prod_{s=1}^{l} (z-(-q)^{3n+m-l-4+2s})   }{\displaystyle \prod_{s=1}^{l} (z-(-q)^{n-l+m-4+2s}) }.
\]
On the other hand, we have 
\[
\dfrac{a_{l,n^{m-1}}((-q)z)}{a_{l,n^{m}}(z)a_{l,n}((-q)^{-2n-m+3}z)} = \dfrac{(z-(-q)^{n+m-l-2})}{(z-(-q)^{n+m+l-2})}.
\]
Thus we have
\begin{align*}
& \dfrac{ d_{l,n^{m}}(z)  \displaystyle  \prod_{s=1}^{l} (z-(-q)^{3n+m-l-4+2s})   }{\displaystyle \prod_{s=1}^{l} (z-(-q)^{n-l+m-4+2s}) } \times \dfrac{(z-(-q)^{n+m-l-2})}{(z-(-q)^{n+m+l-2})} \allowdisplaybreaks \\
& \hspace{10ex} = \dfrac{ d_{l,n^{m}}(z)  \displaystyle  \prod_{s=1}^{l} (z-(-q)^{3n+m-l-4+2s})   }{\displaystyle \prod_{s=1}^{l} (z-(-q)^{n-l+m-2+2s}) } \in \ko[z^{\pm 1}],
\end{align*}
which implies our assertion, since $n-l+m-2+2s'\ne 3n+m-l-4+2s$ for $1 \le s \le l <n$.  
\end{proof}

\begin{lemma}
For $1 \le l \le n-2$ and $p \ge 1$, we have
\[
d_{l^p,n}(z)= \prod_{s=1}^{l} (z-(-q)^{n-l+p-2+2s}).
\]
\end{lemma}

\begin{proof}
Since $a_{l^p,n}(z) = a_{l,n^p}(z)$, we can apply the same argument of the previous lemma.
\end{proof}

As Lemma~\ref{lemma: division B}, we obtain the following lemma:

\begin{lemma} \label{lem: division D ln}
Let $1 \le l \le n-2$ and $p , m \ge 1$, then
\[
d_{l^p,n^m}(z) \text{ divides }  \prod_{t=0}^{\min(p,m)-1}\prod_{s=1}^{l} (z-(-q)^{n-l-1+|p-m|+2(s+t)}).
\]
\end{lemma}

\begin{lemma} For $1 \le l \le n-2$ and $|m-p| \ge l-1$, we have 
\[
d_{l^p,n^m}(z)=  \prod_{t=0}^{\min(p,m)-1}\prod_{s=1}^{l} (z-(-q)^{n-l-1+|p-m|+2(s+t)}).
\]
\end{lemma}

\begin{proof}
Assume that $p \le m$ and $m-p \ge l -1$. By the homomorphism
\[
\Vkm{l^p}  \otimes \Vkm{l}_{(-q)^{2n+p-3}}  \twoheadrightarrow `\Vkm{l^{p-1}}_{(-q)^{-1}},
\]
we have
\[
\dfrac{ d_{l^p,n^{m}}(z)d_{l,n^{m}}((-q)^{-2n-p+3}z)}{d_{l^{p-1},n^m}((-q)z)} \times \dfrac{a_{l^{p-1},n^m}((-q)z)}{a_{l^p,n^{m}}(z)a_{l,n^{m}}((-q)^{-2n-p+3}z)} \in \ko[z^{\pm 1}].
\]
By induction, we have
\begin{align*}
\dfrac{ d_{l^p,n^{m}}(z)d_{l,n^{m}}((-q)^{-2n-p+3}z)}{d_{l^{p-1},n^m}((-q)z)}
= \dfrac{ d_{l^p,n^{m}}(z) \displaystyle \prod_{s=1}^{l} (z-(-q)^{3n+p-l+m-5+2s})   }{\displaystyle \prod_{t=0}^{p-2}\prod_{s=1}^{l} (z-(-q)^{n-l-1+m-p+2(s+t)}) }
\end{align*}
On the other hand, we have 
\[
\dfrac{a_{l^{p-1},n^m}((-q)z)}{a_{l^p,n^{m}}(z)a_{l,n^{m}}((-q)^{-2n-p+3}z)} = \prod_{s=1}^{l} \dfrac{(z-(-q)^{n-l-m+p-3+2s})}{(z-(-q)^{n-l+m+p-3+2s})}.
\]
Thus we have
\begin{equation}\label{eq: alpnm 2p le m step2}
\begin{aligned}
& \dfrac{ d_{l^p,n^{m}}(z) \displaystyle \prod_{s=1}^{l} (z-(-q)^{3n+p-l+m-5+2s})   }{\displaystyle \prod_{t=0}^{p-2}\prod_{s=1}^{l} (z-(-q)^{n-l-1+m-p+2(s+t)}) }
\times \prod_{s=1}^{l} \dfrac{(z-(-q)^{n-l-m+p-3+2s})}{(z-(-q)^{n-l+m+p-3+2s})} \allowdisplaybreaks \\
& \hspace{5ex} = \dfrac{ d_{l^p,n^{m}}(z) \displaystyle \prod_{s=1}^{l} (z-(-q)^{3n+p-l+m-5+2s})   }{\displaystyle \prod_{t=0}^{p-1}\prod_{s=1}^{l} (z-(-q)^{n-l-1+m-p+2(s+t)}) }
\times \prod_{s=1}^{l} (z-(-q)^{n-l-m+p-3+2s})
\in \ko[z^{\pm 1}].
\end{aligned}
\end{equation}
Note that $3n+p-l+m-5+2s' \ne n-l-m+p-3+2s$. 
Thus the above equation can be written as follows:
\begin{align} \label{eq: t=0 s=l D}
\dfrac{ d_{l^p,n^{m}}(z) \displaystyle \times \prod_{s=1}^{l} (z-(-q)^{n+l-m+p-1-2s})}{\displaystyle \prod_{t=0}^{p-1}\prod_{s=1}^{l} (z-(-q)^{n-l-1+m-p+2(s+t)}) }
\in \ko[z^{\pm 1}].
\end{align}
Thus our assertion follows, since $(u\seteq m-p \ge l-1)$  
\begin{align*}
n-l-1+m-p+2(s+t) \ne n+l-m+p-1-2s' \iff u \ne l-(s'+s)-t
\end{align*}
for $1\le s \le l$ and $0\le t \le p-1$. 
The case when $\min(m,p)=m$ can be proved in a similar way.
\end{proof}

\begin{proof}[Proof of {\rm Theorem~\ref{thm:denominators_untwisted}} for $l < n-1$ and $k=n$ in type $D_n^{(1)}$]
We first assume that $p \le m$.
Considering the homomorphism obtained from Theorem~\ref{thm: Higher Dorey I}~\eqref{eq: k+l<n homo} with the restriction $\min(k,l)=1$
\[
\Vkm{l^p} \otimes  \Vkm{1^p}_{(-q)^{2n-l-1}}  \twoheadrightarrow  \Vkm{(l-1)^p}_{(-q)},
\]
we have
\[
\dfrac{d_{l^p,n^m}( z ) d_{1^p,n^m}((-q)^{2n-l-1}z)}{d_{(l-1)^p,n^m}((-q)z)} \times \dfrac{a_{(l-1)^p,n^m}((-q)z)}{a_{l^p,n^m}( z ) a_{1^p,n^m}((-q)^{2n-l-1}z)} \in \ko[z^{\pm1}].
\]
By direct computation,
\[
\dfrac{a_{(l-1)^p,n^m}((-q)z)}{a_{l^p,n^m}( z ) a_{1^p,n^m}((-q)^{2n-l-1}z)} = \prod_{t=0}^{p-1}  \dfrac{(z-(-q)^{-n+l-1-m+p-2t})}{(z-(-q)^{-n+l+1+m-p+2t})}
\]
Furthermore, by an induction on $l$, we have 
\begin{align*}
& \dfrac{d_{l^p,n^m}( z ) d_{1^p,n^m}((-q)^{2n-l-1}z)}{d_{(l-1)^p,n^m}((-q)z)}
 =  \dfrac{d_{l^p,n^m}( z )  \displaystyle\prod_{t=0}^{p-1}  (z-(-q)^{-n+l+m-p+1+2t})}
{ \displaystyle\prod_{t=0}^{p-1}\prod_{s=1}^{l-1} (z-(-q)^{n-l-1+m-p+2(s+t)})}
\end{align*}
Here the induction works because the range of $|m-p|$ happening ambiguity depends on $l$.
Thus we have 
\begin{align}
& \dfrac{d_{l^p,n^m}( z )  \displaystyle\prod_{t=0}^{p-1}  (z-(-q)^{-n+l+m-p+1+2t})}
{ \displaystyle\prod_{t=0}^{p-1}\prod_{s=1}^{l-1} (z-(-q)^{n-l-1+m-p+2(s+t)})}  \times
\prod_{t=0}^{p-1}  \dfrac{(z-(-q)^{-n+l-1-m+p-2t})}{(z-(-q)^{-n+l+1+m-p+2t})} \nonumber \allowdisplaybreaks\\
& =  \dfrac{d_{l^p,n^m}( z )  \displaystyle \prod_{t=0}^{p-1}   (z-(-q)^{-n+l-1-m+p-2t})}
{ \displaystyle\prod_{t=0}^{p-1}\prod_{s=1}^{l-1} (z-(-q)^{n-l-1+m-p+2(s+t)})}
\in \ko[z^{\pm1}]\allowdisplaybreaks \nonumber  \\
& \Rightarrow  \dfrac{d_{l^p,n^m}( z ) }{ \displaystyle\prod_{t=0}^{p-1}\prod_{s=1}^{l-1} (z-(-q)^{n-l-1+m-p+2(s+t)})}
\in \ko[z^{\pm1}] \label{eq: the one D}
\end{align}

Consider fusion rule
\[
\Vkm{n^m} \otimes  \Vkm{n}_{(-q)^{-m-1}} \twoheadrightarrow \Vkm{n^{m+1}}_{(-q)^{-1}}
\]
then we have
\[
\dfrac{d_{l^p,n^m}(z) d_{l^p,n}((-q)^{-m-1} z)}{d_{l^p,n^{m+1}}((-q)^{-1} z)} \times \dfrac{a_{l^p,n^{m+1}}((-q)^{-1} z)}{a_{l^p,n^m}(z) a_{l^p,n}((-q)^{-m-1} z)} \in \ko[z^{\pm1}].
\]
By direct computation, we have
\begin{align*}
\dfrac{a_{l^p,n^{m+1}}((-q)^{-1} z)}{a_{l^p,n^m}(z) a_{l^p,n}((-q)^{-m-1} z)} & = 1
\end{align*}
and 
\begin{align}
\dfrac{d_{l^p,n^m}(z) d_{l^p,n}((-q)^{-m-1} z)}{d_{l^p,n^{m+1}}((-q)^{-1} z)}
& = \dfrac{d_{l^p,n^m}(z)   \displaystyle\prod_{s=1}^{l} (z-(-q)^{n-l+p+m-1+2s}) }
{\displaystyle\prod_{t=0}^{p-1}\prod_{s=1}^{l} (z-(-q)^{n-l+1+m-p+2(s+t)})} \nonumber \allowdisplaybreaks \\
& = \dfrac{d_{l^p,n^m}(z)  }
{\displaystyle\prod_{t=1}^{p-1}\prod_{s=1}^{l} (z-(-q)^{n-l-1+m-p+2(s+t)})}
 \in \ko[z^{\pm1}] \label{eq: the one2 D}
\end{align}
which follows from the descending induction on $|m-p|$. 
From~\eqref{eq: t=0 s=l D},~\eqref{eq: the one D}, ~\eqref{eq: the one2 D} and Lemma~\ref{lem: division D ln},  
our assertion follows.
The remaining case of $p > m$ can be proved similarly.  
\end{proof}

\subsubsection{ $d_{k^m,l^p}(z)$ for $1 \le k, l <n-1$}
For $1 \le k, l <n-1$, we have
\begin{align*}
& d_{k,l}(z) = \displaystyle \prod_{s=1}^{\min (k,l)}   \big(z-(-q)^{|k-l|+2s}\big)\big(z-(-q)^{2n-2-k-l+2s}\big),  \\
& a_{k,l}(z) \equiv \dfrac{\PSF{|k-l|}{2n+k+l-2}{2n-k-l-2}{4n-|k-l|-4}}{\PSF{k+l}{2n+k-l-2}{2n-k+l-2}{4n-k-l-4}},
\end{align*}
and surjective homomorphisms
\[
\Vkm{k-1}_{(-q)^{-1}} \otimes  \Vkm{1}_{(-q)^{k-1}} \twoheadrightarrow \Vkm{k} \quad  \text{ and } \quad
\Vkm{1}_{(-q)^{1-l}} \otimes \Vkm{l-1}_{-q} \twoheadrightarrow \Vkm{l},
\]
as special cases in Theorem~\ref{thm: Dorey}.

\begin{lemma}
For any $m \in \Z_{\ge 1}$, we have
\begin{align}\label{eq: d11m D}
 d_{1,1^m}(z) =   (z-(-q)^{m+1})(z-(-q)^{2n+m-3}).
\end{align}
\end{lemma}

\begin{proof}
By the homomorphism,
\[
\Vkm{1^{m-1}}_{(-q)} \otimes \Vkm{1}_{(-q)^{1-m}} \twoheadrightarrow \Vkm{1^m},
\]
we have
\begin{align}
 \dfrac{ d_{1,1^{m-1}}((-q)z)d_{1,1}((-q)^{1-m}z)}{d_{1,1^m}(z)} & \nonumber \allowdisplaybreaks \\
&\hspace{-22ex} \equiv \dfrac{ (z-(-q)^{m-1})(z-(-q)^{2n+m-5})   (z-(-q)^{m+1})(z-(-q)^{2n+m-3})}{d_{1,1^m}(z)} \in \ko[z^{\pm 1}]. \label{eq: a11mstep1 D}
\end{align}
By the homomorphism,
\[
\Vkm{1}_{(-q)^{m-1}} \otimes \Vkm{1^{m-1}}_{(-q)^{-1}} \twoheadrightarrow \Vkm{1^m},
\]
we have
\begin{align}
 \dfrac{ d_{1,1}((-q)^{m-1}z)d_{1,1^{m-1}}((-q)^{-1}z)}{d_{1,1^m}(z)} & \nonumber \allowdisplaybreaks  \\
&\hspace{-22ex} \equiv \dfrac{ (z-(-q)^{3-m})(z-(-q)^{2n-m-1}) (z-(-q)^{m+1})(z-(-q)^{2n+m-3}) }{d_{1,1^m}(z)} \in \ko[z^{\pm 1}]. \label{eq: a11mstep1p D}
\end{align}
Thus a possible ambiguity happens at $m=2$. However, $(-q)$ cannot be a root of $d_{1,1^2}(z)$ by 
Theorem~\ref{thm: i-box commute}, and 
\begin{align} \label{eq: D-1L, LhL' 0}
\de(\Vkm{1},\Vkm{1^2}_{(-q)^{2n-3}}) = \de(\Vkm{1},\Vkm{1}_{(-q)^{2n-2}} \hconv \Vkm{1}_{(-q)^{2n-4}})=0,    
\end{align}
by the following argument: Take $L=\Vkm{1}_{(-q)^{2n-2}}$ and $L'=\Vkm{1}_{(-q)^{2n-4}}$. Then $(L,L')$ satisfies~\eqref{Eq: assumption} and
$\scrD^{-1} L = \Vkm{1}$. Hence Lemma~\ref{Lem: two root modules} implies~\eqref{eq: D-1L, LhL' 0}; i.e., $(-q)^{2n-3}$  can not be a root of $d_{1,1^2}(z)$, either. 

On the other hand, the homomorphism
\[
\Vkm{1^m}  \otimes \Vkm{1}_{(-q)^{2n+m-3}}  \twoheadrightarrow \Vkm{1^{m-1}}_{(-q)^{-1}}
\]
implies that we have
\[
\dfrac{ d_{1,1^{m}}(z)d_{1,1}((-q)^{-2n-m+3}z)}{d_{1,1^{m-1}}((-q)z)} \times \dfrac{a_{1,1^{m-1}}((-q)z)}{a_{1,1^{m}}(z)a_{1,1}((-q)^{-2n-m+3}z)}\in \ko[z^{\pm 1}].
\]
By induction, we have
\[
\dfrac{ d_{1,1^{m}}(z)d_{1,1}((-q)^{-2n-m+3}z)}{d_{1,1^{m-1}}((-q)z)} \equiv
\dfrac{ d_{1,1^{m}}(z)  (z-(-q)^{2n+m-1})(z-(-q)^{4n+m-5})}{ (z-(-q)^{m-1})(z-(-q)^{2n+m-5})}.
\]
Additionally, we have
\[
\dfrac{a_{1,1^{m-1}}((-q)z)}{a_{1,1^{m}}(z)a_{1,1}((-q)^{-2n-m+3}z)}  = \dfrac{(z-(-q)^{m-1})(z-(-q)^{2n+m-5})}{(z-(-q)^{m+1})(z-(-q)^{2n+m-3})},
\]
and so we have
\begin{equation}\label{eq: a11mstep2 D}
\begin{aligned}
&\dfrac{ d_{1,1^{m}}(z)  (z-(-q)^{2n+m-1})(z-(-q)^{4n+m-5})}{ (z-(-q)^{m-1})(z-(-q)^{2n+m-5})} \times \dfrac{(z-(-q)^{m-1})(z-(-q)^{2n+m-5})}{(z-(-q)^{m+1})(z-(-q)^{2n+m-3})} \\
& =\dfrac{ d_{1,1^{m}}(z)  (z-(-q)^{2n+m-1})(z-(-q)^{4n+m-5})}{ (z-(-q)^{m+1})(z-(-q)^{2n+m-3})} \in \ko[z^{\pm 1}].
\end{aligned}
\end{equation}
Thus our assertion follows.  
\end{proof}

\begin{lemma}
For any  $1 \le k \le n-2$ and $m \in \Z_{\ge 1}$, we have
\begin{align}\label{eq: dk1m}
 d_{k,1^m}(z) =  (z-(-q)^{k+m}) (z-(-q)^{2n-k+m-2}).
\end{align}
\end{lemma}

\begin{proof}
Our assertion for $k=1$ holds already. Now we consider the case when $k > 1$.
By Dorey's rule in Theorem~\ref{thm: Dorey}
\[
\Vkm{k-1}_{(-q)^{-1}} \otimes  \Vkm{1}_{(-q)^{k-1}} \twoheadrightarrow \Vkm{k},
\]
we have
\begin{align*}
& \dfrac{ d_{k-1,1^m}((-q)^{-1}z)d_{1,1^m}((-q)^{k-1}z)}{d_{k,1^m}(z)}  \times \dfrac{a_{k,1^m}(z)}{a_{k-1,1^m}((-q)^{-1}z)a_{1,1^m}((-q)^{k-1}z)} \in \ko[z^{\pm 1}]
\end{align*}
Note that
\begin{align*}
\dfrac{a_{k,1^m}(z)}{a_{k-1,1^m}((-q)^{-1}z)a_{1,1^m}((-q)^{k-1}z)} & =  \dfrac{(z-(-q)^{-m-k+2})}{(z-(-q)^{m-k+2})} \allowdisplaybreaks \\
\dfrac{ d_{k-1,1^m}((-q)^{-1}z)d_{1,1^m}((-q)^{k-1}z)}{d_{k,1^m}(z)}  & =
\dfrac{ (z-(-q)^{k+m}) (z-(-q)^{2n-k+m}) (z-(-q)^{m-k+2})(z-(-q)^{2n+m-k-2}) }{d_{k,1^m}(z)},
\end{align*}
and so we have 
\begin{align}
\dfrac{ (z-(-q)^{k+m}) (z-(-q)^{2n-k+m}) (z-(-q)^{m-k+2})(z-(-q)^{2n+m-k-2}) }{d_{k,1^m}(z)} &  \times  \dfrac{(z-(-q)^{-m-k+2})}{(z-(-q)^{m-k+2})}  \nonumber \allowdisplaybreaks\\
&\hspace{-65ex}= \dfrac{ (z-(-q)^{k+m}) (z-(-q)^{2n-k+m})(z-(-q)^{-m-k+2})(z-(-q)^{2n+m-k-2}) }{d_{k,1^m}(z)}     \in \ko[z^{\pm 1}]. \label{eq: ak1mstep1 D}
\end{align} 
By the homomorphism
\[
\Vkm{1}_{(-q)^{1-k}} \otimes  \Vkm{k-1}_{(-q)} \twoheadrightarrow \Vkm{k},
\]
we have
\begin{align*}
& \dfrac{ d_{1,1^m}((-q)^{1-k}z)d_{k-1,1^m}((-q)z)}{d_{k,1^m}(z)}  \times
\dfrac{a_{k,1^m}(z)}{ a_{1,1^m}((-q)^{1-k}z)a_{k-1,1^m}((-q)z)} \in \ko[z^{\pm 1}]
\end{align*}
Note that
\begin{align*}
\dfrac{a_{k,1^m}(z)}{ a_{1,1^m}((-q)^{1-k}z)a_{k-1,1^m}((-q)z)} & =\dfrac{(z-(-q)^{k-m-2})}{(z-(-q)^{k+m-2})} \allowdisplaybreaks \\
\dfrac{ d_{1,1^m}((-q)^{1-k}z)d_{k-1,1^m}((-q)z)}{d_{k,1^m}(z)}  & =
\dfrac{  (z-(-q)^{m+k}) (z-(-q)^{2n+m+k-2})  (z-(-q)^{k+m-2}) (z-(-q)^{2n-k+m-2})   }{d_{k,1^m}(z)},
\end{align*}
and so we have
\begin{align}
\dfrac{  (z-(-q)^{m+k}) (z-(-q)^{2n+m+k-2})  (z-(-q)^{k+m-2}) (z-(-q)^{2n-k+m-2})   }{d_{k,1^m}(z)} & \times \dfrac{(z-(-q)^{k-m-2})}{(z-(-q)^{k+m-2})} \nonumber \\
& \hspace{-65ex} = \dfrac{  (z-(-q)^{m+k}) (z-(-q)^{2n+m+k-2})  (z-(-q)^{k-m-2}) (z-(-q)^{2n-k+m-2})   }{d_{k,1^m}(z)}    \in \ko[z^{\pm 1}]. \label{eq: ak1mstep1p D}
\end{align}

On the other hand, the homomorphism
\[
\Vkm{1^m}  \otimes \Vkm{1}_{(-q)^{2n+m-3}}  \twoheadrightarrow \Vkm{1^{m-1}}_{(-q)^{-1}}
\]
implies that we have
\[
\dfrac{ d_{k,1^{m}}(z)d_{k,1}((-q)^{-2n-m+3}z)}{d_{k,1^{m-1}}((-q)z)} \times \dfrac{a_{k,1^{m-1}}((-q)z)}{a_{k,1^{m}}(z)a_{k,1}((-q)^{-2n-m+3}z)}\in \ko[z^{\pm 1}].
\]
By induction, we have
\[
\dfrac{ d_{k,1^{m}}(z)d_{k,1}((-q)^{-2n-m+3}z)}{d_{k,1^{m-1}}((-q)z)} =
\dfrac{ d_{k,1^{m}}(z)  \times  (z-(-q)^{2n+m+k-2}) (z-(-q)^{4n+m-k-4})}{ (z-(-q)^{k+m-2}) (z-(-q)^{2n-k+m-4}) }.
\]
Furthermore, we have
\begin{align*}
&\dfrac{a_{k,1^{m-1}}((-q)z)}{a_{k,1^{m}}(z)a_{k,1}((-q)^{-2n-m+3}z)} = \dfrac{(z-(-q)^{m+k-2})(z-(-q)^{2n+m-k-4})}{(z-(-q)^{2n+m-k-2})(z-(-q)^{m+k})}.
\end{align*}
Thus we have
\begin{align*}
& \dfrac{ d_{k,1^{m}}(z)  \times  (z-(-q)^{2n+m+k-2}) (z-(-q)^{4n+m-k-4})}{ (z-(-q)^{k+m-2}) (z-(-q)^{2n-k+m-4}) } \times  \dfrac{(z-(-q)^{m+k-2})(z-(-q)^{2n+m-k-4})}{(z-(-q)^{2n+m-k-2})(z-(-q)^{m+k})} \\
& = \dfrac{ d_{k,1^{m}}(z)  \times  (z-(-q)^{2n+m+k-2}) (z-(-q)^{4n+m-k-4})}{(z-(-q)^{2n+m-k-2})(z-(-q)^{m+k})} \in \ko[z^{\pm 1}],
\end{align*}
which implies our assertion with~\eqref{eq: ak1mstep1 D} and~\eqref{eq: ak1mstep1p D}.  
\end{proof}

\begin{lemma}
For any  $1 \le k \le n-2$ and $m \in \Z_{\ge 1}$, we have
\begin{align}\label{eq: d1km}
 d_{1,k^m}(z) =  (z-(-q)^{k+m}) (z-(-q)^{2n-k+m-2}).
\end{align}
\end{lemma}

\begin{proof}
Our assertion for $k=1$ holds already. Now we consider the remaining cases.

By the homomorphism
\[
\Vkm{k^{m-1}}_{(-q)} \otimes  \Vkm{k}_{(-q)^{1-m}} \twoheadrightarrow \Vkm{k^m},
\]
we have
\begin{equation}\label{eq: a1kmstep1}
\begin{aligned}
& \dfrac{ d_{1,k^{m-1}}((-q)z)d_{1,k}((-q)^{1-m}z)}{d_{1,k^m}(z)} \\
& \equiv  \dfrac{ (z-(-q)^{k+m-2}) (z-(-q)^{2n-k+m-4}) (z-(-q)^{k+m}) (z-(-q)^{2n+m-k-2})   }{d_{1,k^m}(z)} \in \ko[z^{\pm1}].
\end{aligned}
\end{equation}
By the homomorphism
\[
\Vkm{k}_{(-q)^{m-1}} \otimes  \Vkm{k^{m-1}}_{(-q)^{-1}} \twoheadrightarrow \Vkm{k^m},
\]
we have
\begin{equation}\label{eq: a1kmstep1p}
\begin{aligned}
& \dfrac{ d_{1,k}((-q)^{m-1}z)d_{1,k^{m-1}}((-q)^{-1}z)}{d_{1,k^m}(z)} \\
& \equiv  \dfrac{ (z-(-q)^{k-m+2}) (z-(-q)^{2n-k-m}) (z-(-q)^{k+m}) (z-(-q)^{2n-k+m-2})   }{d_{1,k^m}(z)} \in \ko[z^{\pm1}].
\end{aligned}
\end{equation}
By~\eqref{eq: a1kmstep1} and~\eqref{eq: a1kmstep1p}, an ambiguity possible happens when $m=2$.
Let us resolve the ambiguity: Note that $d_{1,k}(z)=(z-(-q)^{k+1})(z-(-q)^{2n-k-1}).$ Then we have
$$
\de(\Vkm{1},\Vkm{k^2}_{(-q)^{k+1}}) =\de(\Vkm{1},\Vkm{k}_{(-q)^{k-1}}\hconv \Vkm{k}_{(-q)^{k+1}}) =0
$$
by Theorem~\ref{thm: i-box commute}. 
By taking $L = \Vkm{1}$, $L'=\Vkm{k}_{(-q)^{2n-k-1}}$ and $X=\Vkm{k}_{(-q)^{2n-k-3}}$, one can easily check that
the root modules satisfies the assumption in Lemma~\ref{Lem:LL'X}~\eqref{it: LL'X i} since $k>1$. Since $\de(\scrD L,X) =1$ and 
$\de(\scrD L,L' \hconv X) =0$, the root modules satisfies the assumption in {\rm (a)} of Lemma~\ref{Lem:LL'X}~\eqref{it: LL'X i}. Hence 
$\de(L,L' \hconv X) =\de(L, X)=0$. Hence $(-q)^k$ and $(-q)^{2n-k-2}$ can not be roots of $d_{1,k^2}(z)$.
Then the result follows similar to the previous cases.
\end{proof}

As Proposition~\ref{prop: B dlkm} and Lemma~\ref{lemma: division B}, we can obtain the following proposition and lemma.

\begin{proposition}  
For any  $1 \le l,k \le n-2$ and $m \in \Z_{\ge 1}$, we have
\begin{align}\label{eq: dlkm D}
 d_{l,k^m}(z) = \prod_{s=1}^{\min(k,l)} (z-(-q)^{|k-l|+m-1+2s})(z-(-q)^{2n-k-l+m-3+2s}).
\end{align}
\end{proposition}

\begin{lemma} \label{lem: division D}
For $1\leq k,l \le n-2$ and $\max(m,p) \ge  1$, we have
\[
\text{$d_{l^p,k^m}(z)$ divides $\prod_{t=0}^{\min(m,p)-1}\prod_{s=1}^{\min(k,l)} (z-(-q)^{|k-l|+|m-p|+2(s+t)}) (z-(-q)^{2n-k-l+|m-p|-2+2(s+t)})$.}
\]
\end{lemma}

\begin{proof}[Proof of {\rm Theorem~\ref{thm:denominators_untwisted}} for $\max(l,k) < n-1$ in type $D_n^{(1)}$]
Without loss of generality, assume $p \leq m$. We consider the case when $l \le k$.
Recall that, for $1 \le k,l \le n-2$ and $n',n'' \in \{n-1,n\}$ such that $n'-n'' \equiv_2 n-k$,  we have
\[
 \Vkm{{n'}^m}_{(-q)^{-n+k+1}} \otimes \Vkm{{n''}^m}_{(-q)^{n-k-1}} \twoheadrightarrow \Vkm{k^m},
\]
by Theorem~\ref{thm: Higher Dorey I}~\eqref{eq: spin D homo} \emph{without} the restriction on $k$. 
By duality, we have
\[
\Vkm{k^m} \otimes \Vkm{{n'}^m}_{(-q)^{n+k-1}} \twoheadrightarrow  \Vkm{{n''}^m}_{(-q)^{n-k-1}}.
\]
Assume that $n-k \equiv_2 0$ and take $n'=n''=n$. Then we have
\[
\dfrac{d_{k^m,l^p}(z) d_{n^m,l^p}((-q)^{n+k-1} z)}{d_{n^m,l^{p}}((-q)^{n-k-1} z)} \times
\dfrac{a_{n^m,l^{p}}((-q)^{n-k-1} z)}{a_{k^m,l^p}(z) a_{n^m,l^p}((-q)^{n+k-1} z)} \in \ko[z^{\pm1}].
\]
By direct calculation, we have
\begin{align*}
 \dfrac{a_{n^m,l^{p}}((-q)^{n-k-1} z)}{a_{k^m,l^p}(z) a_{n^m,l^p}((-q)^{n+k-1} z)} & =  \prod_{t=0}^{p-1}\prod_{s=1}^{l}  \dfrac{(z-(-q)^{l-k-m+p-2(s+t)})}{(z-(-q)^{-l-k+m-p+2(s+t)})}, \allowdisplaybreaks \\
 \dfrac{d_{k^m,l^p}(z) d_{n^m,l^p}((-q)^{n+k-1} z)}{d_{n^m,l^{p}}((-q)^{n-k-1} z)} & =
\dfrac{d_{k^m,l^p}(z) \displaystyle \prod_{t=0}^{p-1}\prod_{s=1}^{l} (z-(-q)^{-l-k+m-p+2(s+t)}) }
{\displaystyle\prod_{t=0}^{p-1}\prod_{s=1}^{l} (z-(-q)^{k-l+m-p+2(s+t)})}.
\end{align*}
Thus we have
\begin{align*}
& \dfrac{d_{k^m,l^p}(z) \displaystyle \prod_{t=0}^{p-1}\prod_{s=1}^{l} (z-(-q)^{-l-k+m-p+2(s+t)}) }
{\displaystyle\prod_{t=0}^{p-1}\prod_{s=1}^{l} (z-(-q)^{k-l+m-p+2(s+t)})}
\times \prod_{t=0}^{p-1}\prod_{s=1}^{l}  \dfrac{(z-(-q)^{l-k-m+p-2(s+t)})}{(z-(-q)^{-l-k+m-p+2(s+t)})} \allowdisplaybreaks\\
& = \dfrac{d_{k^m,l^p}(z) \displaystyle \prod_{t=0}^{p-1}\prod_{s=1}^{l} (z-(-q)^{l-k-m+p-2(s+t)}) }
{\displaystyle\prod_{t=0}^{p-1}\prod_{s=1}^{l} (z-(-q)^{k-l+m-p+2(s+t)})} \in \ko[z^{\pm1}] \allowdisplaybreaks\\
& \Rightarrow \dfrac{d_{k^m,l^p}(z) }
{\displaystyle\prod_{t=0}^{p-1}\prod_{s=1}^{l} (z-(-q)^{k-l+m-p+2(s+t)})} \in \ko[z^{\pm1}].
\end{align*}
Then the assertion for this case follows from Lemma~\ref{lem: division D}. The remaining cases can be proved in a similar way.
\end{proof}

\begin{corollary} \label{cor: root module D}
For $m \le n-2$, the KR module $\Vkm{1^m}_x$ $(x \in \bfk^{\times})$ is a root module.
\end{corollary}

\subsection{Remark on twisted affine types}

Using the exact monoidal functor $\calF_A^0$, which preserves (i) $\de$-invariants between fundamental modules and (ii) KR modules, 
we can derive the denominator formulas for $\g=A_n^{(2)}$ from those of $\g=A_n^{(1)}$. 
For $\g=D_{n+1}^{(2)}$ and $\g=D_4^{(3)}$, the exact monoidal functor $\calF^{1,t}_\calQ$ $(t=1,2,3)$, which preserves $\de$-invariants and KR modules in $\scrC_\calQ^{(t)}$, 
ensures that the denominator formulas between KR modules in $\scrC_\calQ^{(t)}$ can be obtained. 
Using these formulas as a base, together with the framework developed in the previous subsections and the restricted higher Dorey’s rule, 
we establish the denominator formulas for $\g=D_{n+1}^{(2)}$ and $\g=D_4^{(3)}$.

\section{Applications}
\label{sec:applications}

In this section, we investigate the application of the denominator formulas.
As we did, we shall focus on the untwisted affine types, since the results in this section can be extended
to corresponding twisted affine types in a canonical way.

\subsection{Higher Dorey's rule without restriction}

Now, by using the denominator formulas for KR modules, we can remove the restriction $\min(k,l)=1$ in Theorem~\ref{thm: Higher Dorey I}, thereby completing the generalization of Dorey's rule in comparison to Theorem~\ref{thm: Higher Dorey I}. 

Note that we have a difficulty for computing $d^{C_n^{(1)}}_{k^m,l^p}(z)$ when $m,p>1$ are odd and $k = l<n$. 
In the following lemma, we refine
$d^{C_n^{(1)}}_{k^m,k^m}(z)$ for an odd $m$ and $2k\le n$.
\begin{lemma} \label{lem: C-refine mm}
For $\g=C_n^{(1)}$ and $k+l \le n$,  
we have
\bnum
\item \label{it: defining 1} $\de(\Vkm{l^{m}}_{(-\qs)^{-k}}, \Vkm{k^{m}}_{(-\qs)^{l}}) = \de(\Vkm{l^m}_{(-\qs)^{-k}}, \Vkm{k^{m+1}}_{(-\qs)^{l+1}})$,
\item $\de(\Vkm{l^{m}}_{(-\qs)^{-k-1}}, \Vkm{k^{m}}_{(-\qs)^{l+1}}) = \de(\Vkm{l^m}_{(-\qs)^{-k-1}}, \Vkm{k^{m+1}}_{(-\qs)^{l+2}})$
\ee
for any $m \in \Z_{\ge2}$.
\end{lemma}

\begin{proof}
Since their proofs are similar, we shall prove only~\eqref{it: defining 1}. Without loss of generality, we assume that $k \ge l$.
Note that 
$$
\Vkm{k^{m+1}}_{(-\qs)^{l+1}} \iso \Vkm{k}_{(-\qs)^{m+l+1}} \hconv
\Vkm{k^{m}}_{(-\qs)^{l}}
$$
and
$$
\de( \Vkm{l^m}_{(-\qs)^{-k}},\Vkm{k}_{(-\qs)^{m+l+1}}) 
= \de( \Vkm{l^m}_{(-\qs)^{-2n-2-k}},\Vkm{k}_{(-\qs)^{m+l+1}}) =0
$$
by
$$
d_{k,l^m}(z) =   
\prod_{s=1}^{l}  (z-(-\qs)^{k-l+m-1+2s})(z-(-\qs)^{2n-k-l+m+1+2s} ).
$$
Namely,
\bna
\item $k+l+m+1 \ne k-l+m-1+2s \iff l+1 > s$,
\item $k+l+m+1 \ne 2n-k-l+m+1+2s \iff  (n-k-l)+s >0$,
\item $2n+k+l+m+3 \ne k-l+m-1+2s \iff  n+(l-s)+2>0$,
\item $2n+k+l+m+3 \ne 2n-k-l+m+1+2s \iff k+(l-s)+1 > 0$, 
\ee
for $1 \le s \le l$ and $k+l \le n$. Thus we have the assertion by Lemma~\ref{lem: de=de}~\eqref{it: de=de 2}.     
\end{proof}

For an odd $m$, we have ambiguities for $d^{C_n^{(1)}}_{k^m,k^m}(z)$, while we do \emph{not} have such ambiguities for  
$d^{C_n^{(1)}}_{k^m,k^{m+1}}(z)$. Thus the above lemma refines 
$d^{C_n^{(1)}}_{k^m,k^m}(z)$ for an odd $m$ and $2k \le n$ as we desired.

\begin{theorem}[Higher Dorey's rule II] \label{thm: Higher Dorey II}
For $\uqpg$ of non-exceptional untwisted affine type, let $m \in \Z_{\ge 1}$ and 

\begin{enumerate}[{\rm (1)}]
\begin{subequations} \label{eq: higher homo II}
\item \label{thmitem:higher_Dorey_general II}
For $1 \le k, l < n$, let us assume $k + l < n - \delta(\g = D_n^{(1)})$.  
Then, we have  
\begin{align} \label{eq: k+l<n homo II}
\Vkm{l^{ m}}_{(-\check{q}_l)^{-k}} \hconv \Vkm{k^{ m}}_{(-\check{q}_k)^{l}} \iso \Vkm{(k+l)^{ m}}.     
\end{align}
In particular, if $\g \ne A_{n-1}^{(1)}$ and $k+l =n- \delta(\g = D_n^{(1)})$, we have 
\begin{align} \label{eq: k+l=n homo II}
\bc
\Vkm{k^m}_{(-1)^{m+l}q^{-l}} \hconv \Vkm{l^m}_{(-1)^{k+m}q^k} \iso \Vkm{n^{2m}} & \text{ if } \g = B_n^{(1)}, \\
\Vkm{l^{m}}_{(-\qs)^{-k}} \hconv \Vkm{k^{m}}_{(-\qs)^{l}} \iso  \Vkm{n^{\ceil{m/2}}}_{(-\qs)^{-\epsilon} }\otimes \Vkm{n^{\floor{m/2}}}_{(-\qs)^{\epsilon}}
& \text{ if } \g = C_n^{(1)}, \\
\Vkm{l^m}_{(-q)^{-k}} \hconv \Vkm{k^m}_{(-q)^{l}}  \iso \Vkm{(n-1)^{m}} \otimes \Vkm{n^m} & \text{ if } \g = D_n^{(1)},
\ec
\end{align} 
where $\epsilon = \delta(m \equiv_2 0)$. 
\item For $\g$ of type $B_n^{(1)}$ and  $1 \le l < n < k <2n-1$ with $k+l \le 2n-1$
\begin{equation}  \label{eq: B folded homo II} 
\begin{aligned}
& \Wkm{l^m}_{(-q)^{-k+1}} \hconv \Wkm{k^m}_{-(-q)^{l}}  \iso     \Wkm{(k+l)^m} \\
& \hspace{7ex} \iff \Vkm{\overline{l}^m}_{(-q)^{k-1}} \hconv \Vkm{\overline{k}^m}_{-(-q)^{l}}  \iso     \Vkm{(\overline{k+l})^m}_{(-1)}.
\end{aligned} 
\end{equation}
\item \label{thmitem:higher_Dorey_B II}
For $\g$ of type $B_n^{(1)}$, $1 \le l < n$ and $m$ is even, we have 
\begin{align} \label{eq: spin B homo II} 
\Vkm{n^m}_{(-1)^{n+l}\qs^{-2(n-l)+1}} \hconv \Vkm{n^m}_{(-1)^{n+l}\qs^{2(n-l)-1}}\iso 
\Vkm{ l^{m/2} }_{(-1)^{m/2} \qs} \otimes \Vkm{l^{m/2}}_{(-1)^{m/2} \qs^{-1}}. 
\end{align}
\end{subequations}
\end{enumerate}
\end{theorem}

\begin{proof} [Proof of~\eqref{eq: k+l<n homo II}--\eqref{eq: B folded homo II}] 
Since the proofs are similar, we only give a proof for $\g=C^{(1)}_{n}$ of~\eqref{eq: k+l<n homo II} and $l \le k$.
Let us apply induction on $2 \le l \le k$ with $k+l \le n-1$, since the case when $l=1$ is proved in Theorem~\ref{thm: Higher Dorey I}.
By Lemma~\ref{lem: C-refine mm}, we have
\begin{align*}
\de(\Vkm{k^m}_{(-\qs)^{l}}, \Vkm{l^m}_{(-\qs)^{-k}}) & = \de(\Vkm{k^{m+1}}_{(-\qs)^{l+1}}, \Vkm{l^m}_{(-\qs)^{-k}}), \\
\de(\Vkm{k^m}_{(-\qs)^{l}},\Vkm{(l-1)^m}_{(-\qs)^{-1-k}}) & = \de(\Vkm{k^{m+1}}_{(-\qs)^{l+1}},\Vkm{(l-1)^m}_{(-\qs)^{-1-k}}),
\end{align*}
while 
\begin{align*}
d_{k^{m+1},l^m}(z) &= \displaystyle \prod_{t=0}^{m-1} \prod_{s=1}^{l}  (z-(-\qs)^{k-l+1+2(s+t)})(z-(-\qs)^{2n+3-k-l+2(s+t)} ),   \\
d_{k^{m+1},(l-1)^m}(z) &=  \displaystyle \prod_{t=0}^{m-1} \prod_{s=1}^{l-1}  (z-(-\qs)^{k-l+2+2(s+t)})(z-(-\qs)^{2n+4-k-l+2(s+t)} ).
\end{align*}
Thus we have
\begin{align*}
\de(\Vkm{k^m}_{(-\qs)^{l}}, \Vkm{l^m}_{(-\qs)^{-k}}) &= \min(l,m),    \\
\de(\Vkm{k^m}_{(-\qs)^{l}},\Vkm{(l-1)^m}_{(-\qs)^{-1-k}}) &=\min(l,m)-1.
\end{align*}
On the other hand,
$$
\de(\Vkm{k^m}_{(-\qs)^{l}},\Vkm{1^m}_{(-\qs)^{l-1-k}})=1
$$
by Corollary~\ref{cor: de=1 Dorey w res}.
Hence we know 
\begin{align*}
& \de(\Vkm{k^m}_{(-\qs)^{l}}, \Vkm{l^m}_{(-\qs)^{-k}}) \\
& \hspace{15ex} = \de(\Vkm{k^m}_{(-\qs)^{l}},\Vkm{1^m}_{(-\qs)^{l-1-k}})+ \de(\Vkm{k^m}_{(-\qs)^{l}},\Vkm{(l-1)^m}_{(-\qs)^{-1-k}}),    
\end{align*}
where $\Vkm{1^m}_{(-\qs)^{l}-1-k} \hconv \Vkm{(l-1)^m}_{(-\qs)^{-1-k}}=\Vkm{l^m}_{(-\qs)^{-k}}$.  
Thus the sequence 
$$
(\Vkm{(l-1)^m}_{(-\qs)^{-1-k}},\Vkm{1^m}_{(-\qs)^{l-1-k}},\Vkm{k^m}_{(-\qs)^{l}}) \text{ is normal},
$$
by Proposition~\ref{lem: normal seq d}~(b-iv). Hence we have the following commutative diagram 
\begin{equation} \label{eq: induction l>1}
\begin{aligned} 
\xymatrix@R=6ex@C=9ex{
\Vkm{(l-1)^m}_{(-\qs)^{-1-k}} \otimes \Vkm{1^m}_{(-\qs)^{l-1-k}} \otimes \Vkm{k^m}_{(-\qs)^{l}} 
\ar@{->>}[d]_{\dagger}\ar@{->>}[r]^{\qquad\qquad\qquad\dagger} 
& \Vkm{l^m}_{(-\qs)^{-k}} \otimes \Vkm{k^m}_{(-\qs)^{l}} \ar@{->>}[d]\\
\Vkm{(l-1)^m}_{(-\qs)^{-1-k}} \otimes \Vkm{(k+1)^m}_{(-\qs)^{l-1}} \ar@{->>}[r]^{\qquad\qquad\qquad*} 
& \Vkm{(k+l)^m},
}
\end{aligned}
\end{equation} 
where $\overset{*}{\twoheadrightarrow}$ follows from the induction
and $\overset{\dagger}{\twoheadrightarrow}$ follows from Theorem~\ref{thm: Higher Dorey I}. 
Hence the assertion follows.  
\end{proof}

\begin{remark} \label{rmk: non roots to root}
Let us consider $\g = A_{n-1}^{(1)}$, $1< m < n$, $k+l = n-1$, and $\min(k,l) > 1$.
In this case, although neither $\Vkm{k^m}$ nor $\Vkm{l^m}$ are root modules, we have
\[
M \seteq \Vkm{l^{ m}}_{(-q)^{-k}} \hconv \Vkm{k^{ m}}_{(-q)^{l}} \iso \Vkm{(n-1)^{ m}},
\]
which is a root module as shown in Corollary~\ref{cor: root module A}. 
To the best of the authors’ knowledge, this provides the first example in which two non-root modules produce a root module as their head.
Notably, it occurs even when $\de(\Vkm{l^{ m}}_{(-q)^{-k}} , \Vkm{k^{ m}}_{(-q)^{l}} ) > 1$ (see~\eqref{eq: de conjecture} below), which has not been seen before either (cf.\ Lemma~\ref{Lem: two root modules}).
Also, it would be natural to ask whether $N \seteq \Vkm{l^{ m}}_{(-q)^{-k}} \sconv \Vkm{k^{ m}}_{(-q)^{l}}$ is also a root module or not, since $\La^\infty(M,M) = \La^\infty(N,N)= - 2$.
\end{remark}

\begin{corollary} Let us take $k,l \in I$ as in the above theorem and set 
$$M \seteq \Vkm{l^m}_{(-\chq_l)^{-k}} \qtq N \seteq \Vkm{k^m}_{(-\chq_l)^{l}}.$$
\bna
\item $M \hconv N$ is real
\item $M \hconv N$ strongly commutes with $M$ and $N$.
\item For any $a,b \in \Z$, we have
$$
M^{\otimes a} \hconv N^{\otimes b} \iso 
\bc 
(M \hconv N)^{\otimes a} \otimes N^{\otimes (b-a)} \text{ if } a \le b, \\
 (M \hconv N)^{\otimes a} \otimes M^{\otimes (a-b)} \text{ if } b \le a.
\ec
$$
\ee
\end{corollary}

\begin{proof}
The first assertion is obvious. The second assertion follows from $i$-box argument. For instance, when $\g=A_{n}^{(1)}$, we have
$$
\exrch_{k+l}(\Vkm{l^m}_{(-q)^{-k}}) = \range{-m-2k,m} \qtq  \exrch_{k+l}(\Vkm{k^m}_{(-q)^{l}}) = \range{-m,m+2l}
$$
while 
$\rch(\Vkm{(k+l)^m})=\range{1-m,m-1}$. 

For the last assertion, we shall prove only the first isomorphism. We argue by induction on $a \le b$. If $a=0$, the assertion is obvious.
Assume that $a>0$. Then
\begin{align*}
M^{\tens a} \otimes N^{\tens b}   &\twoheadrightarrow  M^{\tens a-1} \otimes (M\hconv N) \otimes N^{\tens b-1} \\
& \iso (M\hconv N) \otimes M^{\tens a-1} \otimes N^{\tens b-1} \\
& \twoheadrightarrow (M\hconv N) \otimes ( (M\hconv N)^{\tens a-1} \otimes N^{\otimes b-a} ) \\
& \iso (M\hconv N)^{\tens a} \otimes N^{\otimes b-a}.
\end{align*}
Now the assertion follows from the fact that $(M\hconv N)^{\tens a} \otimes N^{\otimes b-a}$ is a simple quotient of
$M^{\tens a} \otimes N^{\tens b}$ which has a simple head. 
\end{proof}

\begin{proof} [Proof of~\eqref{eq: spin B homo II}]  
Let us write $m=2m'$. In this case, we have
\begin{align*} 
& \Vkm{n^{2m'}}_{(-1)^{n+l}\qs^{-2(n-l)+1}} \tens \Vkm{n^{2m'}}_{(-1)^{n+l}\qs^{2(n-l)-1}}    \\
& \hspace{15ex} \rightarrowtail \Vkm{l^{m'}}_{(-1)^{m'}\qs}   
       \otimes \Vkm{(n-l)^{m'}}_{(-1)^{n+m'}\qs^{-2n+1}}  \tens \Vkm{n^{2m'}}_{(-1)^{n+l}\qs^{2(n-l)-1}}  \\
& \hspace{30ex} \twoheadrightarrow \Vkm{l^{m'}}_{(-1)^{m'} \qs} \otimes \Vkm{l^{m'}}_{(-1)^{m'} \qs^{-1}}, 
\end{align*}
by~\eqref{eq: k+l=n homo II}, which implies the assertion. 
\end{proof}

Here we propose the following conjecture, where the case $l = n - 1$ is the T-system.

\begin{conjecture} \label{conj: BC-conjecture}  
For $\g$ of type $B_n^{(1)}$, $1 \le l < n-1$ and $m \in 2\Z_{\ge1}+1$, we have    
\begin{align} \label{eq: spin B homo conj} 
\Vkm{n^m}_{(-1)^{n+l}\qs^{-2(n-l)+1}} \hconv \Vkm{n^m}_{(-1)^{n+l}\qs^{2(n-l)-1}}\iso 
\Vkm{ l^{\lceil m/2 \rceil} }_{(-1)^{\lceil m/2 \rceil} } \otimes \Vkm{l^{\lfloor m/2 \rfloor}}_{(-1)^{\lfloor m/2 \rfloor}}. 
\end{align}
\end{conjecture}

Now we can remove the restriction $\min(a,b)=1$ in Theorem~\ref{thm: higher mesh} by using Theorem~\ref{thm: Higher Dorey II}: 

\begin{theorem}[Higher Dorey's rule of mesh type II] \label{thm: higher mesh II}
For the classical untwisted affine types $\g = A_{n-1}^{(1)}$, $B_n^{(1)}$, $C_n^{(1)}$, $D_n^{(1)}$, $m \in \Z_{\ge1}$ and
$l < n'  \seteq n- \delta(\g = D_n^{(1)} )$, assume $1\le a,b < n'$ with $a+b < l$. Then we have 
\begin{subequations} \label{eq: higher homo mesh II}
\begin{align}  \label{eq: mesh type Dorey's rule higher II}
\Vkm{(l-b)^{ m}}_{(-\chq_{l-b})^{-b}} \hconv \Vkm{(l-a)^{ m}}_{(-\chq_{l-a})^{a}} \iso \Vkm{l^{ m}} \otimes \Vkm{(l-a-b)^{  m}}_{(-\chq_{l-a-b})^{a-b}}.
\end{align}    
\bna
\item In particular, if $l=n'$, we have
\begin{equation} \label{eq: mesh BCD II}
\begin{aligned}
&\Vkm{(n'-b)^{ m}}_{(-\chq_{n'-b})^{-b}} \hconv \Vkm{(n'-a)^{ m}}_{(-\chq_{n'-a})^{a}}  \\  & \hspace{5ex}
\iso
\bc
\Vkm{n^{2m}}_{(-1)^m} \otimes \Vkm{(n'-a-b)^m}_{(-\chq_{n'-a-b})^{a-b}} & \text{ if } \g =B_n^{(1)}, \\
\Vkm{n^{\ceil{m/2}}}_{(-\qs)^{-\epsilon} }\otimes \Vkm{n^{\floor{m/2}}}_{(-\qs)^{\epsilon}} \otimes \Vkm{(n'-a-b)^{m}}_{(-\qs)^{a-b}} & \text{ if } \g =C_n^{(1)}, \\
\Vkm{n^{m}} \otimes \Vkm{(n-1)^{m}}  \otimes  \Vkm{(n'-a-b)^m}_{(-q)^{a-b}} & \text{ if } \g =D_n^{(1)},
\ec
\end{aligned}
\end{equation}    
for $a,b < n'$. Here we understand $\Vkm{(l-a-b)^s}$ $(s\in \Z_{\ge1})$ as the trivial module $\mathbf{1}$ when $l-a-b = 0$, and  $\epsilon = \delta(m \equiv_2 0)$.  
\item  For $\g=B_n^{(1)}$ and $l > n$, take $a,b \in [1,2n-1]$ with $l-a >n$, $l-b <n$ and $a+b < l$ if there are. Then we have
\begin{align} \label{eq: folded homo higher II}
\Wkm{(l-b)^m}_{(-q)^{-b+1}} \hconv \Wkm{(l-a)^m}_{-(-q)^{a}} \iso \Wkm{l^m}_{(-1)} \otimes 
\Wkm{(l-a-b)^m}_{(-q)^{a-b+1}}. 
\end{align}
\ee
\end{subequations}
\end{theorem}

\begin{proof} Since the proof is similar, we give only the proof of~\eqref{eq: mesh type Dorey's rule higher II} for $A^{(1)}_{n}$. 
Note that we have
$$
\Vkm{(l-b)^m}_{(-q)^{-b}} \rightarrowtail  \Vkm{(l-a-b)^m}_{(-q)^{-b+a}} \otimes \Vkm{a^m}_{(-\chq_{a})^{a-l}},
$$
and
$$
\Vkm{a^m}_{(-q)^{a-l}} \otimes \Vkm{(l-a)^m}_{(-q)^{a}} \twoheadrightarrow \Vkm{l^m}
$$
by~\eqref{eq: k+l<n homo II}. 
Hence we have the sequence of homomorphism
\begin{align*}
& \Vkm{(l-b)^m}_{(-q)^{-b}}  \otimes \Vkm{(l-a)^m}_{(-q)^{a}} \\    
&\hspace{5ex}\rightarrowtail \Vkm{(l-a-b)^m}_{(-q)^{-b+a}} \otimes \Vkm{a^m}_{(-q)^{a-l}} \otimes \Vkm{(l-a)^m}_{(-q)^{a}}\\
& \hspace{10ex} \twoheadrightarrow \Vkm{(l-a-b)^m}_{(-q)^{-b+a}} \otimes \Vkm{l^m}
\end{align*}
whose composition does not vanish by Proposition~\ref{prop: non-van}. Since  
$\Vkm{(l-a-b)^m}_{(-q)^{-b+a}} \otimes \Vkm{l^m}$
is simple by the $i$-box argument, the assertion follows. 
\end{proof}
 
The following result generalizes~\eqref{eq: d-value min(a,b)=1} and follows from the denominator formulas and Lemma~\ref{lem: C-refine mm}.

\begin{proposition}
\label{prop:de_invar_KR}
For a classical untwisted affine type $\g=A_{n-1}^{(1)}$, $B_n^{(1)}$, $C_n^{(1)}$, $D_n^{(1)}$, take $m \in \Z_{\ge1}$,
$l < n' \seteq n - \delta(\g = D_n^{(1)})$, and $a + b < l$ such that $a,b \geq 1$. 
Then we have 
\begin{align}  \label{eq: de conjecture}
\de(\Vkm{(l-b)^{ m}}_{(-\chq_{l-b})^{-b}} , \Vkm{(l-a)^{ m}}_{(-\chq_{l-a})^{a}}) =\min(a,b,m).
\end{align}  
\end{proposition}

\begin{corollary} \label{cor: prime socle}
Assume that $a \neq b$. Then 
\[ \text{the socle }
\Vkm{(l-b)^m}_{(-\chq_{l-b})^{-b}} \sconv 
\Vkm{(l-a)^m}_{(-\chq_{l-a})^{a}} \text{ is prime.}
\]
\end{corollary}

\begin{proof}
By Remark~\ref{rmk: gen cogen}, the Drinfel'd polynomial of the socle is explicitly known and is of the form described in~\eqref{eq: Drin poly}.  
Then the assertion follows from Corollary~\ref{cor:two_factor_prime}.
\end{proof}

\begin{problem}
Determine when the prime socle in Corollary~{\rm\ref{cor: prime socle}} is real.
\end{problem}

\subsection{Schur positivity via KR modules}

In this subsection, we first review the result in~\cite{FH14} briefly.
Then we recover and extend it by using the our results in this paper. 

For a positive integer $k$, we denote by $\calP(k)$ the set of partitions of $k$:
$$
\calP(k) = \{  \seq{k_1 \ge \cdots \ge k_r \ge 0 } \ | \  k_s \in \Z_{\ge 0} \text{ and } k_1 + \cdots + k_r = k \}. 
$$
The reverse dominance relation $\preceq^r$ on $\calP(k)$ is defined as follows:
Let $\bla =\seq{\la_1 \ge \cdots \ge  \la_{r_1} >0 }$, $\bmu =\seq{\mu_1 \ge \cdots \ge  \mu_{r_2} >0 } \in \calP(k)$. Then
$$
\bla \preceq^\bfr \bmu \iff \sum_{s=1}^j \la_s \ge \sum_{s=1}^j \mu_s \quad \text{for all }1 \le j \le \min(r_1,r_2).  
$$
We say that $\bmu$ \defn{covers} $\bla$ if 
(i) $\bla \preceq^\bfr \bmu$ and 
(ii) $\bla \preceq^\bfr \bnu \preceq^\bfr \bmu$ implies $\bnu=\bmu$ or $\bnu=\bla$. Since $\calP(k)$ is a finite set, for each pair $\bla \preceq^\bfr \bmu$, we can find $\bnu_0,\ldots,\bnu_l$ such that
$$
\bla=\bnu_0 \preceq^\bfr \bnu_1 \ldots \preceq^\bfr \bnu_l= \bmu \text{ and $\bnu_i$ covers $\bnu_{i-1}$ for all $i$.} 
$$

\begin{proposition}[{\cite[Proposition 3.5]{CFS14}}]
Let $\bla=\seq{\la_1 \ge \cdots  \ge \la_{s} \ge 0 }$, $\bmu=\seq{\mu_1 \ge \cdots \ge  \mu_{s} \ge 0 } \in \calP(k)$ and 
$\max(\la_s,\mu_s)>0$. Suppose $\bmu$ covers $\bla$. Then there exists $i<j$ such that
$$
\mu_l = \la_l -\delta(l=i)+\delta(l=j) \text{ for all $1 \le l \le s$}.
$$
The cover relation on $\calP(k)$ is completely determined by the cover relation on partitions of length $2$. 
\end{proposition}

\begin{proposition} \label{prop: KR-surjective}
Let $i \in I_0$ and $m_1 > m_2 \ge 0$.
Then there exists a surjective of $\uqpg$-homomorphism 
\begin{align} \label{eq: schur positive KR}
\Vkm{i^{(m_2+1)}}_{(-\chq_i)^{m_1-1-m_2}} \tens \Vkm{i^{(m_1-1)}}_{(-\chq_i)^{-1}}
\twoheadrightarrow
\Vkm{i^{m_2}}_{(-\chq_i)^{m_1-m_2-2}} \tens \Vkm{i^{m_1}}.
\end{align}
\end{proposition}

\begin{proof}
The assertion for $m_2 =0$ is obvious, and so we assume that $m_2 > 0$.

Let us set  
$$L_1 \seteq \Vkm{i^{m_1}} \iso \Vkm{i}_{(-\chq_i)^{m_1-1}} \hconv L_1' \qtq L_2 \seteq \Vkm{i^{m_2}}_{(-\chq_i)^{m_1-m_2-2}} \iso  L_2' \hconv \scrD \Vkm{i}_{(-\chq_i)^{m_1-1}},$$ where
$$L_1' = \Vkm{i^{(m_1-1)}}_{(-\chq_i)^{-1}} \qtq L_2' = \Vkm{i^{(m_2+1)}}_{(-\chq_i)^{m_1-1-m_2}}.$$ Then we have 
$$
\exrch_i(L_1) = \range{d_i(-1-m_1),d_i(m_1+1)} \qtq  \rch_i(L_2) = \range{d_i(m_1-2m_2-1),d_i(m_1-3)},
$$
which implies $L_2 \otimes L_1$ is simple as a $\uqpg$-module. 
On the other hand, we have the composition of $\uqpg$-module homomorphisms
\begin{align*}
L_2' \otimes L_1' \rightarrowtail L_2 \otimes \Vkm{i}_{(-\chq_i)^{m_1-1}}  \otimes L_1'
\twoheadrightarrow L_2 \otimes L_1
\end{align*}
which does not vanish by Proposition~\ref{prop: non-van} is surjective since $L_2 \otimes L_1$ is simple.
Hence the assertion for $m_2 >0$ follows.
\end{proof}

Let $\KR(m\varpi_i)$ be the restriction of $\Vkm{i^m}_a$ as a $U_q(\g_0)$-module, which does \emph{not} depend on the spectral parameter $a \in \bfk^\times$.  By following the notation in~\cite{FH14}, for $i \in I_0$ and $\bla=\seq{\la_1 \ge  \cdots \ge \la_r >0} $, we denote by 
$$
\KR(\bla,i) \seteq \KR(\la_1\varpi_i) \otimes \cdots  \otimes \KR(\la_r\varpi_i).
$$

The following was first proved in~\cite[Theorem~2.1]{FH14} (the main result of the paper), but we show that it follows as a direct consequence of Proposition~\ref{prop: KR-surjective}.

\begin{corollary}
Let $i \in I_0$ and $m_1 > m_2 \ge 0$. Then there exists an embedding of $U_q(\g_0)$-modules 
\begin{align} \label{eq: schur positive}
\KR(m_2\varpi_i) \otimes \KR(m_1\varpi_i) \rightarrowtail \KR((m_2+1)\varpi_i) \otimes \KR((m_1-1)\varpi_i).    
\end{align}
\end{corollary} 

As a consequence of the above corollary, we can obtain the following as explained in~\cite{FH14}.

\begin{theorem}[{\cite[Theorem 2.1]{FH14}}]
\label{thm:dimhom_FH}
Let $k \in \Z_{\ge1}$, $i \in I_0$ and $\bla \preceq^\bfr \bmu \in \calP(k)$. Then we have
\[
\dim\bigl( \Hom_{U_q(\g_0)}(\KR(\bla,i), V(\tau)) \bigr) \le \dim\bigl( \Hom_{U_q(\g_0)}(\KR(\bmu,i),V(\tau))\bigr)
\]
for all $\tau \in P_0^+$.
Here $V(\tau)$ denotes the highest weight $U_q(\g_0)$-module corresponding to $\tau$. 
\end{theorem}

\begin{remark}
When $i \in I_0$ is special, Equation~\eqref{eq: schur positive} can be expressed as 
\begin{align} \label{eq: Schur equation}
\calS_{i^{(m_1-1)}} \calS_{i^{(m_2+1)}} - \calS_{i^{m_1}} \calS_{i^{m_2}} = \sum_{\tau \in P_0^+} c_\tau \calS_\tau  \quad
(c_\tau\in \Z_{\ge0}),
\end{align}
where $\calS_\tau$ denotes the character of $V(\tau)$.
Thus~\eqref{eq: schur positive} implies that the difference of the products of characters is positive in the character basis.
\end{remark}

Now let us consider a partition 
$$
\bla = \seq{\la_1 \ge \cdots \ge \la_r > 0} \text{ of $k \in \Z_{\ge 1}$,}
$$
and fix $m \in \Z_{\ge1}$. Then we assume $|I_0| \ge \la_1$. Then we define
$$
\KR(m,\bla) \seteq \KR(m\varpi_{\la_1}) \otimes \cdots  \otimes \KR(m\varpi_{\la_r}). 
$$
Here we understand $\KR(m\varpi_{0})$ as the trivial module. 
Then as a direct consequence of Theorem~\ref{thm: Higher Dorey II} and Theorem~\ref{thm: higher mesh II} we obtain the followings.

\begin{proposition}
Let $m \in \Z_{\ge 1}$ and $n-\delta(\g=D^{(1)}_{n}) > \la_1 > \la_2 \ge 0$.
Then there exists an embedding of $U_q(\g_0)$-modules 
\begin{align} \label{eq: schur positive II}
\KR(m\varpi_{\la_2}) \otimes \KR(m\varpi_{\la_1}) \rightarrowtail \KR(m\varpi_{\la_2+1}) \otimes \KR(m\varpi_{\la_1-1}).  
\end{align}
\end{proposition}

\begin{theorem}  
\label{thm:gen_dimhom_FH}
Let $m,k \in \Z_{\ge1}$ and $\bla \preceq^\bfr \bmu \in \calP(k)$ with $\la_1 < n-\delta(\g=D^{(1)}_{n})$. Then we have
\[
\dim\bigl(\Hom_{U_q(\g_0)}(\KR(m,\bla),V(\tau))\bigr) \le \dim\bigl(\Hom_{U_q(\g_0)}(\KR(m,\bmu),V(\tau))\bigr)
\]
for all $\tau \in P_0^+$. 
\end{theorem}

\begin{remark}
Let us consider when $\g=A_{n-1}^{(1)}$.
In this case, every $\la \in I_0$ is special.
Hence~\eqref{eq: schur positive II} can be expressed as 
\begin{align} \label{eq: Schur equation II}
\calS_{(\la_1-1)^m} \calS_{(\la_2+1)^m} - \calS_{\la_1^m} \calS_{\la_2^m} =\sum_{\tau \in P_0^+} c_\tau \calS_\tau  
\quad
(c_\tau\in \Z_{\ge0}).
\end{align}
Equation~\eqref{eq: Schur equation II} can be obtained by applying the $\omega$ involution to~\eqref{eq: Schur equation} for $i = m$ and $m_s = \la_s$ $(s=1,2)$.
It also can be obtained from a particular case of~\cite[Theorem~5]{LPP07}.
\end{remark}

\subsection{Application to quiver Hecke algebras}  
 
The representation theory of quantum affine algebra $\uqpg$ in closely related to the ones of quiver Hecke algebra $\calR$ of type $\gfin$.
The quiver Hecke algebra $\calR$ is introduced by Khovanov--Lauda \cite{KL1} and Rouquier \cite{R08} independently, which is $\Z$-graded and categorifies the negative half of quantum group $U_q^-(\gfin)$; i.e.,
$$
 \Q(q) \otimes_{\Z[q,q^{-1}]}  K(\calR^\gfin\gmod) \iso U_q^-(\gfin) ,
$$
where $\calR^\gfin\gmod$ denotes the category of finite dimensional graded $\calR^\gfin$-modules. Here $\Z[q,q^{-1}]$-module structure on the Grothendieck group
$K(\calR^\gfin\gmod)$ is induced by the degree shift functor $q$ and the ring structure on $K(\calR^\gfin\gmod)$ is induced by the \defn{convolution product} $\conv$. 

In \cite{KKK18,KKK15}, Kang--Kashiwara--Kim constructed a generalized Schur--Weyl functor
$$
\calF^{(1)}_\calQ \colon \calR^\gfin\gmod \to \scrC^{(1)}_\calQ
$$
for each $\rmQ$-datum $\calQ$ of $\uqpg$ of untwisted affine type $A_n^{(1)}$ and $D_n^{(1)}$, which is monoidal and sends simple modules in $\calR^\gfin\gmod$ to simple modules in $\scrC^{(1)}_\calQ$ bijectively. 
The result is extended to twisted affine type and exceptional type in \cite{KKKOIV,KO18,OhS19}:
$$
\calF_\calQ \colon \calR^\gfin\gmod \to \scrC_\calQ \quad \text{ for \emph{any} quantum affine algebra $\uqpg$}.
$$
Later, it is proved in \cite{Fuj22b,Naoi21} that
the categories $\calR^\gfin\gmod$ and $\scrC_\calQ$ are equivalent. 
Thus our result among KR modules in $\scrC_\calQ$
can be translated into the relationship between their corresponding modules in $\calR^\gfin\gmod$. Those modules are refereed to as \defn{determinantial modules adapted to $\calQ$}.  Furthermore, in \cite{KKOP23P}, it is proved that the $\de$-invariants are preserved by $ \calF_\calQ$, where $\de$-invariants
in $\calR^\gfin\gmod$ is defined also by the $R$-matrices (see \cite{KKKO15} for more detail).

Using the higher Dorey's rule, we have homomorphisms among determinantial modules adapted to $\calQ$ in $\calR^\gfin\gmod$.

\begin{theorem} \label{thm: Dorey in Rgmod}
The higher Dorey's for $\calR^\gfin\gmod$ holds. Namely, if there exists a homomorphism
$$
V_1 \otimes V_2 \twoheadrightarrow \dtens_{j \in J} W_j \quad\text{ in } \scrC_{\calQ}
$$
for KR modules $V_u$ $(u=1,2)$ and $W_j$ $(j \in J)$, there exists a homomorphism
$$
 \calF_\calQ^{-1}(V_1) \conv \calF_\calQ^{-1}(V_2) \twoheadrightarrow \dct{j \in J} \calF_\calQ^{-1}(W_j) \quad\text{ in $\calR^\gfin\gmod$} .
$$
Furthermore, we have
$$
\de(V_1,V_2) = \de( \calF_\calQ^{-1}(V_1), \calF_\calQ^{-1}(V_2)).
$$
\end{theorem}

\begin{remark} \label{rmk: multiplication structure}
In~\cite{VV11,R12}, it was proved that the dual canonical/upper global basis $\bbB$ of $U_q^-(\gfin)$ is categorified by the set of simple modules in $\calR^\gfin\gmod$.
In particular, the simple modules corresponding to KR modules under $\calF_\calQ$ categorify the quantum unipotent minors associated with a $\calQ$-datum and $i$-boxes.
Hence, the higher Dorey's rule in Theorem~\ref{thm: Dorey in Rgmod} can be interpreted in terms of the multiplication structure of $\bbB$.
More precisely, by Leclerc's conjecture, now established as a theorem in~\cite[\S 4]{KKKO18}, Theorem~\ref{thm: Dorey in Rgmod} asserts that for elements (equivalently, quantum minors) $\bfb_i \in \bbB$ ($i=1,2$) corresponding to $\calF_\calQ^{-1}(V_i)$, the element (equivalently, (product of) quantum minor(s)) $\bfb' \in \bbB$ corresponding to $\dct{j \in J} \calF_\calQ^{-1}(W_j)$ appears with multiplicity $q^m$ in the expansion of $\bfb_1\bfb_2$ with respect to $\bbB$, satisfying
\[
\bfb_1\bfb_2 \;=\; q^m \bfb' \;+\; q^s \bfb'' \;+\; \sum_{\bfb \ne \bfb'} \gamma_{\bfb_1,\bfb_2}^\bfc(q)\,\bfc,
\]
where $m < s$ and $\gamma_{\bfb_1,\bfb_2}^\bfc(q) \in q^{m+1}\Z[q] \cap q^{s-1}\Z[q^{-1}]$.
Here, $\bfb''$ denotes the element of $\bbB$ corresponding to the socle $\calF_\calQ^{-1}(V_1) \sconv \calF_\calQ^{-1}(V_2)$. 
\end{remark}

\subsection{Relations among non KR modules}

In~\cite{KKO19}, exact monoidal functors $(t=1,2)$
\[
\calF_{A \rightarrow B}\colon \scrC^0_{A^{(t)}_{2n-1}} \to \scrC^0_{B^{(1)}_{n}},
\qquad\qquad
\calF_{B \rightarrow A}\colon \scrC^0_{B^{(1)}_{n}} \to \scrC^0_{A^{(t)}_{2n-1}}
\]
are constructed that send simple modules to simple modules bijectively thorough the relationship among $\calR^{A_\infty}\gmod$, $\scrC^0_{A^{(t)}_{2n-1}}$, and $\scrC^0_{B^{(1)}_{n}}$.
In contrast to the other Schur--Weyl duality functors discussed earlier, the functors $\calF_{A \leftrightarrow B}$ do \emph{not} preserve KR modules; in fact, they do not even preserve the fundamental modules (see~\cite[Section~3]{KKO19} and~\cite[Section~12]{HO19}).
Therefore, the higher Dorey's rule and the generalized T-system in this paper yields an interesting relations between modules, which in general are \emph{not} KR modules, in the categories $\scrC^0_{A^{(t)}_{2n-1}}$ and $\scrC^0_{B^{(1)}_{n}}$. 

\begin{corollary} \label{cor: non KR rel}
Let $\{ X^{(1)}, Y^{(1)} \} = \{ A_{2n-1}^{(t)}, B_{n}^{(1)} \}$ $(t=1,2)$. 
For each higher Dorey's rule  
\[
V_1 \otimes V_2 \twoheadrightarrow \dtens_{j \in J} W_j \quad \text{in $\scrC^0_{X^{(1)}}$} 
\]
with KR modules $V_u$ $(u=1,2)$ and $W_j$ $(j \in J)$, there exists a homomorphism in $\scrC^0_{Y^{(1)}}$ 
\[
\calF_{X \to Y}(V_1) \otimes \calF_{X \to Y}(V_2) \twoheadrightarrow \dtens_{j \in J} \calF_{X \to Y}(W_j),
\]
where $\calF_{X \to Y}(V_i)$ $(i=1,2)$ and $\calF_{X \to Y}(W_j)$ are not KR modules in general.
\end{corollary}

As seen in the previous section, for $\rmQ$-data $\calQ^{\Ang{1}}=(\Dynkin^{\Ang{1}},\xi^{\Ang{1}},\sigma^{\Ang{1}})$ and $\calQ^{\Ang{2}}=(\Dynkin^{\Ang{2}},\xi^{\Ang{2}},\sigma^{\Ang{2}})$ with $\Dynkin^{\Ang{1}}=\Dynkin^{\Ang{2}}$ of $\g_\fin$-type, 
we have functors
\[
\calF_{\calQ^{\Ang{i}}} \colon \scrC_{\calQ^{\Ang{i}}} \to \calR^\gfin\gmod,
\]
which in turn induce a functor 
\[
\calF_{\calQ^{\Ang{1}} \to \calQ^{\Ang{2}}} \colon \scrC_{\calQ^{\Ang{1}}} \to \scrC_{\calQ^{\Ang{2}}},
\]
sending simple modules to simple modules bijectively.  
However, this functor does \emph{not} preserve KR modules either.  
With this observation, we arrive at the following result:

\begin{corollary} \label{cor: non KR rel2}
Let $\scrC_{\calQ^{\Ang{1}}}$ and $\scrC_{\calQ^{\Ang{2}}}$ be the heart subcategories of $\scrC^0_{\g_1}$ and $\scrC^0_{\g_2}$, where $\g_1$ and $\g_2$ are possibly different but have the requirement $\Dynkin^{\Ang{1}} = \Dynkin^{\Ang{2}}$.
For each higher Dorey's rule  
\[
V_1 \otimes V_2 \twoheadrightarrow \dtens_{j \in J} W_j \qquad \text{in $\scrC_{\calQ^{\Ang{1}}}$} 
\]
with KR modules $V_u$ $(u=1,2)$ and $W_j$ $(j \in J)$, there exists a homomorphism in $\scrC_{\calQ^{\Ang{2}}}$ 
\[
\calF_{\calQ^{\Ang{1}} \to \calQ^{\Ang{2}}}(V_1) \otimes \calF_{\calQ^{\Ang{1}} \to \calQ^{\Ang{2}}}(V_2) \twoheadrightarrow \dtens_{j \in J} \calF_{\calQ^{\Ang{1}} \to \calQ^{\Ang{2}}}(W_j).
\]
Furthermore, we have
\[
\de(V_1,V_2) = \de(\calF_{\calQ^{\Ang{1}} \to \calQ^{\Ang{2}}}(V_1),\calF_{\calQ^{\Ang{1}} \to \calQ^{\Ang{2}}}(V_2)). 
\]
\end{corollary}

Let us again emphasize that $\calF_{\calQ^{\Ang{1}} \to \calQ^{\Ang{2}}}(V_i)$ ($i=1,2$) and $\calF_{\calQ^{\Ang{1}} \to \calQ^{\Ang{2}}}(W_j)$ ($j \in J$) in Corollary~\ref{cor: non KR rel2} are not KR modules in general.
Note that the restriction of $\calF_{A \leftrightarrow B}$ to the heart subcategories $\scrC_Q$ (resp.\ $\scrC_\calQ$) does \emph{not} coincide with $\calF_{Q \leftrightarrow \calQ}$ when $\scrC_Q$ is of type $A_{2n-1}^{(1)}$ and $\scrC_\calQ$ is of type $B_{n}^{(1)}$.


\appendix

\section{Resolving the ambiguities in the \texorpdfstring{$C_n^{(1)}$}{Cn(1)} case} \label{appensec: resolving}

In this appendix, we show that in type $C_n^{(1)}$, the ambiguities in the multiplicities of roots can be resolved in certain cases, by applying auxiliary techniques that are not used in the main body of the paper.  

Let us consider the case of $d_{2^3,2^5}(z)$ for type $C_3^{(1)}$.
The conjectural formula is
\[
(z-\qs^{14})(z-\qs^{12})^2 (z-\qs^{10})^3 (z-\qs^{8})^3 (z-\qs^{6})^2 (z-\qs^{4}),
\]
while~\eqref{eq: factor C 1''} guarantees
\begin{align} \label{eq: 2325 ep}
(z-\qs^{14})(z-\qs^{12})^2 (z-\qs^{10})^{2+\epsilon_{10}} (z-\qs^{8})^{2+\epsilon_{8}} (z-\qs^{6})^{1+\epsilon_{6}} (z-\qs^{4}),
\end{align}
where $\epsilon_{t} \in \{0,1\}$.  

We first prove that the multiplicity of the root $\qs^{6}$ (resp. $\qs^{10}$) is $2$ (resp. $3$), as conjectured.  
Take $V_1 = \Vkm{2^3}_{(-\qs)^{-1}}$ and $V_2 = \Vkm{2^5}_{(-\qs)^5}$. Then by the block decomposition in~\cite{KKOP22} (see also~\cite{CM05}), we have
\[
\La^\infty(V_1,V_2) = 4.
\]
On the other hand, Proposition~\ref{prop: Lambda property}\eqref{it: Lainf d} gives
\[
\La^\infty(V_1,V_2)  = 1+\epsilon_{6} - 1 + 2+\epsilon_{10}
\]
from~\eqref{eq: 2325 ep}.
This confirms that $\epsilon_6=1$ and $\epsilon_{10}=1$, so the multiplicities at $\qs^6$ and $\qs^{10}$ are $2$ and $3$, respectively.  

Finally, we determine the multiplicity at $\qs^8$. Taking $V_1 = \Vkm{2^3}_{(-\qs)^{-1}}$ and $V_2 = \Vkm{2^7}_{(-\qs)^7}$, we obtain
\[
\La^\infty(V_1,V_2) = 6 = 4 + 2\epsilon_{8},
\]
as in the previous case. Hence $\epsilon_8=1$, so the multiplicity at $\qs^8$ is $3$.  

Therefore, all ambiguities in $d_{2^3,2^5}(z)$ for type $C_3^{(1)}$ are removed.

\section{Extended T-systems and the \texorpdfstring{$\de$}{d}-invariants}
\label{sec:extended_T_system}

In this appendix, we generalize the extended T-system in~\cite{MY12} for $A_{n-1}^{(1)}$ and $B_n^{(1)}$, in~\cite{LM13} for $G_2^{(1)}$, and~\cite{Li15} for $C_3^{(1)}$ further using $\de$-invariants between KR modules.
We basically employ the frameworks of \cite{KKOP24A} and~\cite{Naoi24} and consider every affine type.

Recall the compatible reading $\frakR$ of $\hbDynkin_0$, Convention~\ref{conv: identify R} and its fundamental module $\frakR[a]$ in \S\ref{subsec: i-box and T-system}. 
Let us fix the compatible reading $\frakR$ of $\hbDynkin_0$ of \emph{any} classical affine type $\g$.
Then let us take a subsequence
\[
\bsk = (k_1 < k_2 < \cdots <k_p) \text{ of } \frakR
\]
and define
\[
\bbS_\bsk[a,b] \seteq \head(  \frakR[k_b] \otimes \frakR[k_{b-1}]  \otimes  \cdots  \otimes \frakR[k_{a}])
\]
for any sub-interval $[a,b] \subseteq [1,p]$.
Since the sequence of fundamental modules $(\frakR[k_b], \frakR[k_{b-1}], \dotsc, \frakR[k_{a}])$ are strongly unmixed, and hence normal, $\bbS_\bsk[a,b]$ is simple by Lemma~\ref{lem: normal property}. 

Now let us recall the one of the main theorems in~\cite{Naoi24}:

\begin{theorem}[{\cite[Proposition 4.1.1, Theorem 4.2.1]{Naoi24}}] \label{thm: N24}
Assume $p \ge2$ and both of the following two conditions are satisfied:
\bna
\item \label{it: N24 a} for any $1\le a<b \le p$, we have $\de(\frakR[k_a], \bbS_\bsk[a+1,b])=1$, and
\item \label{it: N24 b} for any $1\le a<b \le p$, we have $\de(\frakR[k_b], \bbS_\bsk[a,b-1])=1$.
\ee
Then $ \bbS_\bsk[a,b]$ is a prime real simple module and there exists a short exact sequence
$$
0 \to \Sk[1,p]  \otimes \Sk[2,p-1] \to  \Sk[1,p-1]  \otimes \Sk[2,p]\to \head\left( \stens^{\gets}_{1 \le a <p} \frakR[k_{a}] \hconv \frakR[k_{a+1}] \right  )   \to 0.
$$
\end{theorem}

Now let us take 
\begin{align*}  
\tbsk = (\tk_1 < \tk_2 < \cdots <\tk_p) \text{ of } \frakR    
\end{align*}
such that
\begin{align} \label{eq: tk condi}
\de(\frakR[\tk_s],\frakR[\tk_{s+1}])=1 \text{ for $1 \le s <p$} \qtq  \de(\frakR[\tk_s],\frakR[\tk_{s+t}]) =0 \text{ for $t>1$.} 
\end{align}

Note that the following lemma and corollary holds for an \emph{arbitrary} quantum affine algebra (cf.~\cite{Naoi24}).

\begin{proposition} \label{prop: N cond holds}
The conditions in Theorem~\ref{thm: N24} holds for $\tbsk$.
In particular, the simple module $\Stk[a,b]$ is prime and real. 
\end{proposition}

\begin{proof} \eqref{it: N24 a}
Note that
\begin{align*}
0 \le \de(\frakR[\tk_a], \Stk[a+1,b]) & \le \sum_{s=a+1}^b \de(\frakR[\tk_a], \frakR[\tk_s]) =1
\\
0 \le \de(\frakR[\tk_a], \Stk[a+2,b]) & \le \sum_{s=a+2}^b \de(\frakR[\tk_a], \frakR[\tk_s]) =0
\end{align*}
by Proposition~\ref{prop: de ge 0} and Proposition~\ref{prop: de less than equal to}. On the other hand,
\begin{align*}
\de(\scrD^{-1}\frakR[\tk_a], \Stk[a+2,b]) = 0   
\end{align*}
by the unmixed property and Proposition~\ref{prop: de less than equal to}.
Hence Lemma~\ref{lem: de=de}\eqref{it: de=de 2} implies
$$
\de(\frakR[\tk_a], \Stk[a+1,b]) = \de(\frakR[\tk_a], \frakR[\tk_{a+1}]) = 1.
$$
Similarly,~\eqref{it: N24 b} holds.
\end{proof}

\begin{corollary}
We have a short exact sequence 
\begin{align} \label{eq: ext T}
0 \to \Stk[1,p]  \otimes \Stk[2,p-1]   \to \Stk[1,p-1]  \otimes \Stk[2,p]\to \head\left( \stens^{\to}_{1 \le a <p} \frakR[\tk_{a}] \hconv \frakR[\tk_{a+1}] \right) \to 0.
\end{align}
\end{corollary}

Let us take a sequence of $i$-boxes 
\begin{align} \label{eq: bsc}
\bsc = (\frakc_1,\frakc_2,\ldots,\frakc_p)    
\end{align}
such that
\begin{equation} \label{eq: tc condi}
\begin{aligned}  
& \de(\scrD^{u}\frakR(\frakc_s),\frakR(\frakc_{s+t})) \le^\dagger \delta(t=1)\delta(u=0) \ \ \text{ for $1 \le s <p$, $t>1$ and $u \in \{0,-1\}$.} 
\end{aligned}
\end{equation}
Note that each $\frakR(\frakc_s)$ is a KR module and the sequence
$$
( \frakR(\frakc_p), \frakR(\frakc_{p-1}), \ldots, \frakR(\frakc_1) ) \text{ is normal}.
$$
Hence
\begin{align} \label{eq: Sc}
\Sc[a,b] \seteq \head(  \frakR(\frakc_b) \otimes \frakR(\frakc_{b-1})  \otimes  \cdots  \otimes \frakR(\frakc_{a}))
\end{align}
is simple for any sub-interval $[a,b] \subseteq [1,p]$.

\begin{proposition} \label{prop: Sc pro} \hfill
\bna
\item \label{it: =1} For $1 \le a < b \le p$, we have $\de(\frakR(\frakc_a),\Sc[a+1,b])=\de(\frakR(\frakc_b),\Sc[a,b-1]) \le^{\star} 1$.
  In particular, if $\le^\dagger$ in~\eqref{eq: tc condi} is always an equality, then $\le^{\star}$ is also an equality. 
\item \label{it: subset commute} If $[a',b'] \subseteq [a,b]$, then $\de(\Sc[a',b'],\Sc[a,b]) = 0$. 
\item \label{it: Sc real} $\Sc[a,b]$ is real.
\ee
\end{proposition}

\begin{proof}
The proof of~\eqref{it: =1} is the same as the proof of Proposition~\ref{prop: N cond holds}.
Let us focus on~\eqref{it: subset commute}.
First assume that $[a',b']=[c]$ for some $a \le c \le b$. By the assumption on $\bsc$, Proposition~\ref{prop: three} implies
that
$$
\de(\frakR(\frakc_c),\Sc[a,b]) = \de(\frakR(\frakc_c),\Sc[c+1,b] \hconv \frakR(\frakc_c)) + \de(\frakR(\frakc_c),\frakR(\frakc_c)  \hconv \Sc[a,c-1]).
$$
By Lemma~\ref{lem: decrease} and~\eqref{it: =1}, we have 
$$
\de(\frakR(\frakc_c),\Sc[c+1,b] \hconv \frakR(\frakc_c)) = \de(\frakR(\frakc_c),\frakR(\frakc_c)  \hconv \Sc[a,c-1]) =0
$$
and hence the assertion for $[c]$ follows. Then the assertion for $[a',b']$ directly follows from Proposition~\ref{prop: de less than equal to}. 
The third assertion is obvious. 
\end{proof}

\begin{lemma} \hfill
\bna
\item \label{it: le1 eT} We have $\de(\Sc[1,p-1],\Sc[2,p]) \le 1$.
\item \label{it: head eT} We have $ \Sc[1,p-1] \sconv \Sc[2,p] \iso \Sc[1,p] \otimes \Sc[2,p-1]$.
\ee
\end{lemma}

\begin{proof}
\eqref{it: le1 eT} follows from Proposition~\ref{prop: de less than equal to}. Thus let us focus on~\eqref{it: head eT}. 
Note that we have a non-zero composition of homomorphisms
$$
 \Sc[2,p]  \otimes \Sc[1,p-1] \rightarrowtail  \Sc[2,p-1]  \otimes \frakR(\frakc_p) \otimes \Sc[1,p-1] \twoheadrightarrow
 \Sc[2,p-1]   \otimes \Sc[1,p],
$$
by Proposition~\ref{prop: non-van}. Then the assertion follows from the fact that $\Sc[2,p-1]   \otimes \Sc[1,p]$ is simple. 
\end{proof}

\begin{problem}
It would be interesting to find an additional condition on $\bsc$ that implies
\begin{equation}
\de(\Sc[1,p-1],\Sc[2,p]) = 1,
\end{equation}
and hence the existence of short exact sequence whose socle is isomorphic to $\Sc[1,p] \otimes \Sc[2,p-1]$.
\end{problem}

For instance, when we choose each $\frakc_s =[\tk_s]$ as~\eqref{eq: tk condi}, it implies $\de(\Sc[1,p-1],\Sc[2,p]) = 1$ and~\eqref{eq: ext T}.

\section{Classical T-systems} \label{Appendix: T-system}

The T-system among the KR modules over untwisted quantum affine algebras can be presented in terms of $\Vkm{i^m}_{a'}$ as follows: ($a \in \bfk^\times$)
\begin{align*}
& \text{ Type $ADE_n^{(1)}$:} \allowdisplaybreaks \\
&  \hspace{5ex} 0 \to \Vkm{i^{m-1}}_{a}  \tens  \Vkm{i^{m+1}}_{a}    \to    \Vkm{i^m}_{a(-q)^{-1}} \tens \Vkm{i^m}_{a(-q)}     \to
  \bigotimes_{ \substack{  j \in I_0 \\  a_{ij}=-1}} \Vkm{j^m}_{a} \to 0.  \allowdisplaybreaks \\
& \text{ Type $B_n^{(1)}$: For $1 \le i \le n-2$:} \allowdisplaybreaks \\
& \hspace{5ex} 0 \to  \Vkm{i^{m-1}}_{a}  \tens \Vkm{i^{m+1}}_{a}  \to \Vkm{i^m}_{a(-q)^{-1}}  \tens \Vkm{i^m}_{a(-q)}  \to   \Vkm{(i-1)^m}_{a} \tens \Vkm{(i+1)^m}_{a} \to 0. \allowdisplaybreaks \\
&  \hspace{5ex} 0 \to    \Vkm{(n-1)^{m-1}}_{a} \tens \Vkm{(n-1)^{m+1}}_{a} \to  \Vkm{(n-1)^m}_{a(-q)^{-1}} \tens  \Vkm{(n-1)^m}_{a(-q)} \allowdisplaybreaks\\
& \hspace{53.5ex} \to \Vkm{(n-2)^m}_{a} \tens     \Vkm{n^{2m}}_{(-1)^{m}a }    \to 0. \allowdisplaybreaks \\
&  \hspace{5ex} 0 \to  \Vkm{n^{2m-1}}_a  \tens \Vkm{n^{2m+1}}_a    \to    \Vkm{n^{2m}}_{a(-\qs)^{-1}} \tens  \Vkm{n^{2m}}_{a(-\qs)} \\
& \hspace{43.2ex} \to  \Vkm{(n-1)^m}_{(-1)^ma\qs^{-1}}    \tens \Vkm{(n-1)^m}_{(-1)^{m+1}a\qs} \to 0.  \allowdisplaybreaks \\
&  \hspace{5ex} 0 \to \Vkm{n^{2m}}_a \tens \Vkm{n^{2m+2}}_a \to  \Vkm{n^{2m+1}}_{a(-\qs)^{-1}} \tens  \Vkm{n^{2m+1}}_{a(-\qs)} \\
& \hspace{46ex} \to  \Vkm{(n-1)^{m+1}}_{(-1)^ma}\tens \Vkm{(n-1)^{m}}_{(-1)^{m+1}a} \to 0.  \allowdisplaybreaks \\
& \text{ Type $C_n^{(1)}$: For $1 \le i \le n-2$}  \allowdisplaybreaks \\
& \hspace{5ex} 0 \to   \Vkm{i^{m-1}}_{a}  \tens \Vkm{i^{m+1}}_{a}   \to   \Vkm{i^m}_{a(-\qs)^{-1}} \tens \Vkm{i^m}_{a(-\qs)}  \to  \Vkm{(i-1)^m}_{a} \tens \Vkm{(i+1)^m}_{a} \to 0. \allowdisplaybreaks \\
&  \hspace{5ex} 0 \to  \Vkm{(n-1)^{2m-1}}_{a} \tens \Vkm{(n-1)^{2m+1}}_{a}\to  \Vkm{(n-1)^{2m}}_{a(-\qs)^{-1}}\tens  \Vkm{(n-1)^{2m}}_{a(-\qs)}     \\
& \hspace{33ex}\to \Vkm{(n-2)^{2m}}_{a} \tens \Vkm{n^m}_{(-1)^m a(\qs)^{-1}} \tens \Vkm{n^m}_{(-1)^ma(\qs)}   \to 0.  \allowdisplaybreaks \\
&  \hspace{5ex} 0 \to  \Vkm{(n-1)^{2m}}_{a} \tens \Vkm{(n-1)^{2m+2}}_{a}    \to   \Vkm{(n-1)^{2m+1}}_{a(-\qs)^{-1}}  \tens \Vkm{(n-1)^{2m+1}}_{a(-\qs)} \\
& \hspace{38ex}  \to\Vkm{(n-2)^{2m+1}}_{a} \tens  \Vkm{n^{m+1}}_{(-1)^ma} \tens  \Vkm{n^{m}}_{(-1)^ma}   \to 0.  \allowdisplaybreaks \\
&  \hspace{5ex} 0 \to  \Vkm{n^{m-1}}_{a}  \tens   \Vkm{n^{m+1}}_{a} \to   \Vkm{n^m}_{a(-q)^{-1}} \tens \Vkm{n^m}_{a(-q)}  \to \Vkm{(n-1)^{2m}}_{(-1)^{1+k}a}  \to 0.  \allowdisplaybreaks \\
& \text{ Type $F_4^{(1)}$:}  \allowdisplaybreaks \\
& \hspace{5ex} 0 \to  \Vkm{1^{m-1}}_{a}  \tens  \Vkm{1^{m+1}}_{a}    \to    \Vkm{1^{m}}_{a(-q)^{-1}} \tens \Vkm{1^{m}}_{a(-q)} \to   \Vkm{2^{m}}_{a}  \to 0.  \allowdisplaybreaks \\
& \hspace{5ex} 0 \to  \Vkm{2^{m-1}}_{a}  \tens  \Vkm{2^{m+1}}_{a} \to \Vkm{2^{m}}_{a(-q)^{-1}}  \tens  \Vkm{2^{m}}_{a(-q)}  \to \Vkm{1^{m}}_{a}  \tens \Vkm{3^{2m}}_{(-1)^{1+k}a}  \to 0.  \allowdisplaybreaks \\
& \hspace{5ex} 0 \to \Vkm{3^{2m-1}}_{a}   \tens \Vkm{3^{2m+1}}_{a}  \to    \Vkm{3^{2m}}_{a(-\qs)^{-1}} \tens \Vkm{3^{2m}}_{a(-\qs)} \\
& \hspace{42ex}  \to \Vkm{2^m}_{(-1)^ra\qs^{-1}}  \tens  \Vkm{2^m}_{(-1)^ra\qs}  \tens  \Vkm{4^{2m}}_{a}  \to 0.  \allowdisplaybreaks \\
& \hspace{5ex} 0 \to \Vkm{3^{2m}}_{a}  \tens  \Vkm{3^{2m+2}}_{a}  \to  \Vkm{3^{2m+1}}_{a(-\qs)^{-1}}  \tens   \Vkm{3^{2m+1}}_{a(-\qs)} \\
& \hspace{40ex} \to \Vkm{2^{m+1}}_{(-1)^ra}  \tens \Vkm{2^{m}}_{(-1)^{m-1}a} \tens  \Vkm{4^{2m+1}}_{a}  \to 0.  \allowdisplaybreaks \\
& \hspace{5ex} 0 \to \Vkm{4^{m-1}}_{a}   \tens \Vkm{4^{m+1}}_{a}   \to \Vkm{4^m}_{a(-\qs)^{-1}} \tens  \Vkm{4^m}_{a(-\qs)} \to    \Vkm{3^m}_{a}   \to 0.  \allowdisplaybreaks \\
& \text{ Type $G_2^{(1)}$:}  \allowdisplaybreaks \\
& \hspace{5ex} 0 \to \Vkm{1^{m-1}}_a  \tens \Vkm{1^{m+1}}_a \to \Vkm{1^m}_{a(-q)^{-1}}\tens \Vkm{1^m}_{a(-q)}  \to  \Vkm{2^{3m}}_a \to 0.  \allowdisplaybreaks \\
& \hspace{5ex} 0 \to\Vkm{2^{3m-1}}_a  \tens \Vkm{2^{3m+1}}_a \to \Vkm{2^{3m}}_{a(-\qt)^{-1}} \tens \Vkm{2^{3m}}_{a(-\qt)}  \\
& \hspace{42ex} \to  \Vkm{1^{m}}_{a(-\qt)^{-2}}  \tens  \Vkm{1^{m}}_a  \tens  \Vkm{1^{m}}_{a(-\qt)^2}  \to 0.  \allowdisplaybreaks \\
& \hspace{5ex} 0 \to\Vkm{2^{3m}}_a \tens \Vkm{2^{3m+2}}_a   \to \Vkm{2^{3m+1}}_{a(-\qt)^{-1}} \tens \Vkm{2^{3m+1}}_{a(-\qt)}  \\
& \hspace{41ex}   \to  \Vkm{1^{m+1}}_a  \tens \Vkm{1^m}_{a(-\qt)^{-1}}  \tens  \Vkm{1^m}_{a(-\qt)}  \to 0.  \allowdisplaybreaks \\
& \hspace{5ex} 0 \to \Vkm{2^{3m+1}}_a \tens \Vkm{2^{3m+3}}_a  \to   \Vkm{2^{3m+2}}_{a(-\qt)^{-1}} \tens\Vkm{2^{3m+2}}_{a(-\qt)}  \\
& \hspace{41ex}    \to\Vkm{1^{m+1}}_{a(-\qt)^{-1}} \tens  \Vkm{1^{m+1}}_{a(-\qt)}  \tens \Vkm{1^m}_a \to 0.
\end{align*}

\section{Universal coefficients between fundamentals for non-exceptional types} \label{sec: univ table}

\begin{center}
\fontsize{8}{8}\selectfont
\begin{tabular}{ | c | c | c | c  | c | } \hline
Type & $n$ & $k,l$ &  universal coefficients & $p^*$ \\ \hline  \\[-1em]
$A_{n-1}^{(1)}$& $n \ge 1$ & $1 \le k,l \le n-1$& $ a_{k,l}(z) \equiv \dfrac{[|k-l|][2n-|k-l|]}{[k+l][2n-k-l]}$ & $(-q)^{n}$  \\  \hline \\[-1em]
$B_{n}^{(1)}$& $n \ge 3$&$1 \le k,l \le n-1$ & $
a_{k,l}(z) \equiv  \dfrac{  [|k-l|]\PA{2n+k+l-1}\PA{2n-k-l-1}[4n-|k-l|-2]}{[k+l]\PA{2n+k-l-1}\PA{2n-k+l-1}[4n-k-l-2]}$
 & $-(-q)^{2n-1}$\\ \cline {3-4} \\[-1em]
$\qs^2=q$   &  & $1 \le k \le n-1$ & $a_{k,n}(z)\equiv  \dfrac{\PSN{2n-2k-1}\PSN{6n+2k-3}}{\PSN{2n+2k-1}\PSN{6n-2k-3}} $ & \\ \cline {3-4} \\[-1em]
            &  & $k=l=n$ & $ a_{n,n}(z) \equiv \displaystyle \prod_{s=1}^n \dfrac{[2n+2s-2][2s-2]}{\PA{2s-1}\PA{2n-3+2s}}$ &  \\ \hline  \\[-1em]
$C_{n}^{(1)}$& $n \ge 2
$ & $1 \le k,l \le n$& $ a_{k,l}(z) \equiv
\dfrac{\PS{|k-l|}\PS{2n+2-k-l}\PS{2n+2+k+l}\PS{4n+4-|k-l|}}{\PS{k+l}\PS{2n+2-k+l}\PS{2n+2+k-l}\PS{4n+4-k-l}} $ & $(-\qs)^{2n+2}$ \\ \hline \\ [-1em]
$D_{n}^{(1)}$& $n \ge 4$ & $1 \le k,l \le n-2$ &
$a_{k,l}(z) \equiv \dfrac{\PSF{|k-l|}{2n+k+l-2}{2n-k-l-2}{4n-|k-l|-4}}{\PSF{k+l}{2n+k-l-2}{2n-k+l-2}{4n-k-l-4}}$&$(-q)^{2n-2}$\\ \cline {3-4} \\ [-1em]
             & & $1 \le k \le n-2$ & $ a_{k,n-1}(z)=a_{k,n}(z) \equiv  \dfrac{[n-k-1][3n+k-3]}{[n+k-1][3n-k-3]}$ &  \\ \cline {3-4}  \\ [-1em]
             & & $\{k,l\}=\{n,n-1\}  $ & $a_{n,n-1}(z)=a_{n-1,n}(z) \equiv \dfrac{\prod_{s=1}^{n-1}[4s-2]}{[2n-2]\prod_{s=1}^{n-2}[4s]}$   &  \\  \cline {3-4}  \\ [-1em]
& & $k=l\in\{n,n-1\}  $ & $a_{n,n}(z)=a_{n-1,n-1}(z) \equiv \dfrac{\prod_{s=0}^{n-1}[4s]}{[2n-2]\prod_{s=1}^{n-1}[4s-2]}$ &  \\ \hline \\[-1em]
$A_{2n-1}^{(2)}$& $n \ge 3$ & $1 \le k,l \le n$ & $ a_{k,l}(z)\equiv \dfrac{[|k-l|][4n-|k-l|]\PA{2n+k+l}\PA{2n-k-l}}{[k+l][4n-k-l]\PA{2n+|k-l|}\PA{2n-|k-l|}}$
& $-(-q)^{2n} $ \\ \hline \\ [-1em]
$A_{2n}^{(2)}$& $n \ge 1$ & $1 \le k,l \le n$ & $ a_{k,l}(z)\equiv  \dfrac{[|k-l|][4n+2-|k-l|][2n+1+k+l][2n+1-k-l]}{[k+l][4n+2-k-l][2n+1+|k-l|][2n+1-|k-l|]}  $ & $(-q)^{2n+1}$ \\ \hline \\ [-1em]
$D_{n+1}^{(2)}$& $n \ge 2$ & $1 \le k,l \le n-1$ &  $a_{k,l}(z) \equiv
\dfrac{\PPA{|\frac{k-l}{2}|}\PPA{2n-|\frac{k-l}{2}|}\PPA{n+\frac{k+l}{2}}\PPA{n-\frac{k+l}{2}}}
{\PPA{\frac{k+l}{2}}\PPA{2n-\frac{k+l}{2}}\PPA{n+|\frac{k-l}{2}|}\PPA{n-|\frac{k-l}{2}|}} $ & $-(-q^2)^n$ \\ \cline {3-4}\\ [-1em]
               & & $1 \le k \le n-1$ & $a_{k,n}(z) \equiv
               \dfrac{ \PPAP{\frac{3n+k}{2}}\PPAP{\frac{n-k}{2}}  }{\PPAP{\frac{3n-k}{2}}\PPAP{\frac{n+k}{2}} }$  &  \\ \cline {3-4} \\ [-1em]
               & & $k=l=n$ & $a_{n,n}(z) \equiv
              \displaystyle\prod_{s=1}^n  \dfrac{ \CE{n+s}\CE{n-s} }{ \CO{s} \CO{2n-s}}$  &  \\ \hline

\end{tabular}
\fontsize{11}{11}\selectfont
\end{center}
Here 
\bna
\item $\CE{a}=((-q^2)^{a}z;p^{*2})_\infty$,
\item $\CO{a}=(-(-q^2)^{a}z;p^{*2})_\infty$,
\item $\PPA{a} = ((-q^2)^az;p^{*2})_\infty \times (-(-q^2)^az;p^{*2})_\infty$,
\item $\PPAP{a} = (i(-q^2)^az;p^{*2})_\infty \times (-i(-q^2)^az;p^{*2})_\infty$.
\ee

\section{Denominator formulas for fundamentals modules in non-exceptional types}

\subsection{Denominator formulas between fundamental modules except $E_n^{(1)}$ and $F_4^{(1)}$} \label{subsec: fundamental deno}

\vskip 1em
\begin{center}
\fontsize{9}{9}\selectfont
\begin{tabular}{ | c | c | c | c  | } \hline
Type & $n$ & $k,l$ &  Denominators \\ \hline
$A_{n-1}^{(1)}$& $n \ge 2$ & $1 \le k,l \le n$& $ d_{k,l}(z)= \displaystyle\prod_{s=1}^{ \min(k,l,n-k,n-l)} \big(z-(-q)^{2s+|k-l|}\big)$  \\ \hline
$B_{n}^{(1)}$& $n \ge 3$ & $1 \le k,l \le n-1$ & $d_{k,l}(z) =  \displaystyle \prod_{s=1}^{\min (k,l)} \big(z-(-q)^{|k-l|+2s}\big)\big(z+(-q)^{2n-k-l-1+2s}\big)$ \\ \cline {3-4}
$\qs^2=q$   &  & $1 \le k \le n-1$ & $d_{k,n}(z) = \displaystyle  \prod_{s=1}^{k}\big(z-(-1)^{n+k}\qs^{2n-2k-1+4s}\big)$ \\ \cline {3-4}
            &  & $k=l=n$ & $ d_{n,n}(z)=\displaystyle \prod_{s=1}^{n} \big(z-(\qs)^{4s-2}\big)$ \\ \hline
$C_{n}^{(1)}$& $n \ge 2$ & $1 \le k,l \le n$& $ d_{k,l}(z)= \hspace{-2ex}\displaystyle \hspace{-3ex}\prod_{s=1}^{ \min(k,l,n-k,n-l)} \hspace{-5ex}
\big(z-(-\qs)^{|k-l|+2s}\big)\hspace{-2ex}\prod_{i=1}^{ \min(k,l)}\hspace{-1.5ex} \big(z-(-\qs)^{2n+2-k-l+2s}\big)$   \\ \hline
$D_{n}^{(1)}$& $n \ge 4$ & $1 \le k,l \le n-2$ & $d_{k,l}(z) = \displaystyle \prod_{s=1}^{\min (k,l)}
                \big(z-(-q)^{|k-l|+2s}\big)\big(z-(-q)^{2n-2-k-l+2s}\big)$ \\ \cline {3-4}
             & & $1 \le k \le n-2$ & $ d_{k,n-1}(z)=d_{k,n}(z) = \displaystyle \prod_{s=1}^{k}\big(z-(-q)^{n-k-1+2s}\big)$ \\ \cline {3-4}
             & & $\{k,l\}=\{n,n-1\}$ & $d_{n,n-1}(z)=d_{n-1,n}(z)=\displaystyle \prod_{s=1}^{\lfloor \frac{n-1}{2} \rfloor} \big(z-(-q)^{4s}\big)$\\ \cline {3-4}
             & & $k=l\in\{n,n-1\}$ & $d_{n,n}(z)=d_{n-1,n-1}(z)=\displaystyle \prod_{s=1}^{\lfloor \frac{n}{2} \rfloor} \big(z-(-q)^{4s-2}\big)$ \\ \hline
$A_{2n-1}^{(2)}$& $n \ge 3$ & $1 \le k,l \le n$ & $ d_{k,l}(z)= \displaystyle \prod_{s=1}^{\min(k,l)} \big(z-(-q)^{|k-l|+2s}\big)\big(z+(-q)^{2n-k-l+2s}\big)$ \\ \hline
$A_{2n}^{(2)}$& $n \ge 1$ & $1 \le k,l \le n$ & $d_{k,l}(z) = \displaystyle \prod_{s=1}^{\min(k,l)} \big(z-(-q)^{|k-l|+2s}\big)\big(z-(-q)^{2n+1-k-l+2s}\big)$ \\ \hline
$D_{n+1}^{(2)}$& $n \ge 2$ & $1 \le k,l \le n-1$ &  $d_{k,l}(z) = \displaystyle \prod_{s=1}^{\min(k,l)} \big(z^2 - \mqs^{|k-l|+2s}\big)\big(z^2 - \mqs^{2n-k-l+2s}\big)$ \\ \cline {3-4}
               & & $1 \le k \le n-1$ & $d_{k,n}(z) = \displaystyle \prod_{s=1}^{k}\big(z^2+(-q^{2})^{n-k+2s}\big)$ \\ \cline {3-4}
               & & $k=l=n$ & $ d_{n,n}(z)=\displaystyle \prod_{s=1}^{n} \big(z+\mqs^{s}\big)$ \\ \hline
$D_{4}^{(3)}$ &  & $\omega^3=1$ & $d_{1,1}(z)=(z-q^{2})(z-q^{6})(z-\omega q^{4})(z-\omega^2 q^{4})$           \\
               & &  & $d_{1,2}(z)=(z^3+q^9)(z^3+q^{15}) \qquad\qquad\quad \ \qquad$  \\
               & &  & $d_{2,2}(z)=(z^3-q^6)(z^3-q^{12})^2(z^3-q^{18}) \quad \quad \ \ $ \\ \hline
$G_{2}^{(1)}$ &  & $\qt^3=q$  & $d_{1,1}(z)=(z-\qt^{6})(z-\qt^{8})(z-\qt^{10})(z-\qt^{12})$  \\
              &  &  & $d_{1,2}(z)=(z+\qt^7)(z+\qt^{11}) \qquad\qquad\qquad\quad$  \\
              &  &  & $d_{2,2}(z)=(z-\qt^{2})(z-\qt^{8})(z-\qt^{12}) \qquad\quad \ \ $  \\ \hline
\end{tabular}
\fontsize{11}{11}\selectfont
\end{center}

\subsection{Denominator formulas between KR modules except $E_n^{(1)}$, $F_4^{(1)}$ and $G_2^{(1)}$}

\vskip 1em
\begin{center}
\fontsize{7}{7}\selectfont
\begin{tabular}{ | c | c | c | c  | } \hline
Type & $n,m,p$ & $k,l$ &  Denominators \\ \hline
$A_{n}^{(1)}$& $n \ge 1$ & $1 \le k,l \le n$& $ d_{l^p,k^m}(z)=\displaystyle \prod_{s=1}^{\min(k,l,n-k,n-l)}  \prod_{t=0}^{\min(p,m)-1} \big(z-(-q)^{|k-l|+|p-m|+2(s+t)} \big)  $  \\ \hline
$B_{n}^{(1)}$& $n \ge 3$ & $1 \le k,l \le n-1$ & $d_{l^p,k^m}(z) = \displaystyle \prod_{t=0}^{\min(m,p)-1}\prod_{s=1}^{\min(k,l)} (z-(-q)^{|k-l|+|m-p|+2(s+t)}) (z+(-q)^{2n-k-l+|m-p|-1+2(s+t)}) $ \\ \cline {3-4}
$\qs^2=q$   &  & $1 \le l \le n-1$ & $d_{l^p,n^m}(z) =  \displaystyle \prod_{t=0}^{\min(2p,m)-1}\prod_{s=1}^{l} \big(z-(-1)^{n+l+p+m}(\qs)^{2n-2l-2+|2p-m|+4s+2t} \big)$ \\ \cline {3-4}
            &  & $k=l=n$ & $d_{n^p,n^m}(z) =  \displaystyle\prod_{t=0}^{\min(p,m)-1}\prod_{s=1}^{n} \big(z-(-\qs)^{4s+2t-2+|p-m|} \big)$ \\ \hline
$C_{n}^{(1)}$& $n \ge 2$ & $1 \le k,l \le n$& $ d_{k,l}(z)= \hspace{-2ex}\displaystyle \hspace{-3ex}\prod_{s=1}^{ \min(k,l,n-k,n-l)} \hspace{-5ex}
\big(z-(-\qs)^{|k-l|+2s}\big)\hspace{-2ex}\prod_{i=1}^{ \min(k,l)}\hspace{-1.5ex} \big(z-(-\qs)^{2n+2-k-l+2s}\big)$   \\ \cline {2-4}
             &$\max(m,p)\ge 2$ & $1\le k,l <n$   &  $ d_{l^p,k^m}(z) =   \displaystyle\prod_{t=0}^{\min(m,p)-1} \prod_{s=1}^{\min(k,l)}  \big(z-(-\qs)^{|k-l|+|m-p|+2s+2t}\big)\big(z-(-\qs)^{2n+2-k-l+|m-p|+2s+2t} \big)$ \\ \cline {3-4}
             & & $1\le l <n$   & $d_{l^p,n^m}(z) = \displaystyle\prod_{t=0}^{\min(p,2m)-1} \prod_{s=1}^{l} \big(z-(-\qs)^{n+1-l+|2m-p|+2s+2t}\big)$ \\ \cline {3-4}
             & & $k=l=n$ & $d_{n^p,n^m} =\displaystyle \prod_{t=0}^{\min(p,m)-1} \prod_{s=1}^{n} \big(z-(-\qs)^{2+|2m-2p|+2s+4t}\big)$ \\ \hline
    $D_{n}^{(1)}$& $n \ge 4$ & $1 \le k,l \le n-2$ & $ d_{l^p,k^m} = \displaystyle \prod_{t=0}^{\min(m,p)-1}\prod_{s=1}^{\min(k,l)} (z-(-q)^{|k-l|+|m-p|+2(s+t)}) (z-(-q)^{2n-k-l+|m-p|-2+2(s+t)}) $ \\ \cline {3-4}
             & & $1 \le k \le n-2$ & $d_{l^p,n^m}(z)=d_{l^p,(n-1)^m}(z) =  \displaystyle \prod_{t=0}^{\min(p,m)-1}\prod_{s=1}^{l} (z-(-q)^{n-l-1+|p-m|+2(s+t)})$ \\ \cline {3-4}
             & & $\{k,l\}=\{n,n-1\}$ & $d_{(n-1)^p,n^m}(z)=  \displaystyle\prod_{t=0}^{\min(p,m)-1}\prod_{s=1}^{\lfloor \frac{n-1}{2} \rfloor} (z-(-q)^{4s+2t+|p-m|} )$\\ \cline {3-4}
             & & $k=l\in\{n,n-1\}$ & $d_{n^p,n^m}(z)=d_{(n-1)^p,(n-1)^m}(z)= \displaystyle\prod_{t=0}^{\min(p,m)-1}\prod_{s=1}^{\lfloor \frac{n}{2} \rfloor} (z-(-q)^{4s+2t-2+|p-m|} )$ \\ \hline
$A_{2n-1}^{(2)}$& $n \ge 3$ & $1 \le k,l \le n$ & $  d^{A_{2n-1}^{(2)}}_{l^p,k^m}(z)=  d^{A_{2n-1}^{(1)}}_{l^p,k^m}(z) \times  d^{A_{2n-1}^{(1)}}_{l^p,{(2n-k)}^m}(-z) $ \\ \hline
$A_{2n}^{(2)}$& $n \ge 1$ & $1 \le k,l \le n$ & $  d^{A_{2n}^{(2)}}_{l^p,k^m}(z)=  d^{A_{2n}^{(1)}}_{l^p,k^m}(z) \times  d^{A_{2n}^{(1)}}_{l^p,{(2n+1-k)}^m}(z)$ \\ \hline
$D_{n+1}^{(2)}$& $n \ge 2$ & $1 \le k,l \le n-1$ &  $ d^{D_{n+1}^{(2)}}_{l^p,k^m}(z;q) = d^{D_{n+1}^{(1)}}_{l^p,k^m}(z^2;q^2)$ \\ \cline {3-4}
               & & $1 \le k \le n-1$ & $d^{D_{n+1}^{(2)}}_{l^p,n^m}(z;q) = d^{D_{n+1}^{(1)}}_{l^p,n^m}(-z^2;q^2) $ \\ \cline {3-4}
               & & $k=l=n$ & $ d^{D_{n+1}^{(2)}}_{n^p,n^m}(z;q) =  d^{D_{n+1}^{(1)}}_{n^p,n^m}(-z;q^2)$ \\ \hline
\end{tabular}
\fontsize{11}{11}\selectfont
\end{center}

\bibliographystyle{alpha}
\bibliography{ref}{}
\end{document}